\documentclass[opre,nonblindrev]{informs3}

\DoubleSpacedXI 

\usepackage{amsmath}   
\usepackage{booktabs}  
\usepackage{bbm}       
\usepackage{bm}        
\usepackage[lined,boxed,commentsnumbered, ruled]{algorithm2e}   
\usepackage{indentfirst}  
\usepackage{subcaption}   
\usepackage{mathtools}    
\usepackage{standalone}   
\usepackage[OT1]{fontenc}  
\usepackage[table]{xcolor} 
\usepackage{colortbl}

\usepackage{multibib}
\newcites{ec}{References}

\usepackage[hidelinks,citecolor=blue,urlcolor=blue,pdfauthor=author]{hyperref}  

\usepackage{pgf}
\usepackage{graphics, color, amsfonts, dsfont, graphicx, amssymb, epsfig}
\usepackage[utf8]{inputenc}  
\usepackage{tikz,pgfplots}
\pgfplotsset{compat=1.7}   
\usetikzlibrary{matrix,positioning,decorations.pathreplacing}

\graphicspath{ {Figures/} }

\usepackage{endnotes}
\let\footnote=\endnote
\let\enotesize=\normalsize
\def\notesname{Endnotes}%
\def\makeenmark{\hbox to1.275em{\theenmark.\enskip\hss}}
\def\enoteformat{\rightskip0pt\leftskip0pt\parindent=1.275em
  \leavevmode\llap{\makeenmark}}

\mathchardef\mhyphen="2D 

\DeclareMathOperator{\CVaR}{CVaR}

\newcommand{\ra}[1]{\renewcommand{\arraystretch}{#1}}  
\newcommand{\E}{\mathbb{E}}

\newcommand{\Prob}{\mathbb{P}}

\newcommand{\R}{\mathbb{R}}

\newcommand{\calB}{\mathcal{B}}

\newcommand{\calD}{\mathcal{D}}

\newcommand{\calF}{\mathcal{F}}

\newcommand{\calK}{\mathcal{K}}

\newcommand{\calP}{\mathcal{P}}

\newcommand{\calS}{\mathcal{S}}

\newcommand{\tp}{\top}
\newcommand{\cg}{c^{\textup{\mbox{\tiny g}}}}
\newcommand{\cw}{c^{\textup{\mbox{\tiny w}}}}
\newcommand{\co}{c^{\textup{\mbox{\tiny o}}}}

\newcommand{\hreg}{h^\textup{reg}}
\newcommand{\hcall}{h^\textup{call}}

\newcommand{\dhat}{\widehat{d}}
\newcommand{\dlb}{\underline{d}}
\newcommand{\dub}{\overline{d}}

\newcommand{\lambdaub}{\overline{\lambda}}
\newcommand{\muub}{\overline{\mu}}
\newcommand{\thetaub}{\overline{\theta}}

\newcommand{\rholb}{\underline{\rho}}
\newcommand{\rhoub}{\overline{\rho}}
\newcommand{\rhot}{\widetilde{\rho}}
\newcommand{\blue}{\textcolor{black}}

\usepackage{cases}  
\usepackage{tcolorbox} 
\usepackage{lscape} 
\usepackage{afterpage} 

\newcommand{\st}{\tilde{s}}
\newcommand{\qt}{\tilde{q}}
\newcommand{\ot}{\tilde{o}}
\newcommand{\wt}{\tilde{w}}
\newcommand{\gt}{\tilde{g}}
\newcommand{\pilb}{\underline{\pi}}
\newcommand{\piub}{\overline{\pi}}
\newcommand{\varphilb}{\underline{\varphi}}
\newcommand{\varphiub}{\overline{\varphi}}

\usepackage[shortlabels]{enumitem}

\newcommand{\spe}{SP-E}
\newcommand{\spcvar}{SP-CVaR}
\newcommand{\droe}{DRO-E}
\newcommand{\drocvar}{DRO-CVaR}

\definecolor{lightgray}{gray}{0.9}
\definecolor{orange}{rgb}{1, 0.35, 0.0275} 

\usepackage{natbib}
 \bibpunct[, ]{(}{)}{,}{a}{}{,}%
 \def\bibfont{\small}%
\TheoremsNumberedThrough     
\ECRepeatTheorems

\EquationsNumberedThrough    

\begin{document}


\RUNAUTHOR{Tsang et al.}

\RUNTITLE{Stochastic Optimization Approaches for ORASP}

\TITLE{Stochastic Optimization Approaches for an Operating Room and Anesthesiologist Scheduling Problem}

\ARTICLEAUTHORS{%
\AUTHOR{Man Yiu Tsang, Karmel S. Shehadeh, Frank E. Curtis}
\AFF{Department of Industrial and Systems Engineering, Lehigh University, Bethlehem, PA, USA; \EMAIL{mat420@lehigh.edu}, \EMAIL{kas720@lehigh.edu}, \EMAIL{frank.e.curtis@lehigh.edu}} 

\AUTHOR{Beth Hochman}
\AFF{Divisions of General Surgery \& Critical Care Medicine, Columbia University Medical Center, New York, NY, USA; \EMAIL{brh2106@cumc.columbia.edu}}

\AUTHOR{Tricia E. Brentjens}
\AFF{Department of Anesthesiology, Columbia University Medical Center, New York, NY, USA; \EMAIL{tb164@cumc.columbia.edu}}
} 

\ABSTRACT{%
We propose combined allocation, assignment, sequencing, and scheduling problems under uncertainty involving multiple operation rooms (ORs), anesthesiologists, and surgeries, as well as methodologies for solving such problems.  Specifically, given sets of ORs, regular anesthesiologists, on-call anesthesiologists, and surgeries, our methodologies solve the following decision-making problems simultaneously: (1) an allocation problem that decides which ORs to open and which on-call anesthesiologists to call in, (2) an assignment problem that assigns an OR and an anesthesiologist to each surgery, and (3) a sequencing and scheduling problem that determines the order of surgeries and their scheduled start times in each OR.  To address uncertainty of each surgery's duration, we propose and analyze stochastic programming (SP) and distributionally robust optimization (DRO) models with both risk-neutral and risk-averse objectives.  We obtain near-optimal solutions of our SP models using sample average approximation and propose a computationally efficient column-and-constraint generation method to solve our DRO models.  In addition, we derive symmetry-breaking constraints that improve the models’ solvability.  Using real-world, publicly available surgery data and a case study from a health system in New York, we conduct extensive computational experiments comparing the proposed methodologies empirically and theoretically, demonstrating where significant performance improvements can be gained.  Additionally, we derive several managerial insights relevant to practice.
}%



\maketitle \vspace{-5mm}

%


\setlength{\abovedisplayskip}{0pt}%
\setlength{\belowdisplayskip}{0pt}%
\setlength{\abovedisplayshortskip}{0pt}%
\setlength{\belowdisplayshortskip}{0pt}%

\setlength\floatsep{0.5\baselineskip plus 3pt minus 2pt}
\setlength\textfloatsep{0.5\baselineskip plus 3pt minus 2pt}
 
\section{Introduction} \label{sec:introduction}

\noindent Operating room (OR) planning and scheduling has a significant impact on costs for hospital management and the quality of the health care that a hospital is able to provide.  ORs typically generate 40--70\% of hospital revenues and incur 20--40\% of operating costs \citep{Cardoen_et_al:2010,  Zhu_et_al:2019}.  In addition, it is common for $60$-$70\%$ of patients admitted to a hospital to require surgery \citep{Guerriero_Guido:2011}. As a result, OR planning and scheduling significantly influences overall patient flow, and whether or not they operate efficiently has a large influence on the quality of care that a hospital is able to provide.

On top of their critical nature, OR planning and scheduling problems are extremely complex since they require the coordination of multiple hospital resources, including ORs themselves, anesthesiologists, surgical equipment, and so on.  Their complexity is compounded by the fact that, in addition to limited OR capacity and time, there is an overall shortage in terms of the physicians and anesthesiologists that are required to perform surgeries \citep{De-Simone_et_al:2021,Shanafelt_et_al:2016}.  Consequently, hospital managers could benefit greatly from advanced methodologies to improve OR utilization, surgical care, and quality, as well as to minimize OR operational costs.

Motivated by these important issues and our collaboration with a large health system in New York, we propose new optimization formulations of a scheduling problem in a surgical suite involving multiple parallel ORs, anesthesiologists, and elective surgeries.  Specifically, given sets of ORs, regular anesthesiologists, on-call anesthesiologists, and elective surgeries (each of which requires an OR and an anesthesiologist to be performed), our formulations aim to solve the following decision-making problems simultaneously: (a) an \textit{allocation problem} that determines which ORs to open and which on-call anesthesiologists to call in, (b) an \textit{assignment problem} that assigns an OR and an anesthesiologist to each surgery, and (c) a \textit{sequencing and scheduling} problem that determines surgery order and scheduled start time.  We call this combination an \textit{\underline{o}perating \underline{r}oom and \underline{a}nesthesiologist \underline{s}cheduling \underline{p}roblem} (ORASP).  The objective is to minimize the sum of fixed costs for opening ORs and calling in on-call anesthesiologists along with a weighted average of costs associated with the idling and overtime of anesthesiologists and ORs, as well as the surgery waiting time.

The ORASP is a challenging problem in practice as it requires a significant amount of time for OR managers to make these decisions.  Mathematical formulations of the problem are also challenging to solve for various reasons.  First, it is a complex multi-resource scheduling problem with critical limits in terms of available ORs and anesthesiologists \citep{Liu_et_al:2018, Rath_et_al:2017}.  Some types of surgeries require specialized anesthesiologists, whereas each anesthesiologist might have a different combination of specializations.  This heterogeneity in the set of anesthesiologists increases the complexity of the assignment problem of anesthesiologists to surgeries.  
\begin{figure}[t!]
    \centering
    \includegraphics[scale=0.75]{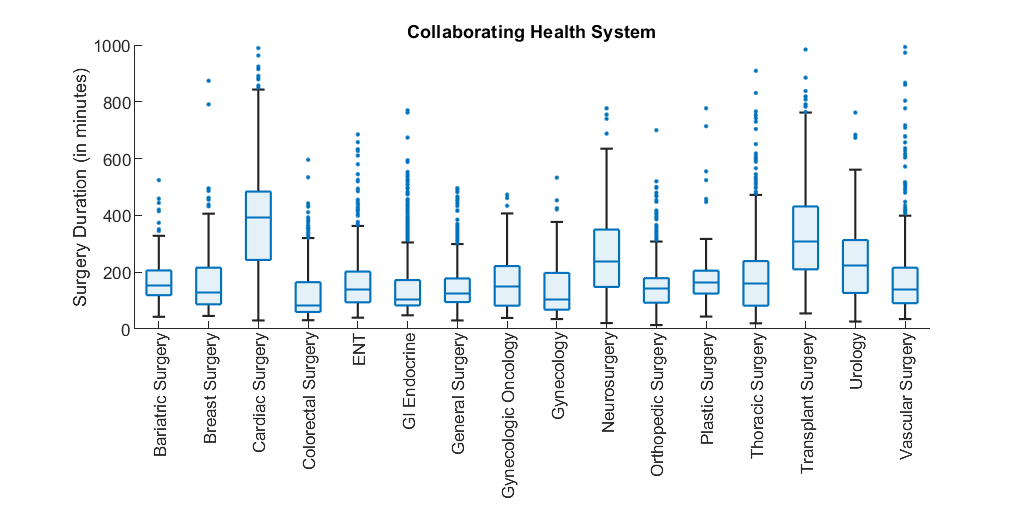}
    \caption{Box plot of surgery duration (in minutes) for different surgery types} 
    \label{fig:duration_boxplot}
\end{figure}
Second, different surgery types have different durations, and even surgery durations of the same type can vary significantly.  To illustrate this variability, we provide Figure \ref{fig:duration_boxplot}, which presents a box plot corresponding to a dataset of surgery durations (in minutes) categorized by surgical specialty.  This data has been provided by our collaborating health system based on half-of-a-year's worth of data.  The figure illustrates clearly that there is significant variability in durations within and across surgery types.  Ignoring such variability in the ORASP may lead to substantial overtime, idling, and/or surgery delays, amongst other schedule deficiencies.  Third, the ORASP is subject to a great deal of symmetry in the solution space, which can lead to computational inefficiencies (see Section~\ref{sec:symm}). 

By building high-quality schedules through solving the ORASP while accounting for the variability in surgery durations, there are substantial opportunities to improve resource utilization (equivalently, reduce overtime and idling time), improve patient and provider satisfaction, reduce delays and costs, and even achieve better surgical care.  In this paper, we propose methodologies for accomplishing these goals by considering two methodologies for handling surgery duration uncertainty: stochastic programming (SP) and distributionally robust optimization (DRO).

SP has been a popular approach for optimization under uncertainty over the past decades \citep{Zhu_et_al:2019, Rahimian_Mehrotra:2022}.  In the SP approach, one essentially needs to assume that decision-makers know the distributions of the durations of each surgery type, or they possess a sufficient amount of high-quality data to estimate these distributions.  Accordingly, one can formulate a two-stage SP model.  The first-stage problem corresponds to determining the allocation, assignment, sequencing, and scheduling decisions while the second-stage problem corresponds to evaluating the performance metrics (i.e., overtime, idle time, and waiting time).

In practice, however, one might not have access to a sufficient amount of high-quality data to estimate surgery duration distributions accurately.  This is especially true when data is limited during the planning stages when the OR schedule is constructed \citep{Shehadeh:2022,Wang_et_al:2019}. As pointed out by \cite{Kuhn_et_al:2019}, even if one employs sophisticated statistical techniques to estimate the probability distribution of uncertain problem parameters using historical data, the estimated distribution may significantly differ from the true distribution. Moreover, future surgery durations do not necessarily follow the same distribution as in the past. Thus, optimal solutions to an SP model that is formulated using an estimated distribution may inherit bias. As such, implementing the (potentially biased) optimal decisions from the SP model may yield disappointing performance in practice, i.e., under unseen data from the true distribution \citep{Smith_Winkler:2006}; in the context of the ORASP, this may correspond to significant overtime, delays, and under-utilization, amongst other negative consequences. While improving estimates of surgery durations may be possible, as pointed out by  \cite{Kayis_et_al:2012} and \cite{Shehadeh_Padman:2021}, the inherent variability in such estimates remains high, necessitating caution in their use when optimizing OR schedules.

One approach to address the above challenges is DRO. In such an approach, one constructs \textit{ambiguity sets} consisting of all distributions that possess certain partial information (e.g., first- and second-order moments) about the surgery durations.  Using these ambiguity sets, one can formulate a DRO problem to minimize the worst-case expectation of the second-stage cost over all distributions residing within the ambiguity set, which effectively means that the probability distribution of the duration of each surgery type is a decision variable \citep{Rahimian_Mehrotra:2022}.  DRO has received substantial attention recently in healthcare applications \citep{Liu_et_al:2019, Shehadeh_et_al:2020,Wang_et_al:2019} and other fields \citep{Huang_et_al:2020, Kang_et_al:2019,Pflug_Pohl:2018,Shang_You:2018} due to its ability to hedge against unfavorable scenarios under incomplete knowledge of the underlying distributions.

\subsection{Contributions}

\noindent In this paper, we propose  the first risk-neutral and risk-averse SP and DRO models for the ORASP, as well as methodologies for solving these models. We summarize our main contributions as follows.
\begin{enumerate}[leftmargin=*]
  \item \textbf{Uncertainty modeling and optimization models.}
  \begin{enumerate}
    \item 
  We propose the first SP and DRO models for the ORASP.  These models consider many of the costs relevant to our collaborating health system: fixed costs related to opening ORs and calling in on-call anesthesiologists, as well as the (random) operational costs associated with OR and anesthesiologist overtime, idle time, and surgery waiting time.  Depending on the risk preference of a decision-maker, these models determine optimal ORASP decisions that minimize the fixed costs plus a risk measure, either expectation or conditional value-at-risk (CVaR), of the operational costs.
  
    \item In the proposed SP model, we minimize the fixed costs plus a risk measure of the operational costs assuming known distributions of the surgery durations. This SP model generalizes recent SP models proposed for multiple OR scheduling problems by incorporating a larger set of important objectives, integrating allocation, assignment, sequencing, and scheduling problems, and modeling decision-makers risk preferences. Moreover, it generalizes that of \cite{Rath_et_al:2017} (a recent SP model for a closely related problem) by incorporating surgery waiting and OR and anesthesiologist idle time decision variables, constraints, and objective terms in the second stage, as well as by considering the decision-maker risk preferences. We show that this generalization can offer more realistic schedules as compared with these models.
   
    \item The proposed DRO model provides an alternative formulation for cases when surgery duration distributions are ambiguous.  The model seeks ORASP decisions that minimize the fixed costs plus a worst-case risk measure (either expectation or CVaR) of the operational costs over all surgery duration distributions defined by mean-support ambiguity sets.  Note that mean and support are two intuitive statistics that capture distribution centrality and dispersion, respectively.  Thus, practitioners could easily adjust the DRO input parameters based on their experience.
     \end{enumerate}

  \item \textbf{Solution Methodologies.} 
\begin{enumerate} 
    \item We derive equivalent solvable reformulations of the proposed mini-max nonlinear expectation and CVaR DRO models, and propose a computationally efficient column-and-constraint generation (C\&CG) method to solve the reformulations.  We also derive valid and efficient lower bound inequalities that efficiently strengthen the master problem in C\&CG, thus improving convergence.
    \item We obtain near-optimal solutions of our SP model using sample average approximation.  We also derive valid lower bounding inequalities to improve the solvability of the SP model. 
    \item We identify structural properties of the ORASP that allow us to decompose the ORASP into smaller problems and thus improve solvability. In addition, we derive new symmetry-breaking constraints, which break symmetry in the solution space of the ORASP’s first-stage decisions and thus improve the solvability of the proposed SP and DRO models.  These constraints are valid for any deterministic or stochastic formulation that employs the first-stage decisions and constraints of the ORASP.
\end{enumerate}
  
  \item \textbf{Computational and Managerial Insights.}  Using real-world, publicly available surgery data and a case study at our collaborating health system, we conduct extensive computational experiments comparing the proposed methodologies empirically and theoretically. Our results show the significance of integrating the allocation, assignment, sequencing, and scheduling problems and the negative consequences associated with (i) adopting existing non-integrated approaches (see Section~\ref{subsec:expt_Rath}) and (ii)  ignoring uncertainty and ambiguity of surgery duration (see Section~\ref{subsec:expt_sol_quality}).  In addition, our results demonstrate the computational efficiency of the proposed methodologies (see Sections~\ref{subsec:expt_comp_time} and \ref{subsec:expt_VI_SBC}) and the potential for impact in practice.

\end{enumerate}

\subsection{Structure of the Paper}

The remainder of the paper is organized as follows. In Section~\ref{sec:literature}, we review relevant literature.  Section \ref{sec:problem_setting} details our problem setting.  In Sections \ref{sec:SP_model} and \ref{sec:DRO_model}, we present and analyze our proposed SP and DRO models for the ORASP, respectively.  In Section \ref{sec:solution_method}, we present our solution strategies of our SP and DRO models, followed by a presentation of our symmetry-breaking constraints in Section~\ref{sec:symm}.  Finally, we present our numerical experiments and corresponding insights in Section~\ref{sec:numerical_experiments}.

\section{Literature Review} \label{sec:literature}

For decades, much work has been done on formulating and solving OR and other healthcare planning and scheduling problems. For comprehensive surveys, we refer to \cite{Ahmadi-Javid_et_al:2017, Cardoen_et_al:2010,shehadeh2022stochastic, Gupta_Denton:2008, Guerriero_Guido:2011, Samudra_et_al:2016, Zhu_et_al:2019}.  In this section, we review recent literature most relevant to our work, namely, studies that propose and analyze stochastic optimization approaches to solve OR planning and scheduling problems.

SP is a useful tool to model uncertainty in surgery duration when distributions are known, and it has been widely applied in the OR planning and scheduling literature \citep{Birge_Louveaux:2011,Zhu_et_al:2019}. \cite{Denton_et_al:2007} proposed the first SP for a single-OR surgery sequencing and scheduling (SAS) problem and heuristic methods to solve it. Recently, \cite{Shehadeh_et_al:2019} proposed a new SP model for the SAS problem that can be solved efficiently.  Their results indicate that remarkable computational improvement can be achieved with their model when compared with those proposed by \cite{Mancilla_Storer:2012} and \cite{Berg_et_al:2014}. \cite{Khaniyev_et_al:2020} discussed the challenges of obtaining exact solutions to the SAS problem under general duration distributions. They proposed an SP model that finds the optimal scheduled times for a given sequence of surgeries that minimize the weighted sum of expected patient waiting times and OR idle time and overtime. They derived an exact alternative reformulation of the objective function that can be evaluated numerically and proposed several scheduling heuristics.

Beyond the single OR setting, \cite{Denton_et_al:2010} introduced an SP model that decides the OR-opening and surgery-to-OR assignments.  An adapted L-shaped algorithm was proposed to solve the model. \cite{Wang_et_al:2014} proposed an SP model that extends that of \cite{Denton_et_al:2010} by considering emergency demand.  They presented several column-generation-based heuristic methods and compared their computational performances. There have also been several studies on parallel and multi-resource OR scheduling problems (e.g., assigning multiple resources, such as surgical staff,  to each surgery). \cite{Batun_et_al:2011} proposed an SP model for an OR and surgeon assignment problem with random incision times. Their numerical results show that OR pooling is beneficial in reducing the operational costs.  Assuming that surgery durations are normally distributed, \cite{Guo_et_al:2014} proposed an SP model for a nurse assignment problem. Unlike in \cite{Batun_et_al:2011}, the surgery-to-OR assignment is assumed to be predetermined.  \cite{Latorre_et_al:2016} generalized the work of \cite{Guo_et_al:2014} by considering an OR scheduling problem with surgeon and other necessary resources (e.g., nurse and anesthesiologist).  To overcome the computational difficulties, they developed a metaheuristic method based on a genetic algorithm.  \cite{Vali-Siar_et_al:2018} investigated an OR scheduling problem that considers the needs for nurses and anesthesiologists. They developed a genetic algorithm to solve their model.

While most existing studies on multiple-OR scheduling problems focus on OR-opening and surgery-to-OR assignment decisions, recent works also consider sequencing and scheduling decisions.  In fact, various empirical and optimization studies have demonstrated the benefits of integrated approaches that incorporate sequencing decisions, including improving OR performance and reducing costs compared to fixed-sequence approaches \citep{Cardoen_et_al:2010, Cayirli_et_al:2006, Denton_et_al:2007}. \cite{Freeman_et_al:2016} proposed the first SP model incorporating these decisions. To deal with the computational challenges associated with solving their model, they proposed a two-step solution approach that reduces the set of surgeries and restricts the maximum number of surgeries in each OR by solving a knapsack problem. \cite{Tsai_et_al:2021} proposed an SP model with chance constraints on overtime and waiting time and developed two approximation algorithms to solve their model. 

SP provides an excellent basis for modeling and solving the ORASP if the distributions of surgery durations are known or one has a sufficient amount of high-quality data to estimate them.  However, high-quality data is often unavailable in most real-world settings, such as for the ORASP.  Accordingly, the distributions are often hard to characterize and subject to ambiguity.  If one solves an SP model with a particular set of training data (i.e., the empirical distribution), the resulting schedule may have disappointing performance (e.g., excessive overtime and waiting time) in practice.  Various studies have shown that decision-makers tend to be averse to ambiguity in distribution \citep{Eliaz_Ortoleva:2016,Halevy:2007}.  In the context of the ORASP, some OR managers may err on the side of caution and prefer robust scheduling decisions that could safeguard the operational performance in adverse scenarios and mitigate the direct and indirect costs of operations (e.g., overtime, surgery delays, and quality of care).

Robust optimization (RO) is an alternative way to model uncertainty when the distributional information is limited. In this approach, one assumes that the random parameters lie in some uncertainty set consisting of possible scenarios and minimizes the worst-case costs over realizations in the uncertainty set.  This could give a more robust solution, potentially reducing surgery waiting time and resource overtime. Examples of RO approaches for OR and surgery scheduling include \cite{Bansal_et_al:2021, Denton_et_al:2010, Addis_et_al:2014, Marques_Captivo:2017, Moosavi_Ebrahimnejad:2020} and references therein. Recently, \cite{Breuer_et_al:2020} proposed an RO model for a combined OR planning and personnel scheduling problem that decides the number of elective surgeries and assigns staff (e.g., nurse, anesthetist, etc.) to surgeries. Unlike our ORASP, their model does not include decisions related to surgery sequences and start times.

Notably, \cite{Rath_et_al:2017} proposed the first and, as far as we are aware, so far the only RO model for an integrated OR and anesthesiologist scheduling problem that is similar to our ORASP.  They employed the uncertainty set of \cite{Bertsimas_Sim:2004} that characterizes surgery duration lower and upper bounds with a tolerance on the maximum number of perturbations with respect to the nominal surgery duration.  They solved their RO model using a decomposition algorithm and discussed the computational challenges of solving large instances.  The model by \cite{Rath_et_al:2017} only considers OR and anesthesiologists fixed and overtime costs in the objective, ignoring surgery waiting times and OR and anesthesiologists idle times.  \blue{This is notable  because, as we show in Section~\ref{subsec:expt_Rath} and \ref{appdx:models_wo_waiting}, by ignoring surgery waiting times, their model could lead to a schedule with multiple surgeries assigned to the same anesthesiologist and/or OR scheduled to start at the same time.}  Moreover, ignoring idle times can lead to poor utilization of the ORs and anesthesiologists.  In this paper, we incorporate both waiting times and idle times, which yield realistic schedules, reduce surgery delay, and improve utilization compared with \cite{Rath_et_al:2017}'s schedules. \blue{Moreover, in \ref{appdx:RO_ORASP}, we demonstrate that an extension of \cite{Rath_et_al:2017}'s RO model which incorporates all elements of the ORSAP is more computationally challenging to solve than our proposed models to the point that it may be considered intractable for real-world settings.}

Focusing on hedging against worst-case scenarios, RO often yields overly conservative decisions \citep{Roos_den-Hertog:2020}. One alternative is DRO, an approach that dates back to \cite{Scarf:1958} and has been of growing interest in recent years  \citep{Delage_Ye:2010, Rahimian_Mehrotra:2022}. Specifically, in DRO, one assumes that the \textit{distribution} of random parameters resides in some ambiguity set, i.e., a family of distributions \citep{Delage_Ye:2010,Goh_Sim:2010,Rahimian_Mehrotra:2022}.  Accordingly, one minimizes the worst-case \textit{expected} behavior over distributions in the ambiguity set.  This reduces conservatism as compared with the RO approach while relaxing the stringent assumption in the SP approach that distributions are known with certainty.  Despite these attractive features, the use of DRO models in the OR scheduling literature has been relatively sparse. 

\blue{\cite{Wang_et_al:2019} derived a DRO model of \cite{Denton_et_al:2010}'s SP model for the simple surgery bock allocation problem, where the ambiguity set captures the support, mean, and mean absolute deviation of surgery durations. The model finds OR opening and surgery-to-OR assignment decisions that minimize OR opening cost plus the worst-case expected OR overtime cost. \cite{Wang_et_al:2019} leveraged the simple structure of their second-stage problem to derive a mixed-integer linear program (MILP) reformulation of their DRO model. In addition, to solve large instances efficiently, they employed the linear decision rule (LDR) technique to derive an MILP approximation of their DRO model and proposed another heuristic approach. We emphasize the following differences between our ORASP model and \cite{Wang_et_al:2019}'s model. First,} the model by \cite{Wang_et_al:2019} does not consider the need to assign both an OR and an anesthesiologist to each surgery and does not consider the surgery sequencing and scheduling decisions that are part of \blue{the planning decisions} in the ORASP. \blue{Hence, the first stage of the ORASP model integrates a larger set of planning decisions (allocation, assignment, sequencing, and scheduling decisions) and involves a more intricate set of constraints. Second, \cite{Wang_et_al:2019}'s second-stage problem includes OR overtime as the only operational metric (second-stage objective). In contrast, we consider a larger set of operational metrics in the second stage of the ORASP (overtime and idle time for ORs and anesthesiologists, and surgery waiting time). Hence, our second-stage formulation is different (and larger in terms of the number of variables and constraints) and more complex. Third, we model decision-makers risk preference while \cite{Wang_et_al:2019} adopts a risk-neutral approach.}

\cite{Deng_et_al:2019} proposed a DRO model with chance constraints on surgery waiting and OR overtime that integrates surgery-to-OR assignment, sequencing, and scheduling decisions in multiple ORs.  \cite{Deng_et_al:2019}'s model cannot be adopted for the ORASP because it does not consider anesthesiologists' scheduling decisions (which on-call anesthesiologists call in, surgery-to-anesthesiologist assignment decisions, order of surgeries assigned to each OR) and the associated operational metrics (anesthesiologist overtime and idle time) which are part of our ORASP. Indeed, the first stage of the ORASP model has a larger and different set of assignment, sequencing, and scheduling decisions and constraints. Our second-stage formulation is also different than that of \cite{Deng_et_al:2019}. \cite{Dean_et_al:2022} proposed a DRO model for the single-OR scheduling problem in \cite{Denton_Gupta:2003} that decides the surgery schedule times for a fixed surgery sequence. The ambiguity set captures quantiles of surgery durations predicted from quantile regression forests. For other recent DRO approaches in healthcare scheduling, see, e.g., \cite{Bansal_et_al:2021a}, \cite{Shehadeh:2022}, \cite{Keyvanshokooh_et_al:2020}, and the references therein.

Several studies have proposed approaches for physician and medical professional scheduling problems but did not integrate OR and anesthesiologist scheduling decisions.  We discuss recent studies on anesthesiologist scheduling and refer to \cite{Abdalkareem_et_al:2021} and \cite{Erhard_et_al:2018} for recent surveys on the state of the art in general healthcare and physician scheduling problems. From an operational perspective, recent advances focus on developing implementable decision-support tools that automate the anesthesiologist scheduling process to replace traditional manual scheduling  \citep{Hoefnagel_et_al:2020, Joseph_et_al:2020}. However, such tools and the underlying models do not consider the ORASP decisions. Other works include empirical studies investigating different anesthesiologist scheduling paradigms (e.g., \citealp{Tsai_et_al:2017, Tsai_et_al:2020}). These studies also do not incorporate the ORASP decisions. Finally, on the optimization end, \cite{Rath_Rajaram:2022} proposed an anesthesiologist scheduling model that decides the number of anesthesiologists that are on regular duty and on call by minimizing the explicit costs (e.g., hiring cost) and implicit costs (e.g., idle cost). These studies do not consider optimizing the ORASP decisions. 

\blue{Finally, it is worth mentioning that various studies have motivated the need for modeling decision-makers risk preferences. In this paper, we adopt CVaR for modeling  OR managers' risk aversion. CVaR is one of the popular risk measures widely adopted in the stochastic optimization literature to model decision-makers risk aversion (see \citealp{Filippi_et_al:2020} for a recent review). In particular, CVaR has been used in various healthcare applications to model decision-makers risk aversion (see, e.g., \citealp{He_et_al:2019, Kishimoto_Yamashita:2018, Lim_et_al:2020, Linz_et_al:2019, Najjarbashi_Lim:2019}).  In the context of the ORASP, incorporating CVaR as a risk measure in the objective reflects the OR manager's risk-averse mindset and desire to err on the side of caution when making surgery planning decisions. This is because CVaR focuses on the tail of the operational cost distribution. Thus, minimizing the CVaR objective could mitigate large values of the operational costs. Our computational results demonstrate significant differences in the optimal planning decisions and operational performances when using the expectation and CVaR objectives.}



\section{Problem Setting} \label{sec:problem_setting}

We start by introducing our ORASP setting.  For a given day, we suppose that there is a set $I$ of elective surgeries to schedule, a set $R$ of available operating rooms (ORs), and a set $A$ of anesthesiologists. Each OR has a pre-allocated length of time $T^\text{end}$ with service hours $[0,T^\text{end}]$. \blue{Moreover, each OR can be dedicated to one or multiple types of surgical specialty (e.g., cardiothoracic, neurosurgery, etc.). Many health systems implement a dedicated OR policy, including our collaborating health system, to better manage elective surgeries. Hence, this policy has been widely adopted in the literature (see, e.g., \citealp{Aringhieri_et_al:2015, Bovim_et_al:2020, Fugener_et_al:2014, Makboul_et_al:2022, Marques_Captivo:2015, Min_Yih:2010, Neyshabouri_Berg:2017, Shehadeh:2022}). Our proposed models can be used to solve ORASP instances with ORs dedicated to one or many surgery types. }

In practice, there are two types of anesthesiologists: regular and on-call \citep{Becker_et_al:2019, Rath_Rajaram:2022, Rath_et_al:2017}. Each of the former type of anesthesiologist is scheduled to work on the given day, whereas each of the latter type is effectively on standby, ready to be called to work, if necessary. Assigning an on-call anesthesiologist to a surgery produces a high cost in some hospitals.  We use the parameter setting $\hreg_a=1$ to indicate that anesthesiologist $a \in A$ is on regular duty ($\hreg_a=0$ otherwise) and the setting $\hcall_a=1$ to indicate that this anesthesiologist is on call ($\hcall_a=0$ otherwise).  Each regular-duty anesthesiologist has a preassigned work shift [$t^\text{start}_a$, $t^\text{end}_a$], where overtime occurs if/when they work beyond the scheduled end of their shift.  In practice (e.g., at our collaborating hospital), some anesthesiologists are dedicated to cover a specific specialty, whereas some can cover a wide range of specialties.  We refer to \ref{appdx:anes_specialty} for an example.


\blue{Each surgery $i\in I$ has a type (e.g., cardiothoracic, breast, etc.), and it can be assigned to any OR that can accommodate surgeries of that type.} Similarly, the assignment of an anesthesiologist to a surgery must respect the specialty required for the surgery. We assume that the surgery-surgeon combination is already known to mimic the current practice in many hospitals. (This is also a common assumption in the literature; see, e.g., \citealp{Doulabi_et_al:2014, Marques_et_al:2014, Rath_et_al:2017}). Thus, one can think of each $i \in I$ as a surgery-surgeon unit. However, this assumption does not prevent surgeons from working in any of the ORs dedicated to their specialty. Surgery durations are random and depend on the surgery type.  We use $d_i$ to denote the duration of surgery $i$ and let $d:=[d_1,\ldots,d_I]^\top$ be the vector of all of the surgery durations.  We assume that a lower bound $\dlb_i$ and upper bound $\dub_i$ of surgery duration $d_i$ are known, which is a realistic assumption recommended by our collaborators and commonly used in healthcare scheduling \citep{Denton_et_al:2010, Shehadeh_Padman:2021, Wang_et_al:2019}.  Mathematically, the random surgery duration $D$ is a measurable function $D:\Omega\rightarrow\calS$ with measurable space $(\Omega,\calF)$, where $\calS\subseteq\R^{I}$ is the bounded support defined as $\calS=\{d\in\R^I \mid \dlb_i \leq d_i \leq \dub_i \text{ for all } i \in I\}$. We use $d$ to denote a realization of $D$.

Given $I$, $R$, and $A$ for each day, our ORASP models solve the following decision problems simultaneously: (a) an allocation problem in which we decide which OR to open, (b) an assignment problem assigning each surgery to an OR and anesthesiologist, and (c) a sequencing and scheduling problem that determines surgery order and scheduled start time.  The objective is to minimize the sum of ORs and anesthesiologists fixed costs and a weighted average of the idling and overtime of anesthesiologists and ORs, and the surgery waiting time.  For notational convenience, we define the following sets to be used in our formulations.  The sets $\calF^A$ and $\calF^R$ consist of all feasible surgery-OR and surgery-anesthesiologist assignments.  The sets $A_i$ and $R_i$ are, respectively, the sets of anesthesiologists and ORs to which a surgery can be assigned for $i\in I$.  The sets $I_a$ and $I_r$ are surgeries that could be performed by anesthesiologist $a \in A$ and in OR $r \in R$, respectively. Mathematically, we let $\kappa^A_{i,a}=1$ indicate that anesthesiologist $a$ can cover surgery $i$, and $\kappa^R_{i,r}=1$ indicate that surgery $i$ can be scheduled in OR $r$. Then, we define $\calF^A=\big\{(i,a)\in I\times A \mid \kappa^A_{i,a}=1\big\}$, $\calF^R=\big\{(i,r)\in I\times R\mid \kappa^R_{i,r}=1\big\}$, $A_i=\big\{a\in A\mid (i,a)\in\calF^A\big\}$, $R_i=\big\{r\in R\mid (i,r)\in\calF^R\big\}$, $I_a =\big\{i\in I\mid  (i,a)\in\calF^A\big\}$, and $I_r =\big\{i\in I\mid (i,r)\in\calF^R\big\}$.  A complete list of our notation can be found in \ref{appdx:notation}.

\section{Stochastic Programming Models} \label{sec:SP_model}

In this section, we present our proposed two-stage SP formulation of the ORASP, which assumes that the probability distributions of surgery durations are known.  First, let us introduce the variables, parameters, and functions defining our first-stage SP model.  For each $r \in R$, we define a binary decision variable $v_r$ that equals 1 if OR $r$ is opened, and is $0$ otherwise.  Similarly, for each $a \in A$, we define a binary variable $y_a$ that equals 1 if on-call anesthesiologist $a$ is called in, and is 0 otherwise. We define binary decision variables $x_{i,a}$ and $z_{i,r}$ taking value 1 if surgery $i$ is assigned to anesthesiologist $a$ and OR $r$ respectively, and are 0 otherwise. To determine the surgery sequence, we proceed as in \cite{Rath_et_al:2017} and define binary variables $u_{i,i'}$, $\alpha_{i,i',a}$, and $\beta_{i,i',r}$ to represent precedence relationships. Specifically, we define $u_{i,i'}$ that takes value 1 if surgery $i$ precedes surgery $i'$, and is 0 otherwise. Variables $\alpha_{i,i',a}$ and $\beta_{i,i',r}$ take value 1 if surgery $i$ precedes surgery $i'$ for anesthesiologist $a$ and in OR $r$ respectively, and are 0 otherwise. For each $i \in I$, we let nonnegative continuous variable $s_i$ represent the scheduled start time of surgery~$i$. 

For the objective function, we define $f_r$ as the nonnegative fixed cost of opening OR $r$ and~$f_a$ as the nonnegative fixed cost of calling in on-call anesthesiologist $a$.  The remaining term in the objective function is a risk measure of the second-stage function (see more below), which, for a given realization $d$ of surgery durations represented by the random variable $D$, is a weighted average of idle time, overtime, and waiting time. Our first-stage SP model can now be stated as follows:
\allowdisplaybreaks
\begin{subequations} 
\begin{align}
 \underset{x,\,y,\,z,\,v,\,u,\,s,\,\alpha,\,\beta}{\text{minimize}} \quad
&  \sum_{r\in R}f_r v_r + \sum_{a\in A}f_a y_a + \varrho_{\Prob}\big(Q(x,y,z,v,u,s,D)\big) \label{eqn:1st_stage_obj} \\
 \text{subject to\,\,\,} \quad 
&  \sum_{a\in A_i}x_{i,a} = 1,\quad \sum_{r\in R_i}z_{i,r} = 1,\quad\forall i\in I, \label{eqn:1st_stage_con1-2} \\
&  z_{i,r} \leq v_r,\quad\forall (i,r)\in\calF^R, \label{eqn:1st_stage_con3}\\
&  x_{i,a} \leq \hreg_a + y_a,\quad\forall (i,a)\in\calF^A, \label{eqn:1st_stage_con4}\\
&  y_a \leq \hcall_a,\quad\forall a\in A, \label{eqn:1st_stage_con5}\\
&  s_i \geq t^\text{start}_a - M(1-x_{i,a}),\quad s_i \leq T^\text{end}, \quad\forall (i,a)\in\calF^A, \label{eqn:1st_stage_con6-7}\\
&  \alpha_{i,i',a}\leq u_{i,i'},\quad\forall \{(i,a),(i',a)\}\subseteq\calF^A, \label{eqn:1st_stage_con8}\\
& \beta_{i,i',r} \leq u_{i,i'}, \quad\forall \{(i,r),(i',r)\}\subseteq\calF^R, \label{eqn:1st_stage_con9}\\
& u_{i,i'}+u_{i',i} \leq 1,\quad\forall \{i,i'\}\subseteq I, \label{eqn:1st_stage_con10}\\
& u_{i,i''} \geq u_{i,i'} + u_{i',i''} -1,\quad\forall \{i,i',i''\}\subseteq I, \label{eqn:1st_stage_con11}\\
& \alpha_{i,i',a}+\alpha_{i',i,a} \leq x_{i,a},\quad \alpha_{i,i',a}+\alpha_{i',i,a} \leq x_{i',a},\quad\forall \{(i,a),(i',a)\}\subseteq\calF^A, \label{eqn:1st_stage_con12-13} \\
& \alpha_{i,i',a}+\alpha_{i',i,a} \geq x_{i,a} + x_{i',a}-1 ,\quad\forall \{(i,a),(i',a)\}\subseteq\calF^A, \label{eqn:1st_stage_con14}\\
& \beta_{i,i',r}+\beta_{i',i,r} \leq z_{i,r} ,\quad \beta_{i,i',r}+\beta_{i',i,r} \leq z_{i',r}, \quad\forall \{(i,r),(i',r)\}\subseteq\calF^R, \label{eqn:1st_stage_con15-16}\\
& \beta_{i,i',r}+\beta_{i',i,r} \geq z_{i,r} + z_{i',r} - 1,\quad\forall \{(i,r),(i',r)\}\subseteq\calF^R, \label{eqn:1st_stage_con17}\\
&  \alpha_{i,i',a} \geq x_{i,a} + x_{i',a} + \beta_{i,i',r}-2,\quad\forall \{(i,a),(i',a)\}\subseteq\calF^A, \{(i,r),(i',r)\}\subseteq\calF^R,  \label{eqn:1st_stage_con18} \\
&  \beta_{i,i',r} \geq z_{i,r} + z_{i',r} + \alpha_{i,i',a}-2,\quad\forall \{(i,a),(i',a)\}\subseteq\calF^A, \{(i,r),(i',r)\}\subseteq\calF^R, \label{eqn:1st_stage_con19} \\
& x_{i,a},\, y_a,\, z_{i,r},\, u_{i,i'},\, v_r,\, \alpha_{i,i',a},\,\beta_{i,i',r}\in\{0,1\},\quad s_i \geq 0, \quad\forall  i\in I,\, a\in A,\, r\in R. \label{eqn:1st_stage_con20-21} 
\end{align} \label{eqn:1st_stage}%
\end{subequations}  \vspace{-8mm} %

The objective \eqref{eqn:1st_stage_obj} aims to find first-stage decisions $(x,y,z,v,u,s,\alpha,\beta)$ that minimize the sum of the fixed cost of opening ORs (first-term), the fixed cost of employing on-call anesthesiologists (second term), and the risk measure $\varrho_{\Prob}$ of the random second stage function $Q$ (third term). A risk-neutral decision-maker may opt to set $\varrho_{\Prob}(\cdot)=\E_{\Prob}(\cdot)$, i.e., the expected total operational costs, which is standard in the OR scheduling literature and intuitive for OR managers. In contrast, a risk-averse decision-maker might set $\varrho_{\Prob}(\cdot)=\Prob\mhyphen\CVaR_\gamma(\cdot)$, i.e., the CVaR of the total operational costs. For simplicity, we let \spe{} and \spcvar{} denote the risk-neutral and risk-averse SP models.

Constraints \eqref{eqn:1st_stage_con1-2} ensure that every surgery is assigned to exactly one anesthesiologist and one OR. Constraints \eqref{eqn:1st_stage_con3} ensure that surgeries are assigned to open ORs. Constraints \eqref{eqn:1st_stage_con4} indicate that an anesthesiologist can be assigned to surgeries if they are on regular duty (i.e., $\hreg_a=1$) or on call (i.e., $y_a=1$). Constraints \eqref{eqn:1st_stage_con5} ensure that $y_a$ may equal 1 if anesthesiologist $a$ is listed as an on-call anesthesiologist (i.e., $\hcall_a=1$). Constraints \eqref{eqn:1st_stage_con6-7} enforce that the scheduled start time of the surgery assigned to anesthesiologist $a$ is greater than or equal to his/her scheduled start time $t^\text{start}_a$ and it is scheduled within the planned service hours [0, $T^\text{end}$].  Constraints \eqref{eqn:1st_stage_con8}--\eqref{eqn:1st_stage_con11} define precedence variables $u_{i,i'}$.  Constraints \eqref{eqn:1st_stage_con8}--\eqref{eqn:1st_stage_con9} ensure that if surgery $i'$ follows surgery $i$ in either an anesthesiologist or an OR schedule, then $u_{i,i'}$ equals $1$. Constraints \eqref{eqn:1st_stage_con10}--\eqref{eqn:1st_stage_con11} maintain the precedence and transitivity relationships that prevent non-implementable schedules (see an example in \ref{appdx:eg_sequencing}). If surgeries $i$ and $i'$ are assigned to anesthesiologist $a$ (i.e., $x_{i,a}=x_{i',a}$), then constraints \eqref{eqn:1st_stage_con12-13} ensure that either $i'$ follows $i$ (i.e., $\alpha_{i,i',a}=1$) or vice versa (i.e., $\alpha_{i',i,a}=1$), but not both. Constraints \eqref{eqn:1st_stage_con14} ensure that the sequencing constraints on $\alpha$ only apply to surgeries assigned to the same anesthesiologist.  Moreover, they enforce that either $\alpha_{i,i',a}$ or $\alpha_{i',i,a}$ takes value one if both surgeries $i$ and $i'$ are assigned to anesthesiologist $a$. Constraints \eqref{eqn:1st_stage_con15-16}--\eqref{eqn:1st_stage_con17} enforce similar precedence and sequencing rules on surgeries assigned to the same OR.  Constraints \eqref{eqn:1st_stage_con18} enforce $\alpha_{i,i',a}=1$ if $i$ and $i'$ are performed by the same anesthesiologist while surgery $i'$ follows surgery $i$ in the same OR, and similarly for constraints \eqref{eqn:1st_stage_con19} with the role of anesthesiologist and OR swapped. Finally, constraints \eqref{eqn:1st_stage_con20-21} specify feasible ranges of the first-stage decision variables.

\begin{remark}
   Our proposed model allows practitioners to accommodate special scheduling requests. For example, suppose that anesthesiologist $a$ must perform a given set of surgeries $I'$. In that case, one can set $x_{i,a}=1$ for all $i \in I'$. If surgery $i$ must be performed in a particular OR, one can set $z_{i,r}=1$. Similarly, one can set $v_r=1$ if a specific OR $r \in R$ must be open and $y_a=1$ if anesthesiologist $a$ must be called in. These are special cases and simplifications of our model.
\end{remark}

Next, we introduce our second-stage (recourse) problem.  For a given set of first-stage decisions $(x,y,z,v,u,s)$ corresponding to a feasible solution of \eqref{eqn:1st_stage} and a realization of surgery durations~$d$, the following second-stage linear program (LP) computes costs related to anesthesiologist and OR idle time and overtime (i.e., the first and second terms in \eqref{eqn:2nd_stage_obj} respectively) and waiting time of surgeries (i.e., third term in \eqref{eqn:2nd_stage_obj}).

In this problem, variable $q_i$ represents the actual start time of surgery~$i$ and $w_i$ represents the waiting time of surgery $i$.  We define the nonnegative continuous variables $o_r$ $(o_a)$ and $g_r$ $(g_a)$ respectively to represent the overtime and idle time of OR $r$ (anesthesiologist $a$).  We define $\co_r$ $(\co_a)$ as the per-unit overtime penalty for OR $r$ (anesthesiologist $a$), $\cg_r$ $(\cg_a)$ as the per-unit idling penalty for OR $r$ (anesthesiologists $a$), and $\cw_i$ as the per-unit surgery waiting penalty.  Finally, $M_\text{seq}$, $M_\text{anes}$ and $M_\text{room}$ are big-$M$ parameters (see \ref{appdx:big_M} for a discussion on these parameters).  For a given realization of surgery duration $d$, our second-stage problem is as follows:
\allowdisplaybreaks
\begin{subequations}
\begin{align}
Q(x,y,z,v,u,s,d):= \underset{q,\,o,\,w,\,g}{\text{minimize}} \quad
&  \sum_{a\in A} \Big(\cg_a g_a + \co_a o_a \Big) +\sum_{r\in R} \Big(\cg_r g_r + \co_r o_r \Big) + \sum_{i\in I} \cw_i w_i\label{eqn:2nd_stage_obj} \\
 \text{subject to} \quad
&  q_{i'} \geq q_i + d_i - M_\text{seq}(1-u_{i,i'}),\quad\forall \{i,i'\}\subseteq I,\, i\ne i', \label{eqn:2nd_stage_con1}\\
&  q_i \geq s_i,\quad\forall i\in I, \label{eqn:2nd_stage_con2} \\
&  o_a \geq q_i + d_i -t^\text{end}_a - M_\text{anes}(1-x_{i,a}+y_a),\quad\forall (i,a)\in\calF^A, \label{eqn:2nd_stage_con3} \\
&  o_r \geq q_i + d_i - T^\text{end} - M_\text{room}(1-z_{i,r}),\quad\forall (i,r)\in\calF^R, \label{eqn:2nd_stage_con4}\\
& w_i \geq q_i - s_i, \quad\forall i\in I, \label{eqn:2nd_stage_con5}\\
& g_a \geq \bigg(t^\text{end}_a - t^\text{start}_a - \sum_{i\in I_a} d_i x_{i,a}\bigg) \hreg_a + o_a,\quad\forall a\in A,\label{eqn:2nd_stage_con6}\\
& g_r \geq T^\text{end} v_r - \sum_{i\in I_r} d_i z_{i,r} + o_r,\quad\forall r\in R,\label{eqn:2nd_stage_con7}\\ 
&  q_i,\, o_a,\, o_r,\, w_i,\, g_a,\, g_r \geq 0,\quad\forall i\in I,\, a\in A,\, r\in R. \label{eqn:2nd_stage_con8} 
\end{align}\label{eqn:2nd_stage}%
\end{subequations}
Constraints \eqref{eqn:2nd_stage_con1}--\eqref{eqn:2nd_stage_con2} ensure that the actual start time of a surgery is not earlier than the scheduled start time and the completion time of the previous surgeries. Constraints \eqref{eqn:2nd_stage_con3} and \eqref{eqn:2nd_stage_con4} yield the overtimes of the anesthesiologists and ORs, respectively. Note that constraints \eqref{eqn:2nd_stage_con3}--\eqref{eqn:2nd_stage_con4} are relaxed if an on-call anesthesiologist is hired or an OR is not open (i.e., overtime is zero in these two cases). Constraints \eqref{eqn:2nd_stage_con5} give the waiting time of each surgery $i \in I$ as the time from the scheduled start time of a surgery to its actual start time. Constraints \eqref{eqn:2nd_stage_con6}--\eqref{eqn:2nd_stage_con7} give the idle times of anesthesiologists and ORs, respectively. Note that the idling cost of on-call anesthesiologists and ORs that are not open are zero. It is easy to verify that formulation \eqref{eqn:2nd_stage} is feasible for any feasible first-stage decisions. Thus, we have a relatively complete recourse.

Our model \eqref{eqn:1st_stage}--\eqref{eqn:2nd_stage} generalizes recent SP models proposed for multiple-OR scheduling problems. For example, the models in \cite{Denton_et_al:2010} and \cite{Wang_et_al:2019} aim to decide the optimal number of ORs to open and surgery assignments to open ORs by minimizing the weighted sum of OR-opening and overtime-penalty costs. Our SP model generalizes these models by (a) incorporating constraints, variables, and objectives related to regular and on-call anesthesiologist scheduling; (b) incorporating a larger set of important objectives; (c) integrating allocation, assignment, sequencing, and scheduling problems; and (d) modeling decision-makers risk preferences.

Furthermore, our SP model generalizes an existing model for a closely related OR-anesthesiologist scheduling problem presented by  \cite{Rath_et_al:2017}. Key differences between our model \eqref{eqn:1st_stage}--\eqref{eqn:2nd_stage} and  \cite{Rath_et_al:2017}'s model include the following.  First, while we use similar sets of first-stage variables and constraints, we add constraints \eqref{eqn:1st_stage_con6-7} to the first-stage model to restrict the scheduled surgery time to be within the planning horizon $[0,T^\text{end}]$, which is common in practice.  Second, our second-stage model \eqref{eqn:2nd_stage} generalizes that of  \cite{Rath_et_al:2017} by considering waiting and idling metrics and the related variables and constraints.  \blue{In particular, incorporating surgery waiting time is essential to minimize delays and avoid scheduling many surgeries to be performed by the same anesthesiologist or in the same OR at the same time. For example, in \ref{appdx:models_wo_waiting}, we analyze the ORASP models without waiting time components and prove that for models such as \cite{Rath_et_al:2017}, it is optimal to schedule surgeries assigned to the same anesthesiologist to start simultaneously at the start time of that anesthesiologist, which is not possible in practice (see also the numerical results in Section \ref{subsec:expt_Rath}).} In addition, we incorporate OR and anesthesiologist idle time in the second-stage objective, which is essential to improve the utilization of these expensive resources \citep{Cardoen_et_al:2010}.  Moreover, different from \cite{Rath_et_al:2017}, we propose a CVaR model that allows for decision-makers risk aversion. Finally, we derive a DRO counterpart of our SP model to address distributional ambiguity, which is the topic of the next section. In Section \ref{subsec:expt_Rath}, we provide examples, results, and detailed discussions demonstrating the importance of incorporating these elements to produce realistic and implementable solutions with superior performance in practice.

\section{Distributionally Robust Models} \label{sec:DRO_model}

In this section, we present our proposed DRO formulation of the ORASP, which does not assume that the probability distributions of surgery durations are known.  That said, we assume that the mean $m:=(m_1, \ldots,m_{|I|})^\top$ and support $\calS=\{d\in\R^I\mid\dlb_i \leq d_i \leq \dub_i,\, i\in I\}$ of surgery durations are known.  These parameters can be estimated based on clinical expert knowledge. Moreover, when data on patient characteristics and medical history are available, one could build statistical and machine-learning models (e.g., regression models) to estimate the mean and support.

We first introduce additional sets and notation defining our ambiguity set. Let  $\calD=\calD(\calS)$ be the set of all probability measures on $(\calS,\calB)$ where $\calB$ is the Borel $\sigma$-field on $\calS$. Elements in $\calD$ can be viewed as probability measures induced by the random vector~$D$. Using this notation, we construct the following mean-support ambiguity set:
\begin{equation} \label{eqn:mean_support_uncertainty_set}
\calP(m,\calS)=\Big\{\Prob\in\calD(\calS) \,\Big|\, \E_\Prob(D) = m \Big\}.
\end{equation}
Using the ambiguity set \eqref{eqn:mean_support_uncertainty_set}, we formulate our DRO model of the ORASP as
\begin{subequations}
\begin{align}
 \underset{x,\,y,\,z,\,v,\,u,\,\alpha,\,\beta,\,s}{\text{minimize}} & \quad   \sum_{r\in R}f_r v_r + \sum_{a\in A}f_a y_a + \bigg\{\sup_{\Prob\in\calP(m,\calS)} \varrho_\Prob\big(Q(x,y,z,v,u,s,D)\big) \bigg\}   \label{eqn:mean_support_obj} \\
 \text{subject to\,\,\,} &   \quad  \text{\eqref{eqn:1st_stage_con1-2}--\eqref{eqn:1st_stage_con20-21}}. \label{eqn:mean_support_con} 
\end{align}  \label{eqn:mean_support_model}%
\end{subequations}
Formulation \eqref{eqn:mean_support_model} finds first-stage decisions $(x,y,z,v,u,s)$ that minimize the first-stage cost and the worst-case of a risk measure of the second-stage cost over distributions residing in $\calP(m,\calS)$. In what follows, we refer to model \eqref{eqn:mean_support_model} with $\rho_{\Prob}(\cdot)=\E_{\Prob}(\cdot)$ as the \droe{} model, and to model \eqref{eqn:mean_support_model} with $\rho_{\Prob} (\cdot)=\Prob\mhyphen\CVaR_{\gamma}(\cdot)$ as the \drocvar{} model. Note that formulation \eqref{eqn:mean_support_model} is a mini-max problem, which is not straightforward to solve in its presented form.  Therefore, our goal is to derive an equivalent solvable formulation of \eqref{eqn:mean_support_model}.  For brevity, we relegate detailed proofs to \ref{appdx:DRO_model_section}.

\subsection{\droe{} Model Reformulation} \label{subsec:DRO_reformulation}

In this section, we derive an equivalent reformulation of the \droe{} model (i.e., model \eqref{eqn:mean_support_model} with $\rho_{\Prob}(\cdot)=\E_{\Prob}(\cdot)$). First, in Proposition~\ref{prop:mean_support_WC_exp}, we present an equivalent reformulation of the inner maximization problem in~\eqref{eqn:mean_support_model}. 

\begin{proposition} \label{prop:mean_support_WC_exp}
  For $(x,y,z,v,u,s)$ satisfying \eqref{eqn:1st_stage_con1-2}--\eqref{eqn:1st_stage_con20-21}, the inner problem in~\eqref{eqn:mean_support_model}, namely, to solve $\sup \limits_{\Prob\in\calP(m,\calS)} \E_\Prob[Q(x,y,z,v,u,s,D)]$, is equivalent to 
\begin{equation}\label{eqn:mean_support_WC_exp_obj}
 \underset{\rho\in\R^I}{\textup{minimize}} \quad \left\{ \sum_{i\in I}\rho_i m_i + \sup_{d\in\calS} \left( Q(x,y,z,v,u,s,d) - \sum_{i\in I} \rho_i d_i \right) \right\}.
\end{equation}
\end{proposition}

Again, the problem in \eqref{eqn:mean_support_WC_exp_obj} involves an inner max-min problem that is not straightforward to solve in its presented form.  However, in Proposition~\ref{prop:mean_support_inner_max_MILP}, we present an equivalent MILP formulation of the inner problem in \eqref{eqn:mean_support_WC_exp_obj} that is solvable. 

\begin{proposition} \label{prop:mean_support_inner_max_MILP}
  Let $\Delta d_i = \dub_i - \dlb_i$ for all $i \in I$.  Then, for $(x,y,z,v,u,s)$ satisfying $\eqref{eqn:1st_stage_con1-2}$--$\eqref{eqn:1st_stage_con20-21}$, there exist $\muub_{i,a}$ for all $(i,a) \in\calF^A$, $\thetaub_{i,r}$ for all $(i,r) \in \calF^R$, and $\lambdaub_{i,i'}$ for all $\{i,i'\}\subseteq I$ such that solving the inner problem in~\eqref{eqn:mean_support_WC_exp_obj}, namely, solving $\sup_{d\in\calS} \left( Q(x,y,z,v,u,s,d) - \sum_{i\in I} \rho_i d_i \right)$, is equivalent to evaluating the following function, which can be done by solving the presented MILP:
\begin{subequations} 
\begin{align}
& \hspace{-20mm}  H(x,y,z,v,u,s,\rho)=\\
\underset{\lambda,\,\mu,\,\theta,\,\zeta,\,b}{\textup{maximize}} \quad
&  \Bigg\{ \sum_{a\in A} \cg_a(t^\textup{end}_a-t^\textup{start}_a)\hreg_a + \sum_{r\in R}\cg_r T^\textup{end} v_r \nonumber \\
& \quad + \sum_{i\in I}\Bigg[\sum_{i'\in I, i'\ne i}(\lambda_{i,i'} - \lambda_{i',i}) + \sum_{a\in A_i} \mu_{i,a} + \sum_{r\in R_i} \theta_{i,r} \Bigg] s_i \nonumber \\ 
& \quad - M_\textup{seq} \sum_{i\in I}\sum_{i'\in I, i'\ne i} \lambda_{i,i'}(1-u_{i,i'}) - \sum_{i\in I}\sum_{a\in A_i}\mu_{i,a}\Big[ t^\textup{end}_a + M_\textup{anes}(1-x_{i,a}+ y_a) \Big] \nonumber \\
& \quad - \sum_{i\in I}\sum_{r\in R_i} \theta_{i,r} \Big[T^\textup{end} +M_\textup{room}(1-z_{i,r})\Big] \nonumber \\
& \quad + \sum_{i\in I} \dlb_i\Bigg[ \sum_{i'\in I,i'\ne i} \lambda_{i,i'} + \sum_{a\in A_i} (\mu_{i,a}- \cg_a \hreg_a x_{i,a}) + \sum_{r\in R_i}(\theta_{i,r}-\cg_r z_{i,r}) - \rho_i\Bigg]  \nonumber \\
& \quad + \sum_{i\in I} \Delta d_i\Bigg[ \sum_{i'\in I,i'\ne i} \zeta^L_{i,i'} + \sum_{a\in A_i} (\zeta^M_{i,a}- \cg_a \hreg_a x_{i,a} b_i) + \sum_{r\in R_i}(\zeta^T_{i,r}-\cg_r z_{i,r} b_i) - \rho_i b_i\Bigg] \Bigg\} \label{eqn:mean_support_inner_max_MILP_obj} \\
 \textup{subject to} \quad
&  \sum_{i\in I_a} \mu_{i,a} \leq \cg_a + \co_a,\quad \sum_{i\in I_r} \theta_{i,r} \leq \cg_r + \co_r, \quad\forall a\in A,\, r\in R,\label{eqn:mean_support_inner_max_MILP_con1} \\
&  \sum_{i'\in I, i'\ne i}(\lambda_{i,i'} - \lambda_{i',i}) + \sum_{a\in A_i} \mu_{i,a} + \sum_{r\in R_i} \theta_{i,r} + \cw_i \geq 0,\quad\forall i\in I, \label{eqn:mean_support_inner_max_MILP_con3} \\
&\zeta^L_{i,i'} \leq \lambda_{i,i'},\,\,\, \zeta^L_{i,i'} \leq b_i\lambdaub_{i,i'},\,\,\, \zeta^L_{i,i'} \geq 0,\,\,\, \zeta^L_{i,i'} \geq \lambda_{i,i'} + \lambdaub_{i,i'}(b_i-1), \,\,\,\,\,\forall \{i,i'\}\subseteq I,  \label{eqn:mean_support_inner_max_MILP_con4}\\
& \zeta^M_{i,a} \leq \mu_{i,a},\,\,\, \zeta^M_{i,a} \leq b_i\muub_{i,a},\,\,\, \zeta^M_{i,a} \geq 0,\,\,\,\zeta^M_{i,a} \geq \mu_{i,a} + \muub_{i,a}(b_i-1),\,\,\,\,\,\forall (i,a)\in \calF^A,  \label{eqn:mean_support_inner_max_MILP_con5} \\
& \zeta^T_{i,r} \leq \theta_{i,r},\,\,\, \zeta^T_{i,r}\leq b_i\thetaub_{i,r},\,\,\, \zeta^T_{i,r} \geq 0,\,\,\, \zeta^T_{i,r} \geq \theta_{i,r} + \thetaub_{i,r}(b_i-1),\,\,\,\,\, \forall(i,r)\in\calF^R, \label{eqn:mean_support_inner_max_MILP_con6} \\
&  \lambda_{i,i'},\,\mu_{i,a},\,\theta_{i,r}\geq 0,\quad b_i\in\{0,1\}, \quad\forall \{i,i'\}\subseteq I,\,a\in A_i,\, r\in R_i. \label{eqn:mean_support_inner_max_MILP_con7} 
\end{align} \label{eqn:mean_support_inner_max_MILP}%
\end{subequations} \vspace{-10mm}%
\end{proposition}

Note that the McCormick inequalities \eqref{eqn:mean_support_inner_max_MILP_con4}--\eqref{eqn:mean_support_inner_max_MILP_con6} in formulation \eqref{eqn:mean_support_inner_max_MILP} involve big-$M$ coefficients $(\lambdaub_{i,i'}, \muub_{i,a}, \thetaub_{i,r})$, i.e., upper bounds on dual variables $(\lambda_{i,i'}, \mu_{i,a}, \theta_{i,r})$, that can undermine computational efficiency if they are set too large. Therefore, in Proposition~\ref{prop:dual_var_UB}, we derive tight upper bounds on these variables to strengthen the MILP reformulation.
\begin{proposition} \label{prop:dual_var_UB}
For any $(x,y,z,v,u,s)$ satisfying $\eqref{eqn:1st_stage_con1-2}$--$\eqref{eqn:1st_stage_con20-21}$, the following bounds are valid.
\begin{subequations} 
\begin{align}
&0 \leq \mu_{i,a}\leq \cg_a + \co_a,\quad\forall (i,a)\in\calF^A;\quad 0 \leq \theta_{i,r} \leq \cg_r + \co_r,\quad\forall (i,r)\in\calF^R; \\
&0 \leq \lambda_{i,i'} \leq \sum_{i\in I} \cw_i + \sum_{a\in A} (\cg_a + \co_a) + \sum_{r\in R}(\cg_r + \co_r) ,\quad\forall \{i,i'\}\subseteq I.
\end{align} 
\end{subequations}
\end{proposition}

Replacing the inner maximization problem in \eqref{eqn:mean_support_WC_exp_obj}  by its equivalent  MILP reformulation $H(x,y,z,v,u,s,\rho)$ in  \eqref{eqn:mean_support_inner_max_MILP}, and combining with the outer minimizing problem in \eqref{eqn:mean_support_model}, we derive the following equivalent reformulation of the \droe{} model. 
\begin{subequations}
\begin{align}
 \underset{x,\,y,\,z,\,v,\,u,\,\alpha,\,\beta,\,s,\,\rho,\,\delta}{\text{minimize}} & \quad  \sum_{r\in R}f_r v_r + \sum_{a\in A}f_a y_a + \sum_{i\in I}\rho_i m_i + \delta \label{eqn:mean_support_final_obj} \\
 \text{subject to\,\,\,\,\,\,} &   \quad  \eqref{eqn:1st_stage_con1-2}-\eqref{eqn:1st_stage_con20-21},\quad \delta \geq H(x,y,z,v,u,s,\rho). \label{eqn:mean_support_final_con1}
\end{align}  \label{eqn:mean_support_final}
\end{subequations} \vspace{-10mm}

\subsection{\drocvar{} Model Reformulation} \label{subsec:DRO-CVaR_reformulation}

In this section, we derive an equivalent reformulation of the \drocvar{} model. In Proposition~\ref{prop:mean_support_WC_CVaR}, we present an equivalent reformulation of the inner problem in~\eqref{eqn:mean_support_model} with $\rho_{\Prob}(\cdot)=\Prob\mhyphen\CVaR_{\gamma}(\cdot)$. 

\begin{proposition} \label{prop:mean_support_WC_CVaR}
For $(x,y,z,v,u,s)$ satisfying \eqref{eqn:1st_stage_con1-2}--\eqref{eqn:1st_stage_con20-21}, the inner problem in~\eqref{eqn:mean_support_model}, namely, to solve $\sup \limits_{\Prob\in\calP(m,\calS)} \Prob\mhyphen\CVaR_{\gamma}\big(Q(x,y,z,v,u,s,D)\big)$, is equivalent to 
\begin{subequations}  \label{eqn:mean_support_WC_CVaR}
\begin{align} 
 \underset{\rho_0\in\R,\,\rho\in\R^I,(\underline{\psi}_i,\overline{\psi}_i)\in\R^I \times \R^I}{\textup{minimize}} & \quad
\Bigg(\frac{1}{1-\gamma}-1\Bigg)\rho_0 + \frac{1}{1-\gamma} \sum_{i\in I}\rho_i m_i + \sup_{d\in\calS} \Bigg\{Q(x,y,z,v,u,s,d)-\sum_{i\in I} \rho_id_i \Bigg\} \label{eqn:mean_support_WC_CVaR_obj}\\
 \textup{subject to\hspace{10mm}} & \quad \rho_0 + \sum_{i\in I} \big(\dlb_i\underline{\psi}_i-\dub_i\overline{\psi}_i\big) \geq 0, \label{eqn:mean_support_WC_CVaR_obj_con1}\\
 &\quad \underline{\psi}_i-\overline{\psi}_i=\rho_i,\quad \underline{\psi}_i\geq 0,\quad \overline{\psi}_i\geq 0,\quad\forall i\in I. \label{eqn:mean_support_WC_CVaR_obj_con2}
\end{align} %
\end{subequations}%
\end{proposition}
Note that the inner maximization problem in \eqref{eqn:mean_support_WC_CVaR} is the same as the inner problem in \eqref{eqn:mean_support_WC_exp_obj}. Therefore, we apply the same techniques in the proof of  Proposition \ref{prop:mean_support_inner_max_MILP} to derive the following equivalent reformulation of the \drocvar{} model:
\begin{subequations} \label{eqn:mean_support_CVaR_final}
\begin{align}
\underset{x,\,y,\,z,\,v,\,u,\,\alpha,\,\beta,\,s,\,\rho_0,\rho,\,\delta}{\text{minimize}} & \quad  \sum_{r\in R}f_r v_r + \sum_{a\in A}f_a y_a + \Bigg(\frac{1}{1-\gamma}-1\Bigg)\rho_0 + \frac{1}{1-\gamma} \sum_{i\in I}\rho_i m_i + \delta \label{eqn:mean_support_CVaR_final_obj} \\
 \text{subject to\,\,\,\,\,\,\,\,\,} &   \quad  \eqref{eqn:1st_stage_con1-2}-\eqref{eqn:1st_stage_con20-21},\quad \eqref{eqn:mean_support_WC_CVaR_obj_con1}-\eqref{eqn:mean_support_WC_CVaR_obj_con2},\quad \delta \geq H(x,y,z,v,u,s,\rho). \label{eqn:mean_support_CVaR_final_con1}
\end{align} %
\end{subequations}

\section{Solution Methods} \label{sec:solution_method}
In this section, we first propose valid inequalities to improve the solvability of the \spe{} and \spcvar{} models (Section \ref{subsec:VI_SP}). Then, we propose a column-and-constraint generation (C\&CG) method to solve our \droe{} and \drocvar{} models (Section \ref{subsec:C&CG}). Also, we propose several valid inequalities to improve the convergence of our C\&CG method (Section \ref{subsec:VI_DRO}). Finally, we discuss the separability of the models (Section \ref{subsec:separability}). \vspace{-5mm}

\subsection{Valid Inequalities for the \spe{} and \spcvar{} Models} \label{subsec:VI_SP}
Note that the idle times of anesthesiologists and operating rooms are nonnegative. Thus, we add the following valid inequalities to the \spe{} and \spcvar{} models, which, as we later show, tightens its linear relaxation.
\begin{equation}\label{eqn:idle_nonneg}
   \bigg(t^\text{end}_a-t^\text{start}_a-\sum_{i\in I_a} d_i x_{i,a} \bigg) \hreg_a+ o_a \geq 0,\quad T^\text{end} v_r-\sum_{i\in I_r} d_i z_{i,r} + o_r \geq 0, \quad\forall a\in A,\, r\in R.  
\end{equation} \vspace{-8mm}

\subsection{The C\&CG Method for the \droe{} and \drocvar{} Models} \label{subsec:C&CG} \label{subsec:constraint_generation}

Note that the \droe{} model in \eqref{eqn:mean_support_final} and \drocvar{} model in \eqref{eqn:mean_support_CVaR_final}  involve inner maximization problems, specifically defining the function value $H(x,y,z,v,u,s,\rho)$ in constraints~\eqref{eqn:mean_support_final_con1} and \eqref{eqn:mean_support_CVaR_final_con1}, respectively. Thus, formulations \eqref{eqn:mean_support_final} and \eqref{eqn:mean_support_CVaR_final} cannot be solved directly using standard techniques. In this section, we develop a C\&CG method to solve our \droe{} model, and note that a similar C\&CG method can be employed to solve our \drocvar{} model. The motivation of this algorithm is as follows. Consider the inner maximization problem in~\eqref{eqn:mean_support_WC_exp_obj}.  The $i$th element of the optimal solution component $d^*\in\calS$ only takes value $\dlb_i$ or $\dub_i$. As a result, the inner maximization problem in \eqref{eqn:mean_support_WC_exp_obj} over $d\in\calS$ only needs to consider the $2^{|I|}$ combinations, i.e. $\bigtimes_{i\in I} \{\dlb_i,\dub_i\}$, to determine $d^*$.  Instead of solving an MILP with exponentially many constraints, we use a C\&CG method to identify scenarios in an iterative manner to obtain an optimal solution. 

Algorithm \ref{algo:constraint_generation_mean_support} presents our C\&CG method. In each iteration, we first solve the master problem~\eqref{eqn:mean_support_master}, which only employs the scenario-based constraints \eqref{eqn:2nd_stage_con1}--\eqref{eqn:2nd_stage_con8} corresponding to a set of scenarios indexed by $\calK$, to obtain a solution (i.e., surgery assignment, sequence, and schedule).  By considering only a subset of surgery durations, the master problem is a relaxation of the original problem. Thus, it provides a lower bound on the optimal value of the \droe{} model.  Then, given the optimal solution from the master problem, we identify a duration vector $d^*$ by solving subproblem \eqref{eqn:mean_support_inner_max_MILP}.  Given that solutions of the master problem are feasible to the original problem,  we obtain an upper bound by solving the subproblem using these solutions. Next, we introduce second-stage variables and constraints associated with the identified duration scenario to the master problem.  We then solve the master problem again with the new information (in the enlarged set $\calK$) from the subproblems.  This process continues until the gap between the lower and upper bound obtained in each iteration satisfies a predetermined termination tolerance $\varepsilon \geq 0$. Given the relatively complete recourse property of our second-stage problem, feasibility cuts are not needed. Moreover, since there are only a finite number of scenarios (i.e., $2^{|I|}$) in total, the algorithm terminates in finite number of iterations (see \citealp{Tsang_et_al:2023, Zeng_Zhao:2013}). 
\IncMargin{0em}
\begin{algorithm}[t] \SetAlgoNoEnd  \OneAndAHalfSpacedXI \small
\SetKwInOut{Initialization}{Initialization}
\Initialization{\vspace{-2mm}  Set $LB=0$, $UB=\infty$, $\varepsilon\geq0$, $\calK=\emptyset$,   $j=1$.} 
\textbf{1. Master problem.} Solve the master problem
\begin{subequations}
\begin{align}
 \underset{\substack{ x,y,z,v,u,s,\\ \alpha,\beta,\rho,\,\delta, \\ q,\,o,\,w,\,g}}{\text{minimize}} \quad
&  \sum_{r\in R}c_r v_r + c_q\sum_{a\in A}y_a + \sum_{i\in I} \rho_i m_i + \delta \label{eqn:mean_support_master_obj} \\
\text{subject to} \quad
&  \eqref{eqn:1st_stage_con1-2}-\eqref{eqn:1st_stage_con20-21},\quad\big\{\eqref{eqn:2nd_stage_con1}-\eqref{eqn:2nd_stage_con8},\, k\in\calK\big\}, \label{eqn:mean_support_master_con1} \\
& \delta \geq \sum_{a\in A} \Big(\cg_a g^k_a + \co_a o^k_a \Big) +\sum_{r\in R} \Big(\cg_r g^k_r + \co_r o^k_r \Big) + \sum_{i\in I} \cw_i w^k_i - \sum_{i\in I} \rho_i d^k_i,\,\,\, \forall k\in\calK,  \label{eqn:mean_support_master_con2}
\end{align}  \label{eqn:mean_support_master}%
\end{subequations}
\hspace{2.7mm}Record the optimal solution $(x^j,y^j,z^j,v^j,u^j,s^j,\alpha^j,\beta^j,\rho^j,\delta^j)$ and value $Z^j$. Set $LB = Z^j$. \\
\textbf{2. Subproblem.} Solve \eqref{eqn:mean_support_inner_max_MILP} with fixed $(x^j,y^j,z^j,v^j,u^j,s^j,\rho^j)$.
\begin{enumerate}[leftmargin=12mm,topsep=0mm,itemsep=0mm]
    \item [\textbf{2.1}] Record the optimal solution $b^*$ and value $Y^j$.  Set $UB = \min\big\{UB, (Z^j - \delta^j)+Y^j\big\}$.
    \item [\textbf{2.2}] If $(UB-LB)/UB<\varepsilon$ or $\delta^j\geq Y^j$, then terminate
    ; else, go to step 3.
\end{enumerate}
\textbf{3. Column-and-constraint generation routine.}
\begin{enumerate}[leftmargin=12mm,topsep=0mm,itemsep=0mm]
    \item [\textbf{3.1}] Using $b^*$ from step 2, compute $d^j_i=\dlb_i + b^*_i(\dub_i-\dlb_i)$ for all $i\in I$.
    \item [\textbf{3.2}] Add variables $(o_r^j,o_a^j,w_i^j,g_a^j,g_r^j)$ and the following constraints to the master problem:
    $$\delta \geq \sum_{a\in A} \Big(\cg_a g^j_a + \co_a o^j_a \Big) +\sum_{r\in R} \Big(\cg_r g^j_r + \co_r o^j_r \Big) + \sum_{i\in I} \cw_i w^j_i - \sum_{i\in I} \rho_i d^j_i,\quad \big\{\eqref{eqn:2nd_stage_con1}-\eqref{eqn:2nd_stage_con8} \,\,\, \text{with } d=d^j\big\}.$$
    Update $j \leftarrow j+1$ and $\calK \leftarrow \calK\cup\{j\}$. Go back to step 1.
\end{enumerate}
\BlankLine
\caption{A column-and-constraint-generation method for the \droe{} model} \label{algo:constraint_generation_mean_support}
\end{algorithm}\DecMargin{1em}

\color{black}
\begin{remark}
As suggested by a reviewer of this paper, in \ref{appdx:LDR_DRO}, we derive approximations of our proposed DRO models using the classical linear decision rules (LDR) approach (see \citealp{Georghiou_et_al:2019} for a recent survey). These approximations are not directly solvable. However, we derive equivalent MILP reformulations of these approximations and show how their sizes could grow significantly with the number of surgeries, ORs, and anesthesiologists. Computational results in \ref{appdx:LDR_DRO} demonstrate how these LDR approximations provide poor-quality solutions to the ORASP and are computationally intractable for large instances. In contrast, using our proposed models and the C\&CG method, we can solve all practical ORASP instances within a reasonable time (see Section~\ref{subsec:expt_comp_time}).
\end{remark}
\color{black}

\subsection{Valid Inequalities for the \droe{} and \drocvar{} Models} \label{subsec:VI_DRO}
Inequalities \eqref{eqn:idle_nonneg} are also valid in the master problem~\eqref{eqn:mean_support_master} for each scenario $d^k$ with $k\in\calK$.  Here, we present three more sets of valid inequalities for the master problem. First, observe that the Dirac measure on $m$ lies in $\calP(m,\calS)$. Therefore, we have
\begin{equation} \label{eqn:global_LB}
 \sup_{\Prob\in\calP(m,\calS)} \varrho_\Prob\big(Q(x,y,z,v,u,s,D)\big)  \geq Q(x,y,z,v,u,s,m) =: L,
\end{equation}
where the lower bound is a deterministic problem with a single scenario $d=m$. This means we can impose the constraint $\sum_{i\in I} m_i\rho_i + \delta \geq L$ and $[(1-\gamma)^{-1}-1]\rho_0+(1-\gamma)^{-1}\sum_{i\in I} m_i\rho_i + \delta \geq L$, which respectively serves as a global lower bound on the objective of the \droe{} and \drocvar{} models. Second, from \eqref{eqn:global_LB}, for any given first-stage decision, the recourse value with $d=m$ provides a lower bound. Therefore, in the initialization step of the C\&CG method, we could include the scenario $d=m$ (see \ref{appdx:VI_recourse_mean_LB} for details).  Third, the dual variable $\rho$ is unrestricted in both the \droe{} and \drocvar{} models; see \eqref{eqn:mean_support_final} and \eqref{eqn:mean_support_CVaR_final}.  We derive valid inequalities that provide a lower and upper bound on $\rho$ in Proposition \ref{prop:VI_dual_bounds} (see \ref{appdx:pf_VI_dual_bounds} for a proof). 

\begin{proposition} \label{prop:VI_dual_bounds}
The following bounds are valid lower and upper bounds on $\rho_i$ for all $I \in [I]$: 
\begin{equation} \label{eqn:VI_dual_bounds}
-\sum_{a\in A_i} \cg_a g_a x_{i,a} - \sum_{r\in R_i} \cg_r z_{i,r} \leq \rho_i \leq  \min\Bigg\{ \sum_{i'\in I,\,i'\ne i} u_{i,i'},\, 2 \Bigg\} \lambdaub + \sum_{a\in A_i}\co_a x_{i,a} + \sum_{r\in R_i} \co_r z_{i,r}
\end{equation}
\end{proposition}

Although including \eqref{eqn:VI_dual_bounds} reduces the search space for $\rho$, in our preliminary experiments, we found that this might worsen the computational performance in some cases. This may follow from the increased model complexity by \eqref{eqn:VI_dual_bounds} due to the presence of first-stage variables.  Therefore, in our experiments, we include \eqref{eqn:global_LB} along with the following variable-free version of~\eqref{eqn:VI_dual_bounds}:
\begin{equation} \label{eqn:VI_dual_bounds_const}
 -\max_{a\in A}\cg_a - \max_{r\in R}\cg_r\leq\rho_i\leq 2\lambdaub+\max_{a\in A}\co_a + \max_{r\in R}\co_r.
\end{equation} 
In \ref{appdx:expt_DRO_VI}, we provide results comparing solution times for solving the \droe{} model using either the variable-free version \eqref{eqn:VI_dual_bounds_const} or the variable-dependent version \eqref{eqn:VI_dual_bounds}. We observe that solution times under the variable-free version are generally similar to or shorter than those under the variable-dependent version. In particular, solution times under the variable-free version are significantly shorter for large instances of the ORASP.

\subsection{Separability of the Models} \label{subsec:separability}

\color{black}
In this section, we show how the recourse problem of the ORASP can be decomposed into smaller problems under the following scheduling policies that are well-known and widely employed in practice. First, recall that some types of surgeries require specialized anesthesiologists and that each anesthesiologist might have a different combination of specializations. Thus, the assignment of an anesthesiologist to a surgery must respect the specialty required for the corresponding surgery type.  This policy is employed in all health systems in the US.

Second, many hospitals (including our collaborating hospital) employ the dedicated OR (or dedicated block) scheduling policy to construct their master schedule, which specifies the assignment of ORs to one or few surgical specialties/types. Moreover, it is common that each OR is dedicated to only one surgical specialty, and there can be multiple blocks for the same specialty within a cycle (e.g., a month) of the OR schedule \citep{Aringhieri_et_al:2015, Breuer_et_al:2020, Makboul_et_al:2022, Schneider_et_al:2020}. This is partly because most surgeries have long surgery durations, such as cardiac surgery, neurosurgery, and orthopedic surgery (see Figure~1), each requiring special surgical equipment and setups. Thus, performing one surgery of these types already occupies a large portion of the OR service hours. Therefore, dedicating the same OR for multiple surgery types could be inefficient. Various studies have also shown how the dedicated block scheduling policy could improve efficiency, reduce planning complexity, and promote coordination among surgical resources (see, e.g., \citealp{Mazloumian_et_al:2022, MHallah_Visintin:2019, Penn_et_al:2017, Zhu_et_al:2019}, and references therein).

Under these two policies, one can decompose the recourse problem into smaller subproblems based on surgery types with a shared pool of ORs and/or anesthesiologists. Let us first illustrate the idea using the example presented in Figure~\ref{fig:eg_separability}. In this example, there are six surgery types $T=\{1,2,\ldots, 6\}$, six sets of surgeries (each consisting of surgeries of the same type), seven sets of ORs, and five sets of anesthesiologists. Each set of ORs consists of all ORs dedicated to a subset of surgery types. For example, $R_{23}$ consists of all ORs to which type 2 and type 3 surgeries can be assigned, and $R_2$ consists of all ORs dedicated to type 2 surgeries. Similarly, each set of anesthesiologists consists of all anesthesiologists with the same specialty (i.e., each covering the same set of surgery type(s)). For example, $A_{45}$ consists of all anesthesiologists that can perform type 4 and type 5 surgeries. 
 \begin{figure}[t]\OneAndAHalfSpacedXI
    \centering
    \includegraphics[scale=0.8]{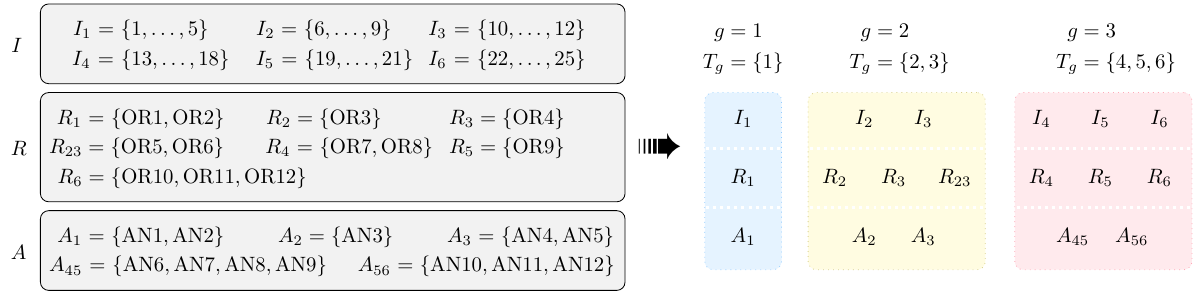}
    \caption{An illustration of the partition of an ORASP instance characterized by sets  $(I,R,A)$ and a set of six surgery types $T=\{1,\dots,6\}$. The left panel shows the sets $(I,R,A)$. The right panel shows the partition $\{(I^g,R^g,A^g)\}_{g=1}^3$ of $(I,R,A)$ based on the partition $\{T_g\}_{g=1}^3$ of $T$.}
    \label{fig:eg_separability}
\end{figure}

Note that type 1 surgeries have dedicated ORs ($R_1$) and require specialized anesthesiologists ($A_1$). In other words, they do not share resources with other surgery types. Types 2 and 3 surgeries have dedicated ORs ($R_2$ and $R_3$, respectively) and require specialized anesthesiologists ($A_2$ and $A_3$, respectively). Also, they can be scheduled in any OR in $R_{23}$ (i.e., they share $R_{23}$). Finally, types 4, 5, and 6 surgeries have dedicated ORs ($R_4$, $R_5$, and $R_6$, respectively). However, while type~4 and type 6 surgeries can only be performed by anesthesiologists in $A_{45}$ and $A_{56}$, respectively, type 5 surgeries can be performed by those in $A_{45}\cup A_{56}$ (i.e., type 5 shares anesthesiologists with types 4 and 6). Accordingly, we partition the set of surgery types $T=\{1,2,\dots,6\}$ into $T_1=\{1\}$, $T_2=\{2,3\}$, and $T_3=\{4,5,6\}$. Each element $T_g$ of $\{T_g\}_{g \in G}$ ($G=\{1,2,3\}$) consists of a unique subset of surgery types that share a subset of ORs and/or a subset of anesthesiologists. Given $\{T_g\}_{g\in G}$, we construct the following partition $\{(I^g,R^g,A^g)\}_{g\in G}$ of $(I,R,A)$: for $g=1$, we have $(I^g,R^g,A^g)=(I_1,R_1,A_1)$; for $g=2$, we have $(I^g,R^g,A^g)=(I_2\cup I_3, R_2\cup R_3\cup R_{23}, A_2\cup A_3)$; for $g=3$, we have $(I^g,R^g,A^g)=(I_4\cup I_5\cup I_6, R_4\cup R_5\cup R_6, A_{45}\cup A_{56})$. Then, we can decompose the recourse problem into three subproblems $Q^g(x^g,y^g,z^g,v^g,u^g,s^g,d^g)$ characterized by $(I^g,R^g,A^g)$ for $g\in\{1,2,3\}$.


For general ORASP instances, one can implement the following recipe to decompose the recourse problem into smaller subproblems. First, one can construct the partition $\{T_g\}_{g \in G}$ of $T$ in the following manner. A subset $T_g$ consists of a single surgery type if this type has dedicated  ORs and anesthesiologists (i.e., does not share any OR and anesthesiologist with other types). On the other hand, type $t \in T$ belongs to a subset $T_g$ with $|T_g|>1$ if (i) there is a subset of ORs to which surgeries of this type and those of any other type $t' \in T_g$ can be assigned and/or (ii) there is a subset of anesthesiologists that could perform type $t$ and any other type $t' \in T_g$ surgeries. Second, given the partition $\{T_g\}_{g\in G}$, one can construct a partition  $\{(I^g,R^g,A^g)\}_{g\in G}$ of $(I,R,A)$, where $I^g=\bigcup_{t \in T_g} I_t$ (here, $I_t$ is the set of type $t$ surgeries), $R^g=\{r\in R\mid \kappa^R_{i,r}=1 \text{ for some }i\in I^g\}$, and $A^g=\{a\in A\mid \kappa^A_{i,a}=1 \text{ for some }i\in I^g\}$. (We recall that $\kappa^R_{i,r}=1$ and $\kappa^A_{i,a}=1$ if surgery $i$ can be performed in OR $r$ and by anesthesiologist $a$, respectively.) Finally, we can decompose the recourse problem as
$$Q(x,y,z,v,u,s,d)=\sum_{g\in G} Q^g(x^g,y^g,z^g,v^g,u^g,s^g,d^g),$$
where each $Q^g(x^g,y^g,z^g,v^g,u^g,s^g,d^g)$ is characterized by $(I^g,R^g,A^g)$.

\color{black}
We can leverage this decomposable structure when solving the \spe{}, \droe{}, and \drocvar{} models; see \ref{appdx:separability} for discussions on the separability of the \droe{} and \drocvar{} models. In contrast, the \spcvar{} model does not admit such a decomposition due to the subadditivity of $\Prob\mhyphen\CVaR_\gamma(\cdot)$. However, our experimental results show that the difference in out-of-sample costs between the \spcvar{} model with and without decomposition is very small (less than $3\%$ in most cases). This indicates that the \spcvar{} model with decomposition could produce near-optimal performance. Hence, we adopt the decomposition approach when solving large instances using the \spcvar{} model.


\color{black}

\section{Symmetry-Breaking Constraints} \label{sec:symm}

Symmetry has long been recognized as a curse for solving (mixed) integer programming problems, such as assignment and sequencing problems. It allows the existence of multiple equivalent solutions and hence identical subproblems, leading to a wasteful duplication of computational effort in algorithms such as branch-and-bound and branch-and-cut. The first-stage problem of the ORASP possesses a great deal of symmetry. Although breaking symmetry is standard in (healthcare) scheduling problems to avoid exploring equivalent solutions, previous studies did not address the issue of symmetry in the ORASP. In this section, we discuss practical situations leading to symmetry in the ORASP and present strategies to break these symmetries. In \ref{appdx:SBC_fixing}, we discuss additional variable-fixing constraints.


\subsection{Operating Rooms Opening Order and Load} \label{subsec:SBC_OR_index}

In practice, each surgery type typically requires specific surgical equipment and setups (see, e.g., \citealp{Diamant_et_al:2018, Deshpande_et_al:2023, Pessoa_et_al:2015}). Thus, it is reasonable that identical ORs, i.e., ORs dedicated to the same set of types, have the same fixed opening cost, overtime cost, and idle time cost. Such a practical setup of these cost parameters is also widely adopted in the OR scheduling literature; see, e.g., numerical experiments of \cite{Breuer_et_al:2020, Denton_et_al:2010, Guo_et_al:2014, Jung_et_al:2019}.  

Swapping the sets of surgeries assigned to any pair of identical ORs with the same fixed, overtime, and idle time costs produces equivalent solutions. To illustrate, we provide an example in Figure~\ref{fig:eg_SBC_OR}. In this example, there are two surgery types and three sets of identical ORs. The set $R^1=\{\text{OR1},\text{OR2}\}$ has two identical ORs dedicated to type 1, the set  $R^2=\{\text{OR3}\}$ has one OR dedicated to type 2, and the set $R^3=\{\text{OR4},\text{OR5}\}$ has two identical ORs where both types 1 and 2 surgeries can be scheduled. Solutions~1 and 2 in Figure~\ref{fig:eg_SBC_OR} are equivalent since the number of ORs opened from each set of identical ORs is the same (i.e., one OR from $R^1$, one OR from $R^2$, and two ORs from $R^3$ are opened), and only sets of surgeries assigned to identical ORs are swapped (i.e., schedules of OR1 and OR2, as well as schedules of OR4 and OR5, are swapped). To prevent exploring such equivalent solutions, one can assume that identical ORs are numbered sequentially and enforce that the OR load (i.e., the number of scheduled surgeries) is non-increasing with the OR index. Mathematically, let $\{R^e\}_{e\in E}$ be the collection of sets of identical ORs, where each $R^e=\{r_{1,e},\dots,r_{|R^e|,e}\}$ is a set of identical ORs. Using this notation, we introduce constraints: 
\begin{subequations} \label{eqn:SBC_OR_index} 
    \begin{align}
         & v_{r_{k-1},e} \geq v_{r_k,e},\quad \forall e\in E,\, k\in[2,|R^{e}|],  \label{eqn:SBC_OR_index_2a}   \\
         & \sum\limits_{i\in I} z_{i,r_{k-1,e}} \geq \sum\limits_{i\in I} z_{i,r_{k,e}}, \quad \forall e\in E,\, k\in[2,|R^{e}|]. \label{eqn:SBC_OR_index_2b}
    \end{align}
\end{subequations}
Constraints~\eqref{eqn:SBC_OR_index_2a} enforce that identical ORs are open in ascending order of their indices.  Constraints~\eqref{eqn:SBC_OR_index_2b} enforce that an OR with a smaller index has a larger number of scheduled surgeries (i.e., larger load). Note that if identical ORs have different fixed, overtime, or idle time costs, one could group ORs with the same costs and apply the proposed symmetry-breaking constraints to each of these smaller groups of ORs.

We derived constraints~\eqref{eqn:SBC_OR_index_2a} based on similar principles presented in prior OR scheduling studies to avoid arbitrary opening of ORs. In contrast to these studies (see, e.g., \citealp{Denton_et_al:2010, HashemiDoulabi_Khalilpourazari:2022, Roshanaei_et_al:2017}), we do not adopt the strong assumption that all ORs are identical, i.e., $|E|=1$. Hence, constraints~\eqref{eqn:SBC_OR_index_2a} generalize existing constraints for breaking the symmetry in OR opening order. Constraints~\eqref{eqn:SBC_OR_index_2b} are derived based on similar symmetry-breaking principles as those outlined in studies on the bin-packing (BP) problem with identical bins, which force the load of bin $j$ to be greater than or equal to bin $j+1$. A key difference between our constraints and those derived for the BP problem is that the load of an OR in the ORASP is the number of surgeries scheduled in this OR, while a bin load in the BP problem equals the sum of the weights of items assigned to the bin.
\begin{figure}[t]\OneAndAHalfSpacedXI
    \centering
    \includegraphics[scale=0.76]{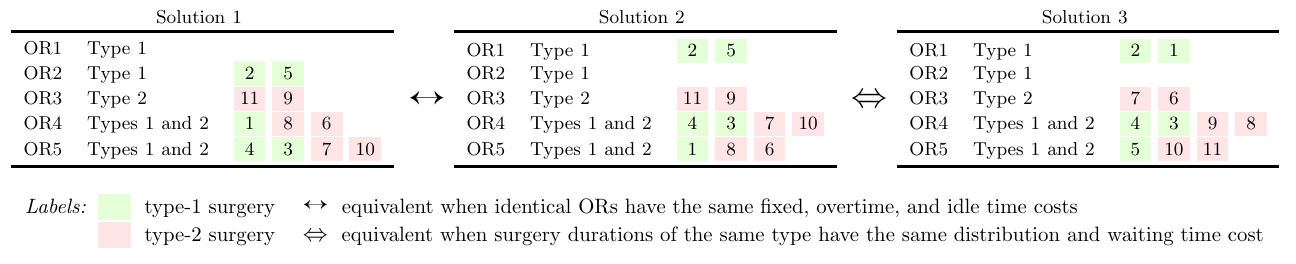}
    \caption{Examples of equivalent solutions due to symmetries in OR opening order and load (solutions~1 and 2), and  surgery-to-OR assignments (solutions~2 and 3).}
    \label{fig:eg_SBC_OR}
\end{figure}

\subsection{Surgery-to-OR Assignments}  \label{subsec:SBC_surgery_OR_assg}

Each surgery type has a clinically acceptable range for its duration that hospitals use as a reference to schedule surgeries of that type. Also, it is common to schedule a surgery using the average of the historical realizations of the duration of the corresponding surgery type. Therefore, many studies have modeled surgery duration distribution by surgery type and assumed a common distribution for durations of surgeries of the same type  (see, e.g., \citealp{Guo_et_al:2014, Kayis_et_al:2015, Pessoa_et_al:2015, Shehadeh_et_al:2019, Stepaniak_et_al:2009, Wang_et_al:2023}). In addition, existing literature on OR scheduling typically uses the same waiting time cost for surgeries of the same type (see, e.g., \citealp{Freeman_et_al:2016, Shehadeh_et_al:2019, Tsai_et_al:2021}). Nevertheless, if surgeries of the same type have different distributions and waiting time costs, one could group those that follow the same distribution and have the same waiting cost as follows. Suppose there are $n$ subtypes $\{t_1,t_2,\ldots t_n\}$ of type $t$ surgeries such that surgery durations of the same sub-type have the same distribution and waiting time cost. Then, one could replace type $t$ by types $\{t_1,t_2,\ldots t_n\}$.


Recall the example described in Section~\ref{subsec:SBC_OR_index}. Since the surgeries of the same type have the same waiting cost and their durations follow the same distribution, surgery-to-OR assignments in solutions 2 and 3 in Figure~\ref{fig:eg_SBC_OR} result in the same objective value. This is because each OR in these two solutions has the same number of scheduled surgeries of each type, and only the assignments of some surgeries of the same type are swapped. For example, in solution 2, surgeries~1 and 5 of type~1 are assigned to OR1 and OR5, respectively, whereas in solution 3, surgery 5 is assigned to OR5 and surgery 1 to OR1. We can prevent exploring such equivalent assignments by enforcing surgeries of smaller indices to be assigned to ORs with smaller indices. To derive the desired symmetry-breaking constraints, we first construct a partition $\{T_\ell\}_{\ell \in L}$ of the set of types $T$, where (i) a subset $T_\ell$ has a single surgery type if this type has dedicated ORs and (ii) type $t \in T$ belongs to a subset $T_\ell$ with $|T_\ell|>1$ if there is a subset of ORs to which surgeries of this type and those of any other type $t' \in T_\ell$ can be assigned.  We then define the partition $\{(I^\ell,R^\ell)\}_{\ell\in L}$ of $(I,R)$, where $I^\ell=\bigcup_{t\in T_\ell} I_t$ and $R^\ell=\{r\in R\mid \kappa^R_{i,r}=1 \text{ for some }i\in I^\ell\}$. We assume that ORs in $R_\ell$ and surgeries in $I_t$ are numbered sequentially, i.e., $R^\ell=\{r_{1,\ell},\dots,r_{|R^\ell|,\ell}\}$  and $I_t=\{i_{1,t},\dots,i_{|I_t|,t}\}$. Using this notation, we introduce the following constraints to break the symmetry in surgery-to-OR assignments:
\begin{subequations} \label{eqn:SBC_surg_index_2} 
    \begin{align}
         &z_{i_{j,t}, r_{k,\ell}}\leq \sum\limits_{k'=1}^k z_{i_{j-1,t},r_{k',\ell}},\quad \forall \ell\in L,\, t\in T_\ell,\, j\in[2,|I_{t}|],\, k\in[2,|R^\ell|],  \label{eqn:SBC_surg_index_2a}  \\
         & z_{i_{j-1,t}, r_{k,\ell}}\leq \sum\limits_{k'=k}^{|R^\ell|} z_{i_{j,t},r_{k',\ell}}, \quad \forall \ell\in L,\, t\in T_\ell,\, j\in[2,|I_{t}|],\, k\in[2,|R^\ell|]. \label{eqn:SBC_surg_index_2b} 
    \end{align}
\end{subequations}
%
Constraints \eqref{eqn:SBC_surg_index_2a} ensure that if surgery $i_{j,t}$ is assigned to OR $r_{k,\ell}$, then surgery $i_{j-1,t}$ is assigned to an OR with index at most $k$, while constraints \eqref{eqn:SBC_surg_index_2b} ensure that surgery $i_{j+1,t}$ is assigned to an OR with index at least $k$. Note that although including either constraints \eqref{eqn:SBC_surg_index_2a} or \eqref{eqn:SBC_surg_index_2b} could break the symmetry, using both of them could tighten the LP relaxation of the proposed models for the ORASP (see \ref{appdx:SBC_surgery_assg}).  We derived constraints \eqref{eqn:SBC_surg_index_2} based on similar principles presented in prior OR scheduling studies. Different from these studies  (see, e.g., \citealp{Denton_et_al:2010, Roshanaei_et_al:2017, Wang_et_al:2023}), we do not assume that all ORs are identical, i.e., $|L|=1$. Hence, constraints \eqref{eqn:SBC_surg_index_2} generalize existing constraints for breaking symmetry in surgery-to-OR assignments.

Now consider the case where one or more sets in the partition $\{T_\ell\}_{\ell\in L}$ consists of a single surgery type, i.e., $|T_\ell|=1$.  Suppose, in addition, that identical ORs dedicated to the surgery type defining each $T_\ell$ with $|T_\ell|=1$ (i.e., ORs in $R^\ell$) have the same fixed, overtime, and idle time costs. In this case, we could replace constraints \eqref{eqn:SBC_surg_index_2a}--\eqref{eqn:SBC_surg_index_2b} for each $T_\ell$ with $|T_\ell|=1$ with the following constraints:
\begin{subequations} \label{eqn:SBC_surg_index_1} 
    \begin{align}
         &z_{i_{j,t}, r_{k,\ell}}\leq z_{i_{j-1,t},r_{k-1,\ell}}+z_{i_{j-1,t},r_{k,\ell}},\quad \forall \ell\in L,\, t\in T_\ell,\, j\in[2,|I_{t}|],\, k\in[2,|R^\ell|]:\, |T_\ell|=1,  \label{eqn:SBC_surg_index_1c}  \\
         &  z_{i_{j-1,t}, r_{k-1,\ell}}\leq \textstyle z_{i_{j,t},r_{k-1,\ell}}+ z_{i_{j,t},r_{k,\ell}}, \quad \forall \ell\in L,\, t\in T_\ell,\, j\in[2,|I_{t}|],\, k\in[2,|R^\ell|]:\, |T_\ell|=1. \label{eqn:SBC_surg_index_1d} 
    \end{align}
\end{subequations}
Constraints \eqref{eqn:SBC_surg_index_1c} ensure that if surgery $i_{j,t}$ is assigned to OR $r_k$, then surgery $i_{j-1,t}$ is assigned to OR with index $k$ or $k-1$ only (as opposed to OR with index at most $k$ in  \eqref{eqn:SBC_surg_index_2a}), and constraints \eqref{eqn:SBC_surg_index_1d} ensure that surgery $i_{j+1,t}$ is assigned to OR with index $k$ or $k+1$ only (as opposed to OR with index at least $k$ in  \eqref{eqn:SBC_surg_index_2b}). Applying constraints~\eqref{eqn:SBC_surg_index_1} instead of \eqref{eqn:SBC_surg_index_2} for each $T_\ell$ with a single surgery type (i.e., $|T_\ell|=1$) removes a larger number of equivalent solutions. To see this, consider the set $T_\ell=\{t\}$ and suppose that there are six surgeries of type $t$ and four ORs dedicated to this type $\{\text{OR1},\ldots,\text{OR4}\}$. If surgeries $1$ to $5$ are assigned to the first two ORs (i.e., OR1 and OR2), constraints~\eqref{eqn:SBC_surg_index_1} ensure that surgery $6$ will be assigned to OR$2$ or OR$3$, but not OR$4$. In contrast, constraints \eqref{eqn:SBC_surg_index_2} allow assigning surgery $6$ to OR2, OR3, or OR$4$. These assignments are equivalent.

\subsection{Surgery Sequence} \label{subsec:SBC_surgery_seq}

As discussed in the previous section,  it is common that surgeries of the same type have the same distribution of duration and waiting time cost. Thus, given a solution with a particular sequence of surgeries assigned to an OR, an equivalent solution can be obtained by swapping the order of any pair of surgeries of the same type in that OR.  Note that constraints~\eqref{eqn:SBC_surg_index_2} and \eqref{eqn:SBC_surg_index_1} do not prevent such symmetry in the surgery sequence. To illustrate, consider the two solutions in Figure~\ref{fig:eg_SBC_surg_seq}, which satisfy constraints~\eqref{eqn:SBC_surg_index_2}. These solutions are equivalent since only the order of some surgeries of the same type in OR2, OR4, and OR5 in solution 1 are swapped to produce solution 2. For example, in OR4, the order of surgeries 3 and 4 of type~1, as well as the order of surgeries 8 and 9 of type~2, are swapped. To avoid exploring such equivalent solutions, we impose the following constraints:
\begin{equation} \label{eqn:SBC_surg_seq_2}  
u_{i_{j-1,t},i_{j,t}} \geq z_{i_{j-1,t},r}+ z_{i_{j,t},r}-1, \quad \forall t\in T,\, j\in[2,|I_{t}|],\, r\in R,
\end{equation}
where $I_t=\{i_{1,t},\dots,i_{|I_t|,t}\}$ is the set of type $t$ surgeries. Constraints~\eqref{eqn:SBC_surg_seq_2} ensure that surgeries of the same type in the same OR are sequenced in ascending order of their indices. Constraints~\eqref{eqn:SBC_surg_seq_2} are generic and can be employed for any instance with ORs dedicated to multiple types.  Moreover, these constraints do not impose any restriction on the order of surgeries of the different types. The optimal sequence of surgeries in each OR is determined by solving the ORASP. Finally, note that we use general precedence variables $u_{i,i'}$ to model sequencing decisions in the ORASP.  Thus, we cannot adopt symmetry-breaking constraints proposed for formulations that employ other types of binary variables for the sequencing problem, such as sequence-position and time-index variables (see \citealp{Baker_Trietsch:2013} for a detailed discussion). Within the related literature that adopts general precedence variables to model sequencing decisions (see, e.g., \citealp{Celik_et_al:2023, Rath_et_al:2017}),  studies did not investigate the issue of symmetry in the surgery sequence in their formulations. 
\begin{figure}[t]\OneAndAHalfSpacedXI
    \centering
    \includegraphics[scale=0.76]{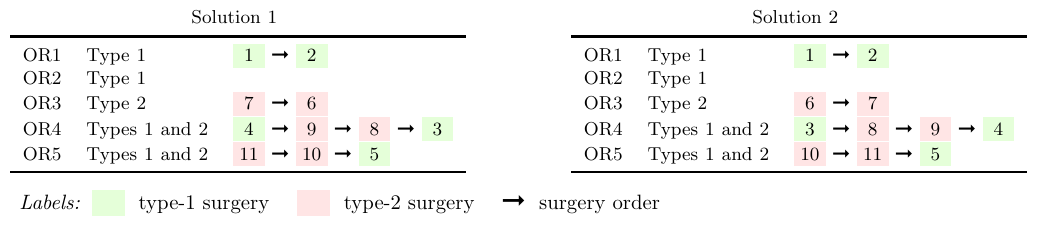}
    \caption{Examples of equivalent solutions due to symmetry in surgery sequence.}
    \label{fig:eg_SBC_surg_seq}
\end{figure}

\subsection{On-Call Anesthesiologist Employment Order} \label{subsec:SBC_oncall_anes}

Recall that specialized anesthesiologists are trained to perform particular types of surgery. Other types of anesthesiologists could have a different combination of specializations. Hence, it is common that identical on-call anesthesiologists (i.e., anesthesiologists with the same set of specializations) have the same fixed cost of being called in (see, e.g., \citealp{Rath_et_al:2017, Rath_Rajaram:2022}). This leads to symmetry in on-call anesthesiologist employment order. For example, suppose that there are three identical on-call anesthesiologists. Then, $y^1=(1,0,0)^\tp$, $y^2=(0,1,0)^\tp$, and $y^3=(0,0,1)^\tp$ are equivalent since only one of the three anesthesiologists is called. To prevent exploring such equivalent solutions, one can assume that identical on-call anesthesiologists are numbered sequentially and enforce that they are called in descending order of their indices. Mathematically, let $\{A^\text{call}_h\}_{h\in H}$ be the collection of sets of identical on-call anesthesiologists, where each $A^\text{call}_h=\{a'_{1,h},\dots,a'_{|A^\text{call}_h|,h}\}$ is a set of identical on-call anesthesiologists. Then, we introduce the constraints
\begin{equation} \label{eqn:SBC_oncall_anex_index} 
\textstyle y_{a'_{k-1,h}}\leq y_{a'_{k,h}}, \quad \forall h\in H,\, k\in[2,|A^\text{call}_h|]. 
\end{equation}
Constraints~\eqref{eqn:SBC_oncall_anex_index} ensure that identical on-call anesthesiologists are called in descending order of their indices.

\color{black}

\section{Numerical Experiments} \label{sec:numerical_experiments}

In this section, we use sets of publicly available surgery data to construct various ORASP instances and perform a case study from our collaborating health system. We conduct extensive computational experiments comparing the proposed methodologies computationally and operationally, demonstrating where significant performance improvements can be obtained and deriving insights relevant to practice. In Section \ref{subsec:expt_description}, we describe the set of ORASP instances constructed and discuss the experimental setup. In Section \ref{subsec:expt_comp_time}, we analyze solution times of the proposed models. We demonstrate the efficiency of the proposed valid inequalities and symmetry-breaking inequalities in Section \ref{subsec:expt_VI_SBC}. In Section \ref{subsec:expt_opt_sol}, we compare the optimal solutions of the proposed models. Then, we compare their operational performance via out-of-sample simulation tests in Sections \ref{subsec:expt_sol_quality} and \ref{subsec:expt_Rath}. Finally, in Section \ref{subsec:expt_real_data}, we present a case study and derive managerial insights.

\subsection{Test Instances and Experimental Setup} \label{subsec:expt_description}
We develop diverse ORASP instances based on prior literature and a publicly available dataset from \cite{Mannino_et_al:2010}. This dataset consists of three years of actual surgery durations for six different surgical specialties. Table \ref{table:instance_summary} summarizes the six ORASP instances we constructed based on this data. Each of these instances is characterized by the number of surgeries and their types, the number of ORs and their types, the number of anesthesiologists, and the master/block schedule. In \ref{appdx:instance_detail}, we provide summary statistics of the datasets and details of the master schedule. Note that instances 1--2 are relatively small, instances 3--4 are medium-sized, and instances 5--6 are large. In  \ref{appdx:num_results_add}, we provide additional computational results for another set of six ORASP instances constructed based on another set of publicly available surgery data.
\begin{table}[t]\centering\small
\footnotesize
\caption{Instance summary} \label{table:instance_summary}
\ra{0.4}  
\begin{tabular}{@{}l|cccccc@{}} \toprule
      & Instance 1 & Instance 2 & Instance 3 & Instance 4 & Instance 5 & Instance 6 \\ \midrule
Number of surgeries $|I|$ &   15 &    20 &    25 &    40 &    55 &    80 \\
Number of ORs $|R|$       &    7 &     8 &     8 &    11 &    20 &    32 \\
Number of anesthesiologists $|A|$ &    5 &     9 &    10 &    16 &    24 &    40 \\
\bottomrule
\end{tabular}
\end{table}

We obtain the parameters for each instance as follows. We estimate the mean $m_i$ and standard deviation $\sigma_i$ of the duration of each surgery type from \cite{Mannino_et_al:2010}. As in prior literature, we set the lower bound $\dlb_i$ and upper bound $\dub_i$ as the $20$th and $80$th percentiles of the data of that surgery type, respectively.  For the \spe{} and \spcvar{} models, we generate the in-sample scenarios using lognormal (logN) distributions with the estimated mean and variance clipped at $\dlb_i$ and $\dub_i$. The overtime costs per hour are set to $\co_r=450$ and $\co_a=150$ \citep{Rath_et_al:2017}. We set the OR fixed cost as $f_r=900$, which is equivalent to double the per-hour OR overtime cost \citep{Denton_et_al:2010}. The on-call anesthesiologist's fixed cost is set to $f_a=1000$ \citep{Rath_et_al:2017}. \blue{\cite{Gupta:2007} argued that $\cw_i=\cw$ is assumed in many practical situations because ``OR managers do not typically have data upon which to base choices of different values of these parameters for different types of surgeries.'' Indeed, prior studies use the same waiting time cost for all surgeries (see, e.g., numerical experiments in \citealp{Denton_Gupta:2003, Denton_et_al:2007, Freeman_et_al:2016, Khaniyev_et_al:2020, Shehadeh_et_al:2019, Tsai_et_al:2021}). Moreover, while there is no gold standard for choosing the per-unit waiting time cost parameter ($\cw$), studies commonly set this parameter such that it is smaller than the per-unit OR overtime cost ($\co_r$), with a ratio $\cw_i:\co_r$ ranging from $1:1.25$ to $1:3$. Based on these considerations, we set $\cw_i=200$ (with $\cw_i:\co_r$ ratio of $1:2.25$)}.   We consider 3 different cost structures for the idling costs per hour: cost 1 ($\cg_r=\cg_a=0$), cost 2 ($\cg_r=300$ and $\cg_a=0$), and cost 3 ($\cg_r=300$ and $\cg_a=100$). We maintain the ratio $\co/\cg$ as $1.5$ as suggested in the literature \citep{Shehadeh_et_al:2019}.

We solve the \spe{} and \spcvar{} models via sample average approximation (SAA) with $N$ scenarios, which replaces the true distribution by the empirical distribution from the data (see \ref{appdx:SP_SAA} complete models). We pick $\gamma=0.95$ for the \spcvar{} model. To decide the number of scenarios $N$, we employ the Monte Carlo optimization (MCO) procedure, which provides statistical lower and upper bounds on the optimal value of the ORASP based on the optimal solution to its SAA \citep{Kleywegt_et_al:2002,Lamiri_et_al:2009}. This in turn provides a statistical estimate of the approximated relative gap (see \ref{appdx:MCO} for a detailed description and corresponding results). Applying the MCO procedure with $N=100$ in the \spe{} model, the approximated relative gaps for the ORASP instances described in Table \ref{table:instance_summary} range from $0.06\%$ to $1.05\%$. Note that a larger $N$ could result in longer solution times without significant improvements in the approximated relative gaps. Therefore, we select $N=100$ for our computational experiments. 

We implemented our proposed models and algorithm in AMPL modeling language and use CPLEX (version 20.1.0.0) as the solver with default settings.  We set the relative MIP tolerance to $2\%$. We solve large DRO instances using an inexact version of our proposed C\&CG method by imposing a time limit of $600$s when solving the master problems \citep{Tsang_et_al:2023}. Unless stated otherwise, we include the proposed symmetry-breaking constraints in all models and the proposed VIs to both SP and DRO models. We conducted all the experiments on a computer with an Intel Xeon Silver processor 2.10 GHz CPU and 128 Gb memory. \vspace{-5mm}

\subsection{Computational Time} \label{subsec:expt_comp_time}

In this section, we analyze the computational times for solving our proposed models. For each instance and cost structure, we solve the \spe{} and \spcvar{} models with $20$ generated SAA instances, while we solve the \droe{} and \drocvar{} models using the lower and upper bounds of surgery durations with $10$ different means generated from a uniform distribution on $[0.9m_i,1.1m_i]$. Table~\ref{table:CPU_time} presents the average solution times in seconds under cost 1.  Throughout this section, for large instances marked with `$\dag$' for the \droe{} model, we apply VIs \eqref{eqn:global_LB} and \eqref{eqn:VI_dual_bounds} with initial scenario $m$ that produces shorter computational times. 

We first observe that solution time increases as the size of the ORASP instance increases.  Second, we can solve all instances using the \spe{}, \droe{}, and \drocvar{} models within a reasonable time.  In fact, we can solve medium-sized instances in less than three minutes while the solution times of large instances range from two minutes to an hour.   Third, solution times of the \droe{} model are slightly longer than the \spe{} model. This is reasonable as the master problem of large ORASP instances is a large-scale scenario-based MILP and its size increases with the number of C\&CG iterations.  Fourth, solution times of the \drocvar{} model, ranging from 1 second to 100 seconds, are significantly shorter than other models. One possible explanation is that the worst-case scenarios that maximize the CVaR objective are always those with long surgery durations. Thus, our C\&CG method quickly identifies these scenarios and terminates in a small number of iterations.  On the other hand, solution times of the \spcvar{} model are the longest among all models, and it cannot solve larger instances (i.e., instances 5 and 6). This is expected since the \spcvar{} model is not separable (see Section \ref{subsec:separability}). Moreover, solving SP problems with CVaR objectives is known to be challenging. Finally, we remark that solution times are generally longer under costs~2 and 3, where we also consider the idle time in the objective (see \ref{appdx:expt_comp_time} for detailed results). Nevertheless, the average solution times of the \spe{} and \droe{} models are within $2$ hours, and that of the \drocvar{} model is within $3$ hours. According to our clinical collaborators, these solution times are reasonable; i.e., our proposed models are tractable for practical purposes. \blue{In \ref{appdx:comp_time_mutliOR}, we present computational times for another set of ORASP instances where some ORs accommodate multiple surgery types. The solution times of such ORASP instances remain reasonable.}

%
\begin{table}[t]\centering\OneAndAHalfSpacedXI
\footnotesize   
\caption{Computational time (in s) with cost 1 (instance with ``$\dag$'': apply \eqref{eqn:global_LB} and \eqref{eqn:VI_dual_bounds} with initial scenario $m$; instance with ``--'': cannot be solved within $10$ hours)} \label{table:CPU_time}
\ra{0.7}  
\begin{tabular}{@{}l|llllllllllll@{}} \toprule
         & Instance 1 & Instance 2 & Instance 3 & Instance 4 & Instance 5 & Instance 6 \\ \midrule
\spe{}       & 1.53       & 9.41       & 3.90       & 21.01      & 109.25     & 1496.49    \\[0.5ex]
\spcvar{}     & 12.54      & 491.83     & 18.14      & 971.94     & --         & --        \\
\droe{}      & 7.20       & 28.21      & 14.14      & 102.95     & 152.11     & 2505.44$^\dag$    \\ [0.5ex]
\drocvar{} & 1.64       & 2.31       & 2.53       & 5.61       & 16.67      & 99.59      \\
\bottomrule
\end{tabular}
\end{table}



\blue{Finally, per a reviewer's suggestion, we also investigate the computational performance of an extension of \cite{Rath_et_al:2017}'s RO model that incorporates all the elements of the ORASP. In \ref{appdx:RO_ORASP},  we present the extended RO model for the ORASP and develop a decomposition algorithm to solve it based on the one presented by \cite{Rath_et_al:2017}. Computational results in \ref{appdx:RO_ORASP} demonstrate that the RO approach takes substantially longer time to solve even small instances of the ORASP. For example, we could not solve instance 2 using the RO decomposition algorithm within the imposed two-hour time limit, and the relative MIP gap at termination ranges from 36\% to 73\%. These results suggest that the proposed SP and DRO approaches for the ORASP are more computationally efficient than the RO approach, further emphasizing our contributions.}

\subsection{Efficiency of Valid Inequalities and Symmetry-Breaking Constraints} \label{subsec:expt_VI_SBC}

In this section, we demonstrate the efficiency  of the proposed valid inequalities (VIs) and symmetry-breaking constraints (SBCs). For brevity and illustrative purposes, we present results for the \spe{} and \droe{} models only.  First, we analyze the impact of including VIs \eqref{eqn:idle_nonneg} in the \spe{} model. For each ORASP instance, we solve $20$ generated SAA instances with (w/) and without (w/o) VIs \eqref{eqn:idle_nonneg}. Table~\ref{table:VI_SP} presents the average solution time w/ and w/o these VIs, and the average ratio of the optimal objective values of LP relaxations of the \spe{} model w/ VIs and w/o VIs. In general, solution times are longer on average w/o these VIs, and the differences in solution times are more significant for large instances. For example, the percentage increase in average solution time ranges from $45\%$ to $75\%$ for large ORASP instances. We attribute the difference in solution time to a weaker LP relaxation of the \spe{} model w/o these VIs. It is clear from Table~\ref{table:VI_SP} the LP relaxations w/ these VIs are strictly tighter, and they can be up to $3$ times the LP relaxations w/o these VIs. These results demonstrate the efficiency of VIs  \eqref{eqn:idle_nonneg} for the \spe{} model. 
\begin{table}[t]\centering\small \OneAndAHalfSpacedXI
\footnotesize
\caption{Average solution time (in s) of the \spe{} model with (w/) and without (w/o) VIs, and average ratio of optimal objective values of LP relaxations of the \spe{} model w/ VI to that w/o VIs}\label{table:VI_SP}
\ra{0.8}  
\begin{tabular}{@{}l|lllllll@{}} \toprule
                    & Instance 1 & Instance 2 & Instance 3 & Instance 4 & Instance 5 & Instance 6 \\ \midrule
Time (w/ VIs)        & 1.53       & 9.41       & 3.90       & 21.01      & 109.25     & 1496.49    \\
Time (w/o VIs)       & 1.57       & 9.48       & 5.47       & 22.85      & 159.67     & 2603.40    \\ \midrule
LP Relaxation Ratio & 1.12       & 1.30       & 1.07       & 1.59       & 2.08       & 2.93   \\
\bottomrule
\end{tabular}
\end{table}

Next, we analyze the impact of including the proposed VIs in Section \ref{subsec:VI_DRO} in the master problem of the C\&CG method for the \droe{} model. Table~\ref{table:VI_DRO} presents the solution time and number of iterations of our C\&CG method w/ and w/o these VIs. While solution times w/ and w/o VIs are similar for small to medium-sized instances, solution times w/ VIs are significantly shorter for large instances. This is because C\&CG w/o these VIs takes a considerably larger number of iterations to converge w/o VIs.  For example, the number of iterations w/o VIs is doubled for instance 6. In addition, we observe that the lower bound (LB) and optimality gap converges faster when we introduce our proposed VIs into the master problem. To illustrate this, we provide Figure~\ref{fig:VI_DRO}, which presents the LB and optimality gap w/ and w/o these VIs in one subproblem of instance 4. We observe that due to the better bounding effect from these VIs, both LB and the optimality gap converge in a smaller number of iterations when we introduce the proposed VIs. We also note that without these VIs, some large instances cannot be solved. For example, instance 6 under cost 2 cannot be solved within 3 hours without these VIs.
\begin{table}[t]\centering\small \OneAndAHalfSpacedXI
\footnotesize
\caption{Solution time (in s) of the \droe{} model with (w/) and without (w/o) VIs, and number of iterations w/ and w/o VIs (instance with $\dag$: apply \eqref{eqn:global_LB} and \eqref{eqn:VI_dual_bounds} with initial scenario $m$)}\label{table:VI_DRO}
\ra{0.8}  
\begin{tabular}{@{}l|llllll@{}} \toprule
                  & Instance 1 & Instance 2 & Instance 3 & Instance 4 & Instance 5 & Instance 6 \\ \midrule
Time (w/ VIs)      & 8.92       & 34.24      & 17.48      & 94.19      & 115.28     & 2501.97$^\dag$    \\
Time (w/o VIs)     & 7.94       & 22.99      & 21.25      & 93.64      & 179.84     & 5300.02    \\ \midrule
No. Iter (w/ VIs)  & 22         & 29         & 34         & 56         & 56         & 62$^\dag$         \\
No. Iter (w/o VIs) & 23         & 31         & 46         & 79         & 93         & 150      \\
\bottomrule
\end{tabular}
\end{table}
\begin{figure}  \OneAndAHalfSpacedXI
    \centering
    \includegraphics[scale=0.65]{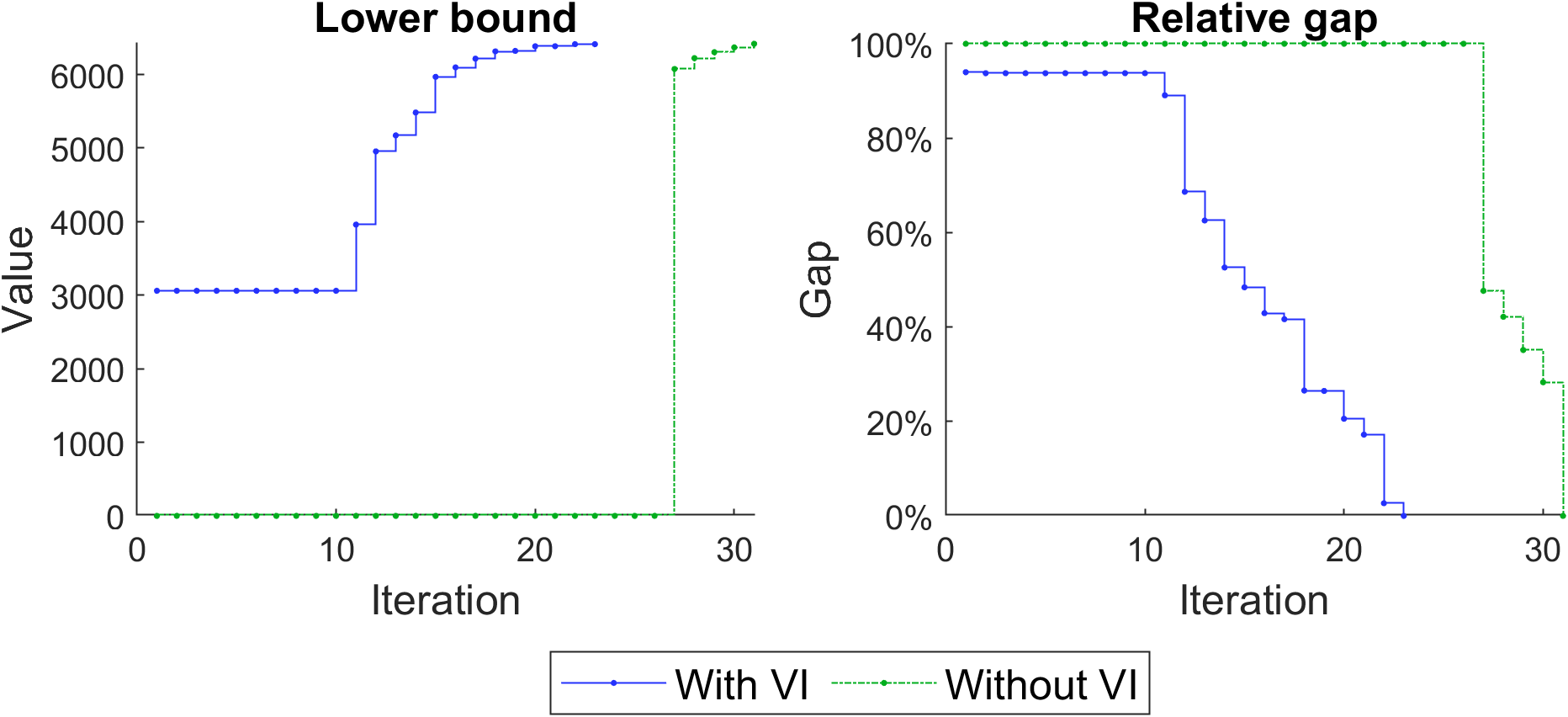}
    \caption{Lower bound and relative gap over iteration with and without VIs in the DRO-E model (Instance 4)} \label{fig:VI_DRO}
\end{figure}

Finally, we study the efficiency of the proposed SBCs. We only focus on the \spe{} model for brevity as the results are similar for the \droe{} model. Given the challenges of solving ORASP instances w/o these SBCs, we use instance 1 in this experiment for illustrative purposes. We first generate 20 sets of scenarios for this instance with number of scenarios $N \in \{10, 20, 50, 100, 200\}$. Then, we separately solve the \spe{} model w/ and w/o these SBCs. Figure \ref{fig:SBC_SP} illustrates the solution time for different $N$, where solid and dashed lines represent the average solution times w/ and w/o SBCs, respectively. The shaded regions are the corresponding $20$th and $80$th percentiles. We observe that solution times are significantly longer without SBCs. Specifically, using our SBCs, we can solve the generated instances 5 to 145 times faster. Indeed, without our SBCs, medium and large SP model instances with $N=100$ cannot be solved within one hour. These results demonstrate the importance of breaking symmetry in the first-stage decisions and the effectiveness of our SBCs.
\begin{figure}  \OneAndAHalfSpacedXI
    \centering
    \includegraphics[scale=0.65]{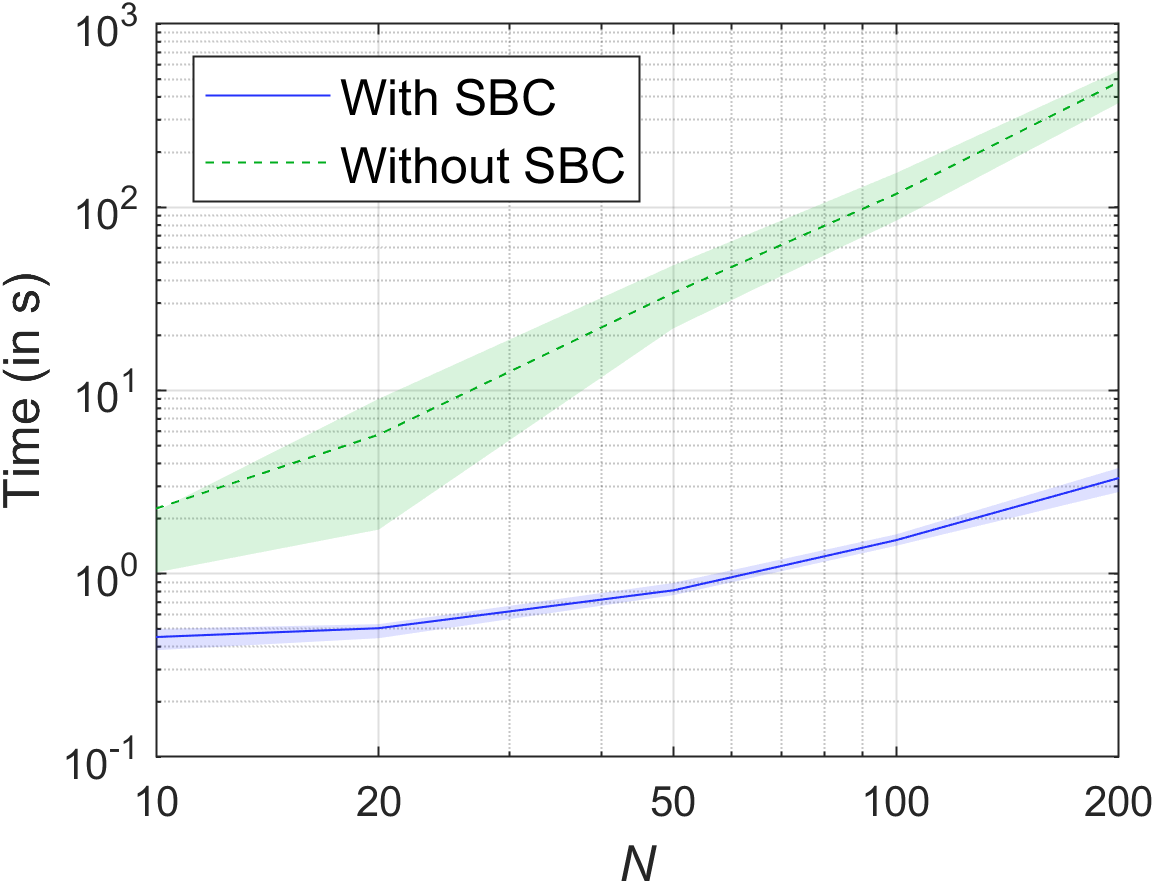}
    \caption{Solution time (in s) of the SP-E model with and without SBCs (Instance 1)}    \label{fig:SBC_SP}
\end{figure}

\subsection{Analysis of Optimal Solutions} \label{subsec:expt_opt_sol}

In this section, we compare the optimal solutions of the proposed \spe{}, \spcvar{}, \droe{}, and \drocvar{} models. \blue{For illustrative purposes and brevity, we present optimal solutions to the \spe{} and \spcvar{} models from one SAA replication (the observations across different SAA replications are similar; see \ref{appdx:expt_analysis_SAA_rep}).} 

Let us first analyze the number of ORs opened by these models presented in Figure~\ref{fig:expt_num_OR}. Under cost~1, the \drocvar{} and \spe{} models open the largest and smallest number of ORs, respectively. Similarly, we observe that the \drocvar{} and \spe{} models call in the largest and smallest number of on-call anesthesiologists, respectively. Under costs 2--3, which include the OR idling cost, all models open fewer ORs than under cost 1 to avoid excessive OR idle time. However, the \droe{} and \drocvar{} models open more ORs than the \spe{} and \spcvar{} models, leading to a smaller number of surgeries scheduled in each OR in general. As we show in the next section, by opening more ORs, scheduling fewer surgeries in each OR, and employing additional on-call anesthesiologists, the \droe{} and \drocvar{} models intend to mitigate surgery delays that may accumulate due to a tighter schedule with fewer ORs and yield a shorter waiting time when compared with \spe{} and \spcvar{} models.

\begin{figure} \OneAndAHalfSpacedXI
    \centering
    \includegraphics[scale=0.7]{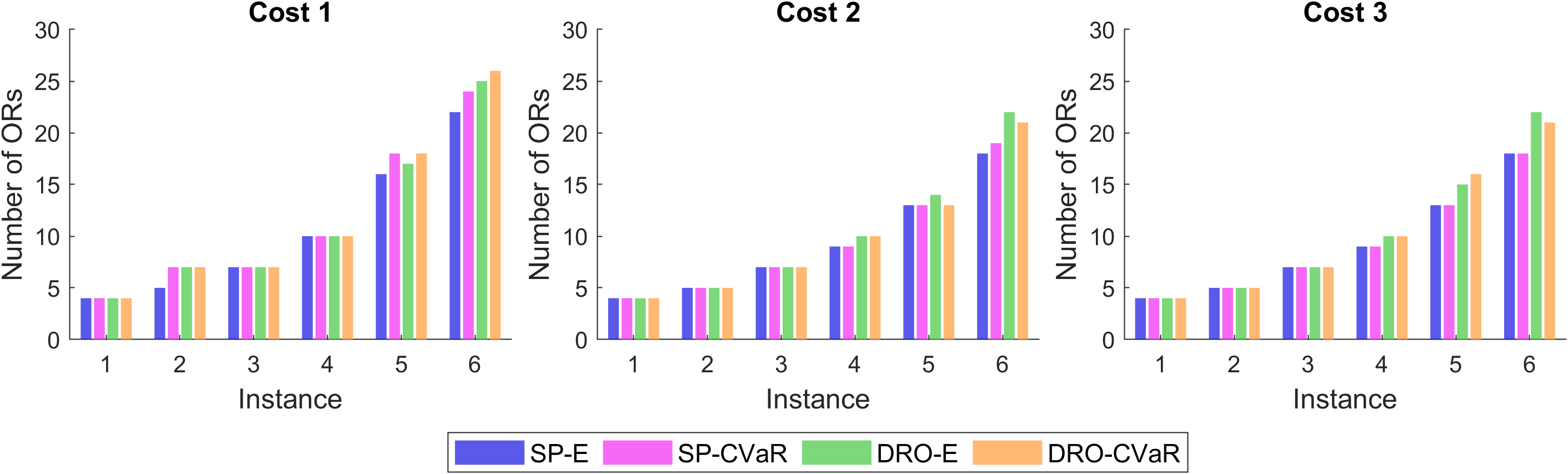}
    \caption{Number of ORs opened for different instances under different costs} \label{fig:expt_num_OR}
\end{figure}

Next, we analyze the structure of the optimal schedules obtained from each model. \blue{We note that the optimal schedules for different ORs in most ORASP instances follow a similar pattern. Hence, we present results for OR~1 in instance 3 under cost 1 for brevity and illustrative purposes, where all models schedule four surgeries.}   Figure~\ref{fig:schedule_OR1} presents the time allocated to each surgery (i.e., the difference between the scheduled start time of a surgery and its subsequent surgery). The two dotted lines in this figure represent the minimum and maximum surgery durations.


\begin{figure} \OneAndAHalfSpacedXI
    \centering
    \includegraphics[scale=0.65]{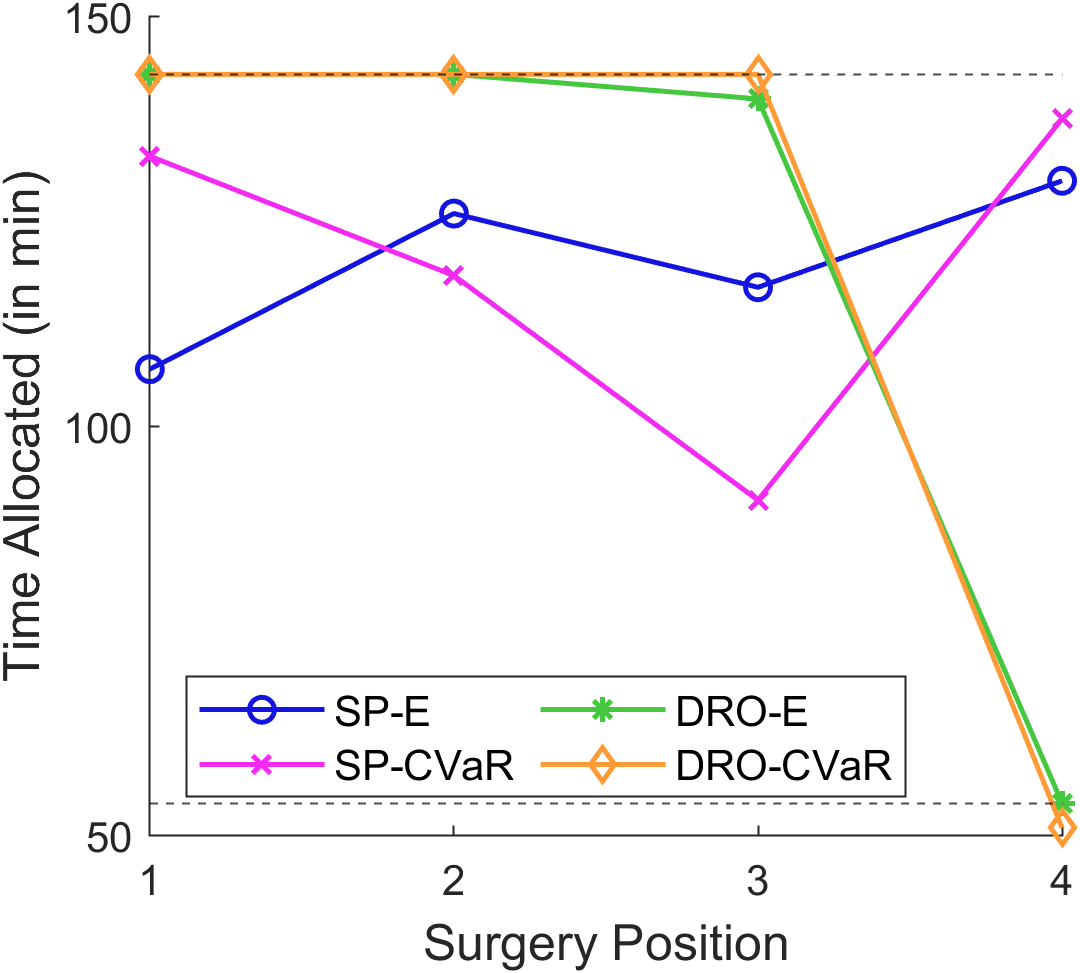}
    \caption{Illustration of the optimal schedules for OR 1 in instance 3. The two dotted lines indicate the minimum and maximum surgery durations.}
    \label{fig:schedule_OR1}
\end{figure}

We observe the following from Figure~\ref{fig:schedule_OR1}. First, the \drocvar{} and \droe{} models intend to protect against the risk of surgery delays that may accumulate due to long surgery durations by allocating longer times to the first three surgeries than the other models (also reflected by shorter waiting times from the \drocvar{} and \droe{} models reported in the next section). Specifically, the \drocvar{} model allocates the maximum surgery duration to surgery 1--3, leaving a shorter time for the last surgery than the other models.  The \droe{} model allocates the maximum surgery duration to surgery 1--2 and slightly less (more) time to surgery 3 (the last surgery) than the \drocvar{} model. Second, the \spcvar{} model allocates a longer time to the first surgery than the \spe{} model and a longer time to the last surgery than the other models, potentially leading to a smaller overtime. In contrast, time allowances in the \spe{} schedule exhibit a zigzag pattern with less (more) time allocated to surgery~1 (the last surgery) than other surgeries.

\subsection{Analysis of Solution Quality} \label{subsec:expt_sol_quality}

In this section, we analyze the operational performance of the optimal schedules via out-of-sample simulation under cost 1 (we provide similar results for other cost structures in \ref{appdx:expt_sim_results}). Specifically, we first generate four sets of $N'=10,000$ out-of-sample scenarios under various distributions, summarized in Table \ref{table:OS_dist}. In setting I, we assume perfect distributional information. That is, we generate $N'$ samples from the same distribution (logN) that we use in the optimization. In settings II-IV, we vary the surgery duration distributions to study the impact of misspecified distributional information, i.e., when the in-sample scenarios do not accurately reflect the true distribution (see \citealp{Shehadeh:2022,Wang_et_al:2020}). Specifically, in setting II, we use a truncated normal distribution with the estimated mean $m$ and variance $\sigma^2$ on $[(1-\Delta)\dlb_i,(1+\Delta)\dub_i]$ with $\Delta \in \{0, 0.25, 0.5\}$. A larger value of $\Delta$ corresponds to a higher variation level and $\Delta=0$ indicates that only the distribution is perturbed with the same support.  In setting III, we use a uniform distribution $U[(1-\Delta)\dlb_i,(1+\Delta)\dub_i]$ with $\Delta\in\{0,0.25,0.5\} $. We denote simulations under setting II (and similarly for setting III) with $\Delta=0$, $\Delta=0.25$, and $\Delta=0.5$ as IIa, IIb, and IIc, respectively.  Finally, in setting IV, we use a beta distribution with the same mean and variance on $[0.5\dlb_i,1.5\dub_i]$. These settings are motivated by our clinical collaborators, who observe significant changes in distribution and range of actual surgery durations between different time frames (e.g., month, year). Our analysis of \cite{Mannino_et_al:2010}'s data also suggest an annual change in the lower and upper bounds of actual surgery durations ranging from $-17\%$ to $31\%$. Second, we solve the second-stage problem with the generated scenarios to compute the out-of-sample performance metrics (i.e., overtime and waiting time). For the sake of brevity in our presentation, we discuss the results for instances 2 and 6; we have similar observations about the results for the other instances.
\begin{table}[t]\centering\small
\footnotesize
\caption{Out-of-sample distributions} \label{table:OS_dist}
\ra{0.4}  
\begin{tabular}{@{}l|l@{}} \toprule
Setting & Distribution of $d_i$ for $i\in I$ \\ \midrule
I       & Same distribution as the in-sample scenarios, i.e., lognormal \\
II      & Truncated normal distribution with mean $m_i$, variance $\sigma_i^2$ on $[(1-\Delta)\dlb_i,(1+\Delta)\dub_i]$, where $\Delta\in\{0,0.25,0.5\} $\\
III     & Uniform distribution on $[(1-\Delta)\dlb_i,(1+\Delta)\dub_i]$, where $\Delta\in\{0,0.25,0.5\} $\\
IV      & Beta distribution with same mean $\mu_i$ and variance $\sigma_i^2$ on $[0.5\dlb_i,1.5\dub_i]$ \\
\bottomrule
\end{tabular}
\end{table}

We first analyze the out-of-sample values of the operational metrics. Table \ref{table:OS_performance_metric} shows the average waiting time and OR overtime under different distributional settings (observations for anesthesiologist overtime are similar to those for OR overtime). It is clear that the \droe{} and \drocvar{} schedules generally yield significantly shorter waiting times and slightly longer overtime than the \spe{} and \spcvar{} models under all simulation settings. Furthermore, the \spcvar{} schedules lead to shorter waiting times than the \spe{} model and the shortest overtime in most settings. Under misspecified distributional settings II--IV (i.e., when the true distribution is different from the in-sample distribution used to generate the data for optimization),  the waiting time and overtime are generally longer than the perfect distributional setting I. Notably, the performance of the \spe{} solutions significantly deteriorates with longer waiting times and OR overtime. Finally, when the true distribution significantly deviates from the in-sample distribution (e.g., setting IIIc), the \droe{} and \drocvar{} models produce shorter or approximately the same overtime as the \spcvar{} model.

\begin{table}[t]\centering\small
\footnotesize
\caption{Average out-of-sample waiting time and OR overtime under settings I--IV (Instances 2 and 6)} \label{table:OS_performance_metric}
\ra{0.47}  
\begin{tabular}{@{}l|rrrr|rrrr@{}} \toprule
\textbf{Waiting   Time} & \multicolumn{4}{c|}{\textbf{Instance 2}}                                                                    & \multicolumn{4}{c}{\textbf{Instance 6}}                                                                    \\ [0.25ex]
                        & \multicolumn{1}{c}{\,\,\,\,\spe{}\,\,\,\,} & \multicolumn{1}{c}{\,\,\spcvar{}\,\,} & \multicolumn{1}{c}{\,\,\droe{}\,\,} & \multicolumn{1}{c|}{\drocvar{}} & \multicolumn{1}{c}{\,\,\,\,\spe{}\,\,\,\,} & \multicolumn{1}{c}{\,\,\spcvar{}\,\,} & \multicolumn{1}{c}{\,\,\droe{}\,\,} & \multicolumn{1}{c}{\drocvar{}} \\ \midrule \rowcolor{lightgray}
Setting I               & 259                    & 91                       & 0                       & 0                            & 346                    & 181                      & 2                       & 0                            \\
Setting IIa             & 250                    & 69                       & 0                       & 0                            & 248                    & 135                      & 0                       & 0                            \\ \rowcolor{lightgray}
Setting IIb             & 469                    & 176                      & 62                      & 69                           & 768                    & 464                      & 227                     & 200                          \\ 
Setting IIc             & 567                    & 235                      & 123                     & 136                          & 1175                   & 769                      & 491                     & 439                          \\ \rowcolor{lightgray}
Setting IIIa            & 257                    & 74                       & 0                       & 0                            & 267                    & 142                      & 0                       & 0                            \\ 
Setting IIIb            & 620                    & 247                      & 110                     & 121                          & 1093                   & 669                      & 381                     & 336                          \\ \rowcolor{lightgray} 
Setting IIIc            & 1026                   & 487                      & 350                     & 376                          & 2347                   & 1614                     & 1219                    & 1113                         \\ 
Setting IV              & 546                    & 258                      & 167                     & 186                          & 1472                   & 1027                     & 733                     & 670                          \\ \midrule
\textbf{OR Overtime}    & \multicolumn{4}{c|}{\textbf{Instance 2}}                                                                    & \multicolumn{4}{c}{\textbf{Instance 6}}                                                                    \\  [0.25ex]
                        & \multicolumn{1}{c}{\,\,\,\,\spe{}\,\,\,\,} & \multicolumn{1}{c}{\,\,\spcvar{}\,\,} & \multicolumn{1}{c}{\,\,\droe{}\,\,} & \multicolumn{1}{c|}{\drocvar{}} & \multicolumn{1}{c}{\,\,\,\,\spe{}\,\,\,\,} & \multicolumn{1}{c}{\,\,\spcvar{}\,\,} & \multicolumn{1}{c}{\,\,\droe{}\,\,} & \multicolumn{1}{c}{\drocvar{}} \\ \midrule \rowcolor{lightgray}
Setting I               & 116                    & 41                       & 69                       & 69                            & 79                    & 77                      & 215                       & 147                          \\                        
Setting IIa             & 115                    & 28                       & 71                      & 71                           & 35                     & 47                       & 225                     & 154                          \\ \rowcolor{lightgray}
Setting IIb             & 212                    & 91                       & 113                     & 110                          & 249                    & 213                      & 347                     & 271                          \\ 
Setting IIc             & 256                    & 128                      & 141                     & 136                          & 433                    & 360                      & 458                     & 383                          \\ \rowcolor{lightgray}
Setting IIIa            & 118                    & 29                       & 70                      & 70                           & 41                     & 51                       & 221                     & 151                          \\ 
Setting IIIb            & 280                    & 134                      & 140                     & 136                          & 412                    & 331                      & 437                     & 352                          \\ \rowcolor{lightgray} 
Setting IIIc            & 480                    & 304                      & 290                     & 271                          & 1114                   & 893                      & 896                     & 805                          \\ 
Setting IV              & 241                    & 144                      & 153                     & 144                          & 544                    & 469                      & 525                     & 461                         \\
\bottomrule
\end{tabular}
\end{table}

To further illustrate the impact of model misspecification, we compare out-of-sample costs under setting III with $\Delta\in\{0,0.25,0.5\}$ using instance 6. (Recall that a larger $\Delta$ corresponds to a larger extent of misspecification).  First, Figure~\ref{fig:OS_OpCost_CDF_settingIII} shows the out-of-sample distributions of the operational (i.e., second-stage) costs for instance 6. When $\Delta=0$, since the distributional change is mild, the operational costs of the \spe{} and \spcvar{} schedules are generally lower than that of the \droe{} and \drocvar{} schedules. When $\Delta\in\{0.25,0.5\}$, i.e., deviations from the in-sample distribution are large, the \spcvar{}, \droe{}, and \drocvar{} schedules yield significantly lower operational costs than the \spe{} schedule on average and at all quantiles. Notably, the \drocvar{} schedules lead to the lowest operational costs under large $\Delta$, and the \droe{} model leads to lower operational costs than the \spcvar{} model when $\Delta=0.5$. These results demonstrate the robustness of the \spcvar{}, \droe{}, and \drocvar{} solutions against distributional changes and the potential operational benefits of adopting the \droe{} and \drocvar{} solutions when the true distribution significantly deviates from the in-sample distribution.

\begin{figure}[t] \OneAndAHalfSpacedXI
    \centering
    \includegraphics[scale=0.72]{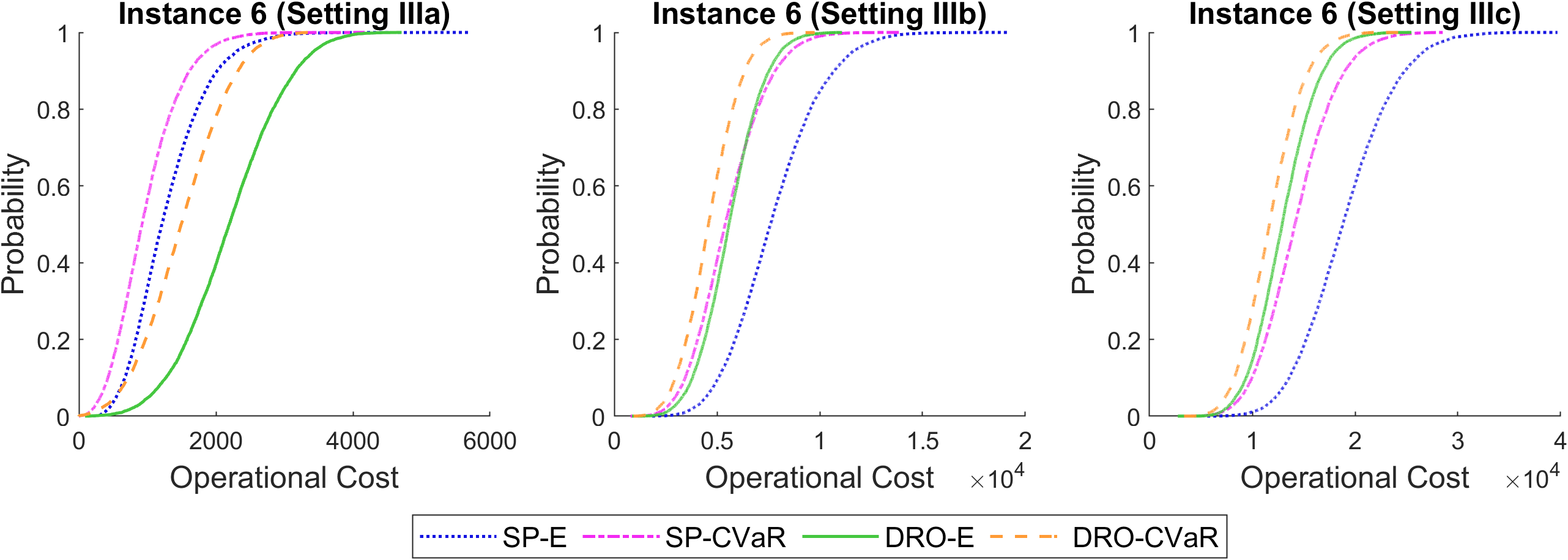}
    \caption{Out-of-sample cumulative distribution function of the operational cost (Instance 6)}   \label{fig:OS_OpCost_CDF_settingIII}
\end{figure}

Second, Figure \ref{fig:OS_CDF_instance6_settingIII} shows the distributions of the total cost, as a sum of the fixed and operational costs, under setting III with $\Delta\in\{0,0.25,0.5\}$. When $\Delta=0$, both \spe{} and \spcvar{} schedules produce lower total costs than the \droe{} and \drocvar{} schedules since the \spe{} and \spcvar{} models open fewer ORs. However, when $\Delta\in\{0.25,0.5\}$, i.e., deviations from the in-sample distribution are large, the \droe{} and \drocvar{} models yield lower total costs at upper quantiles, especially when $\Delta=0.5$. These results further illustrate that \droe{} and \drocvar{} models can protect against distributional ambiguity and show the trade-off between fixed and operational costs. Specifically, by opening more ORs, the \droe{} and \drocvar{} models incur higher fixed costs, but significantly smaller operational costs. We note that while the fixed cost is a one-time cost (i.e., fixed once the ORs are open and on-call anesthesiologists are called in), the operational cost represents recurring long-run costs. With reference to our results, practitioners could decide which model to adopt based on their preferences and actual situations.
\begin{figure}[t] \OneAndAHalfSpacedXI
    \centering
    \includegraphics[scale=0.72]{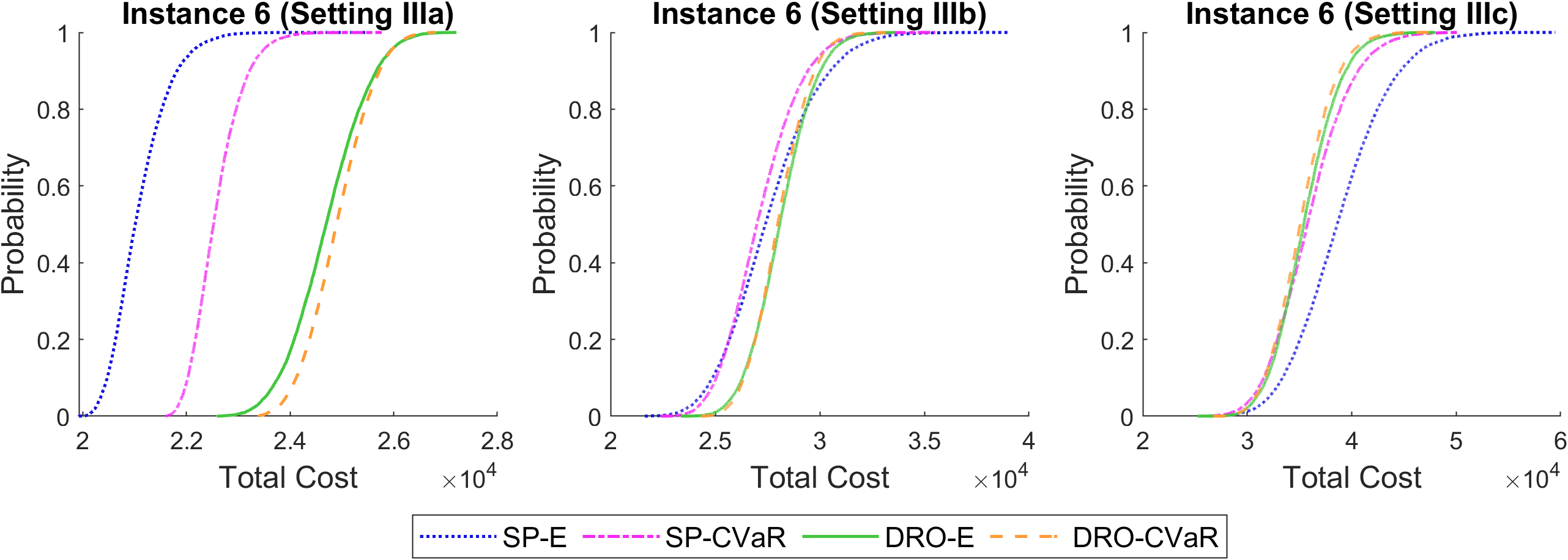}
    \caption{Out-of-sample cumulative distribution function of the total cost (Instance 6)}   \label{fig:OS_CDF_instance6_settingIII}
\end{figure}

Finally, we investigate the value of distributional robustness from the perspective of out-of-sample disappointment, which measures the extent to which the out-of-sample cost exceeds the model's optimal value \citep{Van-Parys_et_al:2021}. Let $\texttt{V}^{\texttt{opt}}$ and $\texttt{V}^{\texttt{out}}$ be the model's optimal value and the out-of-sample objective value, respectively. That is, $\texttt{V}^{\texttt{opt}}$ and $\texttt{V}^{\texttt{out}}$ can be viewed as the estimated and actual costs of implementing the model's optimal solutions, respectively. Then, we define the out-of-sample disappointment as $\max \{(\texttt{V}^{\texttt{out}}-\texttt{V}^{\texttt{opt}})/\texttt{V}^{\texttt{opt}},\, 0 \}$. A disappointment of zero implies that $\texttt{V}^{\texttt{out}} \leq\texttt{V}^{\texttt{opt}}$, which in turn indicates that the model is more conservative and avoids underestimating costs. Figure~\ref{fig:OS_disap} presents the distributions of the out-of-sample disappointments for instance 6. Notably, the \drocvar{} model yields significantly smaller out-of-sample disappointments at all quantiles. Moreover, the out-of-sample disappointment of the \drocvar{} model is the most stable with the smallest standard deviation. On the other hand, disappointments of the \spe{} model are significantly larger than all models, especially at the upper quantiles (e.g., exceeding 100\%). Both \droe{} and \spcvar{} models yield smaller disappointments than the \spe{} model, with the \droe{} model having smaller disappointments than the \spcvar{} model under settings IIIc and IV. These results demonstrate that the \spcvar{}, \droe{}, and \drocvar{} models provide a more robust estimate of the actual cost that the hospital will incur in practice. Thus, risk-averse OR managers who seek a robust financial plan and operational performance may prefer these solutions.

\begin{figure} \OneAndAHalfSpacedXI
    \centering
    \includegraphics[scale=0.72]{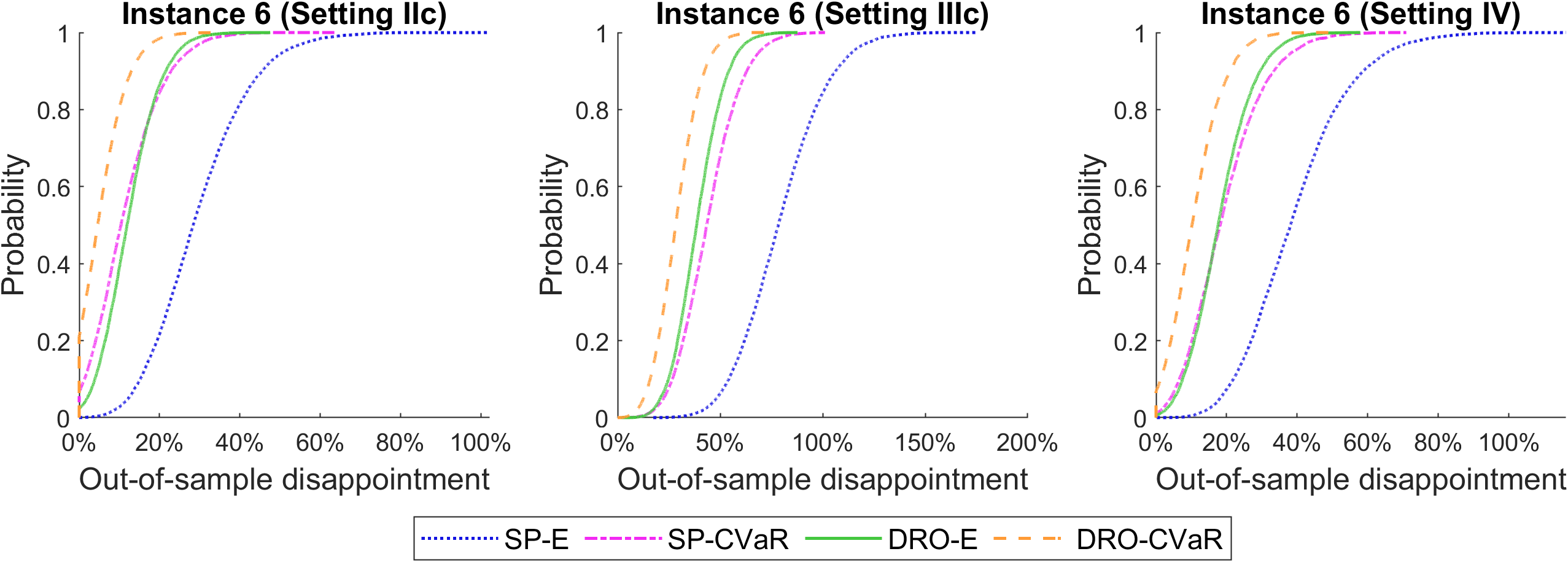}
    \caption{Cumulative distribution function of the out-of-sample disappointment (Instance 6) \vspace{-3mm}}
    \label{fig:OS_disap}
\end{figure}

\subsection{Comparison with Non-Integrated Approaches} \label{subsec:expt_Rath}

In this section, we present results comparing the operational and computational performance of solutions to our proposed models \blue{with related non-integrated approaches,} those of \cite{Rath_et_al:2017}'s RO model and a sequential approach detailed later in this section. We focus on our \spe{} and \droe{} models only for the sake of brevity in our presentation.

First, we compare the performance of our solutions with solutions to \cite{Rath_et_al:2017}'s RO model. Recall that \cite{Rath_et_al:2017}'s model does not include idle, waiting time components, and other components, and thus can schedule multiple surgeries to start simultaneously; see \ref{appdx:expt_Rath_et_al} for details of this model. For a fair comparison, we include symmetry-breaking constraints for all models. To compare the operational performance of the optimal solutions, we solve instance 1 under cost 1 using the three models. (The observations are similar for other ORASP instances.) In particular, we solve \cite{Rath_et_al:2017}'s RO model by their proposed decomposition method with different sizes $\tau\in\{0.1,0.2,0.4\}$ of the uncertainty set adopted in their paper. We provide a detailed comparison of the optimal solutions in \ref{appdx:expt_Rath_et_al}. We highlight that the RO model schedules all surgeries at time zero \blue{(which is consistent with our theoretical results in \ref{appdx:models_wo_waiting})}, while our models do not schedule surgeries performed by the same anesthesiologist or in the same OR to start at the same time. Moreover, as shown in \ref{appdx:expt_Rath_et_al}, the RO model assigns more surgeries to some anesthesiologists, which leads to the possibility of larger overtimes when compared with our proposed models. Figure~\ref{fig:OS_Rath_et_al_wait} shows the histograms of the out-of-sample waiting time associated with the optimal schedules under setting I. These plots indicate that solutions from \cite{Rath_et_al:2017}'s RO model lead to significantly larger waiting times than our \spe{} and \droe{} models. These results demonstrate the importance of our proposed generalization of the second-stage problem in the ORASP. \textcolor{black}{Finally, although this model does not include all components of the ORASP, it is challenging to solve (see \ref{appdx:expt_Rath_et_al}).}


Second, we also compare our proposed \spe{} model with a sequential approach that separates the OR assignment decisions from the remaining decisions (i.e., anesthesiologist assignment, sequencing, and scheduling decisions). Specifically, in the sequential approach, we first solve \cite{Denton_et_al:2010}'s classical surgery-to-OR assignment model to obtain $(v^\star,z^\star)$; see \ref{appdx:expt_Rath_et_al} for the formulation. Then, we solve our ORASP model by fixing ($v, z$) to $(v^\star,z^\star)$. We follow the same experiment settings in Section \ref{subsec:expt_description} to solve instances 1--6 under cost 1 using the two approaches.  Table \ref{table:sep_OR_assg_numOR} in \ref{appdx:expt_Rath_et_al}  presents the number of ORs opened, and Table \ref{table:sep_OR_assg_OS} summarizes the associated average out-of-sample waiting time, OR overtime, and operational costs. We observe that the sequential approach opens the same or smaller number of ORs than our \spe{} model. Moreover, the former results in a packed schedule, thus leading to longer waiting times, OR overtime, and, consequently, higher operational costs. In particular, the operational cost of the sequential approach is two times higher than that of our \spe{} model for large instances (e.g., instances 5--6). These results suggest that our integrated approach could yield better operational performance than the sequential approach.

\begin{figure} \OneAndAHalfSpacedXI
    \centering
    \includegraphics[scale=0.65]{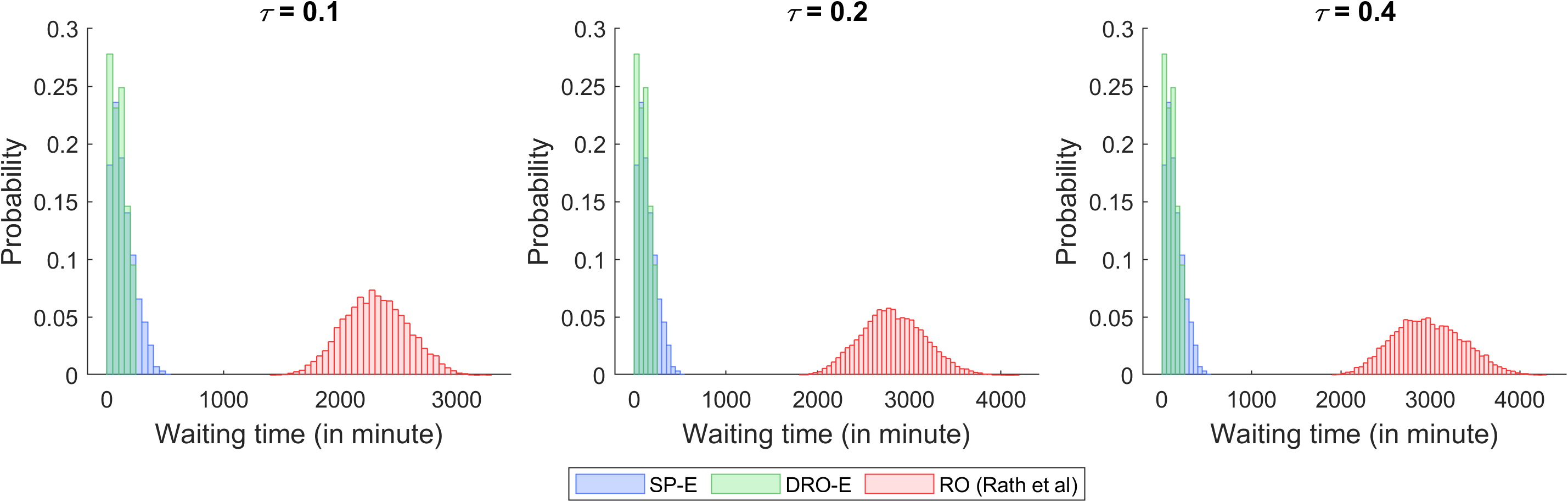}
    \caption{Waiting time from three different models in instance 1 ($\tau$ controls the size of RO uncertainty set) \vspace{-2mm}}
    \label{fig:OS_Rath_et_al_wait}
\end{figure}

\subsection{Case Study} \label{subsec:expt_real_data}
In this section, we present a case study based on questions asked by our collaborators and information from their health systems. Specifically, we use the master schedule of a day in September~2021 to examine the sensitivity of the proposed models' optimal solutions to the cost parameters in the objective (which they can easily adjust) and determine the number of surgeries to schedule. This master schedule consists of $29$ ORs (all open) and $14$ surgery types. There are $35$ anesthesiologists, $8$ of whom are on call. In what follows, we use the same experimental setup described in Section \ref{subsec:expt_description}. We estimate $m$ and $\sigma$ of duration $d$ from the data provided by our collaborators.

First, we investigate the impact of the waiting cost parameter $\cw_i$ on the optimal schedule. Specifically, we keep other cost parameters as in cost 1 and solve the models with $\cw_i\in\{100,200,300,400\}$. For brevity, in Figure~\ref{fig:sen_waiting}, we illustrate the optimal time allowances for four colorectal surgeries assigned to one of the ORs.  (We observe similar results for other surgery types and ORs). As $\cw_i$ increases, the \spe{} and \spcvar{} models assign more time to each surgery except the last one. In contrast, the \droe{} and \drocvar{} models appear to be less sensitive to $\cw_i$, and, in particular, time allowances are the same except for $\cw_i=400$ in the \droe{} model. This is because the \droe{} and \drocvar{} schedules already assign longer time to each surgery to hedge against waiting time. Given that the OR service hour is fixed, there is less or no room to assign more time to each surgery when $\cw_i$ increases in the \droe{} and \drocvar{} schedules.

\begin{figure} \OneAndAHalfSpacedXI
     \centering
     \includegraphics[scale=0.56]{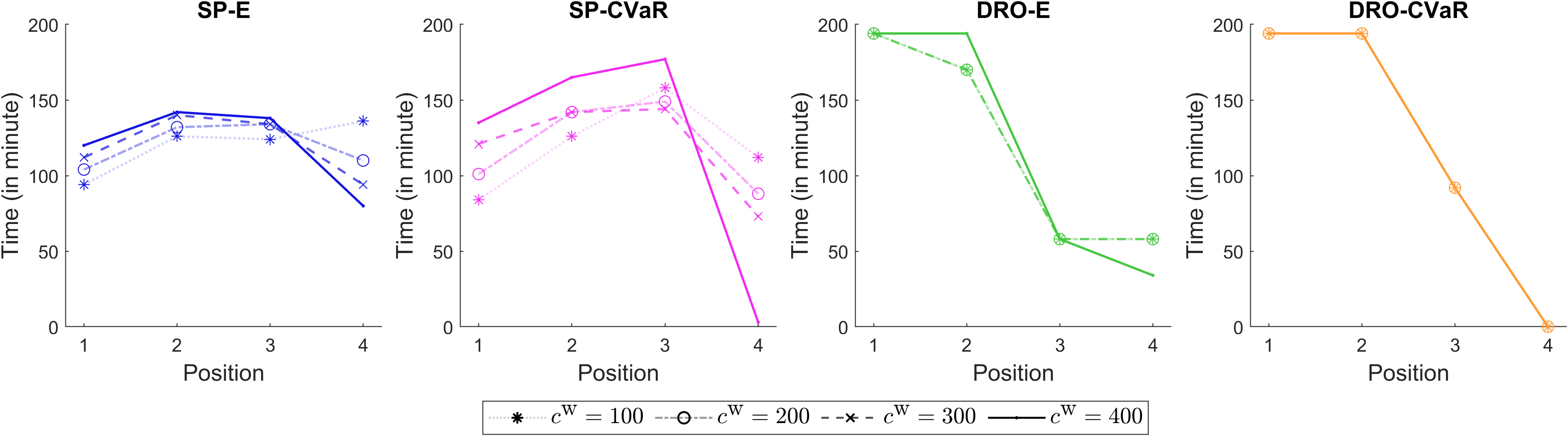}
     \caption{Time allowance to each surgery in one of the ORs with different $\cw_i \in \{100, 200, 300, 400\}$.}  \label{fig:sen_waiting}
\end{figure}

Second, we investigate the impact of $\co_a$ on the optimal schedule. Again, we keep the other cost parameters as in cost 1 and consider $\co_a\in\{50,150,250,350,450\}$. Figure \ref{fig:sen_anes_OT} shows the time allocated to colorectal surgeries in the same OR we used in Figure \ref{fig:sen_waiting}. 
In general, when $\co_a$ increases, the time allowances of the \spe{} and \spcvar{} models decrease for surgeries in earlier positions (i.e., positions 1 and 2), and it increases in the last position. This is reasonable as the \spe{} and \spcvar{} models tend to avoid overtime when $\co_a$ increases. In contrast, the \droe{} and \drocvar{} models are insensitive to the change of $\co_a$. This makes sense because these models already schedule longer time for surgeries in earlier positions to avoid delays, which potentially mitigates overtime. Finally, we observe that when $\co_a$ increases from $350$ to $450$ (resp. from $150$ to $250$), the \spcvar{} and \droe{} models (resp. the \drocvar{} model) employ an extra on-call anesthesiologist for colorectal surgeries. This is to avoid excessive anesthesiologist overtime.

\begin{figure} \OneAndAHalfSpacedXI
     \centering
     \includegraphics[scale=0.56]{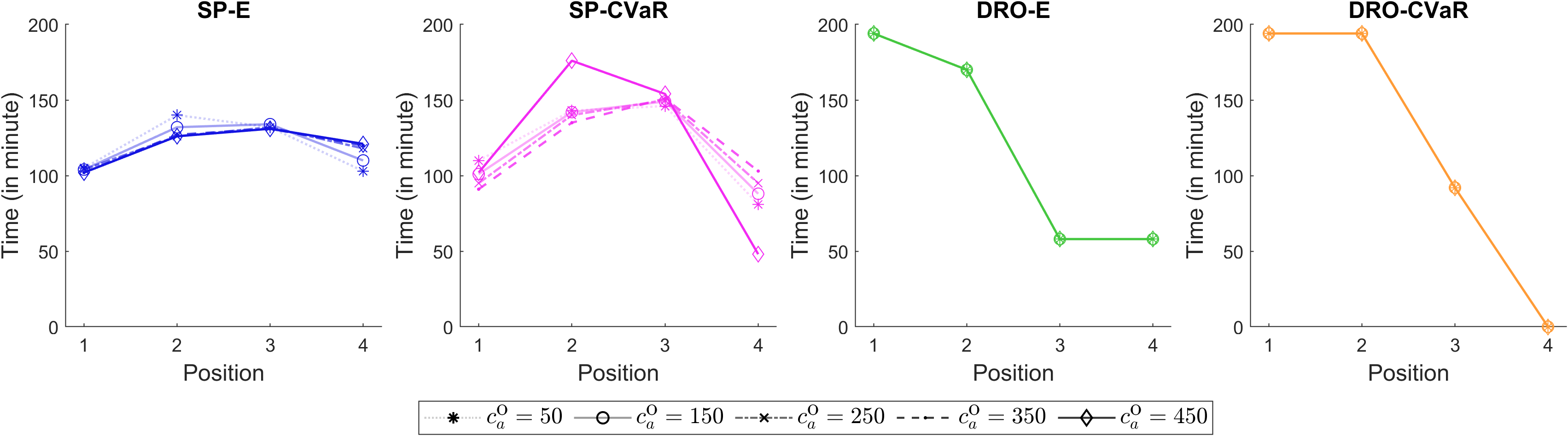}
     \caption{Time allowance to each surgery in one of the ORs with different $\co_a$}  \label{fig:sen_anes_OT}
\end{figure}

Finally, our collaborators are interested in determining the number of surgeries to schedule in each OR to maintain good operational performances that balance variable operational cost and revenue.  Therefore, we use the same technique in \cite{Berg_et_al:2014} to determine the number of surgeries to schedule that maximizes the difference between the profit from scheduling (i.e., performing) surgeries and the associated estimated cost (i.e., optimal value of our models). We use cost 1 and cardiac surgery as an example, for which there are six specialized anesthesiologists (one of whom is on call) and five ORs in the master schedule. 

Figure \ref{fig:sen_I} shows the revenue under different profits $\{2000, 2500, 3000\}$ per surgery. When the profit is relatively low (i.e., $2000$ or $2500$), all models schedule $5$ surgeries (to maximize revenue), precisely one surgery per OR. This is because cardiac surgery has a long duration with a mean of $384$ minutes, and each OR is available for $480$ minutes.  However, when the profit is larger (i.e., $3000$), the \spe{} model schedules $10$ surgeries (i.e., two surgeries per OR) while the other models schedule $5$ since these models tend to hedge against duration scenarios that may lead to high operational costs.  We note that the resulting packed OR schedule of the \spe{} model leads to poor operational performance. For example, the waiting time and anesthesiologist overtime associated with the \spe{} schedule are $511$ and $1227$ minutes, respectively, compared with zero waiting and overtime of $26$ minutes in other models. Our collaborators indicate that they prefer a less packed schedule to ensure smooth operations with fewer delays and overtime.

\begin{figure} \OneAndAHalfSpacedXI
     \centering
     \includegraphics[scale=0.65]{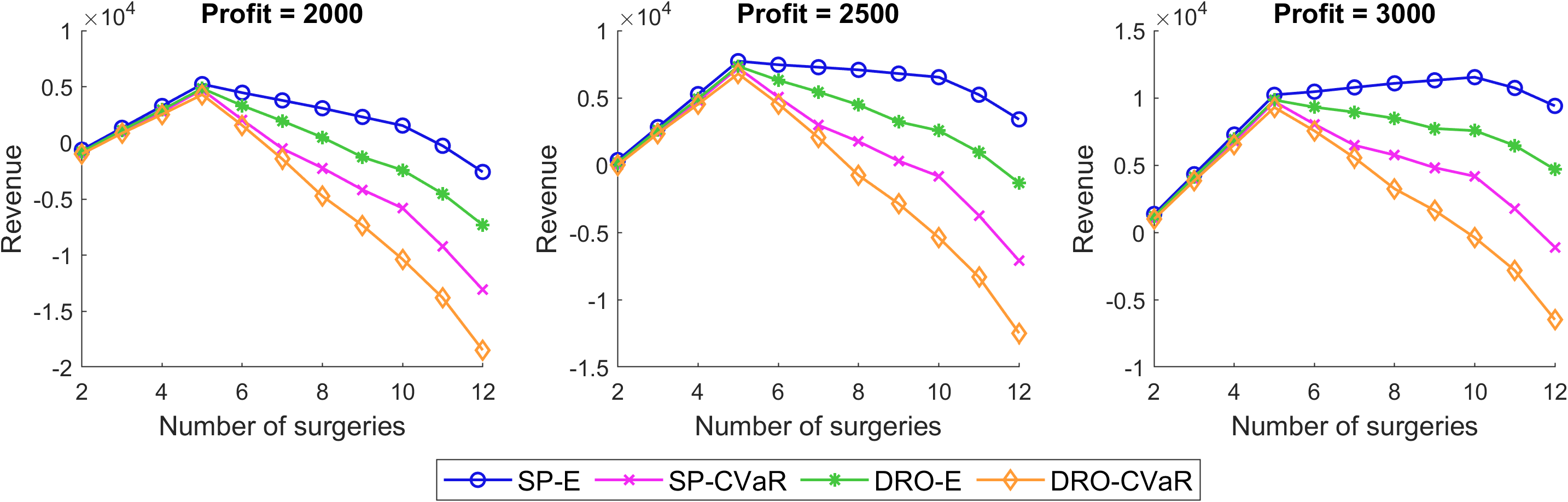}
     \caption{Revenue from scheduling different number of surgeries}  \label{fig:sen_I}
\end{figure}

\section{Conclusion} \label{sec:conclusion}

In this paper, we study an operating room and anesthesiologist scheduling problem (ORASP) under uncertainty. We propose the first risk-neutral and risk-averse SP models to address uncertainty in surgery durations in the ORASP, which generalizes the state-of-the-art models by (a) incorporating a larger set of important objectives, (b) integrating allocation, assignment, sequencing, and scheduling problems, and (c) modeling the decision-maker's risk preference. In addition, recognizing that high-quality data to estimate surgery duration distribution accurately is often not available, we propose DRO counterparts of our SP models based on a mean-support ambiguity set to account for distributional ambiguity. We derive equivalent solvable reformulations of these DRO models and propose a C\&CG method that efficiently solves the reformulations.

Using publicly available surgery data and a case study from a large health system in New York, we construct various ORASP instances and conduct extensive computational experiments to compare the proposed methodologies, illustrating the potential benefits of our proposed integrated approach in practice.  We observe significant differences in the optimal \spe{}, \spcvar{}, \droe{}, and \drocvar{} schedules and thus, substantial differences in these solutions' impact on the operational performance and costs. Our results also show the significance of integrating the allocation, assignment, sequencing, and scheduling problems and the negative consequences of adopting existing non-integrated approaches and  ignoring the uncertainty and ambiguity of surgery duration. We also conduct sensitivity analysis using the case study data from our collaborating health system to derive insights relevant to OR managers. Finally, our results demonstrate the computational efficiency of the proposed methodologies.

We suggest the following areas for future research. First, we would like to incorporate other sources of uncertainty, such as arrival of emergency surgeries. Second, our model can serve as a building block towards data-driven and robust OR planning. In particular, we aim to extend our models to more comprehensive OR and surgery planning models by considering all relevant organizational and technical constraints such as optimizing the block (master) schedule and incorporating recovery units. Finally, it would be interesting to explore advanced data-driven and machine learning methods that exploit, for example, patients' characteristics to model surgery duration and other random factors and hence, to further optimize the OR planning process.

\ACKNOWLEDGMENT{The authors are grateful to anonymous referees, the Editor, and the Associate Editor for their constructive comments and helpful suggestions. Dr. Karmel S.~Shehadeh dedicates her effort in this paper to every little dreamer in the whole world who has a dream so big and so exciting. Believe in your dreams and do whatever it takes to achieve them--the best is yet to come for you.}


\ECSwitch


\ECHead{Appendix}

%
%
%


\section{Anesthesiologists Specialty Example} \label{appdx:anes_specialty}

Table \ref{table:anes_specialty_eg} provides an example of an anesthesiologist's specialty and assignment at our collaborating health system. Some specialized anesthesiologists are dedicated to a specific specialty (e.g., cardiothoracic, obstetrics and pediatric). Some can perform a wide range of surgeries (i.e., cross-cover) such as general, orthopedic, neurosurgery and transplant surgeries.

%
 \begin{table}[h]\centering\small
\ra{0.6}  
{\OneAndAHalfSpacedXI \caption{Anesthesiologist specialties. Empty slots indicate that the cross coverage is possible and slots with NA indicate that such a cross-coverage is not possible. (ORTH: orthopedic, NSG: neurosurgery, CARD: Cardiothoracic, OB: Obstetrics)} \label{table:anes_specialty_eg} }
\begin{tabular}{@{}l|llllllll@{}} \toprule
& \multicolumn{8}{c}{\textbf{Specialty cross coverage}}   \\
\textbf{Primary}       & General & ORTH & Pain Medicine & NSG & Transplant & CARD & OB & Pediatric \\ \midrule
General       &         &      &               &     &            & NA   & NA & NA        \\ 
ORTH          &         &      &               &     &            & NA   & NA & NA        \\
Pain Medicine &         &      &               &     &            & NA   & NA & NA        \\
NSG           &         &      &               &     &            & NA   & NA & NA        \\
Transplant    &         &      &               &     &            & NA   & NA & NA        \\
CARD          &         &      &               &     &            &      & NA & NA        \\
OB            & NA      & NA   & NA            & NA  & NA         & NA   &    & NA        \\
Pediatric     & NA      & NA   & NA            & NA  & NA         & NA   & NA &           \\
\bottomrule
\end{tabular}
\end{table}

\section{Notation} \label{appdx:notation}

Table \ref{table:notation} summarizes the notation we use in the models, including parameters, sets, and decision variables.
\begin{table}[t] 
\footnotesize
\center
\caption{Notation} \label{table:notation}
   \renewcommand{\arraystretch}{0.7}
\begin{tabular}{ll}
\hline 
 \multicolumn{2}{l}{\textbf{Indices}} \\
$i$ & index of surgery, $i\in I$ \\
$a$ & index of anesthesiologist, $a\in A$ \\
$r$ & index of operating room, $r\in R$ \\
$l$ & index of surgery type, $l\in L$ \\
\multicolumn{2}{l}{\textbf{Parameters and sets}} \\
$I$ & set of surgeries \\
$A$ & set of anesthesiologists \\
$R$ & set of operating rooms \\
$L$ & set of surgery types \\
$\calF^A$ / $\calF^R$ & set of $(i,a)$ / $(i,r)$ pairs such that surgery $i$ can be performed by anesthesiologist $a$ / in operating room $r$ \\
$A_i$ / $R_i$ & set of anesthesiologists / operating rooms to which surgery $i$ can be assigned \\
$I_a$ / $I_r$ & set of surgeries that could be performed by anesthesiologist $a$ / in operating room $r$ \\
$t^\text{start}_a$ / $t^\text{end}_a$ & start / end time of anesthesiologist $a$ \\
$T^\text{end}$ & end time of the day \\
$f_r$ & fixed cost of using operating room $r$\\
$f_a$ & fixed cost of employing an on-call anesthesiologist \\
$\cg_a$ / $\cg_r$ & idling cost of anesthesiologist $a$ / operating room $r$\\
$\co_a$ / $\co_r$ & overtime cost of anesthesiologist $a$ / operating room $r$\\
$\cw_i$ & waiting cost of surgery $i$\\
$\kappa^A_{i,a}$ / $\kappa^R_{i,a}$ & $1$ if surgery $i$ can be done by anesthesiologist $a$ / in operating room $r$, $0$ otherwise\\
$\hreg_a$ & $1$ if anesthesiologist $a$ is on regular duty, $0$ otherwise\\
$\hcall_a$ & $1$ if anesthesiologist $a$ is on call, $0$ otherwise \\
$D_i$ & random duration of surgery $i$ ($d_i$ as a realization) \\
$\dlb_i$ / $\dub_i$ & lower / upper bound of duration of surgery $i$ \\
  \multicolumn{2}{l}{\textbf{First-stage decision variables}} \\
$x_{i,a}$ & $1$ if surgery $i$ is performed by anesthesiologist $a$, $0$ otherwise \\
$z_{i,r}$ & $1$ if surgery $i$ is performed in operating room $r$, $0$ otherwise \\
$y_a$ & $1$ if anesthesiologist $a$ is assigned from on call, $0$ otherwise \\
$v_r$ & $1$ if operating room $r$ is used, $0$ otherwise \\
$s_i$ & scheduled start time of surgery $i$ \\
$u_{i,i'}$ & $1$ if surgery $i'$ follows surgery $i$, $0$ otherwise \\ 
$\alpha_{i,i',a}$ & $1$ if surgery $i'$ follows surgery $i$ for anesthesiologist $a$, $0$ otherwise \\
$\beta_{i,i',r}$ & $1$ if surgery $i'$ follows surgery $i$ in operating room $r$, $0$ otherwise \\
  \multicolumn{2}{l}{\textbf{Second-stage decision variables}} \\
$q_i$ & actual start time of surgery $i$\\
$o_a$ / $o_r$ & overtime of anesthesiologist $a$ / operating room $r$ \\
$w_i$ & waiting time of surgery $i$ \\
$g_a$ / $g_r$ & idling time of anesthesiologist $a$ / operating room $r$ \\
\hline
\end{tabular}
\end{table}

\clearpage
\section{Example and Discussions in Section \ref{sec:SP_model}} \label{appdx:SP_model}

\subsection{Example on Sequencing Variable Constraints}  \label{appdx:eg_sequencing}
\begin{example}  \label{eg:1st_stage_eg}
There are two possible types of non-implementable decisions without constraints \eqref{eqn:1st_stage_con9} and \eqref{eqn:1st_stage_con10}. First, consider that surgeries $1$ and $2$ are assigned to the same anesthesiologist $a$ and in the same OR $r$, but we have that surgery $1$ is performed followed by surgery $2$ in the anesthesiologist schedule, but vice versa in the OR schedule, with both scheduled start time being $0$. That is, the anesthesiologist perform surgery $1$ first, but surgery $2$ is scheduled first in the OR. Obviously, this schedule is not implementable and constraints \eqref{eqn:1st_stage_con9} excludes such a possibility.

The second type of non-implementable decision corresponds to cycles. As an example, assume that we have four surgeries with the same scheduled start times, say $0$, as follows.
\begin{center}
Anesthesiologist 1: $1\rightarrow 2$ \quad OR 1: $3\rightarrow1$\\
Anesthesiologist 2: $4\rightarrow 3$ \quad OR 2: $2\rightarrow4$ 
\end{center}
For anesthesiologist 1 to perform surgery $1$, one needs to wait for the completion of surgery $3$ in OR~1. However, to perform surgery $3$ in OR~1, anesthesiologist 2 has to perform surgery $4$ first, which follows surgery $2$ in OR~2. Lastly, to perform surgery~$2$, anesthesiologist $1$ needs to perform surgery~$1$. That is, a cycle occurs and no surgeries could be conducted. Note that with the transitivity constraints \eqref{eqn:1st_stage_con10}, decisions with cycles can be excluded.
\end{example}

\subsection{Choice of Big \texorpdfstring{$M$}{TEXT} Parameters} \label{appdx:big_M}
In the SP model \eqref{eqn:1st_stage}--\eqref{eqn:2nd_stage}, there are four big $M$ parameters. We provide a suitable choice of these parameters for actual implementation. First, from constraints \eqref{eqn:1st_stage_con6-7}, we have that $M=\max_{a\in A} \big\{t^\text{start}_a\big\}$ is a suitable choice. To derive a suitable choice of the remaining three parameters which are related to the maximum of $q_i+d_i$, we provide the following lemma. (Note that $q_i$ is the second-stage variable).

\begin{lemma} \label{lem:big_M_actual}
Suppose that $d_i\in[\dlb_i,\dub_i]$ for all $i\in I$. Then, for any first-stage decision $(x,y,z,v,u,s)$ and realization $d$, we have $\max_{i\in I} \big\{ q_i + d_i \big\} \leq T^\text{end} + \sum_{i\in I}\dub_i$.
\end{lemma}   

\begin{proof}{Proof.}
The last possible actual surgery start time corresponds to the situation that the scheduled start time for all the surgeries are $T^\text{end}$. Moreover, all surgeries are scheduled to one anesthesiologist only. This leads to the desired upper bound. \Halmos
\end{proof}

As a result of Lemma \ref{lem:big_M_actual}, it suffices to set the big-$M$ parameters as $M_\textup{seq} = T^\text{end} + \sum_{i\in I}\dub_i$, $M_\textup{anes} =  T^\text{end} + \sum_{i\in I}\dub_i - \min_{a\in A}\big\{ t^\textup{end}_a \big\}$ and $M_\textup{room} = \sum_{i\in I}\dub_i$.

\begin{remark} \label{appdx:rem_big_M}
The proof of \ref{lem:big_M_actual} is based on the hypothetical situation that an anesthesiologist conducts all the surgeries. In practical ORASP instances, we have anesthesiologists specialized for a single surgery type $\ell\in\{1,\dots,L^\text{spec}\}$ and a pool of anesthesiologists for general surgery types $\{L^\text{spec}+1,\dots,L\}$. In this case, we could tighten the upper bound to 
$$\max_{i\in I} \big\{ q_i + d_i \big\} \leq T^\text{end} + \max\Bigg\{\max_{\ell = 1,\dots,L^\text{spec}} \Big\{ \dub_\ell|I_\ell| \Big\},\, \sum_{\ell=L^\text{spec}+1}^L \dub_\ell|I_\ell| \Bigg\}, $$
where $\dub_\ell$ and $|I_\ell|$ are the maximum surgery duration and number of surgeries of type $\ell$, respectively. Note that it is not optimal to schedule all surgeries at $T^\text{end}$. Indeed, scheduling all surgeries at $\max_{a\in A}t^\text{start}_a$ could produce a better solution by reducing both idle time and overtime. Therefore, it suffices to choose 
$$M_\text{seq}= \max_{a\in A} \big\{t^\text{start}_a\big\} + \max\Bigg\{\max_{\ell = 1,\dots,L^\text{spec}} \Big\{ \dub_\ell|I_\ell| \Big\},\, \sum_{\ell=L^\text{spec}+1}^L \dub_\ell|I_\ell| \Bigg\}$$
as in \cite{Rath_et_al:2017}. Note that we exploit the practical ORASP instance structures to result in tighter big-$M$ parameters.
\end{remark}


\color{black}
\subsection{Analysis of Models without Waiting Time Components} \label{appdx:models_wo_waiting}

In this section, we analyze ORASP models that do not incorporate waiting time components in the second-stage problem. In Proposition~\ref{prop:waiting_cost_zero}, we show that for models without waiting time components, it is optimal to schedule surgeries assigned to the same anesthesiologist to start simultaneously at the start time of that anesthesiologist, which is not possible in practice.

\begin{proposition} \label{prop:waiting_cost_zero}
Let $(x,y,z,v,u,s,\alpha,\beta)$ be any feasible first-stage solution satisfying \eqref{eqn:1st_stage_con1-2}--\eqref{eqn:1st_stage_con20-21}, and define $\st_i=\sum_{a\in A_i} t^\textup{start}_a x_{i,a}$ for all $i\in I$. Then, the following assertions hold. 
\begin{itemize}
    \item[(i)]  $(x,y,z,v,u,\st,\alpha,\beta)$ is a feasible first-stage solution with $\tilde{s}_i\leq s_i$ for all $i\in I$.
    \item[(ii)] If $\cw_i=0$ for all $i\in I$, then $Q(x,y,z,v,u,\st,d)\leq Q(x,y,z,v,u,s,d)$ for all $d\in\calS$. Thus, $\sum\limits_{r\in R} f_r v_r+\sum \limits_{a\in A}f_a y_a+\varrho_{\Prob}\big(Q(x,y,z,v,u,\st,D)\big)\leq \sum \limits_{r\in R} f_r v_r+\sum\limits_{a\in A}f_a y_a+\varrho_{\Prob}\big(Q(x,y,z,v,u,s,D)\big)$. In particular, this inequality holds irrespective of the choice of the risk measure $\rho(\cdot)$ and the probability distribution $\Prob$. 
\end{itemize}
\end{proposition}

\begin{proof}{Proof.}
We first prove part (i). Since $(x,y,z,v,u,s,\alpha,\beta)$ is a feasible first-stage solution satisfying  \eqref{eqn:1st_stage_con1-2}--\eqref{eqn:1st_stage_con20-21}, we only need to show that replacing $s$ in $(x,y,z,v,u,s,\alpha,\beta)$ with $\st$ results in another feasible first-stage solution  $(x,y,z,v,u,\st,\alpha,\beta)$. That is, we need to show that $\st_i=\sum_{a\in A_i} t^\textup{start}_a x_{i,a}$ satisfies constraints \eqref{eqn:1st_stage_con6-7} (the only constraints involving $\st$). Since $x$ satisfies \eqref{eqn:1st_stage_con1-2}, for each $i \in I$, we have $x_{i,a_i}=1$ for exactly one  $a_i \in A_i$ and $x_{i,a'}=0$ for all $a' \in A_i \setminus \{a_i\}$. Moreover, since $s$ satisfies \eqref{eqn:1st_stage_con6-7}, we have $t^\text{start}_{a_i} \leq s_i \leq T^\text{end}$ for all $i\in I$. It follows that $\st_i=\sum_{a\in A_i}t^\text{start}_a x_{i,a}=t^\text{start}_{a_i} \leq s_i \leq T^\text{end}$ for all $i \in I$. Thus, $\st$ satisfies constraints \eqref{eqn:1st_stage_con6-7}. Accordingly, we conclude  that $(x,y,z,u,v,\st,\alpha,\beta)$ is a feasible first-stage solution  satisfying \eqref{eqn:1st_stage_con1-2}--\eqref{eqn:1st_stage_con20-21}. This completes the proof of part (i).

Next, we prove part (ii). We first claim that for any given $d\in\calS$ and first-stage solution $(\bar{x},\bar{y},\bar{z},\bar{v},\bar{u},\bar{s})$, an optimal solution $(\bar{q}, \bar{o}, \bar{w}, \bar{g})$ to the second-stage problem $Q(\bar{x},\bar{y},\bar{z},\bar{v},\bar{u},\bar{s},d)$ defined in \eqref{eqn:2nd_stage} satisfies 
\begin{equation}\label{eqn_pf:prop_waiting_cost_zero_1}
   \bar{q}_i=\max \bigg\{\bar{s}_i,\, \max_{i': \bar{u}_{i',i}=1} \{\bar{q}_{i'}+d_{i'}\} \bigg\},\quad\forall i\in I.
\end{equation}
To show \eqref{eqn_pf:prop_waiting_cost_zero_1}, note from \eqref{eqn:2nd_stage_con1}--\eqref{eqn:2nd_stage_con2} that $\bar{q}_i \geq \max \big\{\bar{s}_i,\, \max_{i': \bar{u}_{i',i}=1} \{\bar{q}_{i'}+d_{i'}\} \big\}$ for all $i\in I$. Also, it is easy to verify from the objective of minimizing $(o_a, o_r,g_a, g_r)$  and constraints \eqref{eqn:2nd_stage_con3}--\eqref{eqn:2nd_stage_con4} and \eqref{eqn:2nd_stage_con6}--\eqref{eqn:2nd_stage_con7} that at optimality, we have $\bar{o}_a=\max_{i\in A_i:\bar{x}_{i,a}=1,\,\bar{y}_a=0}\big\{\bar{q}_i+d_i-t^\text{end}_a\big\}$, $\bar{o}_r=\max_{i\in R_i:\bar{z}_{i,r}=1}\big\{\bar{q}_i+d_i-T^\text{end}\big\}$, $\bar{g}_a = \bar{o}_a + \big(t^\text{end}_a - t^\text{start}_a - \sum_{i\in I_a} d_i \bar{x}_{i,a}\big) \hreg_a$, and $\bar{g}_r=\bar{o}_r + T^\text{end} \bar{v}_r - \sum_{i\in I_r} d_i \bar{z}_{i,r}$. It follows that $(\bar{o}_a, \bar{o}_r, \bar{g}_a, \bar{g}_r)$ and hence the second-stage objective value are non-decreasing in $\bar{q}$. This implies that $\bar{q}_i=\max \big\{\bar{s}_i, \max_{i': u_{i',i}=1} \{\bar{q}_{i'}+d_{i'}\} \big\}$, which proves the claim in \eqref{eqn_pf:prop_waiting_cost_zero_1}. Now, let $(q^*,o^*,w^*, g^*)$ and $(\qt^*,\ot^*,\wt^*,\gt^*)$ be optimal solutions to the second-stage problem \eqref{eqn:2nd_stage} with first-stage solutions $(x,y,z,u,v,s,\alpha,\beta)$ and $(x,y,z,u,v,\st,\alpha,\beta)$, respectively. Since $\st_i \leq s_i$ for all $i\in I$ (shown in part (i)), from \eqref{eqn_pf:prop_waiting_cost_zero_1}, we have $\qt^*_i \leq q^*_i$ for all $i\in I$. As a result, we have  $\ot^*_a\leq o^*_a$ and $\gt^*_a\leq g^*_a$ for all $a\in A$, and we have $\ot^*_r\leq o^*_r$ and $\gt^*_r\leq g^*_r$ for all $r\in R$. It follows that $Q(x,y,z,u,v,\st,d)\leq Q(x,y,z,u,v,s,d)$ for all $d\in\calS$. Since the two solutions have the same number of open ORs ($v$) and employed on-call anesthesiologists ($y$)  and  $Q(x,y,z,u,v,\st,d)\leq Q(x,y,z,u,v,s,d)$ holds for all $d\in\calS$, we have
\begin{align}
\sum\limits_{r\in R} f_r v_r+\sum \limits_{a\in A}f_a y_a+\varrho_{\Prob}\big(Q(x,y,z,v,u,\st,D)\big)\leq \sum \limits_{r\in R} f_r v_r+\sum\limits_{a\in A}f_a y_a+\varrho_{\Prob}\big(Q(x,y,z,v,u,s,D)\big). \label{eq:TC}
\end{align}
Note that inequality \eqref{eq:TC} holds irrespective of the choice of the risk measure and the probability distribution. This completes the proof of part (ii). \Halmos
\end{proof}


Proposition~\ref{prop:waiting_cost_zero} shows that when the waiting time components are not included in the second-stage problem of the ORASP, it is optimal to set the scheduled start time of each surgery to the start time of the anesthesiologist assigned to that surgery, i.e., $s_i=t^\textup{start}_a$ if $x_{i,a}=1$. This, in turn, implies that surgeries assigned to the same anesthesiologist are scheduled to start simultaneously, which is not possible in practice. Mathematically, if one removes the waiting time metric (i.e., $\sum_{i \in I}\cw_iw_i$) from the objective function in (2a) and constraints (2f), one can, without loss of optimality, remove the scheduled start time variables $s_i$ and constraints \eqref{eqn:1st_stage_con6-7} from the first-stage formulation and replace $s_i$ in constraints \eqref{eqn:2nd_stage_con2} with  $\sum_{a\in A_i}t^\text{start}_a x_{i,a}$. The resulting SP formulation is
\begin{subequations} 
\begin{align}
 \underset{x,\,y,\,z,\,v,\,u,\,\alpha,\,\beta}{\text{minimize}} \quad
&  \sum_{r\in R}f_r v_r + \sum_{a\in A}f_a y_a + \varrho_{\Prob}\big(\widetilde{Q}(x,y,z,v,u,D)\big)  \\
 \text{subject to\,\,\,} \quad
&  \text{\eqref{eqn:1st_stage_con1-2}--\eqref{eqn:1st_stage_con5},  \eqref{eqn:1st_stage_con8}--\eqref{eqn:1st_stage_con19},} \\
&  x_{i,a},\, y_a,\, z_{i,r},\, u_{i,i'},\, v_r,\, \alpha_{i,i',a},\,\beta_{i,i',r}\in\{0,1\}, \quad\forall  i\in I,\, a\in A,\, r\in R,
\end{align} \label{eqn:1st_stage_wo_scheduling}%
\end{subequations}%
where the second-stage problem is given by
\begin{subequations}
\begin{align}
\widetilde{Q}(x,y,z,v,u,d):=\quad \underset{q,\,o,\,g}{\text{minimize}} \quad
&  \sum_{a\in A} \Big(\cg_a g_a + \co_a o_a \Big) +\sum_{r\in R} \Big(\cg_r g_r + \co_r o_r \Big) \\
 \text{subject to} \quad
&  \text{\eqref{eqn:2nd_stage_con1}, \eqref{eqn:2nd_stage_con3}--\eqref{eqn:2nd_stage_con4}, \eqref{eqn:2nd_stage_con6}--\eqref{eqn:2nd_stage_con7},}\\
&  q_i \geq \sum_{a\in A_i} t^\text{start}_a x_{i,a},\quad\forall i\in I,\\
&  q_i,\, o_a,\, o_r,\, g_a,\, g_r \geq 0,\quad\forall i\in I,\, a\in A,\, r\in R. 
\end{align}\label{eqn:2nd_stage_wo_scheduling}%
\end{subequations}
Our analyses indicate that scheduled start time decisions are redundant in models that do not incorporate waiting time objectives, variables, and constraints in the second-stage formulation. In other words, models that do not consider the waiting time component will produce schedules that are not implementable in practice, with, for example, surgeries assigned to each anesthesiologist scheduled to start simultaneously. In contrast, our proposed models for the ORASP, which incorporate the waiting time metric, variables, and constraints in the second-stage problem, produce implementable schedules with each surgery having a specific scheduled start time, different from other surgeries sharing the same resource (anesthesiologist and OR). Note that the OR and healthcare scheduling literature have long recognized scheduling decisions as crucial planning decisions impacting patient flow and system performance. Indeed, most healthcare systems use an appointment schedule that specifies the sequence of non-overlapping scheduled start times of surgeries sharing the same resources to manage patient flow in the ORs and facilitate the coordination of different resources (e.g., anesthesiologists and ORs in the ORASP).

\color{black}

\section{Proofs and Discussions Related to \droe{} and \drocvar{} Models in Section \ref{sec:DRO_model}} \label{appdx:DRO_model_section}

\subsection{Proof of Proposition \ref{prop:mean_support_WC_exp}} \label{appdx:pf_mean_support_WC_exp}

\begin{proof}{Proof of Proposition \ref{prop:mean_support_WC_exp}}
Note that $\calS$ is compact, $\psi_0(d):=Q(x,y,z,v,u,s,d)$ and $\psi_i(d):=d_i$ for $i\in I$ are continuous functions in $d$. From Proposition 6.68 of \cite{Shapiro_et_al:2014}, strong duality holds and the worst-case expectation equals
\begin{align*}
 \underset{\rho}{\text{minimize}} \quad
&  \rho_0 + \sum_{i\in I}\rho_i m_i \\
 \text{subject to} \quad
& Q(x,y,z,v,u,s,d) - \sum_{i\in I} \rho_i d_i \leq \rho_0,\quad\forall d\in\calS.
\end{align*}
Since the constraint holds for all $d\in\calS$, we can rewrite it as 
$$\rho_0 \geq \sup_{d\in\calS} \Big\{Q(x,y,z,v,u,s,d) - \sum_{i\in I} \rho_i d_i  \Big\}.$$
As we minimize $\rho_0$ in the objective, we have that the dual problem is equivalent to \eqref{eqn:mean_support_WC_exp_obj}. \Halmos
\end{proof}

\subsection{Proof of Proposition \ref{prop:mean_support_inner_max_MILP}} \label{appdx:pf_mean_support_inner_max_MILP}

Before proceeding to the proof of Proposition \ref{prop:mean_support_inner_max_MILP}, we first derive the dual of the second-stage problem \eqref{eqn:2nd_stage} in Lemma \ref{lem:second_stage_dual}.

\begin{lemma} \label{lem:second_stage_dual}
The second-stage problem \eqref{eqn:2nd_stage} is equivalent to
\begin{subequations} 
\begin{align}
 \underset{\lambda,\,\mu,\,\theta}{\textup{maximize}} \quad
&  \sum_{a\in A} \cg_a(t^\textup{end}_a-t^\textup{start}_a)\hreg_a + \sum_{r\in R}\cg_r T^\textup{end} v_r  \nonumber \\
& \quad + \sum_{i\in I}\Bigg[\sum_{i'\in I, i'\ne i}(\lambda_{i,i'} - \lambda_{i',i}) + \sum_{a\in A_i} \mu_{i,a} + \sum_{r\in R_i} \theta_{i,r} \Bigg] s_i \nonumber \\ 
& \quad - M_\textup{seq} \sum_{i\in I}\sum_{i'\in I, i'\ne i} \lambda_{i,i'}(1-u_{i,i'}) - \sum_{i\in I}\sum_{a\in A_i}\mu_{i,a}\Big[ t^\textup{end}_a + M_\textup{anes}(1-x_{i,a}+ y_a) \Big] \nonumber \\
& \quad - \sum_{i\in I}\sum_{r\in R_i} \theta_{i,r} \Big[T^\textup{end} +M_\textup{room}(1-z_{i,r})\Big] \nonumber \\
& \quad + \sum_{i\in I} \Bigg[ \sum_{i'\in I,i'\ne i} \lambda_{i,i'} + \sum_{a\in A_i} (\mu_{i,a}- \cg_a \hreg_a x_{i,a}) + \sum_{r\in R_i}(\theta_{i,r}-\cg_r z_{i,r})\Bigg] d_i \label{eqn:2nd_stage_dual_reform_obj} \\
 \textup{subject to} \quad
&  \sum_{i\in I_a} \mu_{i,a} \leq \cg_a + \co_a,\quad\forall a\in A, \label{eqn:2nd_stage_dual_reform_con1}\\
&  \sum_{i\in I_r} \theta_{i,r} \leq \cg_r + \co_r,\quad\forall r\in R, \label{eqn:2nd_stage_dual_reform_con2} \\
&  \sum_{i'\in I, i'\ne i}(\lambda_{i,i'} - \lambda_{i',i}) + \sum_{a\in A_i} \mu_{i,a} + \sum_{r\in R_i} \theta_{i,r} + \cw_i \geq 0,\quad\forall i\in I, \label{eqn:2nd_stage_dual_reform_con3} \\
&  \lambda_{i,i'},\,\mu_{i,a},\,\theta_{i,r}\geq 0,\quad\forall i\in I,\,a\in A_i,\, r\in R_i,\,i'\in I\setminus\{i\}. \label{eqn:2nd_stage_dual_reform_con4}
\end{align} \label{eqn:2nd_stage_dual_reform}%
\end{subequations} \vspace{-5mm}%
\end{lemma}

\begin{proof}{Proof of Lemma \ref{lem:second_stage_dual}}
First, from the second-stage problem \eqref{eqn:2nd_stage}, the optimal solution $q_i$, $o_a$ and $o_r$ give the actual surgery start time, the overtime of anesthesiologists and OR, respectively. From constraints \eqref{eqn:2nd_stage_con2} and the minimization nature of the problem, we have  $w_i=q_i-s_i$. Moreover, constraints \eqref{eqn:2nd_stage_con6}--\eqref{eqn:2nd_stage_con7} achieve equality at optimality, which characterize the idle times of anesthesiologists and ORs, respectively. Therefore, the second-stage problem \eqref{eqn:2nd_stage} is equivalent to
\begin{subequations} 
\begin{align}
 \underset{q,\,o,\,w,\,g}{\text{minimize}} \quad
&  \sum_{a\in A} \Bigg\{\cg_a \Bigg[\bigg(t^\text{end}_a - t^\text{start}_a-\sum_{i\in I_a} d_i x_{i,a}\bigg)\hreg_a + o_a \Bigg] + \co_a o_a \Bigg\} \nonumber \\
& \quad +\sum_{r\in R} \Bigg[\cg_r \bigg( T^\text{end} v_r - \sum_{i\in I_r} d_i z_{i,r} + o_r\bigg) + \co_r o_r \Bigg] + \sum_{i\in I} \cw_i (q_i-s_i) \\
 \text{subject to} \quad
&  q_{i'} \geq q_i + d_i - M_\text{seq}(1-u_{i,i'}),\quad\forall \{i,i'\}\subseteq I,\, i\ne i', \\
&  q_i \geq s_i,\quad\forall i\in I,  \\
&  o_a \geq q_i + d_i -t^\text{end}_a - M_\text{anes}(1-x_{i,a}+y_a),\quad\forall (i,a)\in\calF^A,  \\
&  o_r \geq q_i + d_i - T^\text{end} - M_\text{room}(1-z_{i,r}),\quad\forall (i,r)\in\calF^R, \\
&  q_i,\, o_a,\, o_r \geq 0,\quad\forall i\in I,\, a\in A,\, r\in R. 
\end{align}%
\end{subequations}%
By LP duality, this is equivalent to
\begin{subequations} 
\begin{align}
 \underset{\lambda,\,\mu,\,\theta,\,\phi}{\textup{maximize}} \quad
& \Bigg\{ \sum_{a\in A} \cg_a(t^\textup{end}_a-t^\textup{start}_a)\hreg_a + \sum_{r\in R}\cg_r T^\textup{end} v_r + \sum_{i\in I}(\phi_i-\cw_i) s_i \nonumber \\ 
& \quad - M_\textup{seq} \sum_{i\in I}\sum_{i'\in I, i'\ne i} \lambda_{i,i'}(1-u_{i,i'}) - \sum_{i\in I}\sum_{a\in A_i}\mu_{i,a}\Big[ t^\textup{end}_a + M_\textup{anes}(1-x_{i,a}+ y_a) \Big] \nonumber \\
& \quad - \sum_{i\in I}\sum_{r\in R_i} \theta_{i,r} \Big[T^\textup{end} +M_\textup{room}(1-z_{i,r})\Big] \nonumber \\
& \quad + \sum_{i\in I} \Bigg[ \sum_{i'\in I,i'\ne i} \lambda_{i,i'} + \sum_{a\in A_i} (\mu_{i,a}- \cg_a \hreg_a x_{i,a}) + \sum_{r\in R_i}(\theta_{i,r}-\cg_r z_{i,r})\Bigg] d_i \Bigg\}\label{eqn:2nd_stage_dual_obj} \\
 \textup{subject to} \quad
&  \sum_{i\in I_a} \mu_{i,a} \leq \cg_a + \co_a,\quad\forall a\in A, \label{eqn:2nd_stage_dual_con1}\\
&  \sum_{i\in I_r} \theta_{i,r} \leq \cg_r + \co_r,\quad\forall r\in R, \label{eqn:2nd_stage_dual_con2} \\
&  \sum_{i'\in I, i'\ne i}(\lambda_{i,i'} - \lambda_{i',i}) + \sum_{a\in A_i} \mu_{i,a} + \sum_{r\in R_i} \theta_{i,r} - \phi_i \geq -\cw_i,\quad\forall i\in I, \label{eqn:2nd_stage_dual_con3} \\
&  \lambda_{i,i'},\,\mu_{i,a},\,\theta_{i,r},\,\phi_i\geq 0,\quad\forall i\in I,\,a\in A_i,\, r\in R_i,\,i'\in I\setminus\{i\}. \label{eqn:2nd_stage_dual_con4}
\end{align} \label{eqn:2nd_stage_dual}%
\end{subequations}%
Note that from constraints \eqref{eqn:2nd_stage_dual_con3}, we have
$$0\leq \phi_i \leq \sum_{i'\in I, i'\ne i}(\lambda_{i,i'} - \lambda_{i',i}) + \sum_{a\in A_i} \mu_{i,a} + \sum_{r\in R_i} \theta_{i,r}  + \cw_i,\quad\forall i\in I.$$
Since we maximize $\phi_i s_i$ with $s_i \geq 0$ in the objective, the optimal solution  $\phi^*_i$ is the upper bound derived from constraints \eqref{eqn:2nd_stage_dual_con3}. This shows the equivalence of \eqref{eqn:2nd_stage_dual} and \eqref{eqn:2nd_stage_dual_reform}. \Halmos
\end{proof}

\begin{proof}{Proof of Proposition \ref{prop:mean_support_inner_max_MILP}}
Substituting \eqref{eqn:2nd_stage_dual_reform} into \eqref{eqn:mean_support_WC_exp_obj}, we have
\begin{subequations} 
\begin{align}
 \underset{\lambda,\,\mu,\,\theta,\,d}{\textup{maximize}} \quad
&  \sum_{a\in A} \cg_a(t^\textup{end}_a-t^\textup{start}_a)\hreg_a + \sum_{r\in R}\cg_r T^\textup{end} v_r  \nonumber \\
& \quad + \sum_{i\in I}\Bigg[\sum_{i'\in I, i'\ne i}(\lambda_{i,i'} - \lambda_{i',i}) + \sum_{a\in A_i} \mu_{i,a} + \sum_{r\in R_i} \theta_{i,r} \Bigg] s_i \nonumber \\ 
& \quad - M_\textup{seq} \sum_{i\in I}\sum_{i'\in I, i'\ne i} \lambda_{i,i'}(1-u_{i,i'}) - \sum_{i\in I}\sum_{a\in A_i}\mu_{i,a}\Big[ t^\textup{end}_a + M_\textup{anes}(1-x_{i,a}+ y_a) \Big] \nonumber \\
& \quad - \sum_{i\in I}\sum_{r\in R_i} \theta_{i,r} \Big[T^\textup{end} +M_\textup{room}(1-z_{i,r})\Big] \nonumber \\
& \quad + \sum_{i\in I} \Bigg[ \sum_{i'\in I,i'\ne i} \lambda_{i,i'} + \sum_{a\in A_i} (\mu_{i,a}- \cg_a \hreg_a x_{i,a}) + \sum_{r\in R_i}(\theta_{i,r}-\cg_r z_{i,r}) - \rho_i\Bigg] d_i \label{eqn:mean_support_inner_max_naive_obj} \\
 \textup{subject to} \quad
&  \text{\eqref{eqn:2nd_stage_dual_reform_con1}--\eqref{eqn:2nd_stage_dual_reform_con4}},\label{eqn:mean_support_inner_max_naive_con1} \\
&  \dlb_i \leq d_i\leq \dub_i,\quad\forall i\in I. \label{eqn:mean_support_inner_max_naive_con2}
\end{align} \label{eqn:mean_support_inner_max_naive}%
\end{subequations}%
Problem \eqref{eqn:mean_support_inner_max_naive} is not linear due to the quadratic terms between the dual variables and $d_i$ in the objective function. Note that we can first perform maximization over $d$, i.e.,
\begin{subequations} 
\begin{align}
 \underset{d}{\textup{maximize}} \quad
&  \sum_{i\in I} \Bigg[ \sum_{i'\in I,i'\ne i} \lambda_{i,i'} + \sum_{a\in A_i} (\mu_{i,a}- \cg_a \hreg_a x_{i,a}) + \sum_{r\in R_i}(\theta_{i,r}-\cg_r z_{i,r}) - \rho_i\Bigg] d_i \label{eqn:mean_support_inner_max_d_obj} \\
 \textup{subject to} \quad
&  \dlb_i \leq d_i\leq \dub_i,\quad\forall i\in I. \label{eqn:mean_support_inner_max_d_con}
\end{align} \label{eqn:mean_support_inner_max_d}%
\end{subequations}%
Problem \eqref{eqn:mean_support_inner_max_d} is separable in $i$ and the objective function is linear in $d_i$. Therefore, either $\dlb_i$ or $\dub_i$ is an optimal solution. Hence, problem \eqref{eqn:mean_support_inner_max_d} is equivalent to the binary program
\begin{subequations} 
\begin{align}
 \underset{b}{\textup{maximize}} \quad
&  \sum_{i\in I} \Bigg[ \sum_{i'\in I,i'\ne i} \lambda_{i,i'} + \sum_{a\in A_i} (\mu_{i,a}- \cg_a \hreg_a x_{i,a}) + \sum_{r\in R_i}(\theta_{i,r}-\cg_r z_{i,r}) - \rho_i\Bigg] (\dlb_i + b_i \Delta d_i) \label{eqn:mean_support_inner_max_d_IP_obj} \\
 \textup{subject to} \quad
&  b_i\in\{0,1\},\quad\forall i\in I, \label{eqn:mean_support_inner_max_d_IP_con}
\end{align} \label{eqn:mean_support_inner_max_d_IP}%
\end{subequations}%
where $\Delta d_i=\dub_i - \dlb_i$. With the reformulation \eqref{eqn:mean_support_inner_max_d_IP}, problem \eqref{eqn:mean_support_inner_max_naive} is equivalent to  
\begin{subequations} 
\begin{align}
 \underset{\lambda,\,\mu,\,\theta,\,b}{\textup{maximize}} \quad
&  \sum_{a\in A} \cg_a(t^\textup{end}_a-t^\textup{start}_a)\hreg_a + \sum_{r\in R}\cg_r T^\textup{end} v_r  \nonumber \\
& \quad + \sum_{i\in I}\Bigg[\sum_{i'\in I, i'\ne i}(\lambda_{i,i'} - \lambda_{i',i}) + \sum_{a\in A_i} \mu_{i,a} + \sum_{r\in R_i} \theta_{i,r} \Bigg] s_i \nonumber \\ 
& \quad - M_\textup{seq} \sum_{i\in I}\sum_{i'\in I, i'\ne i} \lambda_{i,i'}(1-u_{i,i'}) - \sum_{i\in I}\sum_{a\in A_i}\mu_{i,a}\Big[ t^\textup{end}_a + M_\textup{anes}(1-x_{i,a}+ y_a) \Big] \nonumber \\
& \quad - \sum_{i\in I}\sum_{r\in R_i} \theta_{i,r} \Big[T^\textup{end} +M_\textup{room}(1-z_{i,r})\Big] \nonumber \\
& \quad + \sum_{i\in I} \Bigg[ \sum_{i'\in I,i'\ne i} \lambda_{i,i'} + \sum_{a\in A_i} (\mu_{i,a}- \cg_a \hreg_a x_{i,a}) + \sum_{r\in R_i}(\theta_{i,r}-\cg_r z_{i,r}) - \rho_i\Bigg] (\dlb_i + b_i\Delta d_i) \label{eqn:mean_support_inner_max_MINLP_obj} \\
 \textup{subject to} \quad
&  \text{\eqref{eqn:2nd_stage_dual_reform_con1}--\eqref{eqn:2nd_stage_dual_reform_con4}},\label{eqn:mean_support_inner_max_MINLP_con1} \\
&  b_i\in\{0,1\},\quad\forall i\in I. \label{eqn:mean_support_inner_max_MINLP_con2}
\end{align} \label{eqn:mean_support_inner_max_MINLP}%
\end{subequations}%
Finally, to reformulate it into an MILP, we introduce auxiliary variables $\zeta^L_{i,i'}=\lambda_{i,i'} b_i$, $\zeta^M_{i,a}=\mu_{i,a} b_i$ and $\zeta^T_{i,r}=\theta_{i,r} b_i$ with McCormick inequalities given by \eqref{eqn:mean_support_inner_max_MILP_con4}--\eqref{eqn:mean_support_inner_max_MILP_con6}. The existence of the upper bounds of dual variables for McCormick inequalities will be shown in Proposition \ref{prop:dual_var_UB}. \Halmos
\end{proof}

\subsection{Proof of Proposition \ref{prop:dual_var_UB}} \label{appdx:pf_dual_var_UB}

To obtain an upper bound for the dual variables, we first derive the complementary properties of the optimal dual solutions in Lemma \ref{lem:2nd_stage_complementarity}.

\begin{lemma} \label{lem:2nd_stage_complementarity}
For a given first-stage decision $(x,y,z,v,u,s)$ and realization $d$, let $(\lambda^*,\,\mu^*,\,\theta^*)$ be an optimal solution to \eqref{eqn:2nd_stage_dual_reform}. Then, we have the following equations:
\begin{subequations} 
\begin{align}
\mu^*_{i,a}(1-x_{i,a}+y_a) &= 0,\quad\forall (i,a)\in\calF^A, \\
\theta^*_{i,r}(1-z_{i,r}) &=0,\quad\forall (i,r)\in\calF^R, \\
\lambda^*_{i,i'}(1-u_{i,i'}) &= 0,\quad\forall i\in I,\, i'\in I\setminus\{i\}.
\end{align} 
\end{subequations}%
\end{lemma}

\begin{proof}{Proof of Lemma \ref{lem:2nd_stage_complementarity}}
 For brevity, we only argue that $\lambda^*_{i,i'}(1-u_{i,i'})=0$ and the remaining two equations can be derived in the same manner. If $u_{i,i'}=1$, then the equation holds immediately. Consider the case that $u_{i,i'}\ne 0$. Then, by a sufficiently large choice of the big $M$ parameters (see \ref{appdx:big_M}), the slack variable associated to constraint \eqref{eqn:2nd_stage_con1} is strictly positive. By the complementary slackness condition, the dual variable $\lambda^*_{i,i'}$ is zero. This completes the proof. \Halmos
\end{proof}

\begin{proof}{Proof of Proposition \ref{prop:dual_var_UB}}
The bounds for $\mu_{i,a}$ follow immediately from the its non-negativity constraint \eqref{eqn:2nd_stage_dual_con4} and the constraint \eqref{eqn:2nd_stage_dual_con1}. Similarly, the bounds for $\theta_{i,r}$ follow from constraints \eqref{eqn:2nd_stage_dual_con2} and \eqref{eqn:2nd_stage_dual_con4}. For $\lambda_{i,i'}$, the lower bound follows from \eqref{eqn:2nd_stage_dual_con4}. 

Next, we derive the upper bound for $\lambda_{i,i'}$. Consider a given realization $d$ and denote the actual surgery start time as $q_i$ for $i\in I$. Note that this is known when given a feasible first-stage decision. Without loss of generality, let $q_1 \geq q_2 \geq \dots \geq q_{|I|}$. Define $C_i = \cw_i+\sum_{a\in A_i}\mu_{i,a} + \sum_{r\in R_i} \theta_{i,r}$ for $i\in I$. From Lemma \ref{lem:2nd_stage_complementarity}, at optimality, we immediately have $\lambda_{i,i'}=0$ for any $i'\geq i$ (since surgery $i$ cannot precede surgery $i'$). Using this observation, the constraints \eqref{eqn:2nd_stage_dual_con3} read as follows:
\begin{align*}
\sum_{i'\geq 2}\lambda_{i',1} &\leq \cw_1 + \sum_{a\in A_1}\mu_{1,a} + \sum_{r\in R_1}\theta_{1,r}, \\
\sum_{i'\geq 3}\lambda_{i',2} &\leq \lambda_{2,1} + \cw_2 + \sum_{a\in A_2}\mu_{2,a} + \sum_{r\in R_2}\theta_{2,r}, \\
\sum_{i'\geq 4}\lambda_{i',3} &\leq \lambda_{3,1} + \lambda_{3,2} + \cw_3 + \sum_{a\in A_3}\mu_{3,a} + \sum_{r\in R_3}\theta_{3,r}, \\
& \vdots \\
\sum_{i'\geq |I|}\lambda_{i',|I|-1} &\leq \lambda_{|I|-1,1} + \cdots + \lambda_{|I|-1,|I|-2} +  \cw_{|I|-1} + \sum_{a\in A_{|I|-1}}\mu_{|I|-1,a} + \sum_{r\in R_{|I|-1}}\theta_{|I|-1,r}.
\end{align*}
Note that the constraint for $|I|$ (i.e., the last inequality) holds automatically since the summation on the left is zero. 
%
%
If we sum from the first to the $j$th inequalities, we obtain
$$\sum_{i=1}^j \sum_{i'=j+1}^{|I|} \lambda_{i',i} \leq \sum_{i=1}^j C_i,\quad\forall j\in\{1,\dots,|I|-1\}.$$
Note that all $\lambda_{i,i'}$ with $i'<i$ appear in at least one of the above $|I|-1$ inequalities. Since $\lambda_{i,i'}\geq 0$, this concludes that an upper bound for $\lambda_{i,i}$ is given by
\begin{align*}
\sum_{i=1}^{|I|} C_i &= \sum_{i\in I} \cw_i + \sum_{i\in I}\sum_{a\in A_i} \mu_{i,a} + \sum_{i\in I}\sum_{r\in R_i}\theta_{i,r} \\
&\leq \sum_{i\in I} \cw_i + \sum_{a\in A} (\cg_a + \co_a) + \sum_{r\in R}(\cg_r + \co_r) ,
\end{align*}
where the last inequality follows from constraints \eqref{eqn:2nd_stage_dual_con1} and \eqref{eqn:2nd_stage_dual_con2}. \Halmos
\end{proof}

\begin{remark}
From the proof of Proposition \ref{prop:dual_var_UB}, we could obtain a tighter upper bound for $\lambda$ by summing over only the largest $|I|-1$ constants $C$. That is, 
$$\lambda_{i,i'} \leq \Bigg(\sum_{i\in I} \cw_i - \min_{i\in I} \cw_i \Bigg) + \sum_{a\in A}(\cg_a+\co_a) + \sum_{r\in R} (\cg_r + \co_r).$$
In particular, if $\cw_i=c^W$ for all $i\in I$, $\cg_a=\cg_A$, $\co_a=\co_A$ for all $a\in A$ and $\cg_r=\cg_R$, $\co_r=\co_R$ for all $r\in R$, the upper bound is given by
$$\lambda_{i,i'} \leq \Big( |I|-1 \Big) \cw + |A| (\cg_A+\co_A) + |R| (\cg_R + \co_R).$$
In Example \ref{appdx:eg_tightUB}, we show that this bound is the tightest possible constant upper bound of $\lambda_{i,i'}$.
\end{remark}

\begin{example} \label{appdx:eg_tightUB}
Suppose we have one anesthesiologist, two ORs and five surgeries. Their schedules (i.e., first-stage decisions) and the surgery durations are given as follows. \par \vspace{2mm}
\makebox[3cm]{Anesthesiologist 1}: $5\rightarrow 4\rightarrow 3\rightarrow 2\rightarrow 1$ \par
\makebox[3cm]{Operating room 1}: $5\rightarrow 4\rightarrow 3$ \qquad\qquad\qquad \makebox[3cm]{Operating room 2}: $2\rightarrow 1$  \vspace{-2.5mm}
$$\hspace{-6cm} s = (400, 300, 200, 100, 0) \quad \quad \quad d = (150, 160, 170, 180, 190)$$
Assume the cost structure $c^W=100$, $\cg_A=30$, $\co_A=150$, $\cg_R=20$ and $\co_R=450$. We can solve the second-stage model to obtain the optimal dual solutions, which are given by
$$\OneAndAHalfSpacedXI
\mu = \begin{pmatrix} 180 \\ 0 \\ 0 \\ 0 \\ 0\end{pmatrix}, \quad 
\theta = \begin{pmatrix} 0 & 470 \\ 0 & 0\\ 470 & 0 \\ 0 & 0\\ 0 & 0 \end{pmatrix}, \quad 
\lambda = \begin{pmatrix} 0 & 0 & 0 & 0 & 0 \\ 750 & 0 & 0 & 0 & 0 \\ 0 & 850 & 0 & 0 & 0 \\ 0 & 0 & 1420 & 0 & 0 \\ 0 & 0 & 0 & 1520 & 0  \end{pmatrix}.
$$
Note that the upper bound on $\lambda$ we derived in Proposition \ref{prop:dual_var_UB} is 
$$\Big( |I|-1 \Big) \cw + |A| (\cg_A+\co_A) + |R| (\cg_R + \co_R) = 4(100)+(30+150)+2(20+450)=1520.$$
This shows that the upper bound on $\lambda$ we obtain in the proof of Proposition \ref{prop:dual_var_UB} is tight. It is straightforward to see that the upper bounds on $\mu$ and $\theta$ are tight as well. 
\end{example}

\subsection{Proof of Proposition \ref{prop:mean_support_WC_CVaR}} \label{appdx:pf_mean_support_WC_CVaR}
\begin{proof}{Proof of Proposition \ref{prop:mean_support_WC_CVaR}}
 For notational simplicity, we suppress the dependence of $\Prob$ in $\E_{\Prob}(\cdot)$ and $\Prob\mhyphen\CVaR_\gamma(\cdot)$ in the proof. First, recall the definition of CVaR \citep{Rockafellar_Uryasev:2000}:
\begin{equation} \label{eqn:CVaR_def}
    \CVaR_\gamma\big(Z)=\min_{\tau\in\R} \Bigg\{ \tau + \frac{1}{1-\gamma}\E\big[ \max\{Z-\tau,0\}\big] \Bigg\}.
\end{equation}
With this definition, we have
\begin{equation}
\sup_{\Prob\in\calP(m,\calS)} \CVaR_\gamma\big(Q(x,y,z,v,u,s,d)\big)= \sup_{\Prob\in\calP(m,\calS)} \min_{\tau\in\R} \Bigg\{ \tau + \frac{1}{1-\gamma}\E\big[ \max\{Q(x,y,z,v,u,s,d)-\tau,0\}\big] \Bigg\}.
\end{equation}
The objective function is convex in $\tau$ and concave in $\Prob$. Moreover, the set $\calP(m,\calS)$ is weakly compact (under the topology of weak convergence of probability measures) (see, e.g., \cite{Sun_Xu:2016ec}). Then, by Sion's minimax theorem \citep{Sion:1958ec},  
\begin{align}
\sup_{\Prob\in\calP(m,\calS)} \CVaR_\gamma\big(Q(x,y,z,v,u,s,d)\big)&= \sup_{\Prob\in\calP(m,\calS)} \min_{\tau\in\R} \Bigg\{ \tau + \frac{1}{1-\gamma}\E\big[ \max\{Q(x,y,z,v,u,s,d)-\tau,0\}\big] \Bigg\} \nonumber \\
&= \min_{\tau\in\R} \Bigg\{ \tau + \frac{1}{1-\gamma} \sup_{\Prob\in\calP(m,\calS)} \E\big[ \max\{Q(x,y,z,v,u,s,d)-\tau,0\}\big] \Bigg\}. \label{eqn:mean_support_CVaR_inner}
\end{align}
Applying the same argument for strong duality of the worst-case expectation problem in Proposition \ref{prop:mean_support_WC_exp}, we have that $\sup_{\Prob\in\calP(m,\calS)} \E\big[ \max\{Q(x,y,z,v,u,s,d)-\tau,0\}\big]$ is equivalent to
\begin{subequations}
\begin{align}
 \underset{\rho_0\in\R,\,\rho\in\R^I}{\text{minimize}} & \quad
\rho_0 + \sum_{i\in I}\rho_i m_i \\
 \text{subject to} &   \quad \rho_0 + \sum_{i\in I}\rho_i d_i\geq Q(x,y,z,v,u,s,d)-\tau ,\quad\forall d\in\calS,\label{eqn:mean_support_CVaR_model_reform2_C1} \\
 & \quad \rho_0 + \sum_{i\in I}\rho_i d_i\geq 0 ,\quad\forall d\in\calS, \label{eqn:mean_support_CVaR_model_reform2_C2}
\end{align}  \label{eqn:mean_support_CVaR_model_reform2}%
\end{subequations}
where \eqref{eqn:mean_support_CVaR_model_reform2_C1} and \eqref{eqn:mean_support_CVaR_model_reform2_C2} follows from the definition of $\max \{\cdot, 0\}$. Note that these constraints hold for all $d\in\calS$. Therefore, we can reformulate \eqref{eqn:mean_support_CVaR_model_reform2} as
\begin{subequations}
\begin{align}
 \underset{\rho_0\in\R,\,\rho\in\R^I}{\text{minimize}} & \quad
\rho_0 + \sum_{i\in I}\rho_i m_i \\
 \text{subject to} &   \quad \tau \geq \max_{d\in\calS} \Bigg\{Q(x,y,z,v,u,s,d)-\sum_{i\in I} \rho_id_i \Bigg\} - \rho_0, \\
 & \quad \rho_0 + \min_{d\in\calS} \sum_{i\in I}\rho_i d_i\geq 0.
\end{align}  \label{eqn:mean_support_CVaR_model_reform3}%
\end{subequations}
Combining \eqref{eqn:mean_support_CVaR_model_reform3} with the outer minimization problem in $\tau$ in \eqref{eqn:mean_support_CVaR_inner}, we derive the following equivalent reformulation of $\sup_{\Prob\in\calP(m,\calS)} \CVaR_\gamma\big(Q(x,y,z,v,u,s,d)\big)$:
%
\begin{subequations}
\begin{align}
 \underset{\tau\in\R,\,\rho_0\in\R,\,\rho\in\R^I}{\text{minimize}} & \quad
\tau + \frac{1}{1-\gamma} \Bigg(\rho_0 + \sum_{i\in I}\rho_i m_i \Bigg)\\
 \text{subject to} &   \quad \tau \geq \max_{d\in\calS} \Bigg\{Q(x,y,z,v,u,s,d)-\sum_{i\in I} \rho_id_i \Bigg\} - \rho_0, \label{eqn:mean_support_CVaR_model_reform4_con1} \\
 & \quad \rho_0 + \min_{d\in\calS} \sum_{i\in I}\rho_i d_i\geq 0. 
\end{align}  \label{eqn:mean_support_CVaR_model_reform4}%
\end{subequations}
Note that for any given $\rho_0\in\R$ and $\rho\in\R^I$, the optimal solution $\tau^*=\max_{d\in\calS} \big\{Q(x,y,z,v,u,s,d)-\sum_{i\in I} \rho_id_i \big\} - \rho_0$ from constraint \eqref{eqn:mean_support_CVaR_model_reform4_con1}. Therefore, we can reformulate  problem \eqref{eqn:mean_support_CVaR_model_reform4} as
\begin{subequations}
\begin{align}
 \underset{\rho_0\in\R,\,\rho\in\R^I}{\text{minimize}} & \quad
\Bigg(\frac{1}{1-\gamma}-1\Bigg)\rho_0 + \frac{1}{1-\gamma} \sum_{i\in I}\rho_i m_i + \max_{d\in\calS} \Bigg\{Q(x,y,z,v,u,s,d)-\sum_{i\in I} \rho_id_i \Bigg\}\\
 \text{subject to} & \quad \rho_0 + \min_{d\in\calS} \sum_{i\in I}\rho_i d_i\geq 0. \label{eqn:mean_support_CVaR_model_reform5_con1}
\end{align}  \label{eqn:mean_support_CVaR_model_reform5}%
\end{subequations}
Finally, by LP duality, we have
$$\min_{d\in\calS} \sum_{i\in I}\rho_id_i=\max_{\psi\in\Psi(\rho)} \sum_{i\in I} \big(\dlb_i\underline{\psi}_i-\dub_i\overline{\psi}_i\big),$$
where $\Psi(\rho)=\{(\underline{\psi}_i,\overline{\psi}_i)\in\R_+^I \times \R_+^I \mid \underline{\psi}_i-\overline{\psi}_i=\rho_i,\,\forall i\in I\}$.
Therefore, if there exists $\psi\in\Psi(\rho)$ such that $\rho_0 + \sum_{i\in I} \big(\dlb_i\underline{\psi}_i-\dub_i\overline{\psi}_i\big) \geq 0$, constraint \eqref{eqn:mean_support_CVaR_model_reform5_con1} holds, and vice versa. Hence, we can derive the following reformulation of $\sup_{\Prob\in\calP(m,\calS)} \CVaR_\gamma\big(Q(x,y,z,v,u,s,d)\big)$:
\begin{subequations}
\begin{align}
 \underset{\rho_0\in\R,\,\rho\in\R^I,(\underline{\psi}_i,\overline{\psi}_i)\in\R^I \times \R^I}{\text{minimize}} & \quad
\Bigg(\frac{1}{1-\gamma}-1\Bigg)\rho_0 + \frac{1}{1-\gamma} \sum_{i\in I}\rho_i m_i + \max_{d\in\calS} \Bigg\{Q(x,y,z,v,u,s,d)-\sum_{i\in I} \rho_id_i \Bigg\}\\
 \text{subject to\hspace{10mm}} & \quad \rho_0 + \sum_{i\in I} \big(\dlb_i\underline{\psi}_i-\dub_i\overline{\psi}_i\big) \geq 0, \label{eqn:mean_support_CVaR_model_final_reform_con1}\\
 &\quad \underline{\psi}_i-\overline{\psi}_i=\rho_i,\quad\forall i\in I, \label{eqn:mean_support_CVaR_model_final_reform_con2}\\
 &\quad \underline{\psi}_i\geq 0,\, \overline{\psi}_i\geq 0,\quad\forall i\in I. \label{eqn:mean_support_CVaR_model_final_reform_con3}
\end{align}  \label{eqn:mean_support_CVaR_model_final_reform}%
\end{subequations}
This completes the proof. \Halmos
\end{proof}

\begin{remark}
From constraint \eqref{eqn:mean_support_CVaR_model_reform5}, we note that $\rho_0^*= -  \min_{d\in\calS} \sum_{i\in I}\rho_i d_i$ is an optimal solution since the coefficient associated to $\rho_0$ in the objective is $1/(1-\gamma)-1 \geq 0$ for $\gamma\in(0,1)$. Therefore, an equivalent reformulation of $\sup_{\Prob\in\calP(m,\calS)} \CVaR_\gamma\big(Q(x,y,z,v,u,s,d)\big)$ without going through LP duality is
\begin{equation} \label{eqn:mean_support_CVaR_model_final_reform2}
\min_{\rho\in\R^I} \Bigg\{\frac{1}{1-\gamma} \sum_{i\in I}\rho_i m_i -   \Bigg(\frac{1}{1-\gamma}-1\Bigg)  \min_{d\in\calS} \sum_{i\in I}\rho_i d_i + \max_{d\in\calS} \Bigg\{Q(x,y,z,v,u,s,d)-\sum_{i\in I} \rho_id_i \Bigg\}  \Bigg\}.
\end{equation}
However, this does not facilitate the use of the C\&CG algorithm (see Section \ref{subsec:C&CG}) since we have two optimization problems inside the objective function. Nevertheless, we will use \eqref{eqn:mean_support_CVaR_model_final_reform2} to verify the separability of the \drocvar{} model.

\end{remark}

\section{Proofs and Discussions on Valid Inequalities for the \droe{} and \drocvar{} Models} \label{appdx:VI_DRO}

\subsection{Valid Inequalities for the Recourse Function} \label{appdx:VI_recourse_mean_LB}
\begin{lemma} \label{lem:VI_recourse_mean_LB_DRO_E}
The following lower-bounding inequalities are valid for the \droe{} model.
\begin{subequations} 
\begin{align}
 &  \sum_{i\in I} m_i\rho_i + \delta \geq \sum_{a\in A} (\co_a o^\text{m}_a + \cg_a g^\text{m}_a) + \sum_{r\in R} (\co_r o^\text{m}_r + \cg_r g^\text{m}_r)+\sum_{i\in I} \cw_i w^\text{m}_i , \label{eqn:mom_recourse_LB_con1}\\
 & q^\text{m}_{i'} \geq q^\text{m}_i + m_i - M_\text{seq} (1 - u_{i,i'}),\quad\forall \{i,\,i'\}\subseteq I,\, i'\ne i, \label{eqn:mom_recourse_LB_con2}\\
 & q^\text{m}_i \geq s_i,\quad\forall i\in I, \label{eqn:mom_recourse_LB_con3}\\
 & o^\text{m}_a \geq q^\text{m}_i + m_i - t^\text{end}_a - M_\text{anes}(1-x_{i,a} + y_a),\quad\forall (i,a)\in \calF^A, \label{eqn:mom_recourse_LB_con4}\\
 & o^\text{m}_r \geq q^\text{m}_i + m_i - T^\text{end} - M_\text{room}(1 - z_{i,r}),\quad\forall (i,r)\in \calF^R, \label{eqn:mom_recourse_LB_con5}\\
 & w^\text{m}_i \geq q^\text{m}_i - s_i,\quad\forall i\in I, \label{eqn:mom_recourse_LB_con6}\\
 & g^\text{m}_a \geq \Bigg[(t^\textup{end}_a - t^\textup{start}_a) - \sum_{i\in I_a} m_i x_{i,a} + o^\text{m}_a\Bigg] g_a, \quad\forall (i,a)\in\calF^A, \label{eqn:mom_recourse_LB_con7}\\
 & g^\text{m}_r \geq T^\textup{end}v_r - \sum_{i\in I_r} m_i z_{i,r} + o^\text{m}_r, \quad\forall (i,r)\in\calF^R, \label{eqn:mom_recourse_LB_con8}\\
 & q^\text{m}_i,\,o^\text{m}_a,\,o^\text{m}_r,\, w^\text{m}_i,\,\,g^\text{m}_a,\,g^\text{m}_r \geq 0,\quad\forall i\in I,\, a\in A,\, r\in R. \label{eqn:mom_recourse_LB_con9}
\end{align} \label{eqn:mom_recource_LB}%
\end{subequations}%
\vspace{-10mm}
\end{lemma}

\begin{proof}{Proof.}
Note that we minimize $\sum_{i\in I} m_i\rho_i + \delta$ in the master problem. This is equivalent to minimizing the right-hand side expression of \eqref{eqn:mom_recourse_LB_con1}, which is the second-stage cost. \eqref{eqn:mom_recourse_LB_con2}--\eqref{eqn:mom_recourse_LB_con8} are the second-stage constraints with the scenario $d=m$ (i.e., the mean of the surgery duration). Therefore, in view of the inequality
$$\sup_{\Prob\in\calP(m,\calS)} \varrho_\Prob\big(Q(x,y,z,v,u,s,d)\big)  \geq Q(x,y,z,v,u,s,m),$$
the inequalities are valid. \Halmos
\end{proof}

\begin{lemma}
The following lower-bounding inequalities are valid for the \drocvar{} model.
\begin{subequations} 
\begin{align}
 &  \Bigg(\frac{1}{1-\gamma}-1\Bigg)\rho_0 + \frac{1}{1-\gamma} \sum_{i\in I}\rho_i m_i + \delta \geq \sum_{a\in A} (\co_a o^\text{m}_a + \cg_a g^\text{m}_a) + \sum_{r\in R} (\co_r o^\text{m}_r + \cg_r g^\text{m}_r)+\sum_{i\in I} \cw_i w^\text{m}_i , \label{eqn:mom_recourse_LB_DRO_CVaR_con1}\\
 & \eqref{eqn:mom_recourse_LB_con2}-\eqref{eqn:mom_recourse_LB_con9}. \label{eqn:mom_recourse_LB_DRO_CVaR_con22}
\end{align} \label{eqn:mom_recource_LB_DRO_CVaR}%
\end{subequations}%
\vspace{-10mm}
\end{lemma}

\begin{proof}{Proof.}
The validity of inequalities \eqref{eqn:mom_recourse_LB_DRO_CVaR_con1}--\eqref{eqn:mom_recourse_LB_DRO_CVaR_con22} can be easily verified using the same argument and techniques in the proof of Lemma~\ref{lem:VI_recourse_mean_LB_DRO_E}.  \Halmos
\end{proof}

\subsection{Proof of Proposition \ref{prop:VI_dual_bounds}} \label{appdx:pf_VI_dual_bounds}
\begin{proof}{Proof of Proposition \ref{prop:VI_dual_bounds}}
We focus on the \droe{} model since the same argument applies to the \drocvar{} model. For simplicity, we suppress the dependence of $x$, $z$ and $u$ in $\rholb$ and $\rhoub$.  First, we consider the lower bound. For a fixed first-stage decision, suppose that there exists $i\in I$ such that $\rho^*_i < \rholb$, where $\rho^*$ is an arbitrary solution. Let $(\lambda^*,\mu^*,\theta^*,d^*)$ be an optimal solution to \eqref{eqn:mean_support_inner_max_naive}, i.e., the dual reformulation of $\max_{d\in\calS}\Big\{ Q(x,y,z,v,u,s,d)-\sum_{i\in I}\rho_id_i\Big\}$. Since $\rho^*_i<\rholb$, in the objective function, the coefficient associated to $d_i$ is positive. Indeed,
\begin{align*}
&\quad\,\,\sum_{i'\in I,i'\ne i}\lambda^*_{i,i'} + \sum_{a\in A_i} (\mu^*_{i,a} - \cg_a \hreg_a x_{i,a} ) + \sum_{r\in R_i} (\theta^*_{i,r} - \cg_r z_{i,r} ) - \rho^*_i\\
& \geq -\sum_{a\in A_i} \cg_a \hreg_a x_{i,a} - \sum_{r\in R_i} \cg_r z_{i,r} - \rho^*_i\\
& \geq -\sum_{a\in A_i} \cg_a \hreg_a x_{i,a} - \sum_{r\in R_i} \cg_r z_{i,r} - \rholb \\
&= 0.
\end{align*}
Therefore, we must have $d^*_i = \dub_i$. Next, for any $\epsilon\in(0,\rholb-\rho^*_i]$, define another solution $\rhot$ by
$$\rhot_j=\begin{cases} \rho^*_j+\epsilon, &\text{ if } j=i, \\ \rho^*_j, &\text{ otherwise.}\end{cases}$$
Note that changing $\rho^*_i$ to $\rhot_i$ only affects the coefficient associated to $d_i$ in the objective \eqref{eqn:mean_support_inner_max_naive_obj}. By construction of $\epsilon$, the coefficient associated to $d_i$ is still positive. Therefore, $d^*_i=\dub_i$ is still optimal. After maximizing $d_i$, the remaining program of \eqref{eqn:mean_support_inner_max_naive} is the same. That is, we have
\begin{equation} \label{eqn:pf_VI_dual_bounds}
    \sup_{d\in\calS}\Big\{Q(x,y,z,v,u,s,d)-\sum_{i\in I}\rhot_i d_i \Big\} = \sup_{d\in\calS}\Big\{Q(x,y,z,v,u,s,d)-\sum_{i\in I}\rho^*_i d_i \Big\} - \epsilon \dub_i.
\end{equation}
Hence, we have
\begin{align*}
&\quad\,\,\sum_{i\in I}\rhot_i m_i + \sup_{d\in\calS}\Big\{Q(x,y,z,v,u,s,d)-\sum_{i\in I}\rhot_i d_i \Big\} \\
&= \Bigg(\sum_{i\in I}\rho^*_i m_i+\epsilon m_i \Bigg) + \sup_{d\in\calS}\Big\{Q(x,y,z,v,u,s,d)-\sum_{i\in I}\rho^*_i d_i \Big\} -\epsilon\dub_i \\
&= \Bigg(\sum_{i\in I}\rho^*_i m_i + \sup_{d\in\calS}\Big\{Q(x,y,z,v,u,s,d)-\sum_{i\in I}\rho^*_i d_i \Big\} \Bigg)+ \epsilon(m_i - \dub_i) \\
&\leq \sum_{i\in I}\rho^*_i m_i + \sup_{d\in\calS}\Big\{Q(x,y,z,v,u,s,d)-\sum_{i\in I}\rho^*_i d_i \Big\}.
\end{align*}
The first equality follows from \eqref{eqn:pf_VI_dual_bounds} while the second equality is obtained by re-arraigning the terms. The inequality follows from $m_i\leq\dlb_i$ and $\epsilon>0$. Therefore, without loss of optimality, we can assume that $\rho_i\geq\rholb_i$.

To derive an upper bound for $\rho_i$, it suffices to determine a upper bound for the coefficient associated to $d_i$, i.e., 
$$\sum_{i'\in I,i'\ne i}\lambda_{i,i'} + \sum_{a\in A_i} (\mu_{i,a} - \cg_a \hreg_a x_{i,a} ) + \sum_{r\in R_i} (\theta_{i,r} - \cg_r z_{i,r} ) .$$
We can then derive the upper bound by applying a similar argument when deriving the lower bound. Note, from the complementary slackness condition of the second-stage problem \eqref{eqn:2nd_stage}, $\lambda_{i,i'}\ne 0$ if the constraint \eqref{eqn:2nd_stage_con1} achieves equality. Since we only have two resources (anesthesiologist and OR), at most two surgeries will immediately follow surgery $i$. Hence, using the upper bound on $\lambda$ provided in Proposition \ref{prop:dual_var_UB}, we have $\sum_{i'\in I,i'\ne i} \lambda_{i,i'} \leq \min\big\{\sum_{i'\in I,i'\ne i} u_{i,i'}, 2\big\} \,\lambdaub$.


Next, we consider the second term. For regular anesthesiologists (i.e., $\hreg_a = 1$), by the complementary slackness condition in Lemma \ref{lem:2nd_stage_complementarity}, we have $\mu_{i,a} = 0$ if $x_{i,a}=0$. Since $\mu_{i,a}^I \leq \cg_a + \co_a$ from Proposition \ref{prop:dual_var_UB}, we have $\mu_{i,a} - \cg_a \hreg_a x_{i,a} \leq \co_a x_{i,a}$. For on-call anesthesiologists (i.e., $\hreg_a=0$), we have $x_{i,a}=0$ if $y_a=0$ and this follows that $\mu_{i,a}=0$. On the other hand, if $y_a=1$, then $\mu_{i,a}=0$ from Lemma \ref{lem:2nd_stage_complementarity}. Hence, for on-call anesthesiologists, we have $\mu_{i,a} - \cg_a \hreg_a x_{i,a} = 0 \leq \co_a x_{i,a}$. Therefore, this results in the desired upper bound $\co_a x_{i,a}$. Finally, a similar argument using complementarity condition on the last term yields the desired upper bound for the last term. \Halmos

\end{proof}

\section{Separability of the Models} \label{appdx:separability}
\subsection{Discussion on Separability of the \droe{} Model} \label{appdx:DRO_sep}


Note that for some instances of the ORSAP, the \droe{} model can be decomposed into a sum of sub-problems. In this appendix, we discuss how the \droe{} model could be decomposed into a sum of sub-problems. \blue{Recall that in Section~\ref{subsec:separability}, under practical settings, we can decompose the recourse problem as
$$Q(x,y,z,v,u,s,d)=\sum_{g\in G} Q^g(x^g,y^g,z^g,u^g,s^g,d^g),$$
where each $Q^g(x^g,y^g,z^g,u^g,s^g,d^g)$ is characterized by $(I^g,R^g,A^g)$ (a subset of $(I,R,A)$) and a set of surgery types $T^g$ (a subset of the surgery types $T$).} Recall the support of $d$ is $\calS=\calS^1\times\calS^2\times\dots\times\calS^L$, where $\calS^\ell$ is the support for the surgery of type $\ell\in\{1,\dots,L\}$. For notation simplicity, we write $\chi=(x,y,z,v,u,s)$. Hence,
\begin{align} 
    \sup_{\Prob\in\calF(\calS,m)} \E_\Prob\big[Q(\chi,d)\big] 
    &= \min_{\rho} \Bigg\{ \sum_{i\in I}\rho_i m_i + \sup_{d\in\calS} \bigg\{Q(\chi,d) - \sum_{i\in I} \rho_i d_i\bigg\} \Bigg\} \label{eqn:sep_1} \\
    &= \min_{\rho} \Bigg\{ \sum_{g\in G} \big(\rho^g)^\tp m^g +  \sup_{d\in\calS} \bigg\{ \sum_{g\in G} Q^g(\chi^g,d^g) - \sum_{g\in G} \big(\rho^g)^\tp d^g  \bigg\} \Bigg\} \label{eqn:sep_2} \\
    &= \min_{\rho} \Bigg\{ \sum_{g\in G} \big(\rho^g)^\tp m^g +  \sum_{g\in G} \sup_{d\in\calS^g} \Big\{ Q^g(\chi^g,d^g)-\big(\rho^g\big)^\tp d^g  \Big\} \Bigg\}\label{eqn:sep_3} \\
    &= \sum_{g\in G} \min_{\rho^g} \Bigg\{ \big(\rho^g)^\tp m^g +\sup_{d\in\calS^g} \Big\{ Q^g(\chi^g,d^g)-\big(\rho^g)^\tp d^g  \Big\} \Bigg\}  \label{eqn:sep_4}\\
    &=\sum_{g\in G} \sup_{\Prob^g\in\calF(\calS^g,m^g)} \E_{\Prob^g}\big[Q(\chi^g,d^g)\big]   \label{eqn:sep_5}.
\end{align}%
Equation \eqref{eqn:sep_1} follows from Proposition \ref{prop:mean_support_WC_exp} while  equation \eqref{eqn:sep_2} follows from the recourse function decomposition. In particular, vectors $\rho^g$ and $m^g$ correspond to entries of surgeries in $I^g$. Equation \eqref{eqn:sep_3} makes use of the separability of $\calS=\bigtimes_{g\in G} \calS^g$, where $\calS^g$ is the support of durations for surgeries in $I^g$. Equation \eqref{eqn:sep_4} follows from the separability of the objective function in $\rho$. Equation \eqref{eqn:sep_5} is a direct consequence of Proposition \ref{prop:mean_support_WC_exp}.

The derivation also shows that the worst-case distribution of the \droe{} model takes the form $\bigtimes_{g\in G}\Prob^g_\star$ for some $\Prob^g_\star\in\calF(\calS^g,m^g)$. That is, durations of surgeries in $I^g$ are independent of those in $I^{g'}$ for any $\{g,g'\}\subseteq G$. This is because some surgeries within $I^g$ share a pool of ORs and/or anesthesiologists. Hence, ORs in $R^g$ may be assigned with surgeries of multiple types, and anesthesiologists in $A^g$ may perform surgeries of multiple types. These lead to the possibility that durations of surgeries in $I^g$ are not independent in the worst case.


\subsection{Separability of the \drocvar{} Model} \label{appdx:CVaR_DRO_sep}
Next, we show that the \drocvar{} model can be decomposed into a sum of sub-problems as in the \droe{} model. Note that $\Prob\mhyphen\CVaR_\gamma(\cdot)$ is subadditive, i.e., $\Prob\mhyphen\CVaR_\gamma(Z_1+Z_2) \leq \Prob\mhyphen\CVaR_\gamma(Z_1) + \Prob\mhyphen\CVaR_\gamma(Z_2)$ for any integrable random variables $Z_1$ and $Z_2$. Therefore, if we use $\Prob\mhyphen\CVaR_\gamma(\cdot)$ directly as our objective, we may not be able to obtain such a decomposition.

\blue{To study the separability of the \drocvar{} model, using the notation in \ref{appdx:DRO_sep}, we can rewrite the recourse function as
$$Q(x,y,z,v,u,s,d)=\sum_{g\in G} Q^g(x^g,y^g,z^g,u^g,s^g,d^g).$$}
Recall the support of $d$ is $\calS=\calS^1\times\calS^2\times\dots\times\calS^L$, where $\calS^\ell$ is the support for the surgery of type $\ell\in\{1,\dots,L\}$. Again, for notation simplicity, we write $\chi=(x,y,z,v,u,s)$. Hence, we have
\begin{align} 
    &\quad\, \sup_{\Prob\in\calF(\calS,m)} \Prob\mhyphen\CVaR_\gamma\big(Q(\chi,d)\big) \nonumber\\
    &= \min_{\rho} \Bigg\{\frac{1}{1-\gamma} \sum_{i\in I}\rho_i m_i -   \Bigg(\frac{1}{1-\gamma}-1\Bigg)  \min_{d\in\calS} \sum_{i\in I}\rho_i d_i + \sup_{d\in\calS} \Bigg\{Q(\chi,d)-\sum_{i\in I} \rho_id_i \Bigg\}  \Bigg\} \label{eqn:sep_CVaR_1} \\
    &= \min_{\rho} \Bigg\{ \frac{1}{1-\gamma}\sum_{g\in G} \big(\rho^g)^\tp m^g - \Bigg(\frac{1}{1-\gamma}-1\Bigg)  \min_{d\in\calS} \bigg\{\sum_{g\in G}\big(\rho^g\big)^\tp d^g \bigg\} + \sup_{d\in\calS} \bigg\{ \sum_{g\in G} Q^g(\chi^g,d^g) - \sum_{g\in G} \big(\rho^g)^\tp d^g  \bigg\} \Bigg\} \label{eqn:sep_CVaR_2} \\
    &= \min_{\rho} \Bigg\{ \frac{1}{1-\gamma}\sum_{g\in G} \big(\rho^g)^\tp m^g  - \Bigg(\frac{1}{1-\gamma}-1\Bigg) \Bigg[ \sum_{g\in G} \min_{d\in\calS^g} \big(\rho^g\big)^\tp d^g   \Bigg] + \sum_{g\in G} \sup_{d\in\calS^g} \Big\{ Q^g(\chi^g,d^g)-\big(\rho^g)^\tp d^g \Big\} \Bigg\} \label{eqn:sep_CVaR_3} \\
    &= \sum_{g\in G} \min_{\rho^g} \Bigg\{ \frac{1}{1-\gamma} \big(\rho^g)^\tp m^g - \Bigg(\frac{1}{1-\gamma}-1\Bigg)  \min_{d\in\calS^g} \big(\rho^g\big)^\tp d^g + \sup_{d\in\calS^g} \Big\{ Q^g(\chi^g,d^g)-\big(\rho^g)^\tp d^g  \Big\} \Bigg\}  \label{eqn:sep_CVaR_4}\\
    &=\sum_{g\in G} \sup_{\Prob^g\in\calF(\calS^g,m^g)} \Prob^g\mhyphen\CVaR_\gamma\big(Q(\chi^g,d^g)\big)   \label{eqn:sep_CVaR_5}.
\end{align}
Equation \eqref{eqn:sep_CVaR_1} follows from \eqref{eqn:mean_support_CVaR_model_final_reform2} while equation \eqref{eqn:sep_CVaR_2} follows from the recourse function decomposition. IIn particular, vectors $\rho^g$ and $m^g$ correspond to entries of surgeries in $I^g$. Equation \eqref{eqn:sep_CVaR_3} is a result of the separability of $\calS=\bigtimes_{g\in G} \calS^g$, where $\calS^g$ is the support of durations for surgeries in $I^g$. Equation \eqref{eqn:sep_CVaR_4} follows from the separability of the objective function in $\rho$. Equation \eqref{eqn:sep_CVaR_5} is a direct consequence of the \drocvar{} model reformulation.

\section{Discussion and Examples on Symmetry-Breaking Constraints} \label{appdx:SBC}

\subsection{Symmetry-Breaking Constraints on Surgery-to-OR Assignments} \label{appdx:SBC_surgery_assg}

We note that either the constraints \eqref{eqn:SBC_surg_index_2a} or \eqref{eqn:SBC_surg_index_2b} would suffice to enforce the surgery index ordering. However, these two sets of constraints are not equivalent in the sense that there exist feasible solutions to the former one but not the latter one and vice versa in the LP relaxation of the ORASP models (see Example \ref{eg:SBC_eg1}). Therefore, using both of them could tighten the LP relaxation.


\begin{example} \label{eg:SBC_eg1}
Suppose we have surgeries $1$ to $5$ and ORs $1$ to $3$ of same type. We write $Z=(z_{i,r})$ as a $5\times 3$ matrix to denote the set of $z$ variables (i.e., surgery-to-OR assignments). Consider the matrices:
\begin{equation*}\OneAndAHalfSpacedXI
    Z^1 = \begin{pmatrix}
1    & 0    & 0   \\
0.6  & 0.4  & 0   \\
0.4  & 0.2  & 0.4 \\
0.2  & 0.4  & 0.4 \\
0.25 & 0.25 & 0.5 \\
\end{pmatrix}, \quad
Z^2 = \begin{pmatrix}
1    & 0    & 0   \\
0.6  & 0.4  & 0   \\
0.4  & 0.2  & 0.4 \\
0.25 & 0.25 & 0.5 \\
0.2  & 0.4  & 0.4 \\
\end{pmatrix}.
\end{equation*}
Note that $Z^1$ is feasible to constraints \eqref{eqn:SBC_surg_index_1d} but not \eqref{eqn:SBC_surg_index_1c} since we have $z_{5,1}\not\leq z_{4,1}$. On the other hand, $Z^2$ is feasible to constraints \eqref{eqn:SBC_surg_index_1c} but not \eqref{eqn:SBC_surg_index_1d} since $z_{4,3}\not\leq z_{5,3}$. This shows that using both sets of constraints could tighten the LP relaxation. 
\end{example}

\subsection{Variable Fixing Constraints} \label{appdx:SBC_fixing}

One approach to address symmetry in (mixed) integer optimization problems is to fix the value of some integer variables (see, e.g., \citealp{Vo-Thanh_et_al:2018}). This reduces the actual number of integer variables in the model and thus, potentially improves the computational efficiency. In this section, we discuss three classes variable fixing constraints tailored for the ORASP.

In practice, it is common for identical ORs to have the same fixed, overtime and idle time cost, leading to symmetry in surgery-to-OR assignments (see discussions in Section~\ref{subsec:SBC_surgery_OR_assg}). Let us recall, as described in Section~\ref{subsec:SBC_surgery_OR_assg}, that $\{T_\ell\}_{\ell\in L}$ is a partition of the set of types $T$, where (i) a subset $T_\ell$ has a single surgery type if this type has dedicated ORs and (ii) type $t \in T$ belongs to a subset $T_\ell$ with $|T_\ell|>1$ if there is a subset of ORs to which surgeries of this type and those of any other type $t' \in T_\ell$ can be assigned. We then define the corresponding partition $\{(I^\ell,R^\ell)\}_{\ell\in L}$ of $(I,R)$, where $I^\ell=\bigcup_{t\in T_\ell} I_t$ and $R^\ell=\{r\in R\mid \kappa^R_{i,r}=1 \text{ for some }i\in I^\ell\}$.  

First, when $|T_\ell|=1$, we can always assign surgery $i_{j,t}$ to an OR with index $r_{k,\ell}$ where $k\leq j$ by imposing
\begin{equation} \label{eqn:SBC_OR_assg_fix_1}
\textstyle z_{i_{j,t},r_{k,\ell}} = 0, \quad \forall \ell\in L,\, t\in T_\ell,\,j\in[\min\{|I^\ell|,|R^\ell|\}],\,  k\in\{j+1,\dots,|R^\ell|\}:\, |T_\ell|=1,
\end{equation}
where we recall that ORs in $R_\ell$ and surgeries in $I_t$ are numbered sequentially, i.e., $R^\ell=\{r_{1,\ell},\dots,r_{|R^\ell|,\ell}\}$  and $I_t=\{i_{1,t},\dots,i_{|I_t|,t}\}$. Second, when the ORs in $R^\ell$ are identical (i.e., all dedicated to the same set of types $T_\ell$), we can enforce that the first OR in $R^\ell$ is open. Mathematically, let $L_\text{hom}$ be a subset of the index set $L$ such that ORs in $R^\ell$ are identical. We then introduce the constraints
\begin{equation}  \label{eqn:SBC_open_OR_fix_2} 
v_{r_{1,\ell}} = 1,\quad\forall \ell\in L_\text{hom}.
\end{equation} 
Note that constraints~\eqref{eqn:SBC_open_OR_fix_2} is valid irrespective of the number of surgery types that ORs in $R^\ell$ are dedicated to.

Finally, it is common that identical regular anesthesiologists (i.e., anesthesiologists with the same set of specializations and shift) have the same overtime and idle time cost (see discussions in Section~\ref{subsec:SBC_oncall_anes}). This leads to symmetry in surgery-to-anesthesiologist assignments. We could address this symmetry by fixing some surgery-to-anesthesiologist assignments. To derive the desired variable fixing constraints, we first construct a partition $\{T_{\ell'}\}_{\ell'\in L'}$ of the set of types $T$, where (i) a subset $T_{\ell'}$ has a single surgery type if this type has dedicated anesthesiologists and (ii) type $t \in T$ belongs to a subset $T_{\ell'}$ with $|T_{\ell'}|>1$ if there is a subset of anesthesiologists to which surgeries of this type and those of any other type $t' \in T_\ell$ can be assigned.  We then define the partition $\{(I^{\ell'},A^{\ell'})\}_{\ell'\in L'}$ of $(I,A)$, where $I^{\ell'}=\bigcup_{t\in T_{\ell'}} I_t$ and $A^{\ell'}=\{a\in A\mid \kappa^A_{i,a}=1 \text{ for some }i\in I^{\ell'}\}$. We assume that anesthesiologists in $A^{\ell'}$ and surgeries in $I^{\ell'}$ are numbered sequentially, i.e., $A^{\ell'}=\{a_{1,\ell'},\dots,a_{|A^{\ell'}|,\ell'}\}$  and $I^{\ell'}=\{i_{1,\ell'},\dots,i_{|I^{\ell'}|,\ell'}\}$. Here, without loss of generality, we assume that regular anesthesiologists are labeled before on-call anesthesiologists in $A^{\ell'}$. Also, we consider the realistic setting where there is at least one regular anesthesiologist in $A^{\ell'}$. Therefore, anesthesiologist $a_{1,\ell'}$ is always regular. If the anesthesiologists in $A_{\ell'}$ are identical, we can always assign the first surgery in $I^{\ell'}$ to the first regular anesthesiologist in $A^{\ell'}$. (Note that we consider the practical situation where the fixed cost of calling in an on-call anesthesiologist is sufficiently large such that it is always optimal to assign surgeries to all regular anesthesiologists before calling in on-call anesthesiologists.) Mathematically, let $L'_\text{hom}$ be a subset of index set $L'$ such that anesthesiologists in $A^{\ell'}$ are identical. With this notation, we introduce the constraints
\begin{equation} \label{eqn:SBC_anes_assg_fix}
x_{i_{1,\ell'},a_{1,\ell'}}=1, \quad \forall \ell'\in L'_\text{hom}.
\end{equation}
Constraints~\eqref{eqn:SBC_anes_assg_fix} ensure that the first surgery in $I^{\ell'}$ (i.e., $i_{1,\ell'}$) is assigned to the first regular anesthesiologist in $A^{\ell'}$ (i.e., $a_{1,\ell'}$).

\newpage
\color{black}
\section{Linear Decision Rule-based Approximations of the DRO Models} \label{appdx:LDR_DRO}

As suggested by a reviewer of this paper, in this section, we compare our proposed exact approach for reformulating and solving our DRO models with an approximation based on linear decision rules (LDR). In \ref{appdx:LDR_MILP}, we present the LDR approximations of the DRO models and their equivalent MILP reformulations. Then, in \ref{appdx:LDR_Results}, we provide computational results comparing the LDR approximations and the exact approach.

\subsection{The LDR-based Approximations} \label{appdx:LDR_MILP}

In this section, we first formulate the LDR approximations of our proposed \droe{} and \drocvar{} models and then derive equivalent MILP reformulations of these approximations. To formulate the LDR approximation, we introduce the following linear approximations of the second-stage decision variables (as a function of the random surgery duration $d$):
\begin{subequations} \label{eqn:DRO_LDR_variables}
\begin{alignat}{3}
& q_i(d)=q'_i+(q''_i)^\tp d, \qquad && o_a(d)=o'_a+(o''_a)^\tp d, \qquad && o_r(d)=o'_r+(o''_r)^\tp d, \\
& w_i(d)=w'_i+(w''_i)^\tp d, \qquad && g_a(d)=g'_a+(g''_a)^\tp d, \qquad && g_r(d)=g'_r+(g''_r)^\tp d, 
\end{alignat}
\end{subequations}
for all $i\in I$, $r\in R$, and $a\in A$. Here, variables $q'_i$, $o'_a$, $o'_r$, $w'_i$, $g'_a$, and $g'_r$ are of dimension 1, and variables $q''_i$, $o''_a$, $o''_r$, $w''_i$, $g''_a$, and $g''_r$ are of dimension $|I|$. Using \eqref{eqn:DRO_LDR_variables}, we formulate the following LDR approximation of the DRO model:
\begingroup
\small
\begin{subequations} 
\begin{align}
 \underset{\substack{x,\,y,\,z,\,v,\,u,\,s,\,\alpha,\,\beta \\  q,\,o,\,w,\,g}}{\text{minimize}} \quad
&  \sum_{r\in R}f_r v_r + \sum_{a\in A}f_a y_a + \sup_{\Prob\in\calP(m,\calS)} \varrho_{\Prob}\Bigg(\sum_{a\in A} \bigg\{\cg_a \big[g'_a+(g''_a)^\tp d\big] + \co_a \big[o'_a+(o''_a)^\tp d\big] \bigg\} \nonumber \\
& \quad +\sum_{r\in R}  \bigg\{\cg_r \big[g'_r+(g''_r)^\tp d\big] + \co_r \big[o'_r+(o''_r)^\tp d\big] \bigg\}  + \sum_{i\in I} \cw_i \big[w'_i+(w''_i)^\tp d\big]  \Bigg) \label{eqn:DRO_LDR_obj} \\
 \text{subject to\,\,\,} \quad
&  \text{\eqref{eqn:1st_stage_con1-2}--\eqref{eqn:1st_stage_con20-21}}, \label{eqn:DRO_LDR_con1} \\
&  q'_{i'}+(q''_{i'})^\tp d  \geq \big[q'_{i}+(q''_{i})^\tp d\big] + d_i - M_\text{seq}(1-u_{i,i'}),\quad\forall \{i,i'\}\subseteq I,\, i\ne i',\, d\in\calS, \label{eqn:DRO_LDR_con2}\\
&  q'_{i}+(q''_{i})^\tp d \geq s_i,\quad\forall i\in I,\, d\in\calS,  \label{eqn:DRO_LDR_con3} \\
&  o'_{a}+(o''_{a})^\tp d \geq \big[q'_{i}+(q''_{i})^\tp d\big] + d_i -t^\text{end}_a - M_\text{anes}(1-x_{i,a}+y_a),\quad\forall (i,a)\in\calF^A,\, d\in\calS, \label{eqn:DRO_LDR_con4} \\
&  o'_{r}+(o''_{r})^\tp d \geq \big[q'_{i}+(q''_{i})^\tp d\big] + d_i - T^\text{end} - M_\text{room}(1-z_{i,r}),\quad\forall (i,r)\in\calF^R,\, d\in\calS, \label{eqn:DRO_LDR_con5}\\
& w'_{i}+(w''_{i})^\tp d \geq \big[q'_{i}+(q''_{i})^\tp d\big] - s_i, \quad\forall i\in I,\, d\in\calS, \label{eqn:DRO_LDR_con6}\\
& g'_{a}+(g''_{a})^\tp d \geq \bigg(t^\text{end}_a - t^\text{start}_a - \sum_{i\in I_a} d_i x_{i,a}\bigg) \hreg_a + \big[o'_{a}+(o''_{a})^\tp d \big],\quad\forall a\in A,\, d\in\calS, \label{eqn:DRO_LDR_con7}\\
& g'_{r}+(g''_{r})^\tp d\geq T^\text{end} v_r - \sum_{i\in I_r} d_i z_{i,r} + \big[o'_{r}+(o''_{r})^\tp d \big],\quad\forall r\in R,\, d\in\calS, \label{eqn:DRO_LDR_con8}\\ 
&  q'_{i'}+(q''_{i'})^\tp d \geq 0,\,\, o'_{a}+(o''_{a})^\tp d\geq 0,\,\, g'_{a}+(g''_{a})^\tp d\geq 0,\quad\forall i\in I,\, a\in A,\, d\in\calS, \label{eqn:DRO_LDR_con9} \\
&  w'_{i}+(w''_{i})^\tp d \geq 0,\,\, o'_{r}+(o''_{r})^\tp d\geq 0,\,\, g'_{r}+(g''_{r})^\tp d\geq 0,\quad\forall i\in I,\, r\in R,\, d\in\calS. \label{eqn:DRO_LDR_con10} 
\end{align} \label{eqn:DRO_LDR}%
\end{subequations}%
\endgroup
Note that \eqref{eqn:DRO_LDR} is a mini-max problem with an infinite number of constraints (since the set $\calS$ is not finite). Therefore, formulation~\eqref{eqn:DRO_LDR} is not straightforward to solve in the presented form. In Propositions~\ref{prop:DRO_E_LDR_MILP} and \ref{prop:DRO_CVaR_LDR_MILP}, we derive equivalent MILP reformulations of \eqref{eqn:DRO_LDR} with $\varrho_\Prob=\E_\Prob$ (the \droe{} model) and $\varrho_\Prob=\Prob\mhyphen\CVaR_\gamma$ (the \drocvar{} model), respectively.

\begin{proposition} \label{prop:DRO_E_LDR_MILP}
Solving the LDR approximation of the \droe{} model, i.e., problem~\eqref{eqn:DRO_LDR}  with $\varrho_\Prob=\E_\Prob$, is equivalent to solving the following MILP.
\begingroup
\small
\begin{subequations} 
\begin{align}
 \underset{\substack{x,\,y,\,z,\,v,\,u,\,s,\,\alpha,\,\beta \\  \rho,\,q,\,o,\,w,\,g,\,\pi,\,\varphi}}{\textup{minimize}} \quad
&  \sum_{r\in R}f_r v_r + \sum_{a\in A}f_a y_a + \Bigg\{\sum_{a\in A}\big(\cg_a g'_a + \co_a o'_a\big) + \sum_{r\in R}\big(\cg_r g'_r + \co_r o'_r\big) + \sum_{i\in I} \cw_i w'_i\Bigg\} \nonumber \\
& \hspace{80mm} + \Big(m^\tp\rho + \dub^\tp\piub - \dlb^\tp\pilb\Big) \label{eqn:DRO_E_LDR_MILP_obj} \\
 \textup{subject to\,\,\,} \quad
&  \text{\eqref{eqn:1st_stage_con1-2}--\eqref{eqn:1st_stage_con20-21}}, \label{eqn:DRO_E_LDR_MILP_con1} \\
& \piub-\pilb = \sum_{a\in A}\big(\cg_a g''_a + \co_a o''_a\big) + \sum_{r\in R}\big(\cg_r g''_r + \co_r o''_r\big) + \sum_{i\in I} \cw_i w''_i-\rho,\quad \piub\geq 0,\quad\pilb\geq 0, \label{eqn:DRO_E_LDR_MILP_con2}\\
&  q'_{i'} - q'_{i} + M_\textup{seq}(1-u_{i,i'}) \geq \dub^\tp\varphiub^1_{i,i'}-\dlb^\tp\varphilb^1_{i,i'},\quad\forall \{i,i'\}\subseteq I, \label{eqn:DRO_E_LDR_MILP_con3} \\
& \varphiub^1_{i,i'}-\varphilb^1_{i,i'}=q''_{i}-q''_{i'} + e_i,\quad \varphiub^1_{i,i'}\geq 0,\quad \varphilb^1_{i,i'}\geq 0,\quad\forall \{i,i'\}\subseteq I, \label{eqn:DRO_E_LDR_MILP_con4} \\
&  q'_i - s_i \geq \dub^\tp\varphiub^2_{i}-\dlb^\tp\varphilb^2_{i},\quad\forall i\in I, \\
& \varphiub^2_{i}-\varphilb^2_{i}=-q''_i,\quad \varphiub^2_{i}\geq 0,\quad \varphilb^2_{i}\geq 0,\quad\forall i\in I, \\
&  o'_a-q'_i+t^\textup{end}_a + M_\textup{anes}(1-x_{i,a}+y_a) \geq \dub^\tp\varphiub^3_{i,a}-\dlb^\tp\varphilb^3_{i,a},\quad\forall (i,a)\in\calF^A, \\
& \varphiub^3_{i,a}-\varphilb^3_{i,a}= q''_i -o''_a+e_i,\quad \varphiub^3_{i,a}\geq 0,\quad \varphilb^3_{i,a}\geq 0,\quad\forall (i,a)\in\calF^A, \\
&  o'_r-q'_i+T^\textup{end} + M_\textup{room}(1-z_{i,r}) \geq \dub^\tp\varphiub^4_{i,r}-\dlb^\tp\varphilb^4_{i,r},\quad\forall (i,r)\in\calF^R, \\
& \varphiub^4_{i,r}-\varphilb^4_{i,r}= q''_i -o''_r+e_i,\quad \varphiub^4_{i,r}\geq 0,\quad \varphilb^4_{i,r}\geq 0,\quad\forall (i,r)\in\calF^R, \\
&  w'_i-q'_i+s_i \geq \dub^\tp\varphiub^5_{i}-\dlb^\tp\varphilb^5_{i},\quad\forall i\in I, \\
& \varphiub^5_{i}-\varphilb^5_{i}=q''_i-w''_i,\quad \varphiub^5_{i}\geq 0,\quad \varphilb^5_{i}\geq 0,\quad\forall i\in I, \\
&  g'_a-o'_a- \big(t^\textup{end}_a - t^\textup{start}_a \big) \hreg_a  \geq \dub^\tp\varphiub^6_{a}-\dlb^\tp\varphilb^6_{a},\quad\forall a\in A, \\
& \varphiub^6_{a}-\varphilb^6_{a}=o''_a-g''_a-\hreg_a x_a,\quad \varphiub^6_{a}\geq 0,\quad \varphilb^6_{a}\geq 0,\quad\forall a\in A, \\
&  g'_r-o'_r- T^\textup{end}_r v_r \geq \dub^\tp\varphiub^7_{r}-\dlb^\tp\varphilb^7_{r},\quad\forall r\in R, \\
& \varphiub^7_{r}-\varphilb^7_{r}=o''_r - g''_r - z_r,\quad \varphiub^7_{r}\geq 0,\quad \varphilb^7_{r}\geq 0,\quad\forall r\in R, \\
&  q'_i \geq \dub^\tp\varphiub^8_{i}-\dlb^\tp\varphilb^8_{i},\quad \varphiub^8_{i}-\varphilb^8_{i}=-q''_i,\quad \varphiub^8_{i}\geq 0,\quad \varphilb^8_{i}\geq 0,\quad\forall i\in I, \\
&  o'_a  \geq \dub^\tp\varphiub^9_{a}-\dlb^\tp\varphilb^9_{a},\quad \varphiub^9_{a}-\varphilb^9_{a}= -o''_a,\quad \varphiub^9_{a}\geq 0,\quad \varphilb^9_{a}\geq 0,\quad\forall a\in A, \\
&  g'_a  \geq \dub^\tp\varphiub^{10}_{a}-\dlb^\tp\varphilb^{10}_{a},\quad \varphiub^{10}_{a}-\varphilb^{10}_{a}= -g''_a,\quad \varphiub^{10}_{a}\geq 0,\quad \varphilb^{10}_{a}\geq 0,\quad\forall a\in A, \\
&  w'_i \geq \dub^\tp\varphiub^{11}_{i}-\dlb^\tp\varphilb^{11}_{i},\quad \varphiub^{11}_{i}-\varphilb^{11}_{i}=-w''_i,\quad \varphiub^{11}_{i}\geq 0,\quad \varphilb^{11}_{i}\geq 0,\quad\forall i\in I, \\
&  o'_r  \geq \dub^\tp\varphiub^{12}_{r}-\dlb^\tp\varphilb^{12}_{r},\quad \varphiub^{12}_{r}-\varphilb^{12}_{r}= -o''_r,\quad \varphiub^{12}_{r}\geq 0,\quad \varphilb^{12}_{r}\geq 0,\quad\forall r\in R, \\
&  g'_r  \geq \dub^\tp\varphiub^{13}_{r}-\dlb^\tp\varphilb^{13}_{r},\quad \varphiub^{13}_{r}-\varphilb^{13}_{r}= -g''_r,\quad \varphiub^{13}_{r}\geq 0,\quad \varphilb^{13}_{r}\geq 0,\quad\forall r\in R, \label{eqn:DRO_E_LDR_MILP_con22}
\end{align} \label{eqn:DRO_E_LDR_MILP}%
\end{subequations}%
\endgroup
where $x_a=(x_{i,a})_{i\in I}\in\R^I$, $z_r=(z_{i,r})_{i\in I}\in\R^I$ and $e_i\in\R^I$ is the $i$-th standard basis vector.
\end{proposition}

\begin{proof}{Proof.}
Following the proof of Proposition~\ref{prop:mean_support_WC_exp}, we can reformulate the worst-case expectation in the objective~\eqref{eqn:DRO_LDR_obj} as
\begingroup
\small
\begin{align}
\underset{\rho\in\R^I}{\textup{minimize}} \quad
& m^\tp\rho + \max_{d\in\calS} \Bigg\{\sum_{a\in A} \bigg\{\cg_a \big[g'_a+(g''_a)^\tp d\big] + \co_a \big[o'_a+(o''_a)^\tp d\big] \bigg\} \nonumber \\
& \hspace{16mm} +\sum_{r\in R}  \bigg\{\cg_r \big[g'_r+(g''_r)^\tp d\big] + \co_r \big[o'_r+(o''_r)^\tp d\big] \bigg\}  + \sum_{i\in I} \cw_i \big[w'_i+(w''_i)^\tp d\big] -\rho^\tp d\Bigg\},  \label{eqn_pf:prop_DRO_E_LDR_MILP_1} 
\end{align}
\endgroup
which is equivalent to
\begingroup
\small
\begin{align}
\underset{\rho\in\R^I}{\textup{minimize}} \quad
& m^\tp\rho + \Bigg\{\sum_{a\in A}\big(\cg_a g'_a + \co_a o'_a\big) + \sum_{r\in R}\big(\cg_r g'_r + \co_r o'_r\big) + \sum_{i\in I} \cw_i w'_i\Bigg\} \nonumber \\
& \hspace{20mm} + \max_{d\in\calS}\Bigg\{ \bigg[\sum_{a\in A}\big(\cg_a g''_a + \co_a o''_a\big) + \sum_{r\in R}\big(\cg_r g''_r + \co_r o''_r\big) + \sum_{i\in I} \cw_i w''_i-\rho\bigg]^\tp d \Bigg\}.  \label{eqn_pf:prop_DRO_E_LDR_MILP_2} 
\end{align}
\endgroup
Using LP duality, we can reformulate the inner maximization problem over $d\in\calS$ in \eqref{eqn_pf:prop_DRO_E_LDR_MILP_2} into the following minimization problem: $\min_{\pi}\big\{\dub^\tp\piub - \dlb^\tp\pilb \,\,\big|\,\, \eqref{eqn:DRO_E_LDR_MILP_con2}\big\}$. Thus, model~\eqref{eqn:DRO_LDR} with $\varrho_\Prob=\E_\Prob$ is equivalent to
\begingroup
\small
\begin{subequations} \label{eqn_pf:prop_DRO_E_LDR_MILP_2.5}%
\begin{align}
 \underset{\substack{x,\,y,\,z,\,v,\,u,\,s,\,\alpha,\,\beta \\  \rho,\,q,\,o,\,w,\,g,\,\pi}}{\textup{minimize}} \quad
&  \sum_{r\in R}f_r v_r + \sum_{a\in A}f_a y_a + \Bigg\{\sum_{a\in A}\big(\cg_a g'_a + \co_a o'_a\big) + \sum_{r\in R}\big(\cg_r g'_r + \co_r o'_r\big) + \sum_{i\in I} \cw_i w'_i\Bigg\} \nonumber \\
& \hspace{80mm} + \Big(m^\tp\rho + \dub^\tp\piub - \dlb^\tp\pilb\Big) \label{eqn_pf:prop_DRO_E_LDR_MILP_2.5_obj} \\
 \textup{subject to\,\,\,} \quad
&  \text{\eqref{eqn:1st_stage_con1-2}--\eqref{eqn:1st_stage_con20-21}, \eqref{eqn:DRO_E_LDR_MILP_con2}}, \label{eqn_pf:prop_DRO_E_LDR_MILP_2.5_con1} \\
& \text{\eqref{eqn:DRO_LDR_con1}--\eqref{eqn:DRO_LDR_con10}}.
\end{align} 
\end{subequations}%
\endgroup

Next, we reformulate constraints \eqref{eqn:DRO_LDR_con1}--\eqref{eqn:DRO_LDR_con10}. Note that we can rewrite constraints \eqref{eqn:DRO_LDR_con1} as
\begin{equation} \label{eqn_pf:prop_DRO_E_LDR_MILP_3} 
  q'_{i'}-q'_{i} + M_\text{seq}(1-u_{i,i'})\geq \sup_{d\in\calS} \bigg\{\big(q''_{i} - q''_{i'} + e_i\big)^\tp d\bigg\},\quad\forall \{i,i'\}\subseteq I,\, i\ne i'.
\end{equation}
By LP duality, the maximization problem over $d\in\calS$ in \eqref{eqn_pf:prop_DRO_E_LDR_MILP_3}  is equivalent to the following minimization problem: $\min_{\varphi^1_{i,i'}} \big\{\dub^\tp\varphiub^1_{i,i'} - \dlb^\tp\varphilb^1_{i,i'} \,\,\big| \,\, \eqref{eqn:DRO_E_LDR_MILP_con4} \big\}$. This shows that constraint \eqref{eqn_pf:prop_DRO_E_LDR_MILP_3} is equivalent to constraint \eqref{eqn:DRO_E_LDR_MILP_con3}--\eqref{eqn:DRO_E_LDR_MILP_con4}. Repeating the same argument for the remaining constraints, we can show that the LDR approximation of the \droe{} model is equivalent to the MILP in \eqref{eqn:DRO_E_LDR_MILP}. \Halmos
\end{proof}

\begin{proposition} \label{prop:DRO_CVaR_LDR_MILP}
Solving the LDR approximation of the \drocvar{} model, i.e., model~\eqref{eqn:DRO_LDR}  with $\varrho_\Prob=\Prob\mhyphen\CVaR_\gamma$, is equivalent to solving the following MILP.
\begingroup
\small
\begin{subequations}  \label{eqn:DRO_CVaR_LDR_MILP}
\begin{align} 
 \underset{\substack{x,\,y,\,z,\,v,\,u,\,s,\,\alpha,\,\beta \\  \rho_0,\,\rho,\,\psi,\,q,\,o,\,w,\,g,\,\pi,\,\varphi}}{\textup{minimize}} \quad
&  \sum_{r\in R}f_r v_r + \sum_{a\in A}f_a y_a + \Bigg\{\sum_{a\in A}\big(\cg_a g'_a + \co_a o'_a\big) + \sum_{r\in R}\big(\cg_r g'_r + \co_r o'_r\big) + \sum_{i\in I} \cw_i w'_i\Bigg\} \nonumber \\
& \hspace{50mm} + \Bigg\{\Bigg(\frac{1}{1-\gamma}-1\Bigg)\rho_0 + \frac{1}{1-\gamma} m^\tp\rho + \dub^\tp\piub - \dlb^\tp\pilb\Bigg\} \label{eqn:DRO_CVaR_LDR_MILP_obj} \\
 \textup{subject to\,\,\,\,\,\,\,} \quad
&  \text{\eqref{eqn:1st_stage_con1-2}--\eqref{eqn:1st_stage_con20-21}, \eqref{eqn:DRO_E_LDR_MILP_con2}--\eqref{eqn:DRO_E_LDR_MILP_con22}}, \label{eqn:DRO_CVaR_LDR_MILP_con1} \\
&  \rho_0 + \sum_{i\in I} \big(\dlb_i\underline{\psi}_i-\dub_i\overline{\psi}_i\big) \geq 0, \label{eqn:DRO_CVaR_LDR_MILP_con2} \\
& \underline{\psi}_i-\overline{\psi}_i=\rho_i,\quad \underline{\psi}_i\geq 0,\quad \overline{\psi}_i\geq 0,\quad\forall i\in I.  \label{eqn:DRO_CVaR_LDR_MILP_con3}
\end{align}
\end{subequations}%
\endgroup
\end{proposition}

\begin{proof}{Proof.}
Using similar arguments in the proof of Proposition~\ref{prop:DRO_E_LDR_MILP} and  Proposition~\ref{prop:mean_support_WC_CVaR}, one can verify that the LDR approximation of the \drocvar{} model is equivalent to \eqref{eqn:DRO_CVaR_LDR_MILP_con3}. \Halmos
\end{proof}

\begin{remark}\label{rem:LDR_Size}
Formulations \eqref{eqn:DRO_E_LDR_MILP} and \eqref{eqn:DRO_CVaR_LDR_MILP} are large-scale MILPs. Specifically, both \eqref{eqn:DRO_E_LDR_MILP} and \eqref{eqn:DRO_CVaR_LDR_MILP} have $O\big(|I|^2(|I|+|A|+|R|)\big)$ variables and $O\big(|I|^2(|I|+|A|+|R|)\big)$ constraints. Hence, the sizes of these MILPs grow significantly with the number of surgeries, ORs, and anesthesiologists.  In the next section, we show that it is computationally prohibitive to solve large instances of the ORASP using these MILPs.  
\end{remark}


\subsection{Computational Results}\label{appdx:LDR_Results}

In this section, we present computational results comparing the LDR-based approximations of the DRO models with the exact approach proposed in this paper.  For brevity and illustrative purposes, we present results for solving instances 1--6 under costs 2 and 3 using the LDR approximation of the \droe{} model.  We solve the MILP reformulation \eqref{eqn:DRO_E_LDR_MILP} of this approximation using CPLEX  with default settings and a two-hour time limit imposed on each decomposed subproblem. 

In Table~\ref{table:CPU_time_LDR_DRO}, we present the average solution time (across 10 replications; see Section~\ref{subsec:expt_description}) of solving instances 1--6 using the LDR approximation of \droe{} model and the proposed C\&CG for this model. We also present the average optimality gap computed as follows. For each replication $j$ of each ORASP instance, we compute the optimality gap as $\texttt{OptGap}_j=(\hat{\upsilon}_j-\upsilon_j^*)/\upsilon_j^*$, where $\upsilon_j^*$ and $\hat{\upsilon}_j$ are the optimal values obtained from solving the \droe{} model and its LDR approximation, respectively. For instances that the MILP formulation \eqref{eqn:DRO_E_LDR_MILP} cannot solve within the time limit, we set $\hat{\upsilon}$ as the objective value of the best incumbent solution obtained upon termination. Then, we compute average optimality gap as $\sum_{j=1}^{10} \texttt{OptGap}_j/10$. A large positive optimality gap indicates that the quality of the LDR approximation is poor.

Let us first analyze the solution times presented in Table~\ref{table:CPU_time_LDR_DRO}. Similar to the C\&CG method, the solution time of the LDR approximation increases with the size of the instance. This is reasonable because the MILP of this approximation is a large-scale MILP, and the size of which increases with the size of the instance (see Remark~\ref{rem:LDR_Size}). While solving instances 1--5 using the LDR approximation is slightly faster than solving them exactly using the C\&CG method, we could not solve instance 6 (largest ORASP instance) using this approximation within the imposed time limit. In particular, the relative MIP gap upon termination ranges from 4\% to 58\%. In contrast, we can solve all the instances using our C\&CG method within a reasonable time. Next, we analyze the approximation quality. Clearly, the LDR approximation produces poor-quality solutions with large optimality gaps. Specifically, The average optimality gaps range from 9\% to 41\% and from 4\% to 23\% under cost 3 and cost 2, respectively.

Computational results in this section demonstrate that the LDR approximation provides poor-quality solutions to the ORASP and is not computationally tractable for large instances. 
\begin{table}[t]\centering\OneAndAHalfSpacedXI
\footnotesize   
\color{black}
\caption{\blue{Computational time (in s) for solving the \droe{} model and its LDR approximation, and the optimality gap between the optimal values of the \droe{} model and its LDR approximation. A large positive optimality gap indicates that the quality of the LDR approximation is poor. \textit{Note:} Instances that cannot be solved within the time limit are marked with ``--''.}} \label{table:CPU_time_LDR_DRO}
\ra{0.7}  
\begin{tabular}{@{}r|rrrrrr@{}} \toprule
\textbf{Cost 2} & Instance 1 & Instance 2 & Instance 3 & Instance 4 & Instance 5 & Instance 6 \\ \midrule
\droe{}           & 9          & 14         & 16         & 88         & 707        & 5573       \\
LDR             & 1          & 1          & 1          & 7          & 221        & --      \\
Optimality Gap        & 23.25\%    & 4.09\%     & 33.41\%    & 5.82\%     & 5.53\%     & 13.83\%     \\ \midrule
\textbf{Cost 3} & Instance 1 & Instance 2 & Instance 3 & Instance 4 & Instance 5 & Instance 6 \\ \midrule
\droe{}           & 9          & 14         & 16         & 147        & 715        & 6239       \\
LDR             & 1          & 1          & 2          & 6          & 101        & --      \\
Optimality Gap        & 31.46\%    & 9.01\%     & 41.13\%    & 12.71\%    & 16.63\%    & 22.16\%      \\
\bottomrule
\end{tabular}
\end{table}

\color{black}

\color{black}
\section{Extension of \texorpdfstring{\protect\cite{Rath_et_al:2017}}{ }'s RO Approach} \label{appdx:RO_ORASP}
In this section, we derive an extension of \cite{Rath_et_al:2017}'s RO model that incorporates all elements of the ORSAP (\ref{appdx:RO_Model}), present a decomposition method to solve this extension (\ref{appdx:RO_Decomp}), and discuss the computational challenges of this RO approach compared with the proposed DRO approach for the ORASP (\ref{appdx:RO_Results}).

\subsection{The RO Model}\label{appdx:RO_Model}
In this section, we derive an RO model of the ORASP using the same uncertainty set employed in \cite{Rath_et_al:2017}. Specifically, we use the following uncertainty set:
\begin{equation} \label{eqn:RO_uncertainty_set}
  \calD(\tau')=\bigg\{d\in\R^{I}\mid d_i=m_i+b_i\dhat_i,\, i\in I,\, b\in\calB(\tau')\bigg\},  
\end{equation}
where $\calB(\tau')=\big\{b\in\R^{I}\mid \sum_{i\in I}|b_i|\leq \tau',\, -1\leq b_i\leq 1\big\}$. The set $\calD(\tau')$ restricts the duration of surgery $i \in I$ to take values in $[m_i-\dhat_i, m_i+\dhat_i]$, where parameter $\dhat_i$  is the maximum deviation from the mean $m_i$. The parameter $\tau'\geq 0$ bounds the total maximum deviation of surgery duration from the mean across all surgeries. Using uncertainty set $ \calD(\tau')$, we formulate the following RO model of the ORASP:
\begin{subequations} 
\begin{align}
 \underset{x,\,y,\,z,\,v,\,u,\,s,\,\alpha,\,\beta}{\text{minimize}} \quad
&  \Bigg \{\sum_{r\in R}f_r v_r + \sum_{a\in A}f_a y_a + \max_{d\in\calD(\tau')}Q(x,y,z,v,u,s,d) \Bigg\} \label{eqn:1st_stage_RO_extension_obj} \\
 \text{subject to\,\,\,} \quad
&  \text{\eqref{eqn:1st_stage_con1-2}--\eqref{eqn:1st_stage_con20-21},}
\end{align} \label{eqn:1st_stage_RO_extension}%
\end{subequations}%
where $Q(x,y,z,v,u,s,d)$ is the second-stage recourse defined in \eqref{eqn:2nd_stage}. Formulation~\eqref{eqn:1st_stage_RO_extension} finds first-stage decisions $(x,y,z,v,u,s,\alpha,\beta)$ that minimize the first-stage cost plus the worst-case value of the operational cost over a set of surgery duration scenarios characterized by the uncertainty set $\calD(\tau')$. When $\tau'=0$, the uncertainty set $\calD(\tau')$ defined in \eqref{eqn:RO_uncertainty_set} is the singleton $\{m\}$. In this case,  formulation \eqref{eqn:1st_stage_RO_extension} reduces to the deterministic ORASP model with $d=m$. Compared with \cite{Rath_et_al:2017}'s RO model, formulation~\eqref{eqn:1st_stage_RO_extension} additionally includes constraints~\eqref{eqn:1st_stage_con6-7} in the first stage, and in the second stage, it additionally incorporates objectives, variables, and constraints related to waiting and (OR and anesthesiologist) idle time metrics.

Problem~\eqref{eqn:1st_stage_RO_extension} is challenging to solve directly in the presented form due to the inner max-min problem. However, adopting the same assumption made in \cite{Rath_et_al:2017} that parameter $\tau'$ is chosen to be a positive integer, we derive an equivalent solvable reformulation of \eqref{eqn:1st_stage_RO_extension} in the following proposition. 

\begin{proposition} \label{prop:RO_extension_reformulation}
Suppose that parameter $\tau'$ is chosen to be a positive integer. Then, problem~\eqref{eqn:1st_stage_RO_extension} is equivalent to the following problem:
\begin{subequations} 
\begin{align}
 \underset{x,\,y,\,z,\,v,\,u,\,s,\,\alpha,\,\beta,\,\delta}{\textup{minimize}} \quad
&  \sum_{r\in R}f_r v_r + \sum_{a\in A}f_a y_a + \delta \label{eqn:RO_extension_reformulation_obj} \\
 \textup{subject to\,\,\,\,\,} \quad
&  \textup{\eqref{eqn:1st_stage_con1-2}--\eqref{eqn:1st_stage_con20-21},}\quad \delta\geq G(x,y,z,v,u,s), \label{eqn:RO_extension_reformulation_con1}
\end{align} \label{eqn:RO_extension_reformulation}%
\end{subequations}%
where
\begin{subequations} 
\begin{align}
& \hspace{-20mm}  G(x,y,z,v,u,s)=\\
\underset{\lambda,\,\mu,\,\theta,\,\zeta,\,p}{\textup{maximize}} \quad
&  \Bigg\{ \sum_{a\in A} \cg_a(t^\textup{end}_a-t^\textup{start}_a)\hreg_a + \sum_{r\in R}\cg_r T^\textup{end} v_r \nonumber \\
& \quad + \sum_{i\in I}\Bigg[\sum_{i'\in I, i'\ne i}(\lambda_{i,i'} - \lambda_{i',i}) + \sum_{a\in A_i} \mu_{i,a} + \sum_{r\in R_i} \theta_{i,r} \Bigg] s_i \nonumber \\ 
& \quad - M_\textup{seq} \sum_{i\in I}\sum_{i'\in I, i'\ne i} \lambda_{i,i'}(1-u_{i,i'}) - \sum_{i\in I}\sum_{a\in A_i}\mu_{i,a}\Big[ t^\textup{end}_a + M_\textup{anes}(1-x_{i,a}+ y_a) \Big] \nonumber \\
& \quad - \sum_{i\in I}\sum_{r\in R_i} \theta_{i,r} \Big[T^\textup{end} +M_\textup{room}(1-z_{i,r})\Big] \nonumber \\
& \quad + \sum_{i\in I} m_i\Bigg[ \sum_{i'\in I,i'\ne i} \lambda_{i,i'} + \sum_{a\in A_i} (\mu_{i,a}- \cg_a \hreg_a x_{i,a}) + \sum_{r\in R_i}(\theta_{i,r}-\cg_r z_{i,r}) \Bigg]  \nonumber \\
& \quad - \sum_{i\in I} \dhat_i\Bigg[ \sum_{i'\in I,i'\ne i} \zeta^{L,1}_{i,i'} + \sum_{a\in A_i} (\zeta^{M,1}_{i,a}- \cg_a \hreg_a x_{i,a} p^1_i) + \sum_{r\in R_i}(\zeta^{T,1}_{i,r}-\cg_r z_{i,r} p^1_i)\Bigg]   \nonumber \\
& \quad + \sum_{i\in I} \dhat_i\Bigg[ \sum_{i'\in I,i'\ne i} \zeta^{L,2}_{i,i'} + \sum_{a\in A_i} (\zeta^{M,2}_{i,a}- \cg_a \hreg_a x_{i,a} p^2_i) + \sum_{r\in R_i}(\zeta^{T,2}_{i,r}-\cg_r z_{i,r} p^2_i)\Bigg] \Bigg\} \label{eqn:RO_inner_max_MILP_obj} \\
 \textup{subject to} \quad
& \text{\eqref{eqn:2nd_stage_dual_reform_con1}--\eqref{eqn:2nd_stage_dual_reform_con4}}, \label{eqn:RO_inner_max_MILP_con0}  \\
& \sum_{i\in I} (p^1_i+p^2_i) \leq \tau', \label{eqn:RO_inner_max_MILP_con1}  \\
& p^1_i + p^2_i \leq 1,\quad\forall i\in I, \label{eqn:RO_inner_max_MILP_con2} \\
&  \sum_{i\in I_a} \mu_{i,a} \leq \cg_a + \co_a,\quad \sum_{i\in I_r} \theta_{i,r} \leq \cg_r + \co_r, \quad\forall a\in A,\, r\in R,\label{eqn:RO_inner_max_MILP_con3} \\
&  \sum_{i'\in I, i'\ne i}(\lambda_{i,i'} - \lambda_{i',i}) + \sum_{a\in A_i} \mu_{i,a} + \sum_{r\in R_i} \theta_{i,r} + \cw_i \geq 0,\quad\forall i\in I, \label{eqn:RO_inner_max_MILP_con4} \\
&\zeta^{L,k}_{i,i'} \leq \lambda_{i,i'},\,\,\, \zeta^{L,k}_{i,i'} \leq p^k_i\lambdaub_{i,i'},\,\,\, \zeta^{L,k}_{i,i'} \geq 0,\,\,\, \zeta^{L,k}_{i,i'} \geq \lambda_{i,i'} + \lambdaub_{i,i'}(p^k_i-1), \nonumber \\
& \hspace{85mm}\forall \{i,i'\}\subseteq I,\, k\in\{1,2\},  \label{eqn:RO_inner_max_MILP_con5}\\
& \zeta^{M,k}_{i,a} \leq \mu_{i,a},\,\,\, \zeta^{M,k}_{i,a} \leq p^k_i\muub_{i,a},\,\,\, \zeta^{M,k}_{i,a} \geq 0,\,\,\,\zeta^{M,k}_{i,a} \geq \mu_{i,a} + \muub_{i,a}(p^k_i-1), \nonumber \\
& \hspace{85mm}\forall (i,a)\in \calF^A,\, k\in\{1,2\},  \label{eqn:RO_inner_max_MILP_con6} \\
& \zeta^{T,k}_{i,r} \leq \theta_{i,r},\,\,\, \zeta^{T,k}_{i,r}\leq p^k_i\thetaub_{i,r},\,\,\, \zeta^{T,k}_{i,r} \geq 0,\,\,\, \zeta^{T,k}_{i,r} \geq \theta_{i,r} + \thetaub_{i,r}(p^k_i-1), \nonumber \\
& \hspace{85mm} \forall(i,r)\in\calF^R, \, k\in\{1,2\}, \label{eqn:RO_inner_max_MILP_con7} \\
&  \lambda_{i,i'},\,\mu_{i,a},\,\theta_{i,r}\geq 0,\quad p^k_i\in\{0,1\}, \quad\forall \{i,i'\}\subseteq I,\,a\in A_i,\, r\in R_i,\, k\in\{1,2\}. \label{eqn:RO_inner_max_MILP_con8} 
\end{align}\label{eqn:RO_inner_max_MILP}%
\end{subequations} \vspace{-8mm} 
\end{proposition}

\begin{proof}{Proof.}
Using the dual problem \eqref{eqn:2nd_stage_dual_reform} of $Q(x,y,z,v,u,s,d)$, we can reformulate $\max_{d\in\calD(\tau')} Q(x,y,z,v,u,s,d)$ as
\begin{subequations} 
\begin{align}
 \underset{\lambda,\,\mu,\,\theta,\,d,\,b}{\textup{maximize}} \quad
&  \sum_{a\in A} \cg_a(t^\textup{end}_a-t^\textup{start}_a)\hreg_a + \sum_{r\in R}\cg_r T^\textup{end} v_r  \nonumber \\
& \quad + \sum_{i\in I}\Bigg[\sum_{i'\in I, i'\ne i}(\lambda_{i,i'} - \lambda_{i',i}) + \sum_{a\in A_i} \mu_{i,a} + \sum_{r\in R_i} \theta_{i,r} \Bigg] s_i \nonumber \\ 
& \quad - M_\textup{seq} \sum_{i\in I}\sum_{i'\in I, i'\ne i} \lambda_{i,i'}(1-u_{i,i'}) - \sum_{i\in I}\sum_{a\in A_i}\mu_{i,a}\Big[ t^\textup{end}_a + M_\textup{anes}(1-x_{i,a}+ y_a) \Big] \nonumber \\
& \quad - \sum_{i\in I}\sum_{r\in R_i} \theta_{i,r} \Big[T^\textup{end} +M_\textup{room}(1-z_{i,r})\Big] \nonumber \\
& \quad + \sum_{i\in I} \Bigg[ \sum_{i'\in I,i'\ne i} \lambda_{i,i'} + \sum_{a\in A_i} (\mu_{i,a}- \cg_a \hreg_a x_{i,a}) + \sum_{r\in R_i}(\theta_{i,r}-\cg_r z_{i,r})\Bigg] d_i \label{eqn_pf:RO_inner_max_MILP_1_obj} \\
 \textup{subject to} \quad
&  \text{\eqref{eqn:2nd_stage_dual_reform_con1}--\eqref{eqn:2nd_stage_dual_reform_con4}}, \label{eqn_pf:RO_inner_max_MILP_1_con1} \\
&  d_i=m_i+b_i\dhat_i,\quad\forall i\in I, \label{eqn_pf:RO_inner_max_MILP_1_con2} \\
& \sum_{i\in I} |b_i| \leq \tau',\quad -1\leq b_i\leq 1,\quad\forall i\in I. \label{eqn_pf:RO_inner_max_MILP_1_con3}
\end{align} \label{eqn_pf:RO_inner_max_MILP_1}%
\end{subequations}%
We first consider the maximization over $d$ and $b$ in \eqref{eqn_pf:RO_inner_max_MILP_1}, i.e.,
\begin{subequations} 
\begin{align}
 \underset{d,\,b}{\textup{maximize}} \quad & \sum_{i\in I} \Bigg[ \sum_{i'\in I,i'\ne i} \lambda_{i,i'} + \sum_{a\in A_i} (\mu_{i,a}- \cg_a \hreg_a x_{i,a}) + \sum_{r\in R_i}(\theta_{i,r}-\cg_r z_{i,r})\Bigg] d_i \label{eqn_pf:RO_inner_max_MILP_2_obj} \\
 \textup{subject to} \quad
&  d_i=m_i+b_i\dhat_i,\quad\forall i\in I, \label{eqn_pf:RO_inner_max_MILP_2_con1} \\
& \sum_{i\in I} |b_i| \leq \tau',\quad -1\leq b_i\leq 1,\quad\forall i\in I. \label{eqn_pf:RO_inner_max_MILP_2_con2}
\end{align} \label{eqn_pf:RO_inner_max_MILP_2}%
\end{subequations}%
Using \eqref{eqn_pf:RO_inner_max_MILP_2_con1}, we can reformulate \eqref{eqn_pf:RO_inner_max_MILP_2} as 
\begin{subequations} 
\begin{align}
 \underset{b}{\textup{maximize}} \quad & \sum_{i\in I} \Bigg[ \sum_{i'\in I,i'\ne i} \lambda_{i,i'} + \sum_{a\in A_i} (\mu_{i,a}- \cg_a \hreg_a x_{i,a}) + \sum_{r\in R_i}(\theta_{i,r}-\cg_r z_{i,r})\Bigg] \big(m_i+b_i\dhat_i\big) \label{eqn_pf:RO_inner_max_MILP_3_obj} \\
 \textup{subject to} \quad
& \sum_{i\in I} |b_i| \leq \tau',\quad -1\leq b_i\leq 1,\quad\forall i\in I. \label{eqn_pf:RO_inner_max_MILP_3_con1}
\end{align} \label{eqn_pf:RO_inner_max_MILP_3}%
\end{subequations}%
Since the objective is linear in $b$, under the assumption that $\tau'$ is a positive integer, there always exists an optimal solution $b$ such that $b_i\in\{-1,0,1\}$ for all $i\in I$. Thus, we introduce new binary variables $p^1_i$ taking value $1$ if $b_i=-1$ and $0$ otherwise, and $p^2_i$ taking value $1$ if $b_i=1$ and $0$ otherwise, for all $i\in I$. With these new variables, we can reformulate \eqref{eqn_pf:RO_inner_max_MILP_3} as
\begin{subequations} 
\begin{align}
 \underset{p}{\textup{maximize}} \quad & \sum_{i\in I} \Bigg[ \sum_{i'\in I,i'\ne i} \lambda_{i,i'} + \sum_{a\in A_i} (\mu_{i,a}- \cg_a \hreg_a x_{i,a}) + \sum_{r\in R_i}(\theta_{i,r}-\cg_r z_{i,r})\Bigg] \big(m_i-p^1_i\dhat_i+p^2_i\dhat_i\big) \label{eqn_pf:RO_inner_max_MILP_4_obj} \\
 \textup{subject to} \quad
& \sum_{i\in I} \big(p^1_i + p^2_i\big) \leq \tau', \label{eqn_pf:RO_inner_max_MILP_4_con1} \\
& p^1_i+p^2_i \leq 1,\quad\forall i\in I, \label{eqn_pf:RO_inner_max_MILP_4_con2} \\
& p^1_i\in\{0,1\},\, p^2_i\in\{0,1\},\quad\forall i\in I. \label{eqn_pf:RO_inner_max_MILP_4_con3}
\end{align} \label{eqn_pf:RO_inner_max_MILP_4}%
\end{subequations}%
Constraints \eqref{eqn_pf:RO_inner_max_MILP_4_con2} ensure that $p^1_i$ and $p^2_i$ do not take value $1$ simultaneously. With the reformulation \eqref{eqn_pf:RO_inner_max_MILP_4}, problem \eqref{eqn_pf:RO_inner_max_MILP_1} is equivalent to
\begin{subequations} 
\begin{align}
 \underset{\lambda,\,\mu,\,\theta,\,d,\,p}{\textup{maximize}} \quad
&  \sum_{a\in A} \cg_a(t^\textup{end}_a-t^\textup{start}_a)\hreg_a + \sum_{r\in R}\cg_r T^\textup{end} v_r  \nonumber \\
& \quad + \sum_{i\in I}\Bigg[\sum_{i'\in I, i'\ne i}(\lambda_{i,i'} - \lambda_{i',i}) + \sum_{a\in A_i} \mu_{i,a} + \sum_{r\in R_i} \theta_{i,r} \Bigg] s_i \nonumber \\ 
& \quad - M_\textup{seq} \sum_{i\in I}\sum_{i'\in I, i'\ne i} \lambda_{i,i'}(1-u_{i,i'}) - \sum_{i\in I}\sum_{a\in A_i}\mu_{i,a}\Big[ t^\textup{end}_a + M_\textup{anes}(1-x_{i,a}+ y_a) \Big] \nonumber \\
& \quad - \sum_{i\in I}\sum_{r\in R_i} \theta_{i,r} \Big[T^\textup{end} +M_\textup{room}(1-z_{i,r})\Big] \nonumber \\
& \quad + \sum_{i\in I} \Bigg[ \sum_{i'\in I,i'\ne i} \lambda_{i,i'} + \sum_{a\in A_i} (\mu_{i,a}- \cg_a \hreg_a x_{i,a}) + \sum_{r\in R_i}(\theta_{i,r}-\cg_r z_{i,r})\Bigg] \big(m_i-p^1_i\dhat_i+p^2_i\dhat_i\big)  \label{eqn_pf:RO_inner_max_MILP_5_obj} \\
 \textup{subject to} \quad
&  \text{\eqref{eqn:2nd_stage_dual_reform_con1}--\eqref{eqn:2nd_stage_dual_reform_con4}}, \label{eqn_pf:RO_inner_max_MILP_5_con1} \\
&  \sum_{i\in I} \big(p^1_i + p^2_i\big) \leq \tau', \label{eqn_pf:RO_inner_max_MILP_5_con2} \\
& p^1_i+p^2_i \leq 1,\quad\forall i\in I, \label{eqn_pf:RO_inner_max_MILP_5_con3} \\
& p^1_i\in\{0,1\},\quad p^2_i\in\{0,1\},\quad\forall i\in I. \label{eqn_pf:RO_inner_max_MILP_5_con4} 
\end{align} \label{eqn_pf:RO_inner_max_MILP_5}%
\end{subequations}%
Finally, to reformulate \eqref{eqn_pf:RO_inner_max_MILP_5} into an MILP, we introduce auxiliary variables $\zeta^{L,k}_{i,i'}=\lambda_{i,i'}p^k_i$, $\zeta^{M,k}_{i,a}=\mu_{i,a}p^k_i$, and $\zeta^{T,k}_{i,r}=\theta_{i,r}p^k_i$ with McCormick inequalities \eqref{eqn:RO_inner_max_MILP_con5}--\eqref{eqn:RO_inner_max_MILP_con7}.  \Halmos
\end{proof}

Key differences between the MILP reformulation of $\max_{d\in\calD(\tau')} Q(x,y,z,v,u,s,d)$ in our RO model for the ORASP and that of \cite{Rath_et_al:2017}  include the following.  First, recall that the formulation of $Q(x,y,z,v,u,s,d)$ in our RO model additionally includes objectives, variables, and constraints related to OR and anthesiologist idle time, and surgery waiting time. Hence, the dual problem \eqref{eqn:2nd_stage_dual_reform} of $Q(x,y,z,v,u,s,d)$, and accordingly problem $ G(x,y,z,v,u,s)$ in \eqref{eqn:RO_inner_max_MILP}  has different sets of variables, constraints, and objectives. Second, since \cite{Rath_et_al:2017}'s model does not incorporate OR and anesthesiologist idle time, the worst-case scenario $d_i$ in their RO model takes only two values, the mean $m_i$ or the upper bound $m_i+\dhat_i$. To see this, suppose that $\cg_a=\cg_r=0$ for all $a\in A$ and $r\in R$. In this case, the objective function \eqref{eqn_pf:RO_inner_max_MILP_3_obj} of problem \eqref{eqn_pf:RO_inner_max_MILP_3} is non-decreasing in $b_i$ (since $\lambda_{i,i'}\geq 0$, $\mu_{i,a}\geq 0$, and $\theta_{i,r}\geq 0$), and hence there exists a worst-case scenario $d$ such that $d_i\in\{m_i,m_i+\dhat_i\}$ for all $i\in I$. In contrast, since our RO model incorporates the idle time metric, the objective function \eqref{eqn_pf:RO_inner_max_MILP_3_obj} could be non-increasing or non-decreasing in $b$, depending on the first-stage assignment decision $(x,z)$. Hence, the worst-case scenario of surgery duration in our RO model may attain the lower bound ($m_i-\dhat_i$), upper bound ($m_i+\dhat_i$), or mean ($m_i$). Therefore, since the recourse function in our RO model for the ORASP and that of \cite{Rath_et_al:2017}  are different, the reformulation of $\max_{d\in\calD(\tau')} Q(x,y,z,v,u,s,d)$  are also different.

\subsection{A Decomposition Algorithm for the RO Model}\label{appdx:RO_Decomp}

\IncMargin{0em}
\begin{algorithm}[t] \SetAlgoNoEnd  \OneAndAHalfSpacedXI \small
\color{black}
\SetKwInOut{Initialization}{Initialization}
\Initialization{\vspace{-2mm}  Set $LB=0$, $UB=\infty$, $\varepsilon\geq0$, $\calK=\emptyset$,   $j=1$.} 
\textbf{1. Master problem.} Solve the master problem
\begin{subequations}
\begin{align}
 \underset{\substack{ x,y,z,v,u,s,\\ \alpha,\beta,\,\delta}}{\text{minimize}} \quad\,\,
&  \sum_{r\in R}f_r v_r + \sum_{a\in A}f_a y_a + \delta \label{eqn:RO_decomposition_master_obj} \\
\text{subject to} \quad
&  \text{\eqref{eqn:1st_stage_con1-2}--\eqref{eqn:1st_stage_con20-21},}  \label{eqn:RO_decomposition_master_con1} \\
& \delta \geq 0, \label{eqn:RO_decomposition_master_con2}\\
& \delta \geq \sum_{a\in A} \cg_a(t^\textup{end}_a-t^\textup{start}_a)\hreg_a + \sum_{r\in R}\cg_r T^\textup{end} v_r  \nonumber \\
& \qquad + \sum_{i\in I}\Bigg[\sum_{i'\in I, i'\ne i}(\lambda^k_{i,i'} - \lambda^k_{i',i}) + \sum_{a\in A_i} \mu^k_{i,a} + \sum_{r\in R_i} \theta^k_{i,r} \Bigg] s_i \nonumber \\ 
& \qquad - M_\textup{seq} \sum_{i\in I}\sum_{i'\in I, i'\ne i} \lambda^k_{i,i'}(1-u_{i,i'}) - \sum_{i\in I}\sum_{a\in A_i}\mu^k_{i,a}\Big[ t^\textup{end}_a + M_\textup{anes}(1-x_{i,a}+ y_a) \Big] \nonumber \\
& \qquad - \sum_{i\in I}\sum_{r\in R_i} \theta^k_{i,r} \Big[T^\textup{end} +M_\textup{room}(1-z_{i,r})\Big] \nonumber \\
& \qquad + \sum_{i\in I} \Bigg[ \sum_{i'\in I,i'\ne i} \lambda^k_{i,i'} + \sum_{a\in A_i} (\mu^k_{i,a}- \cg_a \hreg_a x_{i,a}) + \sum_{r\in R_i}(\theta^k_{i,r}-\cg_r z_{i,r})\Bigg] d^k_i,\quad \forall k\in\calK.  \label{eqn:RO_decomposition_master_con3}
\end{align}  \label{eqn:RO_decomposition_master}%
\end{subequations}
\hspace{2.7mm}Record the optimal solution $(x^j,y^j,z^j,v^j,u^j,s^j,\alpha^j,\beta^j,\delta^j)$ and value $Z^j$. Set $LB = Z^j$. \\
\textbf{2. Subproblem.} Solve \eqref{eqn:RO_inner_max_MILP} with fixed $(x^j,y^j,z^j,v^j,u^j,s^j)$.
\begin{enumerate}[leftmargin=12mm,topsep=0mm,itemsep=0mm]
    \item [\textbf{2.1}] Record the optimal solution $(\lambda^*,\mu^*,\theta^*,p^{*,1},p^{*,2})$ and value $Y^j$.  Set $UB = \min\big\{UB, (Z^j - \delta^j)+Y^j\big\}$.
    \item [\textbf{2.2}] If $(UB-LB)/UB<\varepsilon$ or $\delta^j\geq Y^j$, then terminate
    ; else, go to step 3.
\end{enumerate}
\textbf{3. Optimality Cut.}
\begin{enumerate}[leftmargin=12mm,topsep=0mm,itemsep=0mm]
    \item [\textbf{3.1}] Using $(p^{*,1},p^{*,2})$ from step 2, compute $d^*_i=m_i+(p^{*,2}_i-p^{*,1}_i)\dhat_i$ for all $i\in I$.
    \item [\textbf{3.2}] Add the optimality cut of the form \eqref{eqn:RO_decomposition_master_con3} with $(\lambda^k,\mu^k,\theta^k,d^k)=(\lambda^*,\mu^*,\theta^*,d^*)$. Update $j \leftarrow j+1$ and $\calK \leftarrow \calK\cup\{j\}$. Go back to step 1.
\end{enumerate}
\BlankLine
\caption{A decomposition algorithm to solve \eqref{eqn:1st_stage_RO_extension}.} \label{algo:decomposition_RO_extension}
\end{algorithm}\DecMargin{1em}

Note that the RO formulation in \eqref{eqn:RO_extension_reformulation} involves a maximization problem in constraint~\eqref{eqn:RO_extension_reformulation_con1} that defines the function value $G(x,y,z,v,u,s)$. Thus, it is not directly solvable using standard techniques. In this section, we propose a decomposition algorithm to solve the RO model \eqref{eqn:RO_extension_reformulation} based on the one presented by \cite{Rath_et_al:2017} for solving their RO model. 

Algorithm~\ref{algo:decomposition_RO_extension} summarizes the steps of our decomposition algorithm. In each iteration, we first solve the master problem \eqref{eqn:RO_decomposition_master}. In Proposition~\ref{prop:RO_optimality_cut}, we show that the right hand side of constraint \eqref{eqn:RO_decomposition_master_con3} in the master problem provides a valid lower bound on $\max_{d\in\calD(\tau')} Q(x,y,z,v,u,s,d)$. Thus, the master problem \eqref{eqn:RO_decomposition_master} is a relaxation of RO model \eqref{eqn:1st_stage_RO_extension}, providing a lower bound on the optimal value of \eqref{eqn:1st_stage_RO_extension}. Given the optimal solution from the master problem, we identify a new dual feasible solution $(\lambda^*,\mu^*,\theta^*)\in\Pi$ and a new scenario $d^*\in\calD(\tau')$ by solving the subproblem \eqref{eqn:RO_inner_max_MILP}. Since solutions of the master problem are feasible to the original problem, we obtain an upper bound on the optimal value of \eqref{eqn:1st_stage_RO_extension}. We then add a new cut associated with  $(\lambda^*,\mu^*,\theta^*)$ and $d^*$ to the master problem. Next, we solve the master problem again with the new information (in the enlarged set $\calK$) from the subproblems. This process continues until the gap between the lower and upper bound obtained in each iteration satisfies a predetermined termination tolerance $\varepsilon\geq 0$.

\begin{proposition} \label{prop:RO_optimality_cut}
For any $d^k\in\calD(\tau')$ and $(\lambda^k,\mu^k,\theta^k)\in\Pi:=\big\{(\lambda,\mu,\theta)\,\,\big|\,\,\text{\eqref{eqn:2nd_stage_dual_reform_con1}--\eqref{eqn:2nd_stage_dual_reform_con4}}\big\}$, the right hand side of \eqref{eqn:RO_decomposition_master_con3} provides a valid lower bound on $\max_{d\in\calD(\tau')} Q(x,y,z,v,u,s,d)$.
\end{proposition}

\begin{proof}{Proof.}
First, for any $d^k\in\calD(\tau')$, we have
\begin{equation} \label{eqn_pf:prop_RO_optimality_cut_1}
\max_{d\in\calD(\tau')} Q(x,y,z,v,u,s,d) \geq  Q(x,y,z,v,u,s,d^k).
\end{equation}
Second, from \eqref{eqn:2nd_stage_dual_reform}, we have  $Q(x,y,z,v,u,s,d^k)$ equals to 
\begin{align}
& \max_{(\lambda,\mu,\theta)\in\Pi} \Bigg\{\sum_{a\in A} \cg_a(t^\textup{end}_a-t^\textup{start}_a)\hreg_a + \sum_{r\in R}\cg_r T^\textup{end} v_r + \sum_{i\in I}\Bigg[\sum_{i'\in I, i'\ne i}(\lambda_{i,i'} - \lambda_{i',i}) + \sum_{a\in A_i} \mu_{i,a} + \sum_{r\in R_i} \theta_{i,r} \Bigg] s_i \nonumber \\ 
& \quad\qquad\qquad - M_\textup{seq} \sum_{i\in I}\sum_{i'\in I, i'\ne i} \lambda_{i,i'}(1-u_{i,i'}) - \sum_{i\in I}\sum_{a\in A_i}\mu_{i,a}\Big[ t^\textup{end}_a + M_\textup{anes}(1-x_{i,a}+ y_a) \Big] \nonumber \\
& \quad\qquad\qquad - \sum_{i\in I}\sum_{r\in R_i} \theta_{i,r} \Big[T^\textup{end} +M_\textup{room}(1-z_{i,r})\Big] \nonumber \\
& \quad\qquad\qquad + \sum_{i\in I} \Bigg[ \sum_{i'\in I,i'\ne i} \lambda_{i,i'} + \sum_{a\in A_i} (\mu_{i,a}- \cg_a \hreg_a x_{i,a}) + \sum_{r\in R_i}(\theta_{i,r}-\cg_r z_{i,r})\Bigg] d^k_i \Bigg\}. \label{eqn_pf:prop_RO_optimality_cut_2}
\end{align}
It follows that from \eqref{eqn_pf:prop_RO_optimality_cut_1} and \eqref{eqn_pf:prop_RO_optimality_cut_2} that the right hand side of \eqref{eqn:RO_decomposition_master_con3}  is a valid lower bound on $\max_{d\in\calD(\tau')} Q(x,y,z,v,u,s,d)$. \Halmos
\end{proof}

There are two key differences between Algorithm~\ref{algo:decomposition_RO_extension} and the decomposition algorithm presented by \cite{Rath_et_al:2017}. First, our second-stage problem $Q(x,y,z,v,u,s,d)$, which includes waiting time and idle time components, is different from that of \cite{Rath_et_al:2017}. Hence, the reformulation of the worst-case recourse function is different (see Proposition~\ref{prop:RO_extension_reformulation}). Second, the optimality cuts for the extended RO model \eqref{eqn:1st_stage_RO_extension} are different from the optimality cuts for \cite{Rath_et_al:2017}'s RO model (see Proposition~\ref{prop:RO_optimality_cut}).

\subsection{Computational Time}\label{appdx:RO_Results}

In this section, we investigate the computational performance of the proposed Algorithm~\ref{algo:decomposition_RO_extension} for the RO model of the ORASP. For brevity and illustrative purposes, we present results for instances~1 and 2 under costs~1--3. We follow the same parameter settings described in Section~\ref{subsec:expt_description}. In addition, similar to \cite{Rath_et_al:2017},  we choose $\tau'=\lceil \tau|I| \rceil$ with $\tau\in\{0.1,0.2,0.4\}$. Finally, we impose a time limit of two hours. 

Table~\ref{table:CPU_time_RO} shows solution times of the instances that Algorithm~\ref{algo:decomposition_RO_extension} solved to optimality and the relative MIP gap for those that Algorithm~\ref{algo:decomposition_RO_extension} failed to solve within the imposed time limit. We also present solution times of instances 1 and 2 using our proposed \spe{} and \droe{} models. It is clear that the RO approach takes a substantially longer time to solve such small instances than our proposed models. Specifically, using the RO model and Algorithm~\ref{algo:decomposition_RO_extension}, we could only solve instance 1 with $\tau=0.1$, and solution times of this instance range form 2214 to 7176 seconds.  For instance 2, Algorithm~\ref{algo:decomposition_RO_extension} terminates with a large relative MIP gap after two hours for all $\tau\in\{0.1,0.2,0.4\}$. In contrast, using our \spe{} and \droe{} models, we can solve instances 1 and 2 in a few seconds. These results suggest that our proposed SP and DRO approaches for the ORASP are more computationally efficient than the RO approach.

\begin{table}[t]\centering\OneAndAHalfSpacedXI
\color{black}
\footnotesize   
\caption{\blue{Computational time in seconds when solving \spe{}, \droe{}, and RO models. For cases when an optimal solution was not found in the 2-hour time limit, the percentage in brackets indicates the relative optimality gap reported from Algorithm~\ref{algo:decomposition_RO_extension} when the time limit was reached.} } \label{table:CPU_time_RO}
\ra{0.7}  
\begin{tabular}{@{}lcrrrrrcrrrrr@{}} \toprule
       && \multicolumn{5}{c}{\textbf{Instance 1}}             && \multicolumn{5}{c}{\textbf{Instance 2}}             \\
       && \spe{} & \droe{} & $\tau = 0.1$ & $\tau = 0.2$ & $\tau = 0.4$ && \spe{} & \droe{} & $\tau = 0.1$ & $\tau = 0.2$ & $\tau = 0.4$ \\ \midrule
Cost 1 && 2 & 7 & 2214         & (21.9\%)     & (46.8\%)     && 9 &  28 & (36.3\%)     & (46.2\%)      & (60.7\%)      \\
Cost 2 && 1 & 9 & 5133         & 6048         & (56.6\%)     && 2 & 14 & (49.8\%)     & (61.3\%)      & (67.5\%)      \\
Cost 3 && 1 & 9 & 7176         & (3.0\%)      & (43.6\%)     && 3 & 14 & (61.2\%)     & (70.1\%)      & (72.6\%) \\
\bottomrule
\end{tabular}
\color{black}
\end{table}

\color{black}

\section{Additional details and results on numerical experiments} \label{appdx:expt}

\subsection{Sample average approximation of the \spe{} and \spcvar{} models} \label{appdx:SP_SAA}
We provide the sample average approximation (SAA) approach for the \spe{} and \spcvar{} models. Given a set of finite scenarios $\{d^n\}_{n=1}^N$, the SAA approach is to solve the \spe{} model with the empirical distribution based on $\{d^n\}_{n=1}^N$. That is,
\begin{subequations}
\begin{align}
 \underset{x,...,s,\,q,\,o,\,w,\,g}{\text{minimize}} \quad
&  \sum_{r\in R}c_r v_r + c_q\sum_{a\in A}y_a + \frac{1}{N}\sum_{n=1}^N \bigg[\Big(\cg_a g_a^n + \co_a o_a^n \Big) +\sum_{r\in R} \Big(\cg_r g_r^n + \co_r o_r^n \Big) + \sum_{i\in I} \cw_i w_i^n \bigg] \label{eqn:SP_SAA_obj} \\
 \text{subject to } \quad
&  \eqref{eqn:1st_stage_con1-2}-\eqref{eqn:1st_stage_con20-21}, \label{eqn:SP_SAA_con1} \\
&  q^n_{i'} \geq q^n_i + d^n_i - M_\text{seq}(1-u_{i,i'}),\quad\forall \{i,i'\}\subseteq I,\, i\ne i', \,n\in[N], \label{eqn:SP_SAA_con2}\\
&  q^n_i \geq s_i,\quad\forall i\in I,\,n\in[N],\label{eqn:SP_SAA_con3} \\
&  o^n_a \geq q^n_i + d^n_i -t^\text{end}_a - M_\text{anes}(1-x_{i,a}+y_a),\quad\forall (i,a)\in\calF^A,\,n\in[N], \label{eqn:SP_SAA_con4} \\
&  o^n_r \geq q^n_i + d^n_i - T^\text{end} - M_\text{room}(1-z_{i,r}),\quad\forall (i,r)\in\calF^R,\,n\in[N], \label{eqn:SP_SAA_con5}\\
& w^n_i \geq q^n_i - s_i, \quad\forall i\in I,\,n\in[N], \label{eqn:SP_SAA_con6}\\
& g^n_a \geq \Bigg(t^\text{end}_a - t^\text{start}_a - \sum_{i\in I_a} d^n_i x_{i,a}\Bigg) \hreg_a + o^n_a,\quad\forall a\in A,\,n\in[N],\label{eqn:SP_SAA_con7}\\
& g^n_r \geq T^\text{end} v_r - \sum_{i\in I_r} d^n_i z_{i,r} + o^n_r,\quad\forall r\in R,\,n\in[N],\label{eqn:SP_SAA_con8}\\ 
&  q^n_i,\, o^n_a,\, o^n_r,\, w^n_i,\, g^n_a,\, g^n_r \geq 0,\quad\forall i\in I,\, a\in A,\, r\in R,\,n\in[N]. \label{eqn:SP_SAA_con9} 
\end{align}  \label{eqn:SP_SAA}%
\end{subequations}
\vspace{-10mm}

Next, for the \spcvar{} model, using the definition of CVaR in \eqref{eqn:CVaR_def}, we can reformulate the \spcvar{} model as follows:
\begin{subequations}
\begin{align}
 \underset{x,\,y,\,z,\,v,\,u,\,\alpha,\,\beta,\,s,\,\tau}{\text{minimize}} & \quad  \sum_{r\in R}f_r v_r + \sum_{a\in A}f_a y_a + \Bigg\{\tau + \frac{1}{1-\gamma}\E_\Prob\Big[\max\big\{ Q(x,y,z,v,u,s,d)-\tau,0\big\}\Big]  \Bigg\}\label{eqn:CVaR_obj} \\
 \text{subject to\,\,\,\,\,\,} &   \quad  \eqref{eqn:1st_stage_con1-2}-\eqref{eqn:1st_stage_con20-21}. \label{eqn:CVaR_con1}
\end{align}  \label{eqn:CVaR_model}%
\end{subequations}
Given a set of finite scenarios $\{d^n\}_{n=1}^N$, we can introduce auxiliary variable $\eta_n$ to linearize the $\max\{\cdot,\,0\}$ operator in the objective. Therefore, we have the following reformulation of the \spcvar{} model.
\begin{subequations}
\begin{align}
 \underset{x,\,y,\,z,\,v,\,u,\,\alpha,\,\beta,\,s,\,\tau,\,\eta}{\text{minimize}} & \quad  \sum_{r\in R}f_r v_r + \sum_{a\in A}f_a y_a + \Bigg\{\tau + \frac{1}{N(1-\gamma)}\sum_{n=1}^N \eta_n \Bigg\}\label{eqn:CVaR_scen_obj} \\
 \text{subject to\,\,\,\,\,\,} &   \quad  \eqref{eqn:1st_stage_con1-2}-\eqref{eqn:1st_stage_con20-21}, \label{eqn:CVaR_scen_con1}\\
 & \ \eta_n \geq 0,\quad \eta_n \geq Q(x,y,z,v,u,s,d^n)-\tau,\quad\forall n\in[N]. \label{eqn:CVaR_scen_con2}
\end{align}  \label{eqn:CVaR_model_scen}%
\end{subequations}
Substituting $Q(x,y,z,v,u,s,d^n)$ by the second-stage problem \eqref{eqn:2nd_stage}, we obtain the final formulation of the \spcvar{} model.
\begin{subequations}
\begin{align}
 \underset{x,...,s,\,\tau,\,\eta,\,q,\,o,\,w,\,g}{\text{minimize}} \quad
&  \sum_{r\in R}c_r v_r + c_q\sum_{a\in A}y_a +  \Bigg\{\tau + \frac{1}{N(1-\gamma)}\sum_{n=1}^N \eta_n \Bigg\} \label{eqn:CVaR_final_obj} \\
 \text{subject to \hspace{3mm}} \quad
&  \eqref{eqn:1st_stage_con1-2}-\eqref{eqn:1st_stage_con20-21}, \label{eqn:CVaR_final_con1} \\
&  \eta_n\geq \sum_{a\in A} \Big(\cg_a g^n_a + \co_a o^n_a \Big) +\sum_{r\in R} \Big(\cg_r g^n_r + \co_r o^n_r \Big) + \sum_{i\in I} \cw_i w^n_i - \tau,\quad\forall n\in[N],  \label{eqn:CVaR_final_con2}\\
&  \eta_n \geq 0,\quad \eqref{eqn:2nd_stage_con1}-\eqref{eqn:2nd_stage_con8},\quad\forall n\in[N]. \label{eqn:CVaR_final_con3} 
\end{align}  \label{eqn:CVaR_final}%
\end{subequations}

\subsection{Details and Additional Results for Section \ref{sec:numerical_experiments}} \label{appdx:supp_num_results}

\subsubsection{Instance Details} \label{appdx:instance_detail}  \phantom{a}

We first provide the summary statistics of the surgery duration for different specialties. In Table \ref{table:data_summary_1}, we present the mean, standard deviation, minimum, and maximum of the random surgery duration of each type for instances 1 to 6. The statistics are computed based on the data set from \cite{Mannino_et_al:2010}. The minimum and maximum surgery duration are chosen as the 20\% and 80\% percentiles of the surgery data respectively. Next, Table \ref{table:instance_detail1} presents the details for instances 1 to 6 with six surgery types (CARD: Cardiology, ORTH: Orthopedics, GYN: Gynecology, MED: Medicine, GASTRO: Gastroenterology, URO: Urology).  


\begin{table}[t]\centering\small
\footnotesize
\caption{Summary statistics of surgery duration for instances 1 to 6 (Std. Dev.: standard deviation)} \label{table:data_summary_1}
\ra{0.6}  
\begin{tabular}{@{}l|cccccc@{}} \toprule
Type      & CARD & ORTH & GYN & MED & GASTRO & URO \\ \midrule
Mean      & 99   & 142  & 78  & 75  & 132    & 72  \\
Std. Dev. & 53   & 58   & 52  & 42  & 76     & 38  \\
Minimum   & 54   & 87   & 31  & 37  & 66     & 44  \\
Maximum   & 143  & 188  & 121 & 111 & 194    & 94  \\
\bottomrule
\end{tabular}
\end{table}
\begin{table}[t]\centering\small
\footnotesize 
{\OneAndAHalfSpacedXI \caption{Instance details based on data from \protect\cite{Mannino_et_al:2010}. \textit{Note:} MED, GASTRO, and URO surgeries can be covered by the same pool of anesthesiologists. For example, in instance 1, there are two regular and one on-call anesthesiologists that can perform MED and GASTRO surgeries.} \label{table:instance_detail1} } 
\ra{0.55}  
\begin{tabular}{@{}lcccccc@{}} \toprule
\textbf{Instance   1}                   &      &      &          &            &  &  \\
Surgery type                            & CARD & ORTH & MED      & GASTRO     &  &  \\
Number of surgeries                     & 3    & 4    & 5        & 3          &  &  \\
Number of ORs                           & 1    & 2    & 2        & 2          &  &  \\
Number of anesthesiologists   (regular) & 1    & 1    & \multicolumn{2}{c}{2} &  &  \\
Number of anesthesiologists   (on call) & 0    & 0    & \multicolumn{2}{c}{1} &  &  \\ \midrule
\textbf{Instance   2}                   &      &      &          &            &  &  \\
Surgery type                            & CARD & ORTH & MED      & GASTRO     &  &  \\
Number of surgeries                     & 5    & 6    & 7        & 2          &  &  \\
Number of ORs                           & 2    & 2    & 2        & 2          &  &  \\
Number of anesthesiologists   (regular) & 2    & 2    & \multicolumn{2}{c}{2} &  &  \\
Number of anesthesiologists   (on call) & 1    & 1    & \multicolumn{2}{c}{1} &  &  \\ \midrule
\textbf{Instance   3}                   & \multicolumn{1}{l}{} & \multicolumn{1}{l}{} & \multicolumn{1}{l}{} & \multicolumn{1}{l}{} & \multicolumn{1}{l}{} & \multicolumn{1}{l}{} \\
Surgery type                            & CARD                 & ORTH                 & GYN                  & MED                  & GASTRO               & URO                  \\
Number of surgeries                     & 7                    & 4                    & 2                    & 4                    & 3                    & 5                    \\
Number of ORs                           & 2                    & 1                    & 1                    & 1                    & 1                    & 2                    \\
Number of anesthesiologists   (regular) & 2                    & 2                    & 1                    & \multicolumn{3}{c}{2}                                              \\
Number of anesthesiologists   (on call) & 1                    & 1                    & 0                    & \multicolumn{3}{c}{1}  \\ \midrule
\textbf{Instance   4}                   & \multicolumn{1}{l}{} & \multicolumn{1}{l}{} & \multicolumn{1}{l}{} & \multicolumn{1}{l}{} & \multicolumn{1}{l}{} & \multicolumn{1}{l}{} \\
Surgery type                            & CARD                 & ORTH                 & GYN                  & MED                  & GASTRO               & URO                  \\
Number of surgeries                     & 6                    & 7                    & 11                   & 3                    & 7                    & 6                    \\
Number of ORs                           & 2                    & 2                    & 2                    & 1                    & 2                    & 2                    \\
Number of anesthesiologists   (regular) & 2                    & 2                    & 2                    & \multicolumn{3}{c}{6}                                              \\
Number of anesthesiologists   (on call) & 1                    & 1                    & 1                    & \multicolumn{3}{c}{1}     \\ \midrule
\textbf{Instance   5}                   & \multicolumn{1}{l}{} & \multicolumn{1}{l}{} & \multicolumn{1}{l}{} & \multicolumn{1}{l}{} & \multicolumn{1}{l}{} & \multicolumn{1}{l}{} \\
Surgery type                            & CARD                 & ORTH                 & GYN                  & MED                  & GASTRO               & URO                  \\
Number of surgeries                     & 8                    & 10                   & 15                   & 4                    & 10                   & 8                    \\
Number of ORs                           & 3                    & 4                    & 5                    & 1                    & 4                    & 3                    \\
Number of anesthesiologists   (regular) & 4                    & 4                    & 6                    & \multicolumn{3}{c}{10}                                             \\
Number of anesthesiologists   (on call) & 0                    & 0                    & 0                    & \multicolumn{3}{c}{0} \\ \midrule
\textbf{Instance   6}                   & \multicolumn{1}{l}{} & \multicolumn{1}{l}{} & \multicolumn{1}{l}{} & \multicolumn{1}{l}{} & \multicolumn{1}{l}{} & \multicolumn{1}{l}{} \\
Surgery type                            & CARD                 & ORTH                 & GYN                  & MED                  & GASTRO               & URO                  \\
Number of surgeries                     & 12                   & 14                   & 22                   & 4                    & 14                   & 14                   \\
Number of ORs                           & 5                    & 6                    & 8                    & 1                    & 6                    & 6                    \\
Number of anesthesiologists   (regular) & 6                    & 7                    & 11                   & \multicolumn{3}{c}{16}                                             \\
Number of anesthesiologists   (on call) & 0                    & 0                    & 0                    & \multicolumn{3}{c}{0}     \\
\bottomrule
\end{tabular}
\end{table}

\subsubsection{Additional Results for Computational Time} \label{appdx:expt_comp_time} \phantom{a}

In this section, we provide additional results for computational time of instances 1--6 under costs 2 and 3 (see Section \ref{subsec:expt_description} for detailed experiment settings). Table \ref{table:CPU_time_cost2_3} presents the average solution times in seconds under costs 2 and 3. For the \spe{}, \droe{}, and \drocvar{} models, we can obtain near optimal solutions within 3 minutes for small to medium-sized instances and within 3 hours for most large instances. Also, we can solve the \spcvar{} model within 1 hour for small to medium-sized instances. While the observations are similar to those under cost 1, we note that the computational times under costs 2 and 3 are longer. We attribute this observation to the inclusion of the OR idle time and anesthesiologist idle time (for cost 3) in the objective function. That is, the objective is a weighted sum of multiple conflicting performance metrics: waiting time, overtime, and idle time.


%
\begin{table}[t]\centering\OneAndAHalfSpacedXI
\footnotesize   
\caption{Computational time (in s) with costs 2 and 3} \label{table:CPU_time_cost2_3}
\ra{0.7}  
\begin{tabular}{@{}l|llllllllllll@{}} \toprule
\textbf{Cost 2} & Instance 1 & Instance 2 & Instance 3 & Instance 4 & Instance 5 & Instance 6 \\ \midrule
\spe{}      & 1.17 & 2.72  & 3.81  & 16.67   & 153.63 & 2076.62 \\ [0.5ex]
\spcvar{}   & 4.30 & 43.00 & 30.32 & 1744.67 & NA     & NA      \\ [0.5ex]
\droe{}     & 8.58 & 13.97 & 16.14 & 88.15   & 706.52 & 5573.29 \\ [0.5ex]
\drocvar{}  & 5.08 & 19.64 & 9.55  & 72.02   & 889.70 & 8194.48  \\ \midrule
\textbf{Cost 3} & Instance 1 & Instance 2 & Instance 3 & Instance 4 & Instance 5 & Instance 6 \\ \midrule
\spe{}       & 1.23 & 2.94  & 3.62  & 14.46   & 211.80 & 2327.63 \\ [0.5ex]
\spcvar{}    & 4.31 & 46.27 & 30.35 & 2655.75 & NA     & NA      \\ [0.5ex]
\droe{}      & 8.50 & 13.78 & 16.33 & 147.38  & 714.91 & 6239.39 \\ [0.5ex]
\drocvar{}   & 4.25 & 45.02 & 9.43  & 95.29   & 638.71 & 9275.96 \\
\bottomrule
\end{tabular}
\end{table}

\newpage
\subsubsection{Comparisons of Valid Inequalities for the \droe{} Model} \label{appdx:expt_DRO_VI} \phantom{a}

In this section, investigate the computational performance of the \droe{} model when solving it with the variable-dependent valid inequalities (VIs) \eqref{eqn:VI_dual_bounds} and the variable-free VIs \eqref{eqn:VI_dual_bounds_const}. Specifically, we follow the same experiment settings detailed in Section \ref{subsec:expt_comp_time} to solve the \droe{} model with either VIs \eqref{eqn:VI_dual_bounds} or the VIs \eqref{eqn:VI_dual_bounds_const}. We keep  VIs \eqref{eqn:global_LB} in both cases.  Table \ref{table:CPU_time_cost_DRO_VI} summarizes the average solution times under costs 1--3. We observe that solution times under the variable-free version are generally similar to or shorter than solution times under the variable-dependent version. In particular, solution times under the variable-free version are significantly shorter for the largest instance, instance 6. This may be explained by the increased model complexity when including VIs~\eqref{eqn:VI_dual_bounds}, which involve the first-stage variables.
\begin{table}[t]\centering\OneAndAHalfSpacedXI
\footnotesize   
\caption{\droe{} computational time (in s) using variable-free VIs \eqref{eqn:VI_dual_bounds_const} or variable-dependent VIs \eqref{eqn:VI_dual_bounds}} \label{table:CPU_time_cost_DRO_VI}
\ra{0.7}  
\begin{tabular}{@{}l|llllllllllll@{}} \toprule
\textbf{Cost 1}      & Instance 1 & Instance 2 & Instance 3 & Instance 4 & Instance 5 & Instance 6 \\ \midrule
Variable-Free & 7.20       & 28.21      & 14.14      & 102.95     & 152.11     & 2845.10    \\ [0.5ex]
Variable-Dependent   & 6.53       & 60.66      & 9.23       & 45.62      & 186.91     & 4455.80    \\ \midrule
\textbf{Cost 2}      & Instance 1 & Instance 2 & Instance 3 & Instance 4 & Instance 5 & Instance 6 \\ \midrule
Variable-Free & 8.58       & 13.97      & 16.14      & 88.15      & 706.52     & 5573.29    \\ [0.5ex]
Variable-Dependent   & 6.56       & 11.33      & 9.61       & 68.32      & 646.44     & 9325.73    \\ \midrule
\textbf{Cost 3}      & Instance 1 & Instance 2 & Instance 3 & Instance 4 & Instance 5 & Instance 6 \\ \midrule
Variable-Free & 8.50       & 13.78      & 16.33      & 147.38     & 714.91     & 6239.39    \\ [0.5ex]
Variable-Dependent   & 6.95       & 13.54      & 10.55      & 191.87     & 681.04     & 9804.87    \\
\bottomrule
\end{tabular}
\end{table}

\color{black}
\subsubsection{Analysis of Optimal Schedules across SAA Replications} \label{appdx:expt_analysis_SAA_rep} \phantom{a}

In this section, we analyze the variability of the optimal \spe{} and \spcvar{} solutions across different SAA replications. Specifically, we solve 10 SAAs of instance 3 using the \spe{} and \spcvar{} models. First, we analyze the variability in the scheduling decision. Figure~\ref{fig:schedule_instance3_OR1} shows the time allocated to each surgery for OR 1 in instance 3, which is the same OR used for illustrative purposes in Figure~\ref{fig:schedule_OR1}. Specifically, the solid line in the \spe{} and \spcvar{} schedules represents the average time allocated to each surgery and the error bar represents the approximate 90\% confidence interval (CI) over $10$ SAA replications (based on the normal approximation). It is clear that the \spe{} and \spcvar{} schedules do not exhibit any structural pattern as in Figure~\ref{fig:schedule_OR1}.  Nonetheless, when considering different samples, it is apparent that while optimal allocations for each sample slightly differ, they consistently align with a shared underlying (unstructured) pattern. 

\begin{figure}[t!]
    \centering\OneAndAHalfSpacedXI
    \includegraphics[scale=0.8]{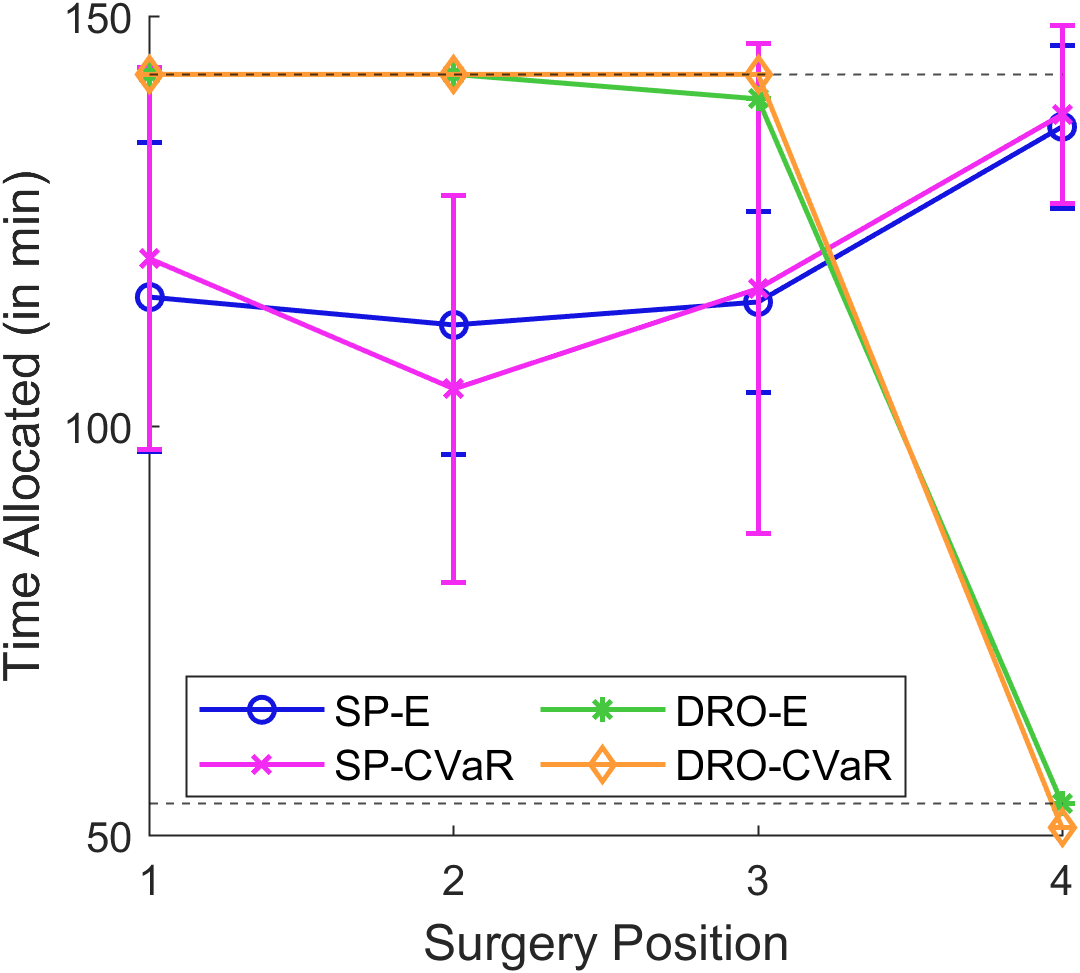}  
    \caption{Illustration of the optimal schedules for OR 1 in instance 3. The two dotted lines indicate the minimum
and maximum surgery durations.}
    \label{fig:schedule_instance3_OR1}
\end{figure}

Next, we examine the variability in the allocation and assignment decisions. In Table~\ref{table:solution_var}, we present the number of surgeries in each OR and the number of ORs open in the optimal SP-E and SP-CVaR schedules, as well as their frequencies of occurrence over $10$ SAA replications for instances 1--4. It is clear from Table~\ref{table:solution_var} that the number of surgeries assigned to each OR and the number of ORs open are relatively stable across different SAA replications. As a result, the sequencing decisions are also similar across replications. 
\begin{table}[t]\centering  \OneAndAHalfSpacedXI
\color{black}
\footnotesize  
\caption{\blue{Number of surgeries in each OR and number of ORs open in the optimal SP-E and SP-CVaR schedules, as well as their frequencies (of occurrence) over $10$ SAA replications.} } \label{table:solution_var}
\ra{0.8}  
\begin{tabular}{@{}ll|lll@{}} \toprule
                    &         & \#surgeries in each OR         & \#ORs open & Frequency \\ \midrule
\textbf{Instance 1} & SP-E    & $(3, 4, 0, 5, 0, 3, 0, 4)^\tp$ & 4          & 10/10        \\ [0.8ex]
                    & SP-CVaR & $(3, 4, 0, 5, 0, 3, 0, 4)^\tp$ & 4          & 10/10        \\  \midrule
\textbf{Instance 2} & SP-E    & $(5,0,3,3,7,0,2,0)^\tp$        & 5          & 10/10        \\  [0.8ex]
                    & SP-CVaR & $(3,2,3,3,4,3,2,0)^\tp$        & 7          & 8/10         \\
                    &         & $(3,2,3,3,5,2,2,0)^\tp$        & 7          & 1/10         \\
                    &         & $(5,0,3,3,4,3,2,0)^\tp$        & 6          & 1/10         \\  \midrule
\textbf{Instance 3} & SP-E    & $(4,3,4,2,4,3,5,0)^\tp$        & 7          & 10/10        \\  [0.8ex]
                    & SP-CVaR & $(4,3,4,2,4,3,5,0)^\tp$        & 7          & 10/10        \\  \midrule
\textbf{Instance 4} & SP-E    & $(3,3,4,3,6,5,3,4,3,6,0)^\tp$  & 10         & 10/10        \\  [0.8ex]
                    & SP-CVaR & $(3,3,4,3,6,5,3,4,3,6,0)^\tp$  & 10         & 10/10        \\
\bottomrule
\end{tabular}
\color{black}
\end{table}


\color{black}

\subsubsection{Additional Results for Solution Quality} \label{appdx:expt_sim_results} \phantom{a}

In this section, we present additional out-of-sample simulation results for instances 1--6 under costs 1--3. First, Tables \ref{table:OS_supp_cost1_waiting}--\ref{table:OS_supp_cost1_anes_OT} present respectively the average waiting time, OR overtime, and anesthesiologists overtime under cost 1.  (In the tables, we abbreviate \drocvar{} as D-CVaR.) As discussed in Section \ref{subsec:expt_sol_quality}, the \droe{} and \drocvar{} models yield significantly shorter waiting times but slightly longer OR and anesthesiologist overtime compared with the \spe{} and \spcvar{} models.

Next, Tables \ref{table:OS_supp_cost2_waiting}--\ref{table:OS_supp_cost2_OR_idle} present respectively the average waiting time, OR overtime, anesthesiologists overtime, and OR idle time under cost 2. (Recall that in this cost structure, we additionally include the cost of the OR idle time in the objective). Similar to our observations under cost 1, the \droe{} and \drocvar{} models yield significantly shorter waiting times than the other models. Moreover, since the \droe{} and \drocvar{} models open more ORs and allocate more time to each surgery (to protect against potential long surgery durations and excessive delays), both the \droe{} and \drocvar{} models yield slightly longer OR idle time compared with the \spe{} and \spcvar{} models.

Finally, Tables \ref{table:OS_supp_cost3_waiting}--\ref{table:OS_supp_cost3_anes_idle} present respectively the average waiting time, OR overtime, anesthesiologists overtime, OR idle time, and anesthesiologist idle time under cost 3. (Recall that in this cost structure, we additionally include the anesthesiologist's idle time in the objective). Again, we have observations similar to those made for cost 1 regarding waiting time, overtime, and OR idle time. Also, we observe that \droe{} and \drocvar{} models yield slightly longer anesthesiologist idle time when compared with the \spe{} and \spcvar{} models, mainly because the \droe{} and \drocvar{} models allocate more time to each surgery.

\renewcommand{\drocvar}{D-CVaR}

\begin{table}[p]\centering\OneAndAHalfSpacedXI
\footnotesize   
\caption{Average waiting time under cost 1} \label{table:OS_supp_cost1_waiting}
\ra{0.7}  
\hspace*{-10mm}
\begin{tabular}{@{}l|rrrr|rrrr|rrrr@{}} \toprule
                      & \multicolumn{4}{c|}{\textbf{Instance 1}} & \multicolumn{4}{c|}{\textbf{Instance 2}} &  \multicolumn{4}{c}{\textbf{Instance 3}} \\ 
\textbf{Waiting Time} & \spe{}      & \spcvar{}    & \droe{}    & \drocvar{}   & \spe{}      & \spcvar{}    & \droe{}    & \drocvar{}   & \spe{}      & \spcvar{}    & \droe{}    & \drocvar{}   \\\midrule \rowcolor{lightgray}
Setting I             & 140     & 61      & 99     & 34         & 259     & 91      & 0      & 0          & 161     & 87      & 100    & 35         \\
Setting IIa           & 120     & 40      & 100    & 36         & 250     & 69      & 0      & 0          & 142     & 59      & 101    & 36         \\ \rowcolor{lightgray}
Setting IIb           & 235     & 125     & 164    & 83         & 469     & 176     & 62     & 69         & 284     & 195     & 193    & 124        \\
Setting IIc           & 291     & 178     & 211    & 125        & 567     & 235     & 123    & 136        & 372     & 296     & 275    & 223        \\ \rowcolor{lightgray}
Setting IIIa          & 123     & 43      & 99     & 35         & 257     & 74      & 0      & 0          & 143     & 62      & 98     & 35         \\
Setting IIIb          & 307     & 184     & 206    & 118        & 620     & 247     & 110    & 121        & 366     & 281     & 253    & 183        \\
\bottomrule
\end{tabular}
\end{table}
\begin{table}[p]\centering\OneAndAHalfSpacedXI
\footnotesize   
\caption{Average OR overtime under cost 1} \label{table:OS_supp_cost1_OR_OT}
\ra{0.7}  
\hspace*{-10mm}
\begin{tabular}{@{}l|rrrr|rrrr|rrrr@{}} \toprule
                     & \multicolumn{4}{c}{\textbf{Instance 1}} & \multicolumn{4}{c}{\textbf{Instance 2}} & \multicolumn{4}{c}{\textbf{Instance 3}} \\
\textbf{OR Overtime} & \spe{}      & \spcvar{}    & \droe{}     & \drocvar{}    & \spe{}      & \spcvar{}     & \droe{}    & \drocvar{}    & \spe{}       & \spcvar{}    & \droe{}    & \drocvar{}    \\ \midrule  \rowcolor{lightgray}
Setting I            & 95      & 162     & 192     & 241       & 116     & 41       & 69     & 69        & 95       & 155     & 199    & 249       \\
Setting IIa          & 91      & 154     & 205     & 252       & 115     & 28       & 71     & 71        & 90       & 147     & 214    & 264       \\ \rowcolor{lightgray}
Setting IIb          & 150     & 207     & 258     & 298       & 212     & 91       & 113    & 110       & 163      & 216     & 275    & 315       \\
Setting IIc          & 179     & 236     & 283     & 322       & 256     & 128      & 141    & 136       & 207      & 256     & 311    & 354       \\ \rowcolor{lightgray}
Setting IIIa         & 90      & 155     & 203     & 251       & 118     & 29       & 70     & 70        & 89       & 148     & 212    & 262       \\
Setting IIIb         & 189     & 243     & 291     & 328       & 280     & 134      & 140    & 136       & 210      & 260     & 315    & 350       \\
\bottomrule
\end{tabular}
\end{table}
\begin{table}[p]\centering\OneAndAHalfSpacedXI
\footnotesize   
\caption{Average anesthesiologist overtime under cost 1} \label{table:OS_supp_cost1_anes_OT}
\ra{0.7}  
\hspace*{-10mm}
\begin{tabular}{@{}l|rrrr|rrrr|rrrr@{}} \toprule
                        & \multicolumn{4}{c}{\textbf{Instance 1}} & \multicolumn{4}{c}{\textbf{Instance 2}} & \multicolumn{4}{c}{\textbf{Instance 3}} \\
\textbf{Anes Overtime} & \spe{}      & \spcvar{}    & \droe{}     & \drocvar{}    & \spe{}      & \spcvar{}     & \droe{}    & \drocvar{}    & \spe{}       & \spcvar{}    & \droe{}    & \drocvar{}    \\ \midrule \rowcolor{lightgray}
Setting I              & 95      & 162     & 192     & 241       & 116     & 41       & 69     & 69        & 95       & 155     & 199    & 249       \\
Setting IIa            & 91      & 154     & 205     & 252       & 115     & 28       & 71     & 71        & 90       & 147     & 214    & 264       \\ \rowcolor{lightgray}
Setting IIb            & 150     & 207     & 258     & 298       & 212     & 91       & 112    & 110       & 163      & 216     & 275    & 315       \\
Setting IIc            & 179     & 236     & 283     & 322       & 256     & 128      & 140    & 135       & 207      & 256     & 311    & 354       \\ \rowcolor{lightgray}
Setting IIIa           & 90      & 155     & 203     & 251       & 118     & 29       & 70     & 70        & 89       & 148     & 212    & 262       \\
Setting IIIb           & 189     & 243     & 291     & 328       & 281     & 134      & 140    & 136       & 210      & 260     & 315    & 350       \\
\bottomrule
\end{tabular}
\end{table}

\begin{table}[p]\centering\OneAndAHalfSpacedXI
\footnotesize   
\caption{Average waiting time under cost 2} \label{table:OS_supp_cost2_waiting}
\ra{0.7}  
\hspace*{-10mm}
\begin{tabular}{@{}l|rrrr|rrrr|rrrr@{}} \toprule
                      & \multicolumn{4}{c}{\textbf{Instance 1}} & \multicolumn{4}{c}{\textbf{Instance 2}} & \multicolumn{4}{c}{\textbf{Instance 3}} \\
\textbf{Waiting Time} & \spe{}      & \spcvar{}    & \droe{}     & \drocvar{}    & \spe{}      & \spcvar{}    & \droe{}     & \drocvar{}    & \spe{}      & \spcvar{}    & \droe{}     & \drocvar{}    \\  \midrule \rowcolor{lightgray}
Setting I             & 176     & 102     & 99      & 96        & 309     & 275     & 274     & 175       & 205     & 133     & 100     & 136       \\
Setting IIa           & 162     & 77      & 100     & 82        & 314     & 271     & 304     & 141       & 181     & 101     & 101     & 121       \\ \rowcolor{lightgray}
Setting IIb           & 284     & 189     & 164     & 168       & 543     & 487     & 421     & 281       & 367     & 253     & 191     & 237       \\
Setting IIc           & 340     & 252     & 211     & 220       & 640     & 595     & 505     & 369       & 487     & 355     & 275     & 313       \\ \rowcolor{lightgray}
Setting IIIa          & 165     & 81      & 99      & 85        & 320     & 277     & 302     & 146       & 182     & 104     & 98      & 123       \\
Setting IIIb          & 359     & 263     & 206     & 228       & 700     & 633     & 514     & 382       & 476     & 341     & 251     & 309       \\
\bottomrule
\end{tabular}
\end{table}
\begin{table}[p]\centering\OneAndAHalfSpacedXI
\footnotesize   
\caption{Average OR overtime under cost 2} \label{table:OS_supp_cost2_OR_OT}
\ra{0.7}  
\hspace*{-10mm}
\begin{tabular}{@{}l|rrrr|rrrr|rrrr@{}} \toprule
                     & \multicolumn{4}{c}{\textbf{Instance 1}} & \multicolumn{4}{c}{\textbf{Instance 2}} & \multicolumn{4}{c}{\textbf{Instance 3}} \\
\textbf{OR Overtime} & \spe{}      & \spcvar{}    & \droe{}     & \drocvar{}    & \spe{}      & \spcvar{}    & \droe{}     & \drocvar{}    & \spe{}      & \spcvar{}    & \droe{}     & \drocvar{}    \\ \midrule  \rowcolor{lightgray}
Setting I            & 85      & 127     & 192     & 202       & 99      & 135     & 298     & 277       & 89      & 125     & 199     & 170       \\
Setting IIa          & 84      & 119     & 205     & 200       & 101     & 134     & 327     & 274       & 85      & 117     & 214     & 161       \\ \rowcolor{lightgray}
Setting IIb          & 140     & 175     & 258     & 265       & 198     & 235     & 412     & 351       & 163     & 191     & 271     & 227       \\
Setting IIc          & 169     & 206     & 283     & 299       & 242     & 283     & 455     & 389       & 214     & 236     & 303     & 267       \\ \rowcolor{lightgray}
Setting IIIa         & 83      & 120     & 203     & 202       & 103     & 137     & 325     & 276       & 83      & 117     & 212     & 162       \\
Setting IIIb         & 179     & 213     & 291     & 310       & 267     & 304     & 473     & 406       & 215     & 239     & 307     & 270       \\
\bottomrule
\end{tabular}
\end{table}
\begin{table}[p]\centering\OneAndAHalfSpacedXI
\footnotesize   
\caption{Average anesthesiologist overtime under cost 2} \label{table:OS_add_cost2_anes_OT}
\ra{0.7}  
\hspace*{-10mm}
\begin{tabular}{@{}l|rrrr|rrrr|rrrr@{}} \toprule
                       & \multicolumn{4}{c}{\textbf{Instance 1}} & \multicolumn{4}{c}{\textbf{Instance 2}} & \multicolumn{4}{c}{\textbf{Instance 3}} \\
\textbf{Anes Overtime} & \spe{}      & \spcvar{}    & \droe{}     & \drocvar{}    & \spe{}      & \spcvar{}    & \droe{}     & \drocvar{}    & \spe{}      & \spcvar{}    & \droe{}     & \drocvar{}    \\ \midrule  \rowcolor{lightgray}
Setting I              & 85      & 127     & 192     & 202       & 99      & 148     & 298     & 277       & 89      & 125     & 199     & 170       \\
Setting IIa            & 84      & 119     & 205     & 200       & 101     & 145     & 327     & 274       & 85      & 117     & 214     & 161       \\ \rowcolor{lightgray}
Setting IIb            & 140     & 175     & 258     & 265       & 198     & 272     & 412     & 351       & 163     & 191     & 271     & 227       \\
Setting IIc            & 169     & 206     & 283     & 299       & 242     & 335     & 456     & 390       & 214     & 236     & 303     & 267       \\ \rowcolor{lightgray}
Setting IIIa           & 83      & 120     & 203     & 202       & 103     & 149     & 325     & 276       & 83      & 117     & 212     & 162       \\
Setting IIIb           & 179     & 213     & 291     & 310       & 268     & 360     & 473     & 406       & 215     & 239     & 307     & 270       \\
\bottomrule
\end{tabular}
\end{table}
\begin{table}[p]\centering\OneAndAHalfSpacedXI
\footnotesize   
\caption{Average OR idle time under cost 2} \label{table:OS_supp_cost2_OR_idle}
\ra{0.7}  
\hspace*{-10mm}
\begin{tabular}{@{}l|rrrr|rrrr|rrrr@{}} \toprule
                      & \multicolumn{4}{c}{\textbf{Instance 1}} & \multicolumn{4}{c}{\textbf{Instance 2}} & \multicolumn{4}{c}{\textbf{Instance 3}} \\
\textbf{OR Idle Time} & \spe{}      & \spcvar{}    & \droe{}     & \drocvar{}    & \spe{}      & \spcvar{}    & \droe{}     & \drocvar{}    & \spe{}      & \spcvar{}    & \droe{}     & \drocvar{}    \\ \midrule  \rowcolor{lightgray}
Setting I             & 468     & 510     & 575     & 585       & 486     & 521     & 685     & 664       & 1133    & 1169    & 1243    & 1214      \\
Setting IIa           & 393     & 428     & 514     & 509       & 399     & 432     & 625     & 572       & 1012    & 1044    & 1141    & 1089      \\ \rowcolor{lightgray}
Setting IIb           & 344     & 379     & 462     & 469       & 363     & 400     & 576     & 515       & 928     & 957     & 1036    & 992       \\
Setting IIc           & 351     & 388     & 465     & 481       & 380     & 422     & 593     & 528       & 944     & 965     & 1032    & 996       \\ \rowcolor{lightgray}
Setting IIIa          & 397     & 434     & 517     & 516       & 408     & 441     & 629     & 580       & 1019    & 1053    & 1148    & 1098      \\
Setting IIIb          & 316     & 350     & 428     & 447       & 345     & 382     & 551     & 484       & 890     & 913     & 982     & 945       \\
\bottomrule
\end{tabular}
\end{table}

\begin{table}[p]\centering\OneAndAHalfSpacedXI
\footnotesize   
\caption{Average waiting time under cost 3} \label{table:OS_supp_cost3_waiting}
\ra{0.7}  
\hspace*{-10mm}
\begin{tabular}{@{}l|rrrr|rrrr|rrrr@{}} \toprule
                      & \multicolumn{4}{c}{\textbf{Instance 1}} & \multicolumn{4}{c}{\textbf{Instance 2}} & \multicolumn{4}{c}{\textbf{Instance 3}} \\
\textbf{Waiting Time} & \spe{}      & \spcvar{}    & \droe{}     & \drocvar{}    & \spe{}      & \spcvar{}    & \droe{}     & \drocvar{}    & \spe{}      & \spcvar{}    & \droe{}     & \drocvar{}    \\ \midrule \rowcolor{lightgray}
Setting I             & 179     & 108     & 99      & 209       & 323     & 308     & 274     & 166       & 197     & 196     & 100     & 201       \\
Setting IIa           & 165     & 86      & 100     & 213       & 331     & 293     & 304     & 141       & 186     & 180     & 101     & 196       \\ \rowcolor{lightgray}
Setting IIb           & 288     & 195     & 164     & 324       & 564     & 537     & 412     & 282       & 322     & 328     & 182     & 339       \\
Setting IIc           & 344     & 255     & 211     & 382       & 660     & 647     & 488     & 373       & 395     & 414     & 250     & 424       \\ \rowcolor{lightgray}
Setting IIIa          & 168     & 90      & 99      & 214       & 337     & 302     & 302     & 144       & 187     & 181     & 98      & 200       \\
Setting IIIb          & 362     & 266     & 206     & 399       & 721     & 694     & 497     & 383       & 399     & 415     & 235     & 424       \\
\bottomrule
\end{tabular}
\end{table}
\begin{table}[p]\centering\OneAndAHalfSpacedXI
\footnotesize   
\caption{Average OR overtime under cost 3} \label{table:OS_supp_cost3_OR_OT}
\ra{0.7}  
\hspace*{-10mm}
\begin{tabular}{@{}l|rrrr|rrrr|rrrr@{}} \toprule
                     & \multicolumn{4}{c}{\textbf{Instance 1}} & \multicolumn{4}{c}{\textbf{Instance 2}} & \multicolumn{4}{c}{\textbf{Instance 3}} \\
\textbf{OR Overtime} & \spe{}         & \spcvar{}   & \droe{}    & \drocvar{}   & \spe{}      & \spcvar{}    & \droe{}     & \drocvar{}    & \spe{}      & \spcvar{}    & \droe{}     & \drocvar{}    \\ \midrule \rowcolor{lightgray}
Setting I            & 84.3443    & 115    & 192    & 160      & 95      & 129     & 298     & 272       & 84      & 90      & 199     & 152       \\
Setting IIa          & 83         & 107    & 205    & 153      & 98      & 127     & 327     & 277       & 83      & 88      & 214     & 148       \\ \rowcolor{lightgray}
Setting IIb          & 140        & 165    & 258    & 201      & 195     & 231     & 406     & 349       & 155     & 155     & 276     & 205       \\
Setting IIc          & 169        & 194    & 283    & 229      & 238     & 278     & 444     & 388       & 199     & 193     & 313     & 245       \\ \rowcolor{lightgray}
Setting IIIa         & 82         & 108    & 203    & 154      & 100     & 130     & 325     & 278       & 81      & 86      & 212     & 148       \\
Setting IIIb         & 178        & 203    & 291    & 235      & 264     & 302     & 462     & 402       & 202     & 198     & 318     & 246       \\
\bottomrule
\end{tabular}
\end{table}
\begin{table}[p]\centering\OneAndAHalfSpacedXI
\footnotesize   
\caption{Average anesthesiologist overtime under cost 3} \label{table:OS_supp_cost3_anes_OT}
\ra{0.7}  
\hspace*{-10mm}
\begin{tabular}{@{}l|rrrr|rrrr|rrrr@{}} \toprule
                       & \multicolumn{4}{c}{\textbf{Instance 1}} & \multicolumn{4}{c}{\textbf{Instance 2}} & \multicolumn{4}{c}{\textbf{Instance 3}} \\
\textbf{Anes Overtime} & \spe{}      & \spcvar{}    & \droe{}     & \drocvar{}    & \spe{}      & \spcvar{}    & \droe{}     & \drocvar{}    & \spe{}      & \spcvar{}    & \droe{}     & \drocvar{}    \\ \midrule \rowcolor{lightgray}
Setting I              & 84      & 115     & 192     & 160       & 95      & 129     & 298     & 272       & 84      & 90      & 199     & 152       \\
Setting IIa            & 83      & 107     & 205     & 153       & 98      & 127     & 327     & 277       & 83      & 88      & 214     & 148       \\ \rowcolor{lightgray}
Setting IIb            & 140     & 165     & 258     & 201       & 195     & 231     & 406     & 349       & 155     & 155     & 276     & 205       \\
Setting IIc            & 169     & 194     & 283     & 229       & 239     & 279     & 445     & 389       & 199     & 193     & 313     & 245       \\ \rowcolor{lightgray}
Setting IIIa           & 82      & 108     & 203     & 154       & 100     & 130     & 325     & 278       & 81      & 86      & 212     & 148       \\
Setting IIIb           & 178     & 203     & 291     & 235       & 265     & 302     & 463     & 402       & 202     & 198     & 318     & 246       \\
\bottomrule
\end{tabular}
\end{table}
\begin{table}[p]\centering\OneAndAHalfSpacedXI
\footnotesize   
\caption{Average OR idle time under cost 3} \label{table:OS_supp_cost3_OR_idle}
\ra{0.7}  
\hspace*{-10mm}
\begin{tabular}{@{}l|rrrr|rrrr|rrrr@{}} \toprule
                      & \multicolumn{4}{c}{\textbf{Instance 1}} & \multicolumn{4}{c}{\textbf{Instance 2}} & \multicolumn{4}{c}{\textbf{Instance 3}} \\
\textbf{OR Idle Time} & \spe{}      & \spcvar{}    & \droe{}     & \drocvar{}    & \spe{}      & \spcvar{}    & \droe{}     & \drocvar{}    & \spe{}      & \spcvar{}    & \droe{}     & \drocvar{}    \\ \midrule  \rowcolor{lightgray}
Setting I             & 467     & 498     & 575     & 543       & 482     & 515     & 685     & 658       & 1128    & 1134    & 1243    & 1196      \\
Setting IIa           & 392     & 416     & 514     & 462       & 396     & 425     & 625     & 575       & 1010    & 1015    & 1141    & 1075      \\ \rowcolor{lightgray}
Setting IIb           & 344     & 369     & 462     & 405       & 359     & 396     & 571     & 514       & 920     & 920     & 1042    & 971       \\
Setting IIc           & 351     & 376     & 465     & 411       & 377     & 417     & 582     & 527       & 928     & 923     & 1042    & 974       \\ \rowcolor{lightgray}
Setting IIIa          & 396     & 421     & 517     & 468       & 405     & 434     & 629     & 582       & 1017    & 1023    & 1148    & 1084      \\
Setting IIIb          & 315     & 340     & 428     & 372       & 342     & 380     & 540     & 480       & 877     & 873     & 992     & 921       \\
\bottomrule
\end{tabular}
\end{table}
\begin{table}[p]\centering\OneAndAHalfSpacedXI
\footnotesize   
\caption{Average anesthesiologist idle time under cost 3} \label{table:OS_supp_cost3_anes_idle}
\ra{0.7}  
\hspace*{-10mm}
\begin{tabular}{@{}l|rrrr|rrrr|rrrr@{}} \toprule
                        & \multicolumn{4}{c}{\textbf{Instance 1}} & \multicolumn{4}{c}{\textbf{Instance 2}} & \multicolumn{4}{c}{\textbf{Instance 3}} \\
\textbf{Anes Idle Time} & \spe{}      & \spcvar{}    & \droe{}     & \drocvar{}    & \spe{}      & \spcvar{}    & \droe{}     & \drocvar{}    & \spe{}       & \spcvar{}    & \droe{}     & \drocvar{}   \\ \midrule \rowcolor{lightgray}
Setting I               & 467     & 498     & 575     & 543       & 962     & 995     & 1165    & 1138      & 1128     & 1134    & 1243    & 1196     \\
Setting IIa             & 392     & 416     & 514     & 462       & 876     & 905     & 1105    & 1055      & 1010     & 1015    & 1141    & 1075     \\ \rowcolor{lightgray}
Setting IIb             & 344     & 369     & 462     & 405       & 839     & 876     & 1051    & 994       & 920      & 920     & 1042    & 971      \\
Setting IIc             & 351     & 376     & 465     & 411       & 857     & 897     & 1063    & 1007      & 928      & 923     & 1042    & 974      \\ \rowcolor{lightgray}
Setting IIIa            & 396     & 421     & 517     & 468       & 885     & 914     & 1109    & 1062      & 1017     & 1023    & 1148    & 1084     \\
Setting IIIb            & 315     & 340     & 428     & 372       & 822     & 860     & 1021    & 960       & 877      & 873     & 992     & 921      \\
\bottomrule
\end{tabular}
\end{table}

\renewcommand{\drocvar}{DRO-CVaR}

\subsubsection{Results Related to the Comparison with Non-Integrated Approaches} \label{appdx:expt_Rath_et_al} \phantom{ }



\textbf{Additional Results for the RO Model in \cite{Rath_et_al:2017}.} The RO model proposed in \cite{Rath_et_al:2017} minimizes the fixed costs and overtime costs under the worst-case scenario residing in the uncertainty set $\calD(\tau')$ given in \eqref{eqn:RO_uncertainty_set}. In the following, we provide further discussions and comparisons between our \spe{} and \droe{} models, and the RO model proposed in \cite{Rath_et_al:2017} based on the experiment setting in Section \ref{subsec:expt_Rath}. 

We start with comparing the optimal assignments as shown in Table \ref{table:assg_Rath_et_al}.  In the RO model, the worst-case surgery duration can either be the mean $m_i$ or the upper bound $m_i+\dhat_i$ (i.e., when $b_i=1$), and $\tau'=\lfloor \tau|I| \rfloor$ controls the number of surgeries that can deviate from mean. Surgeries highlighted in blue in Table \ref{table:assg_Rath_et_al} are those surgeries taking the maximum surgery duration $m_i+\dhat_i$.  First, we observe that the OR assignments are the same across different models. Moreover, the anesthesiologist assignments for \spe{} and \droe{} models are the same. However, the optimal anesthesiologist assignments in the RO model are different from both \spe{} and \droe{} models. In particular, when $\tau=0.2$, the RO model assigns six surgeries to anesthesiologist 3, which may lead to huge overtime as we will show next.
\begin{table}[t]\centering\small
\scriptsize
{\OneAndAHalfSpacedXI \caption{Optimal assignments from \spe{}, \droe{}, and RO models (surgeries highlighted in blue in the RO model are those attaining maximum surgery duration in optimization, i.e., $b_i=1$)}\label{table:assg_Rath_et_al}}
\ra{1.0}  
\begin{tabular}{@{}l|l|l|l@{}} \toprule
\multicolumn{4}{c}{Anesthesiologists}          \\ \midrule
& \spe{}, \droe{}  & RO ($\tau=0.2$)  & RO ($\tau=0.4$)   \\
Anes 1  & 1 $\rightarrow$ 2  $\rightarrow$ 3                                      & 1 $\rightarrow$ 2  $\rightarrow$ 3                                                        & 1 $\rightarrow$ 2  $\rightarrow$ 3                                      \\
Anes 2  & 4 $\rightarrow$ 5  $\rightarrow$ 6   $\rightarrow$ 7                    & 4 $\rightarrow$ 5  $\rightarrow$ 6 $\rightarrow$ 7                                        & 4 $\rightarrow$ 5  $\rightarrow$ 6   $\rightarrow$ \blue{7}                    \\
Anes 3  & 8 $\rightarrow$ 9  $\rightarrow$ 10   $\rightarrow$ 11 $\rightarrow$ 12 & 8 $\rightarrow$ \blue{13}  $\rightarrow$ \blue{14} $\rightarrow$ \blue{15}   $\rightarrow$ 11 $\rightarrow$ 12 & \blue{8} $\rightarrow$ \blue{13}  $\rightarrow$   14 $\rightarrow$ 11                 \\
Anes 4 & 13 $\rightarrow$ 14  $\rightarrow$   15                                 & 9    $\rightarrow$ 10                                                                     & \blue{9} $\rightarrow$ \blue{10}  $\rightarrow$   \blue{15} $\rightarrow$ 12                 \\
Anes 5 & $-$                                                                     & $-$                                                                                       & $-$                                                                     \\ \midrule
\multicolumn{4}{c}{Operating Rooms}   \\ \midrule
 & \spe{}, \droe{}    & RO ($\tau=0.2$)   & RO ($\tau=0.4$)    \\
OR 1 & 1 $\rightarrow$ 2  $\rightarrow$ 3                                      & 1 $\rightarrow$ 2  $\rightarrow$ 3                                                        & 1 $\rightarrow$ 2  $\rightarrow$ 3                                      \\
OR 2 & 4 $\rightarrow$ 5  $\rightarrow$ 6   $\rightarrow$ 7                    & 4 $\rightarrow$ 5  $\rightarrow$ 6 $\rightarrow$ 7                                        & 4 $\rightarrow$ 5  $\rightarrow$ 6   $\rightarrow$ \blue{7}                    \\
OR 3 & $-$                                                                     & $-$                                                                                       & $-$                                                                     \\
OR 4 & 8 $\rightarrow$ 9  $\rightarrow$ 10   $\rightarrow$ 11 $\rightarrow$ 12 & 8 $\rightarrow$ 9  $\rightarrow$ 10 $\rightarrow$ 11   $\rightarrow$ 12                   & \blue{8} $\rightarrow$ \blue{9}  $\rightarrow$ \blue{10}   $\rightarrow$ 11 $\rightarrow$ 12 \\
OR 5 & $-$                                                                     & $-$                                                                                       & $-$                                                                     \\
OR 6 & 13 $\rightarrow$ 14  $\rightarrow$   15                                 & \blue{13} $\rightarrow$ \blue{14}  $\rightarrow$ \blue{15}                                                     & \blue{13} $\rightarrow$ 14  $\rightarrow$   \blue{15}                                 \\
OR 7 & $-$                                                                     & $-$                                                                                     & $-$            \\
\bottomrule
\end{tabular}
\end{table}

Figures \ref{fig:OS_Rath_et_al_anes_OT} and \ref{fig:OS_Rath_et_al_room_OT} show the out-of-sample anesthesiologist and OR overtime from the three models, respectively. We observe that our \spe{} model generally produces a smaller anesthesiologist and OR overtime than the \droe{} and RO models. As discussed in Section \ref{subsec:expt_sol_quality}, the \droe{} model produces a slightly larger overtime than the \spe{} model, as well as the RO model. However, we observe that the RO model may yield an anesthesiologist or OR overtime significantly larger than both the \spe{} and \droe{} models. This could be explained by the construction of the RO uncertainty set $\calD(\tau')$, which consists of scenarios with exactly $\tau'$ surgeries deviating from the mean only. That is, the uncertainty set may not be able to capture the variability of all surgery durations (since only part of the surgery durations are altered) as opposed to our \spe{} and \droe{} models. In contrast, the out-of-sample overtime performances from the \spe{} and \droe{} models are more stable (with a smaller standard deviation).
\begin{figure} 
    \centering
    \includegraphics[scale=0.65]{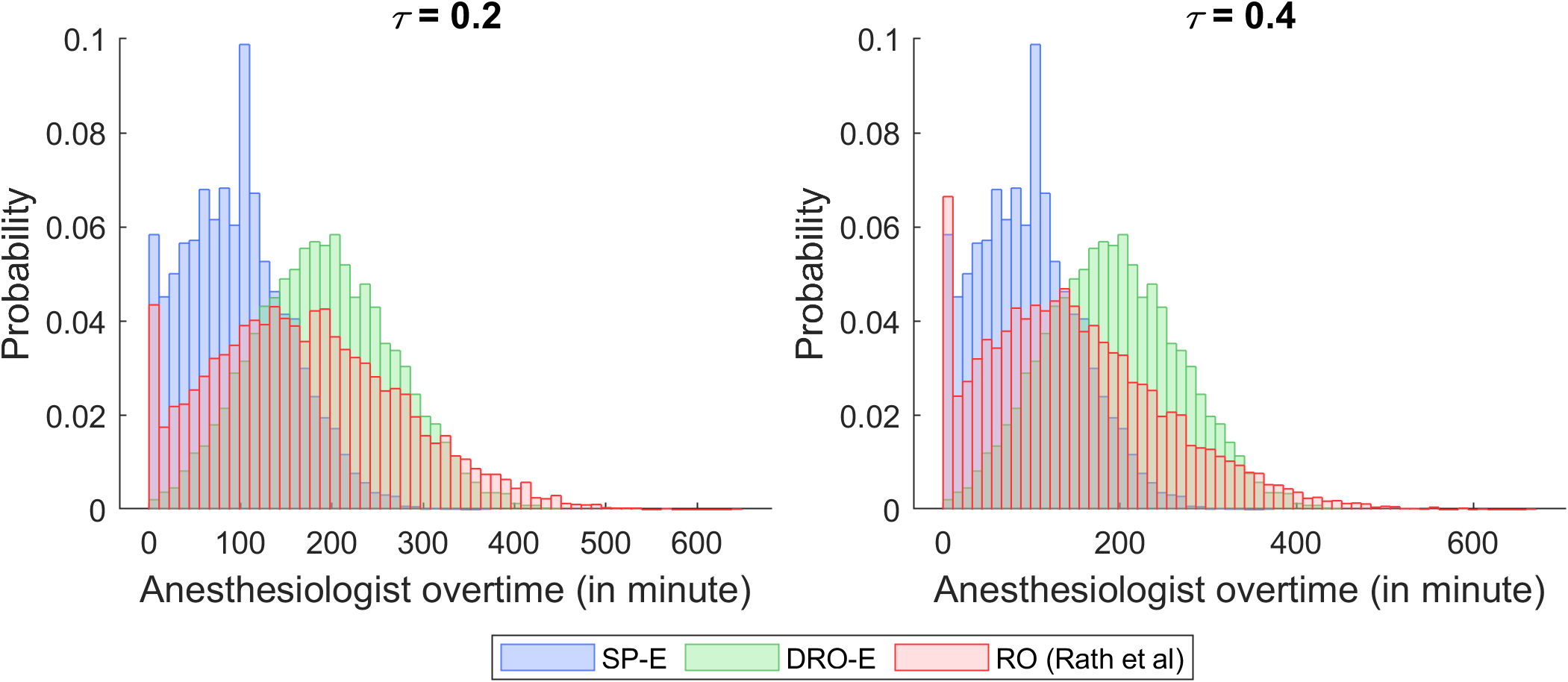}
    \caption{Anesthesiologist overtime from \spe{}, \droe{}, and RO models in instance 1 ($\tau$ controls the size of RO uncertainty set) \vspace{-2mm}}
    \label{fig:OS_Rath_et_al_anes_OT}
\end{figure}
\begin{figure} 
    \centering
    \includegraphics[scale=0.65]{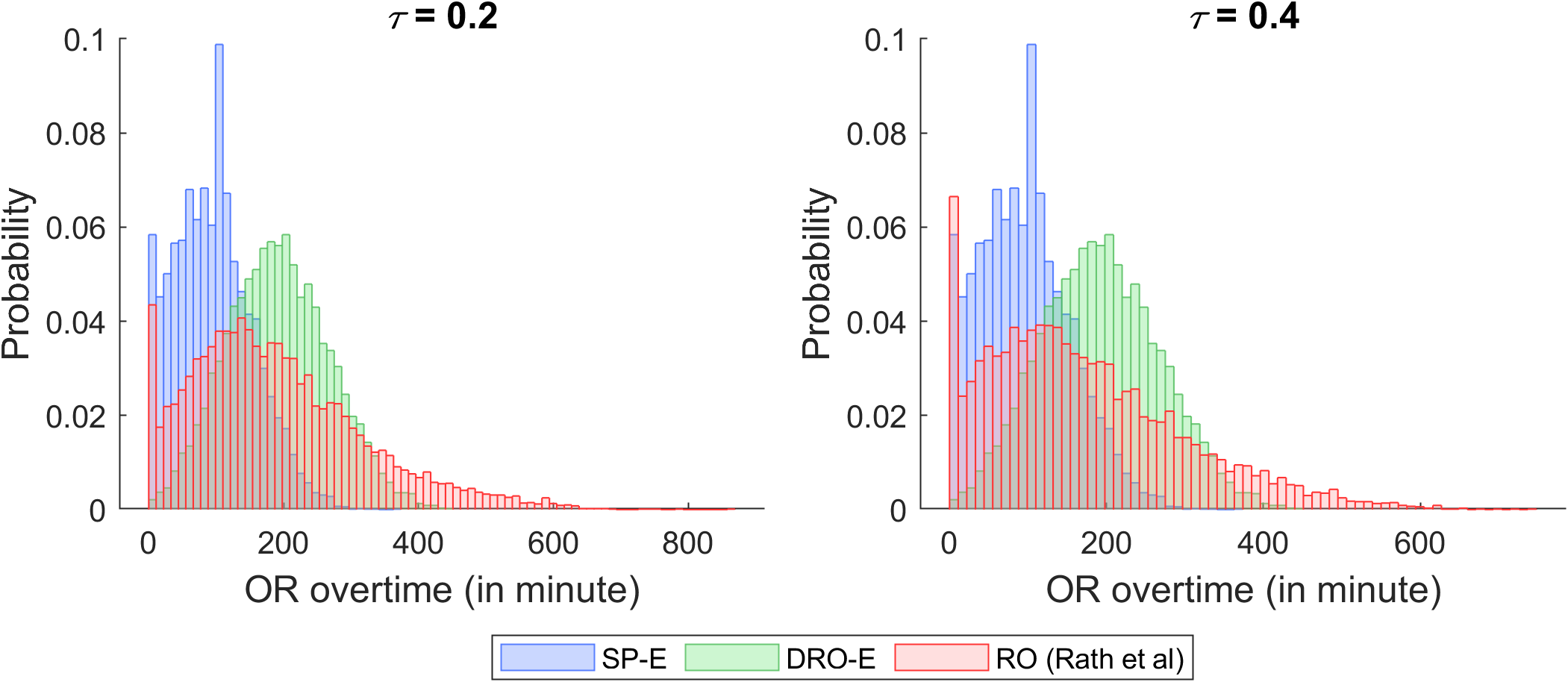}
    \caption{OR overtime from \spe{}, \droe{}, and RO models in instance 1 ($\tau$ controls the size of RO uncertainty set) \vspace{-2mm}}
    \label{fig:OS_Rath_et_al_room_OT}
\end{figure}

Finally, we also compare the computational performance of the three models. Specifically, we solve instances 1, 3, and 5 (small, medium, and large-sized instances) using the three models under cost 1 with $\cw_i=0$. We also impose a one-hour time limit to solve each model. Table \ref{table:Rath_et_al_time} shows the solution time of the three models. Using our proposed models, the instances can be solved to optimality within 2 minutes. In contrast, the RO model takes a longer time to solve instances 1 and 3, and it fails to solve instance 5. In particular, \cite{Rath_et_al:2017}'s decomposition method terminates with a large relative MIP gap after one hour. These computational results indicate that even if we ignore the waiting time component of the ORASP, our proposed models are more computationally efficient to solve than \cite{Rath_et_al:2017}'s RO model.
\begin{table}[t]\centering\small \OneAndAHalfSpacedXI
\footnotesize
\caption{Computational time in seconds when solving \spe{} and \droe{} models along with \cite{Rath_et_al:2017}'s RO model. For cases when an optimal solution was not found in the 1-hour time limit, the percentage in brackets indicates the relative optimality gap when the time limit was reached.} \label{table:Rath_et_al_time}
\ra{0.6}  
\begin{tabular}{@{}rr|rrrrr@{}} \toprule
           &       & \multicolumn{1}{c}{\spe{}} & \multicolumn{1}{c}{\droe{}} & \multicolumn{3}{c}{RO}                                                                                                                           \\
           &       & \multicolumn{1}{c}{}   & \multicolumn{1}{c}{}    & \multicolumn{1}{c}{$\tau = 0.1$}                & \multicolumn{1}{c}{$\tau = 0.2$}                & \multicolumn{1}{c}{$\tau = 0.4$}                \\ \midrule
Instance 1 & Time  & 1                      & \multicolumn{1}{r}{8}   & \multicolumn{1}{r}{106}                        & \multicolumn{1}{r}{264}                        & \multicolumn{1}{r}{1315}                       \\ 
Instance 3 & Time  & 2                      & \multicolumn{1}{r}{10}  & \multicolumn{1}{r}{1992}                       & \multicolumn{1}{r}{2431}                       & \multicolumn{1}{r}{(11.6\%)} \\
Instance 5 & Time  & 111                    & \multicolumn{1}{r}{65}  & \multicolumn{1}{r}{(87.2\%)} & \multicolumn{1}{r}{(88.6\%)} & \multicolumn{1}{r}{(89.5\%)} \\
\bottomrule
\end{tabular}
\end{table}


\textbf{Comparison with the Sequential Approach.} Next, we compare our integrated approach with a sequential approach that separates the OR assignment decisions from the remaining decisions (i.e., anesthesiologist assignment, sequencing, and scheduling decisions). Specifically, in the sequential approach, we first solve the following classical OR assignment model \citep{Denton_et_al:2010}:
\begin{subequations} \label{eqn_apdx:OR_assignement_only}
\begin{align} 
    \underset{v,\,z}{\text{minimize}} &\quad \sum_{r\in R}f_r v_r + \E_\Prob\Bigg[\sum_{r\in R}\co_r\bigg(\sum_{i\in I} d_i z_{i,r} - T_\text{end}\bigg)_+\Bigg] \\
    \text{subject to}
    &\quad \sum_{r\in R_i} z_{i,r}=1,\quad\forall i\in I, \\
    &\quad z_{i,r}\leq v_r,\quad\forall (i,r)\in\calF^R, \\
    &\quad v_r\in\{0,1\},\, z_{i,r}\in\{0,1\},\quad\forall (i,r)\in\calF^R.
\end{align}%
\end{subequations}%
Model \eqref{eqn_apdx:OR_assignement_only} decides which OR to open and assigns surgeries to open ORs. It assumes that a scheduled surgery can start immediately after the preceding surgery (in the same OR) is completed.  The objective is to minimize the sum of the fixed cost of opening ORs and the expected overtime cost. After solving model \eqref{eqn_apdx:OR_assignement_only} to obtain the optimal decision $(v^\star,z^\star)$, we then solve our proposed model \eqref{eqn:1st_stage}--\eqref{eqn:2nd_stage} by fixing $(v,z)$ to $(v^\star,z^\star)$.

We conduct the following experiment to compare the performance of our optimal solution obtained using our integrated approach and the sequential approach. Following the same experiment settings in Section \ref{sec:numerical_experiments}, we solve instances 1--6 under cost 1 using the two approaches, employing the same set of symmetry-breaking constraints for a fair comparison. Table \ref{table:sep_OR_assg_numOR} presents the number of ORs opened. Table \ref{table:sep_OR_assg_OS} summarizes the associated average out-of-sample waiting time, OR overtime, and operational cost under setting I (the perfect distributional setting) and setting IIIc (a mis-specified distributional setting); see Section \ref{subsec:expt_sol_quality} for descriptions of the settings.

First, we observe that the sequential approach opens the same or smaller number of ORs than our \spe{} model. This is because in the sequential approach,  model \eqref{eqn_apdx:OR_assignement_only} assumes that every surgery can start immediately when the preceding surgery is completed and ignores the need for anesthesiologists to perform each surgery. This results in a packed schedule leading to longer waiting times, OR overtime, and consequently higher operational costs. In particular, the operational cost of the sequential approach could be two times higher than that of our \spe{} model for large instances (e.g., instances 5--6). Finally, we note that the differences in the value of operational metrics and associated costs are more significant (in magnitude) under the mis-specified distributional setting.

\begin{table}[t]\centering\small
\footnotesize
{\OneAndAHalfSpacedXI \caption{Number of ORs open from our \spe{} model and the sequential approach}\label{table:sep_OR_assg_numOR} }
\ra{0.52}  
\begin{tabular}{@{}l|rrrrrr@{}} \toprule
\multicolumn{1}{l|}{} & Instance 1 & Instance 2 & Instance 3 & Instance 4 & Instance 5 & Instance 6 \\ \midrule
\spe{}                   & 4          & 5          & 7          & 10         & 16         & 22         \\
Sequential           & 4          & 5          & 7          & 9          & 12         & 19        \\
\bottomrule
\end{tabular}
\end{table}
\begin{table}[t]\centering\small
\footnotesize
{\OneAndAHalfSpacedXI \caption{Average out-of-sample waiting time, OR overtime, and operational cost from our \spe{} model and the sequential approach under Settings I and IIIc}\label{table:sep_OR_assg_OS} }
\ra{0.52}  
\begin{tabular}{@{}ll|rrrrrr@{}} \toprule
&&\multicolumn{6}{c}{\textbf{Setting I}}                                                                                             \\ [0.5ex]
                     & \multicolumn{1}{l|}{\textbf{}} & Instance 1 & Instance 2 & Instance 3 & Instance 4 & Instance 5 & Instance 6 \\ \midrule
\textbf{Waiting}     & \spe{}                        & 155        & 280        & 166        & 423        & 319        & 367        \\
\textbf{Time}        & Sequential                        & 155        & 279        & 168        & 582        & 645        & 911        \\ \midrule
\textbf{OR}          & \spe{}                        & 91         & 108        & 95         & 192        & 75         & 80         \\
\textbf{Overtime}    & Sequential                        & 91         & 108        & 96         & 309        & 301        & 317        \\ \midrule
\textbf{Operational} & \spe{}                        & 1429       & 2018       & 1506       & 3330       & 1815       & 2028       \\
\textbf{Cost}       & Sequential                        & 1427       & 2010       & 1514       & 5042       & 5172       & 6216       \\ \midrule
&&\multicolumn{6}{c}{\textbf{Setting IIIc}}                                                                                          \\ [0.5ex]
                     & \multicolumn{1}{l|}{\textbf{}} & Instance 1 & Instance 2 & Instance 3 & Instance 4 & Instance 5 & Instance 6 \\ \midrule
\textbf{Waiting}     & \spe{}                        & 573        & 1082       & 749        & 1713       & 1700       & 2296       \\
\textbf{Time}        & Sequential                        & 562        & 1080       & 754        & 2195       & 2613       & 3849       \\ \midrule
\textbf{OR}          & \spe{}                        & 324        & 490        & 396        & 843        & 794        & 1104       \\
\textbf{Overtime}    & Sequential                        & 321        & 483        & 409        & 1077       & 1252       & 1737       \\ \midrule
\textbf{Operational} & \spe{}                        & 5146       & 8554       & 6454       & 14184      & 13684      & 18781      \\
\textbf{Cost}       & Sequential                        & 5089       & 8442       & 6605       & 18237      & 21477      & 30703  \\
\bottomrule
\end{tabular}
\end{table}

\subsection{Computational Results on Additional ORASP Instances} \label{appdx:num_results_add}

\subsubsection{Instance Details} \label{appdx:instance_detail_add} \phantom{a}

We first provide the summary statistics of the surgery duration for different specialties of this data set. In Table \ref{table:data_summary_2}, we present the summary statistics of the random surgery duration for instances 7 to 12. The mean and standard deviation are directly obtained from \cite{Min_Yih:2010} while the minimum and maximum surgery duration are computed as the (exact) 20\% and 80\% percentiles of lognormal distribution with given mean and standard deviation respectively (since we do not have the surgery data). Next, Table \ref{table:instance_detail2} presents the details for instances 7 to 12 with various surgery types (ENT: Ear, nose and throat, CARD: Cardiology, VAS: Vascular, ORTHO: Orthopedics, NSG: Neurosurgery, GEN: General, OPHTH: Ophthalmology, URO: Urology).
\begin{table}[t]\centering\small
\footnotesize
\caption{Summary statistics of surgery duration for instances 7 to 12 (Std. Dev.: standard deviation)} \label{table:data_summary_2}
\ra{0.6}  
\begin{tabular}{@{}l|ccccccccc@{}} \toprule
Type      & ENT & OBGYN & ORTHO & Nuro & General & OPHTH & Vascular & Cardiac & Urology \\ \midrule
Mean      & 74  & 86    & 107   & 160  & 93      & 38    & 120      & 240     & 64      \\
Std. Dev. & 37  & 40    & 44    & 77   & 49      & 19    & 61       & 103     & 52      \\
Minimum   & 44  & 54    & 71    & 98   & 54      & 23    & 71       & 156     & 27      \\
Maximum   & 99  & 113   & 138   & 212  & 125     & 51    & 160      & 312     & 90      \\
\bottomrule
\end{tabular}
\end{table}
\begin{table}[t]\centering\small
\footnotesize
{\OneAndAHalfSpacedXI \caption{Instance details based on data from \protect\cite{Min_Yih:2010} \textit{Note:} GEN, OPHTH, and URO surgeries can be covered by the same pool of anesthesiologists. For example, in instance 7, there are two regular and one on-call anesthesiologists that can perform GEN and OPHTH surgeries.}  \label{table:instance_detail2} }
\ra{0.55}  
\begin{tabular}{@{}lcccccccc@{}} \toprule
\textbf{Instance   7}                   &        &        &           &           &  &  &  &  \\
Surgery type                            & ENT    & ORTH   & GEN       & OPHTH     &  &  &  &  \\
Number of surgeries                     & 4      & 5      & 4         & 2         &  &  &  &  \\
Number of ORs                           & 1      & 2      & 2         & 1         &  &  &  &  \\
Number of anesthesiologists   (regular) & 1      & 2      & \multicolumn{2}{c}{2} &  &  &  &  \\
Number of anesthesiologists   (on call) & 0      & 0      & \multicolumn{2}{c}{1} &  &  &  &  \\ \midrule
\textbf{Instance   8}                   &        &        &        &           &           &  &  &  \\
Surgery type                            & ENT    & NSG    & ORTHO  & GEN       & OPHTH     &  &  &  \\
Number of surgeries                     & 5      & 2      & 6      & 8         & 2         &  &  &  \\
Number of ORs                           & 1      & 1      & 2      & 3         & 1         &  &  &  \\
Number of anesthesiologists   (regular) & 1      & 1      & 1      & \multicolumn{2}{c}{3} &  &  &  \\
Number of anesthesiologists   (on call) & 0      & 0      & 1      & \multicolumn{2}{c}{1} &  &  &  \\ \midrule
\textbf{Instance 9}                     &     &      &     &       &        &      &  &  \\
Surgery type                            & ENT & CARD & VAS & GEN   & OPHTH  & URO  &  &  \\
Number of surgeries                     & 8   & 3    & 4   & 9     & 3      & 3    &  &  \\
Number of ORs                           & 2   & 2    & 2   & 2     & 1      & 1    &  &  \\
Number of anesthesiologists   (regular) & 2   & 1    & 2   & \multicolumn{3}{c}{4} &  &  \\
Number of anesthesiologists   (on call) & 0   & 1    & 0   & \multicolumn{3}{c}{0} &  &     \\ \midrule
\textbf{Instance 10}                    &     &      &     &       &        &      &  &  \\
Surgery   type                          & ENT & CARD & VAS & GEN   & OPHTH  & URO  &  &  \\
Number of surgeries                     & 10  & 6    & 6   & 8     & 4      & 6    &  &  \\
Number of ORs                           & 2   & 4    & 2   & 2     & 1      & 2    &  &  \\
Number of anesthesiologists   (regular) & 2   & 3    & 2   & \multicolumn{3}{c}{6} &  &  \\
Number of anesthesiologists   (on call) & 0   & 1    & 1   & \multicolumn{3}{c}{1} &  &  \\ \midrule
\textbf{Instance   11}                  &  &  & &  &  &  &  &  \\
Surgery type                            & ENT & CARD & VAS & ORTHO & GEN   & OPHTH  & URO  &  \\
Number of surgeries                     & 4   & 9    & 11  & 13    & 13    & 2      & 8    &  \\
Number of ORs                           & 1   & 5    & 3   & 4     & 3     & 1      & 3    &  \\
Number of anesthesiologists   (regular) & 1   & 4    & 2   & 3     & \multicolumn{3}{c}{6} &  \\
Number of anesthesiologists   (on call) & 0   & 1    & 2   & 2     & \multicolumn{3}{c}{2} &   \\ \midrule
\textbf{Instance   12}                  & &  &  &   &  &  &  \\
Surgery   type                        & ENT & CARD & VAS & ORTHO & NSG & GEN   & OPHTH   & URO  \\
Number of surgeries                   & 7   & 10   & 9   & 16    & 10  & 15    & 3       & 10   \\
Number of ORs                         & 2   & 5    & 3   & 6     & 3   & 4     & 2       & 5    \\
Number of anesthesiologists (regular) & 1   & 4    & 2   & 6     & 3   & \multicolumn{3}{c}{12} \\
Number of anesthesiologists (on call) & 1   & 2    & 2   & 0     & 0   & \multicolumn{3}{c}{2}   \\ 
\bottomrule
\end{tabular}
\end{table}

\subsubsection{Computational Time} \label{appdx:expt_comp_time_add} \phantom{a}

In this section, we provide the computational times for solving our proposed models under the three cost structures (see Section \ref{subsec:expt_comp_time} for the experiment settings). Table \ref{table:CPU_time_add} presents the average solution times in seconds for instances 7--12 under costs 1--3. We note that the observations are similar to those for instances 1--6.
\begin{table}[t]\centering\OneAndAHalfSpacedXI
\footnotesize   
\caption{Computational time (in s) for additional instances (instance with ``$\dag$'': apply \eqref{eqn:global_LB} and \eqref{eqn:VI_dual_bounds} with initial scenario $m$; instance with ``--'': cannot be solved within $10$ hours)} \label{table:CPU_time_add}
\ra{0.7}  
\begin{tabular}{@{}l|llllllllllll@{}} \toprule
\textbf{Cost 1} & Instance 7 & Instance 8 & Instance 9 & Instance 10 & Instance 11 & Instance 12 \\ \midrule
\spe{}       & 0.82 & 3.23  & 8.60   & 13.70  & 410.82  & 352.00 \\  [0.5ex]
\spcvar{}     & 2.43 & 47.40 & 204.18 & 722.54 & --      & --     \\  [0.5ex]
\droe{}      & 5.22 & 10.73 & 20.74  & 54.51  & 1065.69 & 777.36 \\  [0.5ex]
\drocvar{} & 1.55 & 2.63  & 3.45   & 5.13   & 20.02   & 32.23   \\ \midrule
\textbf{Cost 2} & Instance 7 & Instance 8 & Instance 9 & Instance 10 & Instance 11 & Instance 12 \\  \midrule
\spe{}       & 0.81 & 3.41  & 9.03  & 16.18   & 229.06 & 491.22  \\  [0.5ex]
\spcvar{}     & 2.22 & 35.70 & 99.68 & 4907.91 & --     & --      \\  [0.5ex]
\droe{}      & 5.88 & 12.85 & 23.98 & 57.02   & 553.40 & 1005.93 \\  [0.5ex]
\drocvar{} & 3.27 & 5.52  & 16.50 & 18.52   & 390.80 & 2688.92 \\ \midrule
\textbf{Cost 3} & Instance 7 & Instance 8 & Instance 9 & Instance 10 & Instance 11 & Instance 12 \\  \midrule
\spe{}       & 0.81 & 3.31  & 7.60   & 16.00   & 265.99 & 645.85  \\  [0.5ex]
\spcvar{}     & 2.82 & 41.35 & 132.68 & 6087.30 & --     & --      \\  [0.5ex]
\droe{}      & 5.82 & 12.43 & 28.13  & 78.18   & 630.64 & 961.50  \\  [0.5ex]
\drocvar{} & 3.25 & 7.20  & 15.16  & 19.77   & 547.40 & 2344.09   \\
\bottomrule
\end{tabular}
\end{table}

\subsubsection{Analysis of Optimal Solutions}  \phantom{a}

In this section, we compare the number of ORs opened and the number of on-call anesthesiologists called in from our proposed models. Figure \ref{fig:expt_num_OR_add} shows the optimal number of ORs opened by each model. The observations are similar to those in Section \ref{subsec:expt_opt_sol}.  In addition, we also observe that the \drocvar{} and \spe{} models call in the largest and smallest number of on-call anesthesiologists, respectively.

\begin{figure}[t!]
    \centering
    \includegraphics[scale=0.7]{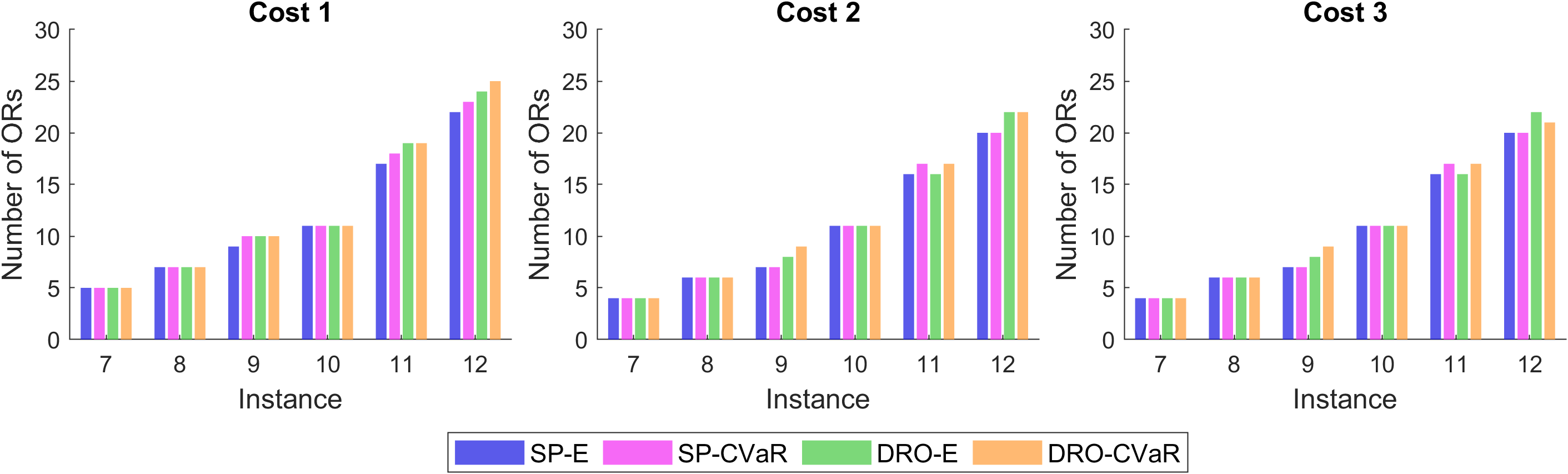}
    \caption{Number of ORs opened for different instances under different costs}
    \label{fig:expt_num_OR_add}
\end{figure}

\newpage
\subsubsection{Analysis of Solution Quality} \label{appdx:expt_sim_results_add} \phantom{a}

In this section, we present the average out-of-sample performance metrics for instances 7--12 under costs 1--3. Tables \ref{table:OS_add_cost1_waiting}--\ref{table:OS_add_cost1_anes_OT} present respectively the average waiting time, OR overtime, and anesthesiologists overtime under cost 1; Tables \ref{table:OS_add_cost2_waiting}--\ref{table:OS_add_cost2_OR_idle} present respectively the average waiting time, OR overtime, anesthesiologists overtime, and OR idle time under cost 2; Tables~\ref{table:OS_add_cost3_waiting}--\ref{table:OS_add_cost3_anes_idle} present respectively the average waiting time, OR overtime, anesthesiologists overtime, OR idle time, and anesthesiologist idle time under cost 3. We note that the observations are similar to those for instances 1--6. 

\renewcommand{\drocvar}{D-CVaR}

\begin{table}[p]\centering\OneAndAHalfSpacedXI
\footnotesize   
\caption{Average waiting time under cost 1} \label{table:OS_add_cost1_waiting}
\ra{0.7}  
\hspace*{-10mm}
\begin{tabular}{@{}l|rrrr|rrrr|rrrr@{}} \toprule
  & \multicolumn{4}{c}{\textbf{Instance 7}}  & \multicolumn{4}{c}{\textbf{Instance 8}}  & \multicolumn{4}{c}{\textbf{Instance 9}}  \\
\textbf{Waiting Time} & \spe{}  & \spcvar{} & \droe{} & \drocvar{} & \spe{}   & \spcvar{} & \droe{}  & \drocvar{} & \spe{}   & \spcvar{} & \droe{}  & \drocvar{} \\ \midrule \rowcolor{lightgray}
Setting I             & 5   & 4    & 0   & 0      & 15   & 14   & 0    & 0      & 168  & 72   & 33   & 9      \\
Setting IIa           & 2   & 1    & 0   & 0      & 4    & 3    & 0    & 0      & 140  & 60   & 36   & 5      \\ \rowcolor{lightgray}
Setting IIb           & 39  & 38   & 36  & 31     & 75   & 64   & 69   & 79     & 342  & 185  & 150  & 91     \\
Setting IIc           & 78  & 74   & 75  & 64     & 142  & 124  & 145  & 176    & 491  & 304  & 290  & 194    \\ \rowcolor{lightgray}
Setting IIIa          & 2   & 1    & 0   & 0      & 4    & 3    & 0    & 0      & 140  & 59   & 36   & 5      \\
Setting IIIb          & 59  & 57   & 55  & 48     & 110  & 94   & 107  & 120    & 424  & 238  & 201  & 133    \\
\bottomrule
\end{tabular}
\end{table}
\begin{table}[p]\centering\OneAndAHalfSpacedXI
\footnotesize   
\caption{Average OR overtime under cost 1} \label{table:OS_add_cost1_OR_OT}
\ra{0.7}  
\hspace*{-10mm}
\begin{tabular}{@{}l|rrrr|rrrr|rrrr@{}} \toprule
 & \multicolumn{4}{c}{\textbf{Instance 7}}  & \multicolumn{4}{c}{\textbf{Instance 8}}  & \multicolumn{4}{c}{\textbf{Instance 9}}  \\
\textbf{OR Overtime} & \spe{}  & \spcvar{} & \droe{} & \drocvar{} & \spe{}   & \spcvar{} & \droe{}  & \drocvar{} & \spe{}   & \spcvar{} & \droe{}  & \drocvar{} \\ \midrule \rowcolor{lightgray}
Setting I            & 0   & 1    & 5   & 5      & 1    & 1    & 14   & 14     & 67   & 67   & 133  & 163    \\
Setting IIa          & 0   & 0    & 3   & 3      & 0    & 0    & 7    & 7      & 51   & 57   & 142  & 165    \\ \rowcolor{lightgray}
Setting IIb          & 9   & 9    & 11  & 11     & 26   & 25   & 31   & 45     & 159  & 133  & 210  & 224    \\
Setting IIc          & 19  & 20   & 20  & 20     & 58   & 55   & 57   & 89     & 238  & 197  & 271  & 268    \\ \rowcolor{lightgray}
Setting IIIa         & 0   & 0    & 3   & 3      & 0    & 0    & 8    & 8      & 51   & 56   & 140  & 163    \\
Setting IIIb         & 13  & 14   & 15  & 15     & 42   & 38   & 42   & 66     & 211  & 169  & 243  & 251    \\
\bottomrule
\end{tabular}
\end{table}
\begin{table}[p]\centering\OneAndAHalfSpacedXI
\footnotesize   
\caption{Average anesthesiologist overtime under cost 1} \label{table:OS_add_cost1_anes_OT}
\ra{0.7}  
\hspace*{-10mm}
\begin{tabular}{@{}l|rrrr|rrrr|rrrr@{}} \toprule
 & \multicolumn{4}{c}{\textbf{Instance 7}}  & \multicolumn{4}{c}{\textbf{Instance 8}}  & \multicolumn{4}{c}{\textbf{Instance 9}}  \\
\textbf{Anes Overtime} & \spe{}  & \spcvar{} & \droe{} & \drocvar{} & \spe{}   & \spcvar{} & \droe{}  & \drocvar{} & \spe{}   & \spcvar{} & \droe{}  & \drocvar{} \\ \midrule \rowcolor{lightgray}
Setting I              & 0   & 1    & 5   & 5      & 1    & 1    & 14   & 14     & 67   & 68   & 70   & 100    \\
Setting IIa            & 0   & 0    & 3   & 3      & 0    & 0    & 7    & 7      & 51   & 57   & 75   & 98     \\ \rowcolor{lightgray}
Setting IIb            & 9   & 9    & 11  & 11     & 21   & 22   & 31   & 41     & 155  & 136  & 115  & 130    \\
Setting IIc            & 19  & 20   & 20  & 20     & 48   & 49   & 57   & 81     & 235  & 205  & 160  & 159    \\ \rowcolor{lightgray}
Setting IIIa           & 0   & 0    & 3   & 3      & 0    & 0    & 8    & 8      & 51   & 56   & 74   & 97     \\
Setting IIIb           & 13  & 14   & 15  & 15     & 33   & 33   & 42   & 58     & 206  & 174  & 134  & 143    \\
\bottomrule
\end{tabular}
\end{table}

\begin{table}[p]\centering\OneAndAHalfSpacedXI
\footnotesize   
\caption{Average waiting time under cost 2} \label{table:OS_add_cost2_waiting}
\ra{0.7}  
\hspace*{-10mm}
\begin{tabular}{@{}l|rrrr|rrrr|rrrr@{}} \toprule
 & \multicolumn{4}{c}{\textbf{Instance 7}}  & \multicolumn{4}{c}{\textbf{Instance 8}}  & \multicolumn{4}{c}{\textbf{Instance 9}}  \\
\textbf{Waiting Time} & \spe{}  & \spcvar{} & \droe{} & \drocvar{} & \spe{}   & \spcvar{} & \droe{}  & \drocvar{} & \spe{}   & \spcvar{} & \droe{}  & \drocvar{} \\ \midrule \rowcolor{lightgray}
Setting I             & 116 & 46   & 100 & 50     & 199  & 144  & 196  & 353    & 469  & 400  & 259  & 106    \\
Setting IIa           & 111 & 26   & 103 & 40     & 193  & 124  & 206  & 349    & 461  & 382  & 279  & 72     \\ \rowcolor{lightgray}
Setting IIb           & 203 & 105  & 164 & 101    & 370  & 264  & 308  & 507    & 815  & 711  & 414  & 231    \\
Setting IIc           & 252 & 155  & 206 & 142    & 475  & 355  & 384  & 610    & 1030 & 904  & 541  & 353    \\ \rowcolor{lightgray}
Setting IIIa          & 108 & 27   & 102 & 41     & 189  & 125  & 203  & 347    & 452  & 376  & 276  & 74     \\
Setting IIIb          & 247 & 143  & 194 & 131    & 453  & 335  & 365  & 582    & 958  & 841  & 480  & 297    \\
\bottomrule
\end{tabular}
\end{table}
\begin{table}[p]\centering\OneAndAHalfSpacedXI
\footnotesize   
\caption{Average OR overtime under cost 2} \label{table:OS_add_cost2_OR_OT}
\ra{0.7}  
\hspace*{-10mm}
\begin{tabular}{@{}l|rrrr|rrrr|rrrr@{}} \toprule
 & \multicolumn{4}{c}{\textbf{Instance 7}}  & \multicolumn{4}{c}{\textbf{Instance 8}}  & \multicolumn{4}{c}{\textbf{Instance 9}}  \\
\textbf{OR Overtime} & \spe{}  & \spcvar{} & \droe{} & \drocvar{} & \spe{}   & \spcvar{} & \droe{}  & \drocvar{} & \spe{}   & \spcvar{} & \droe{}  & \drocvar{} \\ \midrule \rowcolor{lightgray}
Setting I            & 53  & 95   & 101 & 122    & 151  & 189  & 211  & 224    & 353  & 404  & 346  & 155    \\
Setting IIa          & 54  & 93   & 105 & 117    & 156  & 188  & 215  & 223    & 363  & 403  & 375  & 147    \\ \rowcolor{lightgray}
Setting IIb          & 92  & 123  & 137 & 153    & 211  & 237  & 267  & 269    & 509  & 544  & 468  & 238    \\
Setting IIc          & 110 & 140  & 152 & 176    & 249  & 269  & 297  & 304    & 594  & 622  & 521  & 301    \\ \rowcolor{lightgray}
Setting IIIa         & 53  & 92   & 103 & 117    & 154  & 188  & 213  & 223    & 357  & 399  & 368  & 146    \\
Setting IIIb         & 110 & 138  & 152 & 171    & 241  & 262  & 293  & 295    & 572  & 603  & 510  & 279    \\
\bottomrule
\end{tabular}
\end{table}
\begin{table}[p]\centering\OneAndAHalfSpacedXI
\footnotesize   
\caption{Average anesthesiologist overtime under cost 2} \label{table:OS_supp_cost2_anes_OT}
\ra{0.7}  
\hspace*{-10mm}
\begin{tabular}{@{}l|rrrr|rrrr|rrrr@{}} \toprule
 & \multicolumn{4}{c}{\textbf{Instance 7}}  & \multicolumn{4}{c}{\textbf{Instance 8}}  & \multicolumn{4}{c}{\textbf{Instance 9}}  \\
\textbf{Anes Overtime} & \spe{}  & \spcvar{} & \droe{} & \drocvar{} & \spe{}   & \spcvar{} & \droe{}  & \drocvar{} & \spe{}   & \spcvar{} & \droe{}  & \drocvar{} \\ \midrule \rowcolor{lightgray}
Setting I              & 53  & 101  & 101 & 122    & 151  & 189  & 211  & 224    & 353  & 405  & 346  & 109    \\
Setting IIa            & 54  & 95   & 105 & 117    & 156  & 188  & 215  & 223    & 363  & 403  & 375  & 109    \\ \rowcolor{lightgray}
Setting IIb            & 92  & 140  & 137 & 153    & 211  & 237  & 267  & 269    & 510  & 547  & 468  & 161    \\
Setting IIc            & 110 & 164  & 152 & 176    & 249  & 269  & 297  & 304    & 596  & 632  & 521  & 206    \\ \rowcolor{lightgray}
Setting IIIa           & 53  & 94   & 103 & 117    & 154  & 188  & 213  & 223    & 357  & 399  & 368  & 108    \\
Setting IIIb           & 110 & 163  & 152 & 171    & 241  & 262  & 293  & 295    & 575  & 608  & 510  & 184    \\
\bottomrule
\end{tabular}
\end{table}
\begin{table}[p]\centering\OneAndAHalfSpacedXI
\footnotesize   
\caption{Average OR idle time under cost 2} \label{table:OS_add_cost2_OR_idle}
\ra{0.7}  
\hspace*{-10mm}
\begin{tabular}{@{}l|rrrr|rrrr|rrrr@{}} \toprule
 & \multicolumn{4}{c}{\textbf{Instance 7}}  & \multicolumn{4}{c}{\textbf{Instance 8}}  & \multicolumn{4}{c}{\textbf{Instance 9}}  \\
\textbf{OR Idle Time} & \spe{}   & \spcvar{} & \droe{}  & \drocvar{} & \spe{}   & \spcvar{} & \droe{}  & \drocvar{} & \spe{}   & \spcvar{} & \droe{}  & \drocvar{} \\ \midrule \rowcolor{lightgray}
Setting I             & 768  & 810  & 816  & 836    & 1010 & 1047 & 1070 & 1083   & 979  & 1030 & 1451 & 1740   \\
Setting IIa           & 727  & 766  & 778  & 790    & 941  & 973  & 999  & 1008   & 877  & 917  & 1369 & 1622   \\ \rowcolor{lightgray}
Setting IIb           & 691  & 722  & 736  & 753    & 867  & 892  & 922  & 924    & 835  & 870  & 1274 & 1524   \\
Setting IIc           & 690  & 720  & 732  & 756    & 873  & 892  & 921  & 928    & 854  & 881  & 1261 & 1521   \\ \rowcolor{lightgray}
Setting IIIa          & 733  & 773  & 784  & 797    & 949  & 983  & 1008 & 1018   & 891  & 933  & 1382 & 1640   \\
Setting IIIb          & 678  & 707  & 720  & 739    & 843  & 865  & 896  & 898    & 833  & 864  & 1251 & 1500   \\
\bottomrule
\end{tabular}
\end{table}

\begin{table}[p]\centering\OneAndAHalfSpacedXI
\footnotesize   
\caption{Average waiting time under cost 3} \label{table:OS_add_cost3_waiting}
\ra{0.7}  
\hspace*{-10mm}
\begin{tabular}{@{}l|rrrr|rrrr|rrrr@{}} \toprule
 & \multicolumn{4}{c}{\textbf{Instance 7}}  & \multicolumn{4}{c}{\textbf{Instance 8}}  & \multicolumn{4}{c}{\textbf{Instance 9}}  \\
\textbf{Waiting Time} & \spe{}   & \spcvar{} & \droe{} & \drocvar{} & \spe{}   & \spcvar{} & \droe{}  & \drocvar{} & \spe{}   & \spcvar{} & \droe{}  & \drocvar{} \\\midrule  \rowcolor{lightgray}
Setting I             & 145  & 79   & 100 & 51     & 200  & 287  & 311  & 332    & 517  & 393  & 259  & 252    \\
Setting IIa           & 136  & 54   & 103 & 40     & 201  & 262  & 340  & 336    & 511  & 377  & 279  & 242    \\ \rowcolor{lightgray}
Setting IIb           & 250  & 161  & 164 & 103    & 360  & 482  & 473  & 485    & 892  & 695  & 429  & 420    \\
Setting IIc           & 309  & 223  & 206 & 145    & 451  & 594  & 567  & 577    & 1121 & 866  & 573  & 559    \\ \rowcolor{lightgray}
Setting IIIa          & 133  & 54   & 102 & 40     & 198  & 263  & 334  & 333    & 501  & 370  & 276  & 242    \\
Setting IIIb          & 302  & 210  & 194 & 135    & 436  & 575  & 542  & 558    & 1033 & 822  & 502  & 488    \\
\bottomrule
\end{tabular}
\end{table}
\begin{table}[p]\centering\OneAndAHalfSpacedXI
\footnotesize   
\caption{Average OR overtime under cost 3} \label{table:OS_add_cost3_OR_OT}
\ra{0.7}  
\hspace*{-10mm}
\begin{tabular}{@{}l|rrrr|rrrr|rrrr@{}} \toprule
 & \multicolumn{4}{c}{\textbf{Instance 7}}  & \multicolumn{4}{c}{\textbf{Instance 8}}  & \multicolumn{4}{c}{\textbf{Instance 9}}  \\
\textbf{OR Overtime} & \spe{}  & \spcvar{} & \droe{} & \drocvar{} & \spe{}   & \spcvar{} & \droe{}  & \drocvar{} & \spe{}   & \spcvar{} & \droe{}  & \drocvar{} \\ \midrule \rowcolor{lightgray}
Setting I            & 54  & 89   & 101 & 118    & 149  & 195  & 182  & 224    & 348  & 396  & 346  & 145    \\
Setting IIa          & 53  & 86   & 105 & 114    & 155  & 193  & 190  & 222    & 359  & 395  & 375  & 136    \\ \rowcolor{lightgray}
Setting IIb          & 95  & 128  & 137 & 148    & 212  & 265  & 248  & 266    & 505  & 537  & 466  & 231    \\
Setting IIc          & 115 & 155  & 152 & 169    & 246  & 302  & 285  & 294    & 587  & 616  & 518  & 302    \\ \rowcolor{lightgray}
Setting IIIa         & 52  & 85   & 103 & 113    & 154  & 193  & 188  & 222    & 352  & 391  & 368  & 135    \\
Setting IIIb         & 115 & 149  & 152 & 165    & 241  & 298  & 278  & 289    & 568  & 596  & 507  & 277    \\
\bottomrule
\end{tabular}
\end{table}
\begin{table}[p]\centering\OneAndAHalfSpacedXI
\footnotesize   
\caption{Average anesthesiologist overtime under cost 3} \label{table:OS_add_cost3_anes_OT}
\ra{0.7}  
\hspace*{-10mm}
\begin{tabular}{@{}l|rrrr|rrrr|rrrr@{}} \toprule
 & \multicolumn{4}{c}{\textbf{Instance 7}}  & \multicolumn{4}{c}{\textbf{Instance 8}}  & \multicolumn{4}{c}{\textbf{Instance 9}}  \\
\textbf{Anes Overtime} & \spe{}  & \spcvar{} & \droe{} & \drocvar{} & \spe{}   & \spcvar{} & \droe{}  & \drocvar{} & \spe{}   & \spcvar{} & \droe{}  & \drocvar{} \\ \midrule \rowcolor{lightgray}
Setting I              & 54  & 89   & 101 & 118    & 149  & 195  & 182  & 224    & 348  & 168  & 346  & 145    \\
Setting IIa            & 53  & 86   & 105 & 114    & 155  & 193  & 190  & 222    & 359  & 160  & 375  & 136    \\ \rowcolor{lightgray}
Setting IIb            & 95  & 126  & 137 & 148    & 212  & 268  & 248  & 266    & 507  & 289  & 466  & 231    \\
Setting IIc            & 115 & 148  & 152 & 169    & 246  & 312  & 285  & 294    & 592  & 369  & 518  & 302    \\ \rowcolor{lightgray}
Setting IIIa           & 52  & 85   & 103 & 113    & 154  & 193  & 188  & 222    & 352  & 158  & 368  & 135    \\
Setting IIIb           & 115 & 145  & 152 & 165    & 241  & 303  & 278  & 289    & 571  & 345  & 507  & 277    \\
\bottomrule
\end{tabular}
\end{table}
\begin{table}[p]\centering\OneAndAHalfSpacedXI
\footnotesize   
\caption{Average OR idle time under cost 3} \label{table:OS_add_cost3_OR_idle}
\ra{0.7}  
\hspace*{-10mm}
\begin{tabular}{@{}l|rrrr|rrrr|rrrr@{}} \toprule
  & \multicolumn{4}{c}{\textbf{Instance 7}}  & \multicolumn{4}{c}{\textbf{Instance 8}}  & \multicolumn{4}{c}{\textbf{Instance 9}}  \\
\textbf{OR Idle Time} & \spe{}   & \spcvar{} & \droe{}  & \drocvar{} & \spe{}   & \spcvar{} & \droe{}  & \drocvar{} & \spe{}   & \spcvar{} & \droe{}  & \drocvar{} \\ \midrule \rowcolor{lightgray}
Setting I             & 768  & 804  & 816  & 833    & 1008 & 1054 & 1040 & 1083   & 974  & 1021 & 1451 & 1730   \\
Setting IIa           & 726  & 758  & 778  & 786    & 940  & 978  & 975  & 1006   & 873  & 910  & 1369 & 1610   \\ \rowcolor{lightgray}
Setting IIb           & 694  & 727  & 736  & 747    & 867  & 921  & 904  & 921    & 831  & 863  & 1272 & 1517   \\
Setting IIc           & 695  & 735  & 732  & 749    & 869  & 926  & 909  & 918    & 846  & 876  & 1258 & 1522   \\ \rowcolor{lightgray}
Setting IIIa          & 732  & 765  & 784  & 793    & 948  & 988  & 983  & 1017   & 886  & 925  & 1382 & 1629   \\
Setting IIIb          & 683  & 718  & 720  & 733    & 844  & 901  & 881  & 892    & 829  & 857  & 1248 & 1498   \\
\bottomrule
\end{tabular}
\end{table}
\begin{table}[p]\centering\OneAndAHalfSpacedXI
\footnotesize   
\caption{Average anesthesiologist idle time under cost 3} \label{table:OS_add_cost3_anes_idle}
\ra{0.7}  
\hspace*{-10mm}
\begin{tabular}{@{}l|rrrr|rrrr|rrrr@{}} \toprule
 & \multicolumn{4}{c}{\textbf{Instance 7}}  & \multicolumn{4}{c}{\textbf{Instance 8}}  & \multicolumn{4}{c}{\textbf{Instance 9}}  \\
\textbf{Anes Idle Time} & \spe{}   & \spcvar{} & \droe{}  & \drocvar{} & \spe{}   & \spcvar{} & \droe{}  & \drocvar{} & \spe{}   & \spcvar{} & \droe{}  & \drocvar{} \\ \midrule \rowcolor{lightgray}
Setting I               & 1248 & 1284 & 1296 & 1313   & 1008 & 1054 & 1040 & 1083   & 1934 & 1981 & 1931 & 1958   \\
Setting IIa             & 1206 & 1238 & 1258 & 1266   & 940  & 978  & 975  & 1006   & 1833 & 1870 & 1849 & 1845   \\ \rowcolor{lightgray}
Setting IIb             & 1174 & 1205 & 1216 & 1227   & 867  & 924  & 904  & 921    & 1793 & 1824 & 1752 & 1766   \\
Setting IIc             & 1175 & 1208 & 1212 & 1229   & 869  & 936  & 909  & 918    & 1811 & 1839 & 1738 & 1773   \\ \rowcolor{lightgray}
Setting IIIa            & 1212 & 1245 & 1264 & 1273   & 948  & 988  & 983  & 1017   & 1846 & 1885 & 1862 & 1861   \\
Setting IIIb            & 1163 & 1194 & 1200 & 1213   & 844  & 906  & 881  & 892    & 1792 & 1819 & 1728 & 1751   \\
\bottomrule
\end{tabular}
\end{table}

\renewcommand{\drocvar}{DRO-CVaR}

\newpage

\color{black}
\subsubsection{Computation Time for Another Set of ORASP Instances} \label{appdx:comp_time_mutliOR} \phantom{a}

In this section, we present additional computational results for another set of ORASP instances with some ORs accommodating multiple surgery types. Specifically, we construct these new ORASP instances as follows. (i) We construct instances 1$'$--5$'$ from instances 1--5 by allowing ORs dedicated to either medicine (MED) or gastroenterology (GASTRO) to be shared by both MED and GASTRO surgeries. (ii) We construct instances 6$'$--11$'$ from instances 6--11 by allowing an OR dedicated to ophthalmology (OPHTH) to be shared by both general (GEN) and OPHTH surgeries. (iii) We construct a new hypothetical instance 13$'$ (based on instance~$4$) with four surgery types, where there is one OR that can accommodate three surgery types (cardiology, orthopedic, and medicine) and one anesthesiologist that can cover two surgery types (gynecology and medicine). Table~\ref{table:CPU_time_multiOR} presents the average solution times in seconds under cost 1 (see Section~8.2 for the experiment settings). It is clear that we can solve all the new instances using the SP-E, DRO-E, and DRO-CVaR models within a reasonable time. These new results demonstrate that our proposed models are tractable even for these  ORASP instances, further emphasizing our contribution.
%
\begin{table}[t]\centering\OneAndAHalfSpacedXI \color{black}
\footnotesize   
\caption{\blue{Computational time (in s) with cost 1 (instance with ``--'': some replications cannot be solved within $2$ hours)} }\label{table:CPU_time_multiOR}
\ra{0.7}  
\begin{tabular}{@{}l|llllllllllllll@{}} \toprule
         & Instance 1$'$ & Instance 2$'$ & Instance 3$'$ & Instance 4$'$  & \multicolumn{2}{l}{Instance 5$'$} \\ \midrule
SP-E     & 12.53         & 14.07         & 10.08         & 68.57          & 1565.64         &                 \\
SP-CVaR  & 55.80         & 201.96        & 57.93         & --             & --              &                 \\
DRO-E    & 8.06          & 18.28         & 23.44         & 342.13         & 2861.24         &                 \\
DRO-CVaR & 3.43          & 5.80          & 3.06          & 149.53         & 153.37          &                 \\ \midrule
         & Instance 7$'$ & Instance 8$'$ & Instance 9$'$ & Instance 10$'$ & Instance 11$'$  & Instance 13$'$  \\ \midrule
SP-E     & 1.62          & 9.77          & 109.99        & 312.68         & 1819.34         & 914.01          \\
SP-CVaR  & 10.91         & 199.00        & 2167.11       & --             & --              & --              \\
DRO-E    & 3.27          & 16.22         & 43.69         & 198.20         & 3875.95         & 1995.26         \\
DRO-CVaR & 1.36          & 2.63          & 4.13          & 5.76           & 95.86           & 83.41        \\
\bottomrule
\end{tabular}\color{black}
\end{table}

\color{black}

\subsection{Monte Carlo Optimization} \label{appdx:MCO}

In this section, we describe the Monte Carlo optimization procedure to obtain near optimal solution for the SP-E model with possibly a small number of scenarios and provide the corresponding results. The procedure is as follows (see \cite{Jebali_Diabat:2015ec,Kleywegt_et_al:2002,Lamiri_et_al:2009,Shehadeh_et_al:2021ec} for more detail explanations). 
\begin{enumerate}[leftmargin=1.6cm]
    \item [Step 1a.] Generate scenarios $d^{n,k}_i$ for $i\in I$, $m\in\{1,\dots,K\}$ and $n\in\{1,\dots,N\}$.
    \item [Step 1b.] Solve the SP-E model using SAA with scenarios $\{d^{n,k}_i\}_{n=1}^N$ and obtain the optimal first-stage solution $\chi^k:=(x^k,y^k,z^k,v^k,u^k,s^k)$ with optimal value $v^k$ for $k\in\{1,\dots,K\}$.
    \item [Step 1c.] Generate new scenarios $\tilde{d}^{n,k}_i$ for $i\in I$, $k\in\{1,\dots,K\}$ and $n\in\{1,\dots,N'\}$.
    \item [Step 1d.] Obtain the estimate of the true function value $\widehat{v}^{n,k}$ evaluated at $\chi^k$ using samples $\{\tilde{d}^{n,k}_i\}_{n=1}^{N'}$ by solving the second-stage problem \eqref{eqn:2nd_stage} for $m\in\{1,\dots,K\}$ and $n\in\{1,\dots,N'\}$.
    \item [Step 2.] Compute the estimates $\widehat{\mu}=K^{-1}\sum_{k=1}^K v^k$, $\widehat{\mu}^k=(N')^{-1}\sum_{n=1}^{N'} \hat{v}^{n,k}$, 
    $$\widehat{\sigma}^2=\frac{1}{K(K-1)}\sum_{k=1}^K\big(v^k-\widehat{\mu}\big)^2 \quad\text{and}\quad (\widehat{\sigma}^2)^k=\frac{1}{N'(N'-1)}\sum_{n=1}^{N'}\big(\widehat{v}^{n,k} - \widehat{\mu}^k \big)^2.$$
    \item [Step 3.] Obtain the estimated optimality gap $\widehat{\mu}^k-\widehat{\mu}$ and its variance $\widehat{\sigma}^2+ (\widehat{\sigma}^2)^k$ for $k\in\{1,\dots,K\}$.
\end{enumerate}

The quantities $\widehat{\mu}$ and $\widehat{\mu}^k$ provide a statistical lower and upper bound on the optimal objective value of the SP-E model respectively \citep{Kleywegt_et_al:2002,Mak_et_al:1999ec}. Therefore, given a sample size $N$, if the estimated optimality gap $\widehat{\mu}^k-\widehat{\mu}$ and its variance are small, then the sample size $N$ is sufficiently large for producing near optimal solutions. Otherwise, we could increase the sample size $N$ and estimate the new optimality gap. 

We apply the Monte Carlo optimization method to the ORASP with $K=20$, $N=100$ and $N'=10000$. In Table \ref{table:MCO}, we report the absolute value of the mean (over $K=20$ replications) of the estimated gap and the standard deviation estimate. Moreover, we report the normalize gap, i.e., approximated optimality index (AOI) defined as $\big|\sum_{k=1}^K (\widehat{\mu}^k-\widehat{\mu})/\sum_{k=1}^K \widehat{\mu}^k\big|$ \citep{Shehadeh_et_al:2021ec}. It is clear from Table \ref{table:MCO} that almost all the AOIs are less than $1\%$. That is, the relative optimality gap is small. Also, the standard deviations of the estimated optimality gaps are small. Therefore, it is reasonable to solve the SP-E model using SAA with $N=100$ scenarios to obtain near-optimal solutions.

\begin{table}[t]\centering\small \OneAndAHalfSpacedXI
\footnotesize
\caption{Statistics from Monte Carlo optimization method with $K=20$, $N=100$ and $N'=10000$} \label{table:MCO}
\ra{1.0}  
\begin{tabular}{@{}l|rrr|rrr|rrr@{}} \toprule
& \multicolumn{3}{c}{\textbf{Cost 1}} & \multicolumn{3}{c}{\textbf{Cost 2}} & \multicolumn{3}{c}{\textbf{Cost 3}} \\
Instance & Gap    & Std.  & AOI    & Gap    & Std.  & AOI    & Gap    & Std.  & AOI    \\ \midrule
1        & 10.87  & 22.07 & 0.22\% & 34.35  & 22.82 & 0.47\% & 42.15  & 25.14 & 0.52\% \\
2        & 15.34  & 20.36 & 0.24\% & 23.85  & 22.08 & 0.27\% & 29.14  & 25.92 & 0.28\% \\
3        & 43.04  & 21.81 & 0.55\% & 68.01  & 26.75 & 0.50\% & 69.95  & 32.00 & 0.45\% \\
4        & 58.89  & 31.64 & 0.48\% & 64.46  & 48.67 & 0.37\% & 65.95  & 51.62 & 0.31\% \\
5        & 55.83  & 20.39 & 0.35\% & 96.96  & 32.03 & 0.42\% & 108.44 & 40.64 & 0.32\% \\
6        & 102.50 & 19.82 & 0.47\% & 184.59 & 44.29 & 0.58\% & 209.68 & 42.90 & 0.40\% \\
7        & 47.48  & 10.45 & 1.05\% & 54.53  & 17.64 & 0.65\% & 60.71  & 22.16 & 0.58\% \\
8        & 5.95   & 1.71  & 0.08\% & 30.68  & 27.86 & 0.24\% & 26.20  & 32.54 & 0.18\% \\
9        & 6.29   & 17.22 & 0.06\% & 34.81  & 34.02 & 0.21\% & 43.24  & 36.94 & 0.22\% \\
10       & 20.23  & 24.85 & 0.18\% & 30.71  & 27.20 & 0.17\% & 27.92  & 30.06 & 0.12\% \\
11       & 42.40  & 41.07 & 0.19\% & 110.49 & 60.77 & 0.36\% & 120.40 & 55.62 & 0.36\% \\
12       & 134.56 & 37.73 & 0.47\% & 112.24 & 46.57 & 0.30\% & 132.14 & 54.92 & 0.28\% \\
\bottomrule
\end{tabular}
\end{table}

\newpage
\bibliographystyle{informs2014.bst}
\bibliography{references_main}

\begin{thebibliography}{111}
\providecommand{\natexlab}[1]{#1}
\providecommand{\url}[1]{\texttt{#1}}
\providecommand{\urlprefix}{URL }

\bibitem[{Abdalkareem et~al.(2021)Abdalkareem, Amir, Al-Betar, Ekhan,
  \protect\BIBand{} Hammouri}]{Abdalkareem_et_al:2021}
Abdalkareem ZA, Amir A, Al-Betar MA, Ekhan P, Hammouri AI (2021) Healthcare
  scheduling in optimization context: A review. \emph{Health and Technology}
  11:445--469.

\bibitem[{Addis et~al.(2014)Addis, Carello, \protect\BIBand{}
  T{\`a}nfani}]{Addis_et_al:2014}
Addis B, Carello G, T{\`a}nfani E (2014) A robust optimization approach for the
  operating room planning problem with uncertain surgery duration.
  \emph{Proceedings of the International Conference on Health Care Systems
  Engineering}, 175--189 (Springer).

\bibitem[{Ahmadi-Javid et~al.(2017)Ahmadi-Javid, Jalali, \protect\BIBand{}
  Klassen}]{Ahmadi-Javid_et_al:2017}
Ahmadi-Javid A, Jalali Z, Klassen KJ (2017) Outpatient appointment systems in
  healthcare: A review of optimization studies. \emph{European Journal of
  Operational Research} 258(1):3--34.

\bibitem[{Aringhieri et~al.(2015)Aringhieri, Landa, Soriano, Tanfani,
  \protect\BIBand{} Testi}]{Aringhieri_et_al:2015}
Aringhieri R, Landa P, Soriano P, Tanfani E, Testi A (2015) A two level
  metaheuristic for the operating room scheduling and assignment problem.
  \emph{Computers \& Operations Research} 54:21--34.

\bibitem[{Baker \protect\BIBand{} Trietsch(2013)}]{Baker_Trietsch:2013}
Baker KR, Trietsch D (2013) \emph{Principles of Sequencing and Scheduling}
  (John Wiley \& Sons).

\bibitem[{Bansal et~al.(2021{\natexlab{a}})Bansal, Berg, \protect\BIBand{}
  Huang}]{Bansal_et_al:2021a}
Bansal A, Berg B, Huang YL (2021{\natexlab{a}}) A distributionally robust
  optimization approach for coordinating clinical and surgical appointments.
  \emph{IISE Transactions} 1--83.

\bibitem[{Bansal et~al.(2021{\natexlab{b}})Bansal, Berg, \protect\BIBand{}
  Huang}]{Bansal_et_al:2021}
Bansal A, Berg BP, Huang YL (2021{\natexlab{b}}) A value function-based
  approach for robust surgery planning. \emph{Computers \& Operations Research}
  132:105313.

\bibitem[{Batun et~al.(2011)Batun, Denton, Huschka, \protect\BIBand{}
  Schaefer}]{Batun_et_al:2011}
Batun S, Denton BT, Huschka TR, Schaefer AJ (2011) Operating room pooling and
  parallel surgery processing under uncertainty. \emph{INFORMS Journal on
  Computing} 23(2):220--237.

\bibitem[{Becker et~al.(2019)Becker, Steenweg, \protect\BIBand{}
  Werners}]{Becker_et_al:2019}
Becker T, Steenweg PM, Werners B (2019) Cyclic shift scheduling with on-call
  duties for emergency medical services. \emph{Health Care Management Science}
  22:676--690.

\bibitem[{Berg et~al.(2014)Berg, Denton, Erdogan, Rohleder, \protect\BIBand{}
  Huschka}]{Berg_et_al:2014}
Berg BP, Denton BT, Erdogan SA, Rohleder T, Huschka T (2014) Optimal booking
  and scheduling in outpatient procedure centers. \emph{Computers \& Operations
  Research} 50:24--37.

\bibitem[{Bertsimas \protect\BIBand{} Sim(2004)}]{Bertsimas_Sim:2004}
Bertsimas D, Sim M (2004) The price of robustness. \emph{Operations Research}
  52(1):35--53.

\bibitem[{Birge \protect\BIBand{} Louveaux(2011)}]{Birge_Louveaux:2011}
Birge JR, Louveaux F (2011) \emph{Introduction to Stochastic Programming}
  (Springer Science \& Business Media).

\bibitem[{Bovim et~al.(2020)Bovim, Christiansen, Gullhav, Range,
  \protect\BIBand{} Hellemo}]{Bovim_et_al:2020}
Bovim TR, Christiansen M, Gullhav AN, Range TM, Hellemo L (2020) Stochastic
  master surgery scheduling. \emph{European Journal of Operational Research}
  285(2):695--711.

\bibitem[{Breuer et~al.(2020)Breuer, Lahrichi, Clark, \protect\BIBand{}
  Benneyan}]{Breuer_et_al:2020}
Breuer DJ, Lahrichi N, Clark DE, Benneyan JC (2020) Robust combined operating
  room planning and personnel scheduling under uncertainty. \emph{Operations
  Research for Health Care} 27:100276.

\bibitem[{Cardoen et~al.(2010)Cardoen, Demeulemeester, \protect\BIBand{}
  Beli{\"e}n}]{Cardoen_et_al:2010}
Cardoen B, Demeulemeester E, Beli{\"e}n J (2010) Operating room planning and
  scheduling: A literature review. \emph{European Journal of Operational
  Research} 201(3):921--932.

\bibitem[{Cayirli et~al.(2006)Cayirli, Veral, \protect\BIBand{}
  Rosen}]{Cayirli_et_al:2006}
Cayirli T, Veral E, Rosen H (2006) Designing appointment scheduling systems for
  ambulatory care services. \emph{Health Care Management Science} 9:47--58.

\bibitem[{Celik et~al.(2023)Celik, Gul, \protect\BIBand{}
  {\c{C}}elik}]{Celik_et_al:2023}
Celik B, Gul S, {\c{C}}elik M (2023) A stochastic programming approach to
  surgery scheduling under parallel processing principle. \emph{Omega}
  115:102799.

\bibitem[{De~Simone et~al.(2021)De~Simone, Vargas, \protect\BIBand{}
  Servillo}]{De-Simone_et_al:2021}
De~Simone S, Vargas M, Servillo G (2021) Organizational strategies to reduce
  physician burnout: A systematic review and meta-analysis. \emph{Aging
  Clinical and Experimental Research} 33(4):883--894.

\bibitem[{Dean et~al.(2022)Dean, Meisami, Lam, Van~Oyen, Stromblad,
  \protect\BIBand{} Kastango}]{Dean_et_al:2022}
Dean A, Meisami A, Lam H, Van~Oyen MP, Stromblad C, Kastango N (2022) Quantile
  regression forests for individualized surgery scheduling. \emph{Health Care
  Management Science} 1--28.

\bibitem[{Delage \protect\BIBand{} Ye(2010)}]{Delage_Ye:2010}
Delage E, Ye Y (2010) Distributionally robust optimization under moment
  uncertainty with application to data-driven problems. \emph{Operations
  Research} 58(3):595--612.

\bibitem[{Deng et~al.(2019)Deng, Shen, \protect\BIBand{}
  Denton}]{Deng_et_al:2019}
Deng Y, Shen S, Denton B (2019) Chance-constrained surgery planning under
  conditions of limited and ambiguous data. \emph{INFORMS Journal on Computing}
  31(3):559--575.

\bibitem[{Denton \protect\BIBand{} Gupta(2003)}]{Denton_Gupta:2003}
Denton B, Gupta D (2003) A sequential bounding approach for optimal appointment
  scheduling. \emph{IIE Transactions} 35(11):1003--1016.

\bibitem[{Denton et~al.(2007)Denton, Viapiano, \protect\BIBand{}
  Vogl}]{Denton_et_al:2007}
Denton B, Viapiano J, Vogl A (2007) Optimization of surgery sequencing and
  scheduling decisions under uncertainty. \emph{Health Care Management Science}
  10(1):13--24.

\bibitem[{Denton et~al.(2010)Denton, Miller, Balasubramanian, \protect\BIBand{}
  Huschka}]{Denton_et_al:2010}
Denton BT, Miller AJ, Balasubramanian HJ, Huschka TR (2010) Optimal allocation
  of surgery blocks to operating rooms under uncertainty. \emph{Operations
  Research} 58(4-part-1):802--816.

\bibitem[{Deshpande et~al.(2023)Deshpande, Mundru, Rath, Knowles, Rowe,
  \protect\BIBand{} Wood}]{Deshpande_et_al:2023}
Deshpande V, Mundru N, Rath S, Knowles M, Rowe D, Wood BC (2023) Data-driven
  surgical tray optimization to improve operating room efficiency.
  \emph{Operations Research} .

\bibitem[{Diamant et~al.(2018)Diamant, Milner, Quereshy, \protect\BIBand{}
  Xu}]{Diamant_et_al:2018}
Diamant A, Milner J, Quereshy F, Xu B (2018) Inventory management of reusable
  surgical supplies. \emph{Health Care Management Science} 21:439--459.

\bibitem[{Doulabi et~al.(2014)Doulabi, Rousseau, \protect\BIBand{}
  Pesant}]{Doulabi_et_al:2014}
Doulabi SHH, Rousseau LM, Pesant G (2014) A constraint programming-based column
  generation approach for operating room planning and scheduling.
  \emph{International Conference on Integration of Constraint Programming,
  Artificial Intelligence, and Operations Research}, 455--463 (Springer).

\bibitem[{Eliaz \protect\BIBand{} Ortoleva(2016)}]{Eliaz_Ortoleva:2016}
Eliaz K, Ortoleva P (2016) Multidimensional {E}llsberg. \emph{Management
  Science} 62(8):2179--2197.

\bibitem[{Erhard et~al.(2018)Erhard, Schoenfelder, F{\"u}gener,
  \protect\BIBand{} Brunner}]{Erhard_et_al:2018}
Erhard M, Schoenfelder J, F{\"u}gener A, Brunner JO (2018) State of the art in
  physician scheduling. \emph{European Journal of Operational Research}
  265(1):1--18.

\bibitem[{Filippi et~al.(2020)Filippi, Guastaroba, \protect\BIBand{}
  Speranza}]{Filippi_et_al:2020}
Filippi C, Guastaroba G, Speranza MG (2020) Conditional value-at-risk beyond
  finance: A survey. \emph{International Transactions in Operational Research}
  27(3):1277--1319.

\bibitem[{Freeman et~al.(2016)Freeman, Melouk, \protect\BIBand{}
  Mittenthal}]{Freeman_et_al:2016}
Freeman NK, Melouk SH, Mittenthal J (2016) A scenario-based approach for
  operating theater scheduling under uncertainty. \emph{Manufacturing \&
  Service Operations Management} 18(2):245--261.

\bibitem[{F{\"u}gener et~al.(2014)F{\"u}gener, Hans, Kolisch, Kortbeek,
  \protect\BIBand{} Vanberkel}]{Fugener_et_al:2014}
F{\"u}gener A, Hans EW, Kolisch R, Kortbeek N, Vanberkel PT (2014) Master
  surgery scheduling with consideration of multiple downstream units.
  \emph{European Journal of Operational Research} 239(1):227--236.

\bibitem[{Georghiou et~al.(2019)Georghiou, Kuhn, \protect\BIBand{}
  Wiesemann}]{Georghiou_et_al:2019}
Georghiou A, Kuhn D, Wiesemann W (2019) The decision rule approach to
  optimization under uncertainty: Methodology and applications.
  \emph{Computational Management Science} 16(4):545--576.

\bibitem[{Goh \protect\BIBand{} Sim(2010)}]{Goh_Sim:2010}
Goh J, Sim M (2010) Distributionally robust optimization and its tractable
  approximations. \emph{Operations Research} 58(4-part-1):902--917.

\bibitem[{Guerriero \protect\BIBand{} Guido(2011)}]{Guerriero_Guido:2011}
Guerriero F, Guido R (2011) Operational research in the management of the
  operating theatre: A survey. \emph{Health Care Management Science}
  14(1):89--114.

\bibitem[{Guo et~al.(2014)Guo, Wu, Li, \protect\BIBand{} Rong}]{Guo_et_al:2014}
Guo M, Wu S, Li B, Rong Y (2014) Maximizing the efficiency of use of nurses
  under uncertain surgery durations: A case study. \emph{Computers \&
  Industrial Engineering} 78:313--319.

\bibitem[{Gupta(2007)}]{Gupta:2007}
Gupta D (2007) Surgical suites' operations management. \emph{Production and
  Operations Management} 16(6):689--700.

\bibitem[{Gupta \protect\BIBand{} Denton(2008)}]{Gupta_Denton:2008}
Gupta D, Denton B (2008) Appointment scheduling in health care: Challenges and
  opportunities. \emph{IIE Transactions} 40(9):800--819.

\bibitem[{Halevy(2007)}]{Halevy:2007}
Halevy Y (2007) Ellsberg revisited: An experimental study. \emph{Econometrica}
  75(2):503--536.

\bibitem[{Hashemi~Doulabi \protect\BIBand{}
  Khalilpourazari(2022)}]{HashemiDoulabi_Khalilpourazari:2022}
Hashemi~Doulabi H, Khalilpourazari S (2022) Stochastic weekly operating room
  planning with an exponential number of scenarios. \emph{Annals of Operations
  Research} 1--22.

\bibitem[{He et~al.(2019)He, Chaussalet, \protect\BIBand{} Qu}]{He_et_al:2019}
He F, Chaussalet T, Qu R (2019) Controlling understaffing with conditional
  value-at-risk constraint for an integrated nurse scheduling problem under
  patient demand uncertainty. \emph{Operations Research Perspectives} 6:100119.

\bibitem[{Hoefnagel et~al.(2020)Hoefnagel, McLeod, \protect\BIBand{}
  Mongan}]{Hoefnagel_et_al:2020}
Hoefnagel AL, McLeod C, Mongan PD (2020) Daily anesthesia assignment schedule
  automation: Utilizing an electronic scheduling system to export daily
  assignments into the electronic health record. \emph{Perioperative Care and
  Operating Room Management} 21:100135.

\bibitem[{Huang et~al.(2020)Huang, Zheng, \protect\BIBand{}
  Hill}]{Huang_et_al:2020}
Huang W, Zheng W, Hill DJ (2020) Distributionally robust optimal power flow in
  multi-microgrids with decomposition and guaranteed convergence. \emph{IEEE
  Transactions on Smart Grid} 12(1):43--55.

\bibitem[{Jebali \protect\BIBand{} Diabat(2015)}]{Jebali_Diabat:2015ec}
Jebali A, Diabat A (2015) A stochastic model for operating room planning under
  capacity constraints. \emph{International Journal of Production Research}
  53(24):7252--7270.

\bibitem[{Joseph et~al.(2020)Joseph, Wax, Goldstein, Huang, McCormick,
  \protect\BIBand{} Levin}]{Joseph_et_al:2020}
Joseph TT, Wax DB, Goldstein R, Huang J, McCormick PJ, Levin MA (2020) A
  web-based perioperative dashboard as a platform for anesthesia informatics
  innovation. \emph{Anesthesia and Analgesia} 131(5):1640.

\bibitem[{Jung et~al.(2019)Jung, Pinedo, Sriskandarajah, \protect\BIBand{}
  Tiwari}]{Jung_et_al:2019}
Jung KS, Pinedo M, Sriskandarajah C, Tiwari V (2019) Scheduling elective
  surgeries with emergency patients at shared operating rooms. \emph{Production
  and Operations Management} 28(6):1407--1430.

\bibitem[{Kang et~al.(2019)Kang, Li, Li, \protect\BIBand{}
  Zhu}]{Kang_et_al:2019}
Kang Z, Li X, Li Z, Zhu S (2019) Data-driven robust mean-cvar portfolio
  selection under distribution ambiguity. \emph{Quantitative Finance}
  19(1):105--121.

\bibitem[{Kay{\i}{\c{s}} et~al.(2015)Kay{\i}{\c{s}}, Khaniyev, Suermondt,
  \protect\BIBand{} Sylvester}]{Kayis_et_al:2015}
Kay{\i}{\c{s}} E, Khaniyev TT, Suermondt J, Sylvester K (2015) A robust
  estimation model for surgery durations with temporal, operational, and
  surgery team effects. \emph{Health Care Management Science} 18:222--233.

\bibitem[{Kayis et~al.(2012)Kayis, Wang, Patel, Gonzalez, Jain, Ramamurthi,
  Santos, Singhal, Suermondt, \protect\BIBand{} Sylvester}]{Kayis_et_al:2012}
Kayis E, Wang H, Patel M, Gonzalez T, Jain S, Ramamurthi R, Santos C, Singhal
  S, Suermondt J, Sylvester K (2012) Improving prediction of surgery duration
  using operational and temporal factors. \emph{AMIA Annual Symposium
  Proceedings}, volume 2012, 456 (American Medical Informatics Association).

\bibitem[{Keyvanshokooh et~al.(2022)Keyvanshokooh, Kazemian, Fattahi,
  \protect\BIBand{} Van~Oyen}]{Keyvanshokooh_et_al:2020}
Keyvanshokooh E, Kazemian P, Fattahi M, Van~Oyen MP (2022) Coordinated and
  priority-based surgical care: An integrated distributionally robust
  stochastic optimization approach. \emph{Production and Operations Management}
  31(4):1510--1535.

\bibitem[{Khaniyev et~al.(2020)Khaniyev, Kay{\i}{\c{s}}, \protect\BIBand{}
  G{\"u}ll{\"u}}]{Khaniyev_et_al:2020}
Khaniyev T, Kay{\i}{\c{s}} E, G{\"u}ll{\"u} R (2020) Next-day operating room
  scheduling with uncertain surgery durations: Exact analysis and heuristics.
  \emph{European Journal of Operational Research} 286(1):49--62.

\bibitem[{Kishimoto \protect\BIBand{}
  Yamashita(2018)}]{Kishimoto_Yamashita:2018}
Kishimoto S, Yamashita M (2018) A successive {LP} approach with {C-VaR} type
  constraints for {IMRT} optimization. \emph{Operations Research for Health
  Care} 17:55--64.

\bibitem[{Kleywegt et~al.(2002)Kleywegt, Shapiro, \protect\BIBand{} Homem-de
  Mello}]{Kleywegt_et_al:2002}
Kleywegt AJ, Shapiro A, Homem-de Mello T (2002) The sample average
  approximation method for stochastic discrete optimization. \emph{SIAM Journal
  on Optimization} 12(2):479--502.

\bibitem[{Kuhn et~al.(2019)Kuhn, Esfahani, Nguyen, \protect\BIBand{}
  Shafieezadeh-Abadeh}]{Kuhn_et_al:2019}
Kuhn D, Esfahani PM, Nguyen VA, Shafieezadeh-Abadeh S (2019) Wasserstein
  distributionally robust optimization: Theory and applications in machine
  learning. \emph{Operations Research \& Management Science in the Age of
  Analytics}, 130--166 (Informs).

\bibitem[{Lamiri et~al.(2009)Lamiri, Grimaud, \protect\BIBand{}
  Xie}]{Lamiri_et_al:2009}
Lamiri M, Grimaud F, Xie X (2009) Optimization methods for a stochastic surgery
  planning problem. \emph{International Journal of Production Economics}
  120(2):400--410.

\bibitem[{Latorre-N{\'u}{\~n}ez et~al.(2016)Latorre-N{\'u}{\~n}ez,
  L{\"u}er-Villagra, Marianov, Obreque, Ramis, \protect\BIBand{}
  Neriz}]{Latorre_et_al:2016}
Latorre-N{\'u}{\~n}ez G, L{\"u}er-Villagra A, Marianov V, Obreque C, Ramis F,
  Neriz L (2016) Scheduling operating rooms with consideration of all
  resources, post anesthesia beds and emergency surgeries. \emph{Computers \&
  Industrial Engineering} 97:248--257.

\bibitem[{Lim et~al.(2020)Lim, Kardar, Ebrahimi, \protect\BIBand{}
  Cao}]{Lim_et_al:2020}
Lim GJ, Kardar L, Ebrahimi S, Cao W (2020) A risk-based modeling approach for
  radiation therapy treatment planning under tumor shrinkage uncertainty.
  \emph{European Journal of Operational Research} 280(1):266--278.

\bibitem[{Linz et~al.(2019)Linz, Zabinsky, Heim, \protect\BIBand{}
  Fishman}]{Linz_et_al:2019}
Linz D, Zabinsky ZB, Heim J, Fishman P (2019) A multi-objective model for
  optimizing staffing across geographically distributed patient-centered
  medical homes. \emph{IISE Transactions on Healthcare Systems Engineering}
  9(1):55--70.

\bibitem[{Liu et~al.(2018)Liu, Zhang, Luo, \protect\BIBand{}
  Xu}]{Liu_et_al:2018}
Liu H, Zhang T, Luo S, Xu D (2018) Operating room scheduling and surgeon
  assignment problem under surgery durations uncertainty. \emph{Technology and
  Health Care} 26(2):297--304.

\bibitem[{Liu et~al.(2019)Liu, Li, \protect\BIBand{} Zhang}]{Liu_et_al:2019}
Liu K, Li Q, Zhang ZH (2019) Distributionally robust optimization of an
  emergency medical service station location and sizing problem with joint
  chance constraints. \emph{Transportation Research Part B: Methodological}
  119:79--101.

\bibitem[{Mak et~al.(1999)Mak, Morton, \protect\BIBand{}
  Wood}]{Mak_et_al:1999ec}
Mak WK, Morton DP, Wood RK (1999) Monte {C}arlo bounding techniques for
  determining solution quality in stochastic programs. \emph{Operations
  Research Letters} 24(1-2):47--56.

\bibitem[{Makboul et~al.(2022)Makboul, Kharraja, Abbassi, \protect\BIBand{}
  Alaoui}]{Makboul_et_al:2022}
Makboul S, Kharraja S, Abbassi A, Alaoui AEH (2022) A two-stage robust
  optimization approach for the master surgical schedule problem under
  uncertainty considering downstream resources. \emph{Health Care Management
  Science} 25(1):63--88.

\bibitem[{Mancilla \protect\BIBand{} Storer(2012)}]{Mancilla_Storer:2012}
Mancilla C, Storer R (2012) A sample average approximation approach to
  stochastic appointment sequencing and scheduling. \emph{IIE Transactions}
  44(8):655--670.

\bibitem[{Mannino et~al.(2010)Mannino, Nilssen, \protect\BIBand{}
  Nordlander}]{Mannino_et_al:2010}
Mannino C, Nilssen E, Nordlander T (2010) {SINTEF ICT: MSS-Adjusts Surgery
  data}.
  \url{https://www.sintef.no/Projectweb/Health-care-optimization/Testbed/},
  accessed: 2022-04-18.

\bibitem[{Marques \protect\BIBand{} Captivo(2015)}]{Marques_Captivo:2015}
Marques I, Captivo ME (2015) Bicriteria elective surgery scheduling using an
  evolutionary algorithm. \emph{Operations Research for Health Care} 7:14--26.

\bibitem[{Marques \protect\BIBand{} Captivo(2017)}]{Marques_Captivo:2017}
Marques I, Captivo ME (2017) Different stakeholders’ perspectives for a
  surgical case assignment problem: Deterministic and robust approaches.
  \emph{European Journal of Operational Research} 261(1):260--278.

\bibitem[{Marques et~al.(2014)Marques, Captivo, \protect\BIBand{}
  Pato}]{Marques_et_al:2014}
Marques I, Captivo ME, Pato MV (2014) Scheduling elective surgeries in a
  {P}ortuguese hospital using a genetic heuristic. \emph{Operations Research
  for Health Care} 3(2):59--72.

\bibitem[{Mazloumian et~al.(2022)Mazloumian, Baki, \protect\BIBand{}
  Ahmadi}]{Mazloumian_et_al:2022}
Mazloumian M, Baki MF, Ahmadi M (2022) A robust multiobjective integrated
  master surgery schedule and surgical case assignment model at a publicly
  funded hospital. \emph{Computers \& Industrial Engineering} 163:107826.

\bibitem[{M'Hallah \protect\BIBand{} Visintin(2019)}]{MHallah_Visintin:2019}
M'Hallah R, Visintin F (2019) A stochastic model for scheduling elective
  surgeries in a cyclic master surgical schedule. \emph{Computers \& Industrial
  Engineering} 129:156--168.

\bibitem[{Min \protect\BIBand{} Yih(2010)}]{Min_Yih:2010}
Min D, Yih Y (2010) Scheduling elective surgery under uncertainty and
  downstream capacity constraints. \emph{European Journal of Operational
  Research} 206(3):642--652.

\bibitem[{Moosavi \protect\BIBand{}
  Ebrahimnejad(2020)}]{Moosavi_Ebrahimnejad:2020}
Moosavi A, Ebrahimnejad S (2020) Robust operating room planning considering
  upstream and downstream units: A new two-stage heuristic algorithm.
  \emph{Computers \& Industrial Engineering} 143:106387.

\bibitem[{Najjarbashi \protect\BIBand{} Lim(2019)}]{Najjarbashi_Lim:2019}
Najjarbashi A, Lim GJ (2019) A variability reduction method for the operating
  room scheduling problem under uncertainty using {CVaR}. \emph{Operations
  Research for Health Care} 20:25--32.

\bibitem[{Neyshabouri \protect\BIBand{} Berg(2017)}]{Neyshabouri_Berg:2017}
Neyshabouri S, Berg BP (2017) Two-stage robust optimization approach to
  elective surgery and downstream capacity planning. \emph{European Journal of
  Operational Research} 260(1):21--40.

\bibitem[{Penn et~al.(2017)Penn, Potts, \protect\BIBand{}
  Harper}]{Penn_et_al:2017}
Penn ML, Potts CN, Harper PR (2017) Multiple criteria mixed-integer programming
  for incorporating multiple factors into the development of master operating
  theatre timetables. \emph{European Journal of Operational Research}
  262(1):194--206.

\bibitem[{Pess{\^o}a et~al.(2015)Pess{\^o}a, Lins, Da~Silva, \protect\BIBand{}
  Fiszman}]{Pessoa_et_al:2015}
Pess{\^o}a LAM, Lins MPE, Da~Silva ACM, Fiszman R (2015) Integrating soft and
  hard operational research to improve surgical centre management at a
  university hospital. \emph{European Journal of Operational Research}
  245(3):851--861.

\bibitem[{Pflug \protect\BIBand{} Pohl(2018)}]{Pflug_Pohl:2018}
Pflug GC, Pohl M (2018) A review on ambiguity in stochastic portfolio
  optimization. \emph{Set-Valued and Variational Analysis} 26(4):733--757.

\bibitem[{Rahimian \protect\BIBand{} Mehrotra(2022)}]{Rahimian_Mehrotra:2022}
Rahimian H, Mehrotra S (2022) Frameworks and results in distributionally robust
  optimization. \emph{Open Journal of Mathematical Optimization} 3:1--85.

\bibitem[{Rath \protect\BIBand{} Rajaram(2022)}]{Rath_Rajaram:2022}
Rath S, Rajaram K (2022) Staff planning for hospitals with implicit cost
  estimation and stochastic optimization. \emph{Production and Operations
  Management} 31(3):1271--1289.

\bibitem[{Rath et~al.(2017)Rath, Rajaram, \protect\BIBand{}
  Mahajan}]{Rath_et_al:2017}
Rath S, Rajaram K, Mahajan A (2017) Integrated anesthesiologist and room
  scheduling for surgeries: Methodology and application. \emph{Operations
  Research} 65(6):1460--1478.

\bibitem[{Rockafellar et~al.(2000)Rockafellar, Uryasev
  et~al.}]{Rockafellar_Uryasev:2000}
Rockafellar RT, Uryasev S, et~al. (2000) Optimization of conditional
  value-at-risk. \emph{Journal of Risk} 2:21--42.

\bibitem[{Roos \protect\BIBand{} den Hertog(2020)}]{Roos_den-Hertog:2020}
Roos E, den Hertog D (2020) Reducing conservatism in robust optimization.
  \emph{INFORMS Journal on Computing} 32(4):1109--1127.

\bibitem[{Roshanaei et~al.(2017)Roshanaei, Luong, Aleman, \protect\BIBand{}
  Urbach}]{Roshanaei_et_al:2017}
Roshanaei V, Luong C, Aleman DM, Urbach D (2017) Propagating logic-based
  benders’ decomposition approaches for distributed operating room
  scheduling. \emph{European Journal of Operational Research} 257(2):439--455.

\bibitem[{Samudra et~al.(2016)Samudra, Van~Riet, Demeulemeester, Cardoen,
  Vansteenkiste, \protect\BIBand{} Rademakers}]{Samudra_et_al:2016}
Samudra M, Van~Riet C, Demeulemeester E, Cardoen B, Vansteenkiste N, Rademakers
  FE (2016) Scheduling operating rooms: Achievements, challenges and pitfalls.
  \emph{Journal of Scheduling} 19(5):493--525.

\bibitem[{Scarf(1958)}]{Scarf:1958}
Scarf H (1958) A min-max solution of an inventory problem. \emph{Studies in The
  Mathematical Theory of Inventory and Production} .

\bibitem[{Schneider et~al.(2020)Schneider, van Essen, Carlier,
  \protect\BIBand{} Hans}]{Schneider_et_al:2020}
Schneider AT, van Essen JT, Carlier M, Hans EW (2020) Scheduling surgery groups
  considering multiple downstream resources. \emph{European Journal of
  Operational Research} 282(2):741--752.

\bibitem[{Shanafelt et~al.(2016)Shanafelt, Mungo, Schmitgen, Storz, Reeves,
  Hayes, Sloan, Swensen, \protect\BIBand{} Buskirk}]{Shanafelt_et_al:2016}
Shanafelt TD, Mungo M, Schmitgen J, Storz KA, Reeves D, Hayes SN, Sloan JA,
  Swensen SJ, Buskirk SJ (2016) Longitudinal study evaluating the association
  between physician burnout and changes in professional work effort. \emph{Mayo
  Clinic Proceedings}, volume~91, 422--431 (Elsevier).

\bibitem[{Shang \protect\BIBand{} You(2018)}]{Shang_You:2018}
Shang C, You F (2018) Distributionally robust optimization for planning and
  scheduling under uncertainty. \emph{Computers \& Chemical Engineering}
  110:53--68.

\bibitem[{Shapiro et~al.(2014)Shapiro, Dentcheva, \protect\BIBand{}
  Ruszczy{\'n}ski}]{Shapiro_et_al:2014}
Shapiro A, Dentcheva D, Ruszczy{\'n}ski A (2014) \emph{Lectures on Stochastic
  Programming: Modeling and Theory} (SIAM), second edition.

\bibitem[{Shehadeh(2022)}]{Shehadeh:2022}
Shehadeh KS (2022) Data-driven distributionally robust surgery planning in
  flexible operating rooms over a {W}asserstein ambiguity. \emph{Computers \&
  Operations Research} 146:105927.

\bibitem[{Shehadeh et~al.(2019)Shehadeh, Cohn, \protect\BIBand{}
  Epelman}]{Shehadeh_et_al:2019}
Shehadeh KS, Cohn AE, Epelman MA (2019) Analysis of models for the stochastic
  outpatient procedure scheduling problem. \emph{European Journal of
  Operational Research} 279(3):721--731.

\bibitem[{Shehadeh et~al.(2020)Shehadeh, Cohn, \protect\BIBand{}
  Jiang}]{Shehadeh_et_al:2020}
Shehadeh KS, Cohn AE, Jiang R (2020) A distributionally robust optimization
  approach for outpatient colonoscopy scheduling. \emph{European Journal of
  Operational Research} 283(2):549--561.

\bibitem[{Shehadeh et~al.(2021)Shehadeh, Cohn, \protect\BIBand{}
  Jiang}]{Shehadeh_et_al:2021ec}
Shehadeh KS, Cohn AE, Jiang R (2021) Using stochastic programming to solve an
  outpatient appointment scheduling problem with random service and arrival
  times. \emph{Naval Research Logistics} 68(1):89--111.

\bibitem[{Shehadeh \protect\BIBand{} Padman(2021)}]{Shehadeh_Padman:2021}
Shehadeh KS, Padman R (2021) A distributionally robust optimization approach
  for stochastic elective surgery scheduling with limited intensive care unit
  capacity. \emph{European Journal of Operational Research} 290(3):901--913.

\bibitem[{Shehadeh \protect\BIBand{} Padman(2022)}]{shehadeh2022stochastic}
Shehadeh KS, Padman R (2022) Stochastic optimization approaches for elective
  surgery scheduling with downstream capacity constraints: Models, challenges,
  and opportunities. \emph{Computers \& Operations Research} 137:105523.

\bibitem[{Sion(1958)}]{Sion:1958ec}
Sion M (1958) On general minimax theorems. \emph{Pacific Journal of
  Mathematics} 8(1):171--176.

\bibitem[{Smith \protect\BIBand{} Winkler(2006)}]{Smith_Winkler:2006}
Smith JE, Winkler RL (2006) The optimizer's curse: Skepticism and postdecision
  surprise in decision analysis. \emph{Management Science} 52(3):311--322.

\bibitem[{Stepaniak et~al.(2009)Stepaniak, Heij, Mannaerts, de~Quelerij,
  \protect\BIBand{} de~Vries}]{Stepaniak_et_al:2009}
Stepaniak PS, Heij C, Mannaerts GH, de~Quelerij M, de~Vries G (2009) Modeling
  procedure and surgical times for current procedural
  terminology-anesthesia-surgeon combinations and evaluation in terms of
  case-duration prediction and operating room efficiency: A multicenter study.
  \emph{Anesthesia \& Analgesia} 109(4):1232--1245.

\bibitem[{Sun \protect\BIBand{} Xu(2016)}]{Sun_Xu:2016ec}
Sun H, Xu H (2016) Convergence analysis for distributionally robust
  optimization and equilibrium problems. \emph{Mathematics of Operations
  Research} 41(2):377--401.

\bibitem[{Tsai et~al.(2017)Tsai, Cipri, O'Donnell, Fisher, \protect\BIBand{}
  Andritsos}]{Tsai_et_al:2017}
Tsai MH, Cipri LA, O'Donnell SE, Fisher JM, Andritsos DA (2017) Scheduling
  non-operating room anesthesia cases in endoscopy: Using the sandbox analogy.
  \emph{Journal of Clinical Anesthesia} 40:1--6.

\bibitem[{Tsai et~al.(2020)Tsai, Hall, Cardinal, Breidenstein, Abajian,
  \protect\BIBand{} Zubarik}]{Tsai_et_al:2020}
Tsai MH, Hall MA, Cardinal MS, Breidenstein MW, Abajian MJ, Zubarik RS (2020)
  Changing anesthesia block allocations improves endoscopy suite efficiency.
  \emph{Journal of Medical Systems} 44:1--9.

\bibitem[{Tsai et~al.(2021)Tsai, Yeh, \protect\BIBand{} Kuo}]{Tsai_et_al:2021}
Tsai SC, Yeh Y, Kuo CY (2021) Efficient optimization algorithms for surgical
  scheduling under uncertainty. \emph{European Journal of Operational Research}
  293(2):579--593.

\bibitem[{Tsang et~al.(2023)Tsang, Shehadeh, \protect\BIBand{}
  Curtis}]{Tsang_et_al:2023}
Tsang MY, Shehadeh KS, Curtis FE (2023) An inexact column-and-constraint
  generation method to solve two-stage robust optimization problems.
  \emph{Operations Research Letters} 51(1):92--98.

\bibitem[{Vali-Siar et~al.(2018)Vali-Siar, Gholami, \protect\BIBand{}
  Ramezanian}]{Vali-Siar_et_al:2018}
Vali-Siar MM, Gholami S, Ramezanian R (2018) Multi-period and multi-resource
  operating room scheduling under uncertainty: A case study. \emph{Computers \&
  Industrial Engineering} 126:549--568.

\bibitem[{Van~Parys et~al.(2021)Van~Parys, Mohajerin~Esfahani,
  \protect\BIBand{} Kuhn}]{Van-Parys_et_al:2021}
Van~Parys BP, Mohajerin~Esfahani P, Kuhn D (2021) From data to decisions:
  Distributionally robust optimization is optimal. \emph{Management Science}
  67(6):3387--3402.

\bibitem[{Vo-Thanh et~al.(2018)Vo-Thanh, Jans, Schoen, \protect\BIBand{}
  Goos}]{Vo-Thanh_et_al:2018}
Vo-Thanh N, Jans R, Schoen ED, Goos P (2018) Symmetry breaking in mixed integer
  linear programming formulations for blocking two-level orthogonal
  experimental designs. \emph{Computers \& Operations Research} 97:96--110.

\bibitem[{Wang et~al.(2020)Wang, Chen, \protect\BIBand{} Liu}]{Wang_et_al:2020}
Wang S, Chen Z, Liu T (2020) Distributionally robust hub location.
  \emph{Transportation Science} 54(5):1189--1210.

\bibitem[{Wang et~al.(2014)Wang, Tang, \protect\BIBand{}
  Fung}]{Wang_et_al:2014}
Wang Y, Tang J, Fung RY (2014) A column-generation-based heuristic algorithm
  for solving operating theater planning problem under stochastic demand and
  surgery cancellation risk. \emph{International Journal of Production
  Economics} 158:28--36.

\bibitem[{Wang et~al.(2019)Wang, Zhang, \protect\BIBand{}
  Tang}]{Wang_et_al:2019}
Wang Y, Zhang Y, Tang J (2019) A distributionally robust optimization approach
  for surgery block allocation. \emph{European Journal of Operational Research}
  273(2):740--753.

\bibitem[{Wang et~al.(2023)Wang, Zhang, Zhou, \protect\BIBand{}
  Tang}]{Wang_et_al:2023}
Wang Y, Zhang Y, Zhou M, Tang J (2023) Feature-driven robust surgery
  scheduling. \emph{Production and Operations Management} .

\bibitem[{Zeng \protect\BIBand{} Zhao(2013)}]{Zeng_Zhao:2013}
Zeng B, Zhao L (2013) Solving two-stage robust optimization problems using a
  column-and-constraint generation method. \emph{Operations Research Letters}
  41(5):457--461.

\bibitem[{Zhu et~al.(2019)Zhu, Fan, Yang, Pei, \protect\BIBand{}
  Pardalos}]{Zhu_et_al:2019}
Zhu S, Fan W, Yang S, Pei J, Pardalos PM (2019) Operating room planning and
  surgical case scheduling: A review of literature. \emph{Journal of
  Combinatorial Optimization} 37(3):757--805.

\end{thebibliography}

\end{document}